\documentclass{amsart}
\usepackage{amsmath,amssymb,amsthm}
\usepackage[dvips]{graphicx}
\usepackage{xcolor}
\usepackage{color}              
\usepackage{epsfig}
\usepackage{colonequals}
\usepackage{enumitem}
\usepackage{units}
\usepackage{scalerel}
\usepackage{mathabx}
\usepackage{mathtools}
\usepackage{bbm}
\usepackage{url}
\usepackage{color}
\usepackage[colorlinks,citecolor=blue,linkcolor=red]{hyperref}

\newtheorem{theorem}{Theorem}[section]
\newtheorem{proposition}[theorem]{Proposition}
\newtheorem{corollary}[theorem]{Corollary}
\newtheorem{lemma}[theorem]{Lemma}

\newtheorem{observation}[theorem]{Observation}
\newtheorem{claim}[theorem]{Claim}

\newtheorem*{strong-align-lem}{Lemma~\ref{lem:building_strong_alignment}}
\newtheorem*{gen-upper-bound}{Theorem~\ref{thm:general_upper_bound}}

\theoremstyle{definition}
\newtheorem{definition}[theorem]{Definition}
\newtheorem{def-thm}[theorem]{Definition/Theorem}
\newtheorem{notation}[theorem]{Notation}
\newtheorem{terminology}[theorem]{Terminology}
\newtheorem{example}[theorem]{Example}

\newtheorem{convention}[theorem]{Convention}
\newtheorem{strategy}[theorem]{Strategy}
\newtheorem{question}[theorem]{Question}

\theoremstyle{remark}
\newtheorem{remark}[theorem]{Remark}

\makeatletter
\let\c@equation\c@theorem
\makeatother
\numberwithin{equation}{section}

\newcommand{\calC}{{\mathcal C}}

\newcommand{\T}{{\mathcal T}}

\newcommand{\C}{{\mathbb C}}

\renewcommand{\H}{{\mathbb H}}
\newcommand{\R}{{\mathbb R}}
\newcommand{\Z}{{\mathbb Z}}
\newcommand{\N}{{\mathbb N}}

\renewcommand{\d}{{\rm d}}

\newcommand{\abs}[1]{\left\vert#1\right\vert}

\newcommand{\floor}[1]{\left\lfloor#1\right\rfloor}
\newcommand{\inv}{^{-1}}

\makeatletter
\newcommand{\oset}[3][0ex]{%
  \mathrel{\mathop{#3}\limits^{
    \vbox to#1{\kern-2\ex@
    \hbox{$\scriptstyle#2$}\vss}}}}
\makeatother

\newcommand{\eadd}{\oset{+}{\asymp}}
\newcommand{\gadd}{\oset{+}{\succ}}
\newcommand{\ladd}{\oset{+}{\prec}}

\newcommand{\clld}[1]{\curlyeqprec^{#1}}

\newcommand{\define}[1]{\emph{#1}}

\newcommand{\param}{{\mathchoice{\mkern1mu\mbox{\raise2.2pt\hbox{$
\centerdot$}}
\mkern1mu}{\mkern1mu\mbox{\raise2.2pt\hbox{$\centerdot$}}\mkern1mu}{
\mkern1.5mu\centerdot\mkern1.5mu}{\mkern1.5mu\centerdot\mkern1.5mu}}}

\DeclareMathOperator{\Mod}{Mod}

\DeclareMathOperator{\diam}{diam}

\DeclareMathOperator{\base}{base} 
\newcommand{\cc}{\calC} 
\newcommand{\thecircle}{\mathbb{S}^1} 

\newcommand{\thin}{{\epsilon_0}} 
\newcommand{\thinprime}{\epsilon'_0} 
\newcommand{\bers}{L_0} 
\newcommand{\Mactive}{{\sf M}} 
\newcommand{\rafi}{\Mactive} 
\newcommand{\interval}{\mathcal{I}} 
\newcommand{\A}{\mathcal{A}} 
\newcommand{\adderr}{{\sf D}} 
\newcommand{\minsky}{\adderr_0} 
\newcommand{\holmeas}{\mathbf{m}} 

\newcommand{\ball}[2]{\mathrm{Ball}(#1,#2)} 

\newcommand{\netsep}{\sf c} 
\newcommand{\tnet}{\mathcal{N}} 
\newcommand{\bt}{R} 

\newcommand{\bad}{\mathcal{B}} 

\newcommand{\productregion}[2]{\mathcal{P}(#1\vert#2)} 
\newcommand{\productmap}[1]{\Phi_{#1}} 
\newcommand{\back}{{\sf B}} 
\newcommand{\bgit}{{\sf Q}} 
\newcommand{\nestprojscommute}{{\sf k}} 
\newcommand{\consist}{{\sf K}} 
\newcommand{\csistfun}{\mathfrak{C}} 
\newcommand{\lipconst}{{\sf L}} 
\newcommand{\branchal}{\mathfrak{B}} 
\newcommand{\res}[3]{\widehat{#1}^{#2}_{#3}} 
\newcommand{\enc}{\mathcal{E}} 
\newcommand{\lenc}{\mathcal{E}^\ell}
\newcommand{\renc}{\mathcal{E}^r}
\newcommand{\bigd}[1]{\mathcal{F}(#1)} 
\newcommand{\maxbigd}[1]{\underline{\mathcal{F}}(#1)} 

\newcommand{\horoball}{\T^\le}

\newcommand{\horoballnet}{\mathcal{N}^\le}
\newcommand{\netcount}{{\sf P}}
\newcommand{\twist}[1]{D_{#1}} 
\newcommand{\curves}[1]{\Gamma(#1)} 
\newcommand{\bcurves}[1]{\overline{\Gamma}(#1)} 
\newcommand{\markings}[1]{\mathcal{M}_0(#1)} 
\newcommand{\plex}[1]{\xi({#1})} 
\newcommand{\ent}[1]{h_{#1}} 
\newcommand{\entstar}[1]{h_{#1}^\ast} 
\newcommand{\badfunc}[1]{\beta_{#1}}
\newcommand{\savings}[1]{\sigma_{#1}}
\newcommand{\adjust}[1]{\xi_{#1}}
\newcommand{\weight}[1]{\omega_{#1}}
\newcommand{\charfunc}{\mathbbm{1}}
\newcommand{\pfback}{{\sf C}} 
\newcommand{\indrafi}[1]{\irafi_{{#1}}} 
\newcommand{\irafi}{{\sf N}} 
\newcommand{\indfrac}[1]{\eta_{#1}} 
\newcommand{\lset}[2]{\mathcal{L}_{#1}({#2})}
\newcommand{\rset}[2]{\mathcal{R}_{#1}({#2})}
\newcommand{\maxlset}[2]{\underline{\mathcal{L}}_{#1}({#2})}
\newcommand{\maxrset}[2]{\underline{\mathcal{R}}_{#1}({#2})}

\newcommand{\wfal}{\pfback}
\newcommand{\colsup}[2]{\bar{#1}^{#2}}

\newcommand{\leftpad}[1]{\mathcal{L}_0(#1)}
\newcommand{\rightpad}[1]{\mathcal{R}_0(#1)}
\newcommand{\maxleftpad}[1]{\underline{\mathcal{L}}_0(#1)}
\newcommand{\maxrightpad}[1]{\underline{\mathcal{R}}_0(#1)}

\newcommand{\cl}{\mathfrak{L}}
\newcommand{\contr}{\mathcal{A}}
\newcommand{\indsw}[1]{\swarrow\!\!\!\!{}_{#1}}
\newcommand{\indse}[1]{{}_{#1}\!\!\!\!\searrow}

\makeatletter
\newcommand{\superimpose}[2]{%
  {\ooalign{$#1\@firstoftwo#2$\cr\hfil$#1\@secondoftwo#2$\hfil\cr}}}
\makeatother
\newcommand{\cut}{\pitchfork} 

\newcommand\essunion{\mathrel{\mathchoice
{\cup\mkern-9.45mu\raise1.4pt\hbox{$\scriptstyle\circ$}\mkern2mu}
{\cup\mkern-9.45mu\raise1.4pt\hbox{$\scriptstyle\circ$}\mkern2mu}
{\cup\mkern-10mu\raise1.4pt\hbox{$\scriptscriptstyle\circ$}\mkern2mu}
{\cup\mkern-10mu\raise.8pt\hbox{$\scriptscriptstyle\circ$}\mkern2mu}}}
\newcommand\essint{\mathrel{\mathchoice
{\cap\mkern-9.3mu\raise0.8pt\hbox{$\scriptstyle\circ$}\mkern2mu}
{\cap\mkern-9.3mu\raise0.8pt\hbox{$\scriptstyle\circ$}\mkern2mu}
{\cap\mkern-10mu\raise0.8pt\hbox{$\scriptscriptstyle\circ$}\mkern2mu}
{\cap\mkern-10.3mu\raise0.2pt\hbox{$\scriptscriptstyle\circ$}\mkern2mu}}}
\newcommand\esssubset{\mathrel{\mathchoice
{\subset\mkern-14mu\raise0.7pt\hbox{$\scriptstyle\circ$}\mkern2mu}
{\subset\mkern-14mu\raise0.7pt\hbox{$\scriptstyle\circ$}\mkern2mu}
{\subset\mkern-10mu\raise0.35pt\hbox{$\scriptscriptstyle\circ$}\mkern2mu}
{\subset\mkern-11mu\hbox{$\scriptscriptstyle\circ$}\mkern2mu}}}

\newcommand{\ecap}{\sqcap}
\newcommand{\ecup}{\sqcup}
\newcommand{\esubset}{\sqsubset}
\newcommand{\esupset}{\sqsupset}

\newcommand{\esubsetneq}{\sqsubsetneq}
\newcommand{\esupsetneq}{\sqsupsetneq}
\newcommand{\esub}{\esubset}
\newcommand{\esup}{\esupset}
\newcommand{\esubneq}{\esubsetneq}
\newcommand{\esubn}{\esubneq}
\newcommand{\esupneq}{\esupsetneq}
\newcommand{\esupn}{\esupneq}
\newcommand{\disjoint}{\perp} 
\newcommand{\disj}{\disjoint}

\newcommand{\intnum}{i}

\newcommand{\tol}{<\mathrel{\mkern-6mu}\mathrel{\cdot}\mkern1mu}
\newcommand{\tog}{>\mathrel{\mkern-13mu}\mathrel{\cdot}\mkern8mu}

\title[Counting mapping classes by Nielsen--Thurston type]{Counting mapping classes by \\ Nielsen--Thurston type}
\date{\today}
\author[Dowdall]{Spencer Dowdall}
	\address{Department of Mathematics, Vanderbilt University, Nashville, TN 37240}
	\email{spencer.dowdall@vanderbilt.edu}

\author[Masur]{Howard Masur}
	\address{Department of Mathematics, University of Chicago, Chicago, IL 60637}
        \email{masurhoward@gmail.com}

\begin{document}
\maketitle

{\centering\footnotesize 
We dedicate this paper to the memory of Maryam Mirzakhani.\par}

\begin{abstract}
This paper concerns the lattice counting problem for the mapping class group of a surface $S$ acting on Teichm\"uller space with the Teichm\"uller metric. In that problem the goal is to count the number of mapping classes that send a given point $x$ into the ball of radius $R$ centered about another point $y$. For the action of the entire group, Athreya, Bufetov, Eskin and Mirzakhani have shown this quantity is asymptotic to $e^{hR}$, where $h$ is the dimension of the Teichm\"uller space. We refine the problem by considering the action various distinguished subsets of elements and counting these separately. For the set of finite-order elements, we show the associated count grows coarsely at the rate of $e^{hR/2}$, that is, with half the exponent. For the reducible elements, the associated count grows coarsely at the rate of $e^{(h-1)R}$. Finally, for the set of all multitwists, the coarse growth rate is also $e^{hR/2}$.
To obtain these quantitative estimates, we introduce a new notion in Teichm\"uller geometry, called complexity length, which  reflects some aspects of the negative curvature of curve complexes and also has applications to counting problems.
\end{abstract}

\setcounter{tocdepth}{1}
\tableofcontents

\section{Introduction}

\subsection{Lattice point counting}

The  goal of this paper is to count  elements of the mapping class group via its action  on Teichm\"uller space.  
When a group $G$ acts on a metric space $X$ by isometries, counting the number of orbit or ``lattice'' points in metric balls of increasing radius gives a measure of growth in the group 
as reflected in the geometry of $X$. 
For example, the number of lattice points of $\Z^n$ in a large metric ball in Euclidean space $\R^n$  is approximately the volume of the ball.
Relatedly, in his Ph.D.\ dissertation, Margulis \cite{Marguli-thesis} considered the case of a compact negatively curved Riemannian manifold $M$ and showed that for the isometric action of the fundamental group $\pi_1(M)$ on the universal cover $\widetilde{M}$, the number of lattice points in a ball of radius $R$ is asymptotic to a constant times $e^{hR}$, where $h > 0$ is the topological entropy of the geodesic flow.

This paper concerns a refinement of this classical lattice point counting problem in the setting of Teichm\"uller geometry. Fix a connected, orientable surface $S$ of genus $g$ with $p$ punctures whose \emph{complexity} $\plex{S} = 3g-3+p$  is positive. We consider the mapping class group $\Mod(S) = \mathrm{Homeo}^+(S)/\mathrm{Homeo}_0(S)$  of isotopy classes of orientation-preserving homeomorphisms of $S$. 
This group acts isometrically and  properly discontinuously on the Teichm\"uller space $\T(S)$ of marked hyperbolic metrics on $S$ equipped with the Teichm\"uller metric $d_{\T(S)}$.

The geometric and dynamical theory of Teichm\"uller space bears striking parallels to that of negatively curved Riemannian manifolds. Motivated by this, Athreya, Bufetov, Eskin, and Mirzakhani \cite{ABEM} drew on ideas from Margulis \cite{Marguli-thesis} to solve the analogous lattice point counting problem. For $x,y\in \T(S)$, let us write $s_x$ for the finite cardinality of the stabilizer of $x$ in $\Mod(S)$ and 
\[\Lambda(x,y,R) = \big\{\phi\in \Mod(S) \mid d_{\T(S)}(\phi(x), y)\le R\big\}\]
for the set of mapping classes that translate $x$ to within distance $R$ of $y$. 
The cardinality $\abs{\Lambda(x,y,R)}$ then equals $s_x$ times the number of orbit points  $\Mod(S)\cdot x$ in the ball of radius $R$ centered at $y$.
Their result may then be stated as:

\begin{theorem}[Athreya--Bufetov--Eskin--Mirzakhani \cite{ABEM}]
\label{thm:ABEM}
There is a constant $\lambda > 0$ such that for all $x,y\in \T(S)$ one has 
$\abs{\Lambda(x,y,R)} \sim \lambda s_x e^{\ent{S}R}$, where
$\ent{S} = 2\plex{S} = 6g-6+2p$ is the entropy of the Teichm\"uller geodesic flow,
and the notation $f(R)\sim g(R)$ means that ${f(R)}/{g(R)}\to 1$ as $R\to\infty$.
\end{theorem}

While this completely answers the lattice point counting problem for the full mapping class group, one might further refine it by considering the growth of certain naturally distinguished subgroups or subsets. 
This is related to the question of determining what a ``typical'' element of $\Mod(S)$ looks like. 

The famous Nielsen--Thurston classification \cite{Thurston-classification} states that every element of the mapping class group is either finite-order, reducible, or pseudo-Anosov (see Definition~\ref{def:NT-classification}).
Accordingly, we let
\[\Lambda_{\mathrm{fo}}(x,y,R),\qquad \Lambda_{\mathrm{red}}(x,y,R),\qquad\text{and}\qquad \Lambda_{\mathrm{pA}}(x,y,R)\] 
denote the subsets of $\Lambda(x,y,R)$ consisting of finite-order, reducible, and pseudo-Anosov elements, respectively.
Building on \cite{ABEM}, Maher \cite{Maher-lattices,Maher-randomwalks} has used ideas from random walks to show that typical mapping classes are pseudo-Anosov in the sense that $\abs{\Lambda_{\mathrm{pA}}(x,y,R)}\sim \abs{\Lambda(x,y,R)}$. 
In particular, this shows that the proportion of finite-order and reducible lattice points tends to zero as $R\to \infty$, but it does not give any indication of the {rate} of convergence.

The purpose of this paper is to give quantitative estimates for the number of finite-order and reducible mapping classes by counting their respective lattice points.  We show:

\begin{theorem}
\label{thm:main}
Let $S = S_g$ be a closed surface of genus $g\ge 2$. 
For any $\delta>0$ and pair of points $x,y\in \T(S)$, there are constants $K_1,K_2,R_0$ such that for  all $R\geq R_0$ the following inequalities hold:
\begin{align*}
K_1\exp\left(\tfrac{\ent{S}}{2}R\right)\leq & |\Lambda_{\mathrm{fo}}(x,y,R)|\leq K_2 \exp\left((\tfrac{\ent{S}}{2} +\delta)R\right)
\\
K_1\exp\big((\ent{S}-1)R\big)\leq & |\Lambda_{\mathrm{red}}(x,y,R)|\leq K_2\exp\big((\ent{S}-1+\delta )R\big)
\end{align*}
\end{theorem}

We remark that the exponents for the upper and lower bounds do not quite coincide because of the $\delta$ in the exponent for the upper bounds; we do not know if the $\delta$ can be removed. 
Nonetheless our main theorem shows that the finite-order and reducible elements each have exponential growth with respective exponents that are essentially one half of and one less than that of entire group.

\begin{remark}
\label{rem:lower-bound-only-closed}
In fact, the lower bound in the reducible case and the two upper bounds in Theorem~\ref{thm:main} all hold for an arbitrary (possibly punctured) surface with $\plex{S} > 0$; see Theorems~\ref{thm:finite-order-lower-bound}, \ref{thm:reducible-lower-bound} and \ref{thm:general_upper_bound} below. In the finite-order case, our lower bound is given by constructing, on each closed surface, an explicit finite-order element with finite centralizer; see Example~\ref{example:finite-order}. The construction can easily be generalized to a few other cases, such as genus $g$ surfaces with 1, 2, or $2g+1$ punctures. It would be interesting to generalize the construction, and thereby the lower bound in Theorem~\ref{thm:main}, to arbitrary surfaces $S$ with $\plex{S} > 0$.
\end{remark}

Though our original motivation was to count finite-order and reducible elements, our techniques apply to some other natural subsets as well. Recall that each element $\phi\in \Mod(S)$ has a power that can be expressed a product of Dehn twists and partial pseudo-Anosovs with disjoint supports (see Definition~\ref{def:nielsen-thurston-decomp}). The union of the supports of these partial pseudo-Anosovs gives a canonical subsurface $S_p^\phi$ that we call the \define{pseudo-Anosov support} of $\phi$; accordingly we set $\ent{\phi} = \ent{S_\phi^p}$. With this setup, we obtain the following general upper bound:

\begin{gen-upper-bound}
Fix $0\le m \le \ent{S}$ and let $\Gamma\subset \Mod(S)$ be any subset of elements $\phi$ for which $\ent{\phi}\le m$. Then for any $\delta > 0$ and $x,y\in \T(S)$ the following holds for all large $R$:
\[\abs{\big\{\phi\in \Gamma \mid d_{\T(S)}(\phi(x),y)\le R\big\}} 
\le \exp\left(\left(\tfrac{\ent{S}+m}{2} + \delta\right)R\right)\]
\end{gen-upper-bound}
Note that every finite-order element has $\ent{\phi} = 0$ and every reducible element has $\ent{\phi} \le \ent{S}-2$; thus the upper bounds in Theorem~\ref{thm:main} formally follow from Theorem~\ref{thm:general_upper_bound}. 
The lower bounds follow from Theorems~\ref{thm:finite-order-lower-bound} and \ref{thm:reducible-lower-bound} in \S\ref{Sec:lowerbound}.

Every \define{multitwist} $\phi$ (that is, product of Dehn twists about disjoint curves) also has empty pseudo-Anosov part and thus $\ent{\phi} = 0$. Hence if $\Lambda_{\mathrm{mt}}(x,y,R)$ denotes the subset of $\Lambda(x,y,R)$ consisting multitwists, then Theorem~\ref{thm:general_upper_bound} implies $\abs{\Lambda_{\mathrm{mt}}(x,y,R)} \le \exp((\tfrac{\ent{S}}{2}+\delta)R)$. The growth rate of the set of multitwists was also considered by Han in \cite[Theorem 1.8]{han-growthDehnTwists}, where it was shown that $\abs{\Lambda_{\mathrm{mt}}(x,y,R)}$ is bounded below  by $R \exp(\tfrac{\ent{S}}{2}R)$. Combining these, we obtain the following corollary:

\begin{corollary}
[of Theorem~\ref{thm:general_upper_bound} and {\cite[Theorem 1.8]{han-growthDehnTwists}}]
For any $\delta > 0$ and $x,y\in \T(S)$ the following bounds hold for all large $R$: 
\[R \exp\left(\tfrac{\ent{S}}{2}R\right) \le \abs{\Lambda_{\mathrm{mt}}(x,y,R)} \le \exp\left((\tfrac{\ent{S}}{2}+\delta)R\right)\]
\end{corollary}
This shows that the $\delta$ in Theorem~\ref{thm:general_upper_bound} cannot be removed and that there is no general upper bound of the form $K \exp\left(\tfrac{\ent{S}+m}{2}R\right)$. Since the upper bounds in Theorem~\ref{thm:main} are obtained from Theorem~\ref{thm:general_upper_bound}, this suggests that the $\delta$ in Theorem~\ref{thm:main} may in fact be necessary.

\subsection{Heuristics and hazards}
\label{sec:heuristics}
Let us focus on the finite-order case and 
describe a naive picture illustrating why one might expect the finite-order elements to grow at half the exponential rate of the whole group. 
The first observation is that there are only finitely many conjugacy classes of finite-order elements \cite[Theorem 7.13]{FM-Primer}; thus it suffices to count each conjugacy class $[\phi_0]$ separately.
Since 
the points $x,y$ may be adjusted at the cost of increasing the constants $K_1,K_2$, we might as well assume $x = y$ is a fixed point $x_0$ for a given finite-order element $\phi_0$. 

The result of \cite{ABEM} (Theorem~\ref{thm:ABEM}) says there are approximately $\exp({\ent{S}}\tfrac{R}{2})$ mapping classes $f\in \Mod(S)$ so that $d_{\T(S)}(x_0,f(x_0))\le R/2$. Further, each of these produces a conjugate $\phi_f = f\phi_0f\inv$ for which the translate $f(x_0)$ is fixed. The triangle inequality thus implies this finite-order element satisfies $d_{\T(S)}(x_0, \phi_f(x_0))\le R$.

This observation suffices for the lower bound in Theorem~\ref{thm:main}, provided the assignment $f\mapsto \phi_f$ is uniformly finite-to-one, as is the case when $\phi_0$ has finite centralizer. This argument is carried out in detail in \S\ref{Sec:lowerbound}, where we prove the lower bound by constructing explicit examples.

For the upper bound, a hope might be that \emph{all} (or at least most) finite-order elements $\phi$ satisfying $d_{\T(S)}(x,\phi(x))\le R$ have a fixed point withnin distance $R/2$ and that we can then count these closest fixed  points $x_\phi$ and hence count the $\phi$.

This, however, may not be possible:  In \S\ref{Sec:example} we provide an example of a finite-order $\phi_0$ with fixed point $x_0$ which has the property  that for every conjugate $\phi\in [\phi_0]$ with $d_{\T(S)}(x_0,\phi(x_0))\leq R$, the closest fixed point $x_\phi$ satisfies the dual properties that: 1) up to additive error we have $d_{\T(S)}(x_0,x_\phi)\ge R$, and 2) the geodesic from $x_0$ to $x_\phi$ passes through the Teichm\"uller space $\T(V)$ of some subsurface $V$ in such a way that the geodesic from $x_\phi$ on to $\phi(x_0)$ ``backtracks'' through the same Teichm\"uller space $\T(V)$, undoing the progress made in going from $x_0$ to $x_\phi$.

The reasons why such backtracking is problematic are perhaps too technical to elaborate upon in this introduction. Suffice it to say that the theory of subsurface projections developed in \cite{MasurMinsky-heirarchy,Rafi-combinatorial} shows that  Teichm\"uller geodesics are governed by how they move through ``thin regions'' where the boundary curves $\partial V$ of subsurfaces $V$ become short. These regions behave like metric products---in which the Teichm\"uller space $\T(V)$ of the surface is one factor \cite{Minsky-product}---and are the main source of non-negative curvature and many headaches in $\T(S)$.
 
The two issues illustrated by the example---that the closest fixed point may be too far away and that the piecewise geodesic path from $x_0$ to the fixed point and on to $\phi(x_0)$ may backtrack in subsurfaces---are the main obstacles
in the finite-order case. The same issues arise in the reducible case, albeit in different guises.
The thrust of our argument is that in spite of such examples, the hope expressed above does hold in some moral sense, 
albeit in a rather complicated way involving an alternative understanding of distance.

Our proof  divides into two separate parts overcoming these issues. 
The main argument (\S\ref{sec:finding-good-points}--\ref{Sec:Bounding the multiplicity of branch points}) constructs a good (nearly fixed) point $x_\phi$ that is tailored to $\phi$, along with a pair of branch points $a_\phi,b_\phi$ that mitigate the backtracking.
For each possible such pair $(a,b)$ we count the number of maps $\phi$ determining that pair.
This is preceded by a much more elaborate part (\S\ref{sec:witness-families}--\ref{Sec:countingcomplexity}) ultimately showing that while the Teichm\"uller distance $d_{\T(S)}(x_0,x_\phi)$ may be much larger than $R/2$, there is a more apt measure of length that \emph{is} on the order of $R/2$, provided backtracking is controlled. 
This allows us to count the number of possible branch points $(a,b)$ and, combined with the main argument, ultimately count the number of maps $\phi$.
Developing the theory of this length is a major component of the paper, 
 which introduces new ideas and techniques to Teichm\"uller theory that we hope may be of independent interest and lead to other applications. We next describe the key ideas of this measurement of length.

\subsection{A new complexity length for Teichm\"uller space}
\label{sec:intro_complexity_length}
The impetus for our construction is the need to count points in a way that 
incorporates
how geodesics move through Teichm\"uller spaces of subsurfaces.

For the purposes of counting, it is helpful to discretize $\T(S)$ by considering a \emph{net} $\tnet(S)$; this is a $\netsep$--separated subset whose $2\netsep$ balls cover $\T(S)$, for some constant $\netsep$ (see \S\ref{sec:volumes_and_nets}). Eskin and Mirzakhani introduced nets in \cite {EM-countingClosedGeods} and showed there is a uniform constant $C_0$ such that a ball of radius $R$ about any thick point $x\in \T(S)$ contains at most $C_0e^{\ent{S}R}$ net points. When the thickness condition is removed and arbitrary centers are considered, they show that for any $\delta > 0$ there is some $C_\delta$ such that all balls of radius $R$ contain at most $C_\delta e^{(\ent{S}+\delta)R}$ net points. This is one explanation of where the $\delta$ comes from in our main theorem. 

The key observation is as follows: If one fixes a thick center point $x$ and moves distance $R$ to a net point $y$ by \emph{only} moving in the Teichm\"uller space $\T(V)$ of a subsurface and not moving in the complement, then Minsky's product regions theorem \cite{Minsky-product} says this  behaves like a Teichm\"uller geodesic in $\T(V)$ and hence \cite{EM-countingClosedGeods} implies there should only be $C_0 e^{\ent{V}R}$ such net points $y$. That is, imposing a restriction that the geodesic passes through Teichm\"uller spaces of subsurfaces cuts down on the number of net points that can be reached in distance $R$.

Our \emph{complexity length} $\cl(x,y)$ (Definition~\ref{def:complexity-savings_for_tuple_of_points}) 
is designed to implement this observation in a rigorous way that accounts for the fact that geodesics can move through disjoint subsurfaces simultaneously. 
Although the construction is complicated, the rough idea is to take all the subsurfaces $Z$ for which the curve complex projection (\S\ref{sec:projection_to_curve_complexes}) $d_Z(x,y)$ is large (as determined by a parameter $\wfal$) and partition them by picking out a distinguished subfamily $\Omega$, called a \emph{witness family} (see \S\ref{sec:witness-families}), with the property
that each such $Z$ is minimally contained in a unique element $V$ of $\Omega$. This family comes with additional combinatorial structure (a \emph{subordering}; Definition~\ref{def:subordering}) that allows one to take the curve complex data from all these subsurfaces $Z$ contributing to $V$ and reassemble it, via the concept of \emph{consistency} from \cite{BKMM-rigidity}, into a pair of points $\res{x}{\Omega}{V},\res{y}{\Omega}{V}\in \T(V)$ in the Teichm\"uller space of $V$ with the property that $d_Z(\res{x}{\Omega}{V},\res{y}{\Omega}{V})$ and $d_Z(x,y)$ coarsely agree (up to only additive error) for each $Z$. The complexity length is then defined as the weighted sum
\[\cl(x,y) = \sum_{V\in \Omega} \ent{V}d_{\T(V)}(\res{x}{\Omega}{V},\res{y}{\Omega}{V}).\]
Since the distances are weighted by the exponents $\ent{V}$ used for counting, and since the pairs $\res{x}{\Omega}{V},\res{y}{\Omega}{V}$ encode the data of the original points, 
we are able to prove the following analog of Eskin and Mirzakhani's \cite{EM-countingClosedGeods} net point counting result:

\begin{theorem}[c.f.\ Theorem~\ref{thm:countcomplexity}]
\label{thm:intro_complexity_count}
For any sufficiently large parameter $\wfal$,
there exists $k\in \N$ such that each $x\in \T(S)$ has at most $k r^k e^r$ net points $y$ within complexity length $r>0$. That is:
$\#\big\{y\in \tnet(\Sigma) \mid \cl(x,y) \le r\big\} \le k r^k e^{r}$.
\end{theorem}

As explained in Remark~\ref{rem:complexity_features}, our complexity length is not dissimilar to Rafi's \cite{Rafi-combinatorial} Distance Formula (Theorem~\ref{thm:distance_formula}), which roughly says the Teichm\"uller distance $d_{\T(S)}(x,y)$ is comparable, with multiplicative and additive error, to the sum of all large curve complex projections $d_Z(x,y)$. 
In fact, one finds that $\cl(x,y)$ and $d_{\T(S)}(x,y)$ also agree up to bounded multiplicative and additive error, since they both coarsely agree with the sum in the distance formula! However, there are two key differences between these perspectives:

The first is that the distance formula concerns \emph{all} subsurfaces with large projection. 
The sum may therefore have arbitrarily many terms, and this ultimately contributes to a multiplicative error.  
  But we
cannot afford multiplicative error, since the distance $R$ appears in the exponent in our main theorem and the whole point is to calculate the exponent. Throughout the construction we must therefore be careful to utilize witness families $\Omega$ of uniformly bounded cardinality, so that our sum has boundedly many terms and the various additive errors
do not accumulate into an uncontrolled multiplicative error. By arranging things with great care, we are able to relate complexity length to Teichm\"uller distance with \emph{only additive error} and an \emph{explicit} multiplicative factor of $\ent{S}$ (Theorem~\ref{thm:intro_complexity_aligned} below).

The second is that we sum over Teichm\"uller, rather than curve complex, distances, which facilitates the above application to counting. Nevertheless,
we are  still able to tap into the hyperbolic geometry of curve complexes
 in the following sense. Let
us say a triple $(a,b,c)$ in $\T(S)$ is $\theta$--\emph{aligned} if for every subsurface $Z$ the three pairwise curve complex projections satisfy the reverse triangle inequality:
\[d_Z(a,b) + d_Z(b,c) \le d_Z(a,c)+\theta.\]
Since the curve complex $\cc(Z)$ is hyperbolic, this is equivalent to saying the projection of $b$ to $\cc(Z)$ lies near the geodesic joining the projections of $a$ and $c$. 
Because of the multiplicative error and arbitrary length of the sum, knowing this inequality in each term of the distance formula does not translate into a reverse triangle inequality for Teichm\"uller distance: There are $\theta$--aligned triples $(a,b,c)$ for which $d_{\T(S)}(a,b) + d_{\T(S)}(b,c) - d_{\T(S)}(a,c)$ is arbitrarily large. Complexity length, however, does satisfy such a reverse triangle inequality. 
This requires the triple $(a,b,c)$, or more generally tuple $(x_0,\dots,x_n)$, to be \emph{strongly} $\theta$--aligned (Definition~\ref{def:strong_alignment}), which adds a condition on the lengths of curves at $b$ so that its projection to the Teichm\"uller space $\T(A)$ of each annulus $A$ 
lies near the geodesic joining the projections of $a$ and $c$. 
We then have the following key result:

\begin{theorem}[c.f.\ 
Theorem~\ref{thm:main_complexity_length_bound}]
\label{thm:intro_complexity_aligned}
For any $n\ge 1$ and sufficiently large parameter $\wfal$, there exists $K$ such every strongly $\wfal$--aligned $n$-tuple $(x_0,\dots,x_n)$ in $\T(S)$ satisfies
\[\cl(x_0,x_1) + \dots +\cl(x_{n-1},x_n) \le \left(\ent{S} + \frac{n}{\wfal}\right) d_{\T(S)}(x_0,x_n) + K.\]
In particular $\cl(x,y) \le K + (\ent{S}+\frac{1}{\wfal})d_{\T(S)}(x,y)$ for any $x,y\in \T(S)$.
\end{theorem}

\subsection{Summary of proof of upper bound}
\label{sec:summary_of_proof}

Theorems~\ref{thm:intro_complexity_count} and \ref{thm:intro_complexity_aligned}  are the key features that enable complexity length to overcome 
the first main issue described in \S\ref{sec:heuristics}, namely of the closest fixed point being too far away. 
Roughly, the finite-order case goes as follows:
if we were able to find a fixed point $x_\phi$ for $\phi\in [\phi_0]$ so that the triple $(x_0, x_\phi,\phi(x_0))$ were strongly aligned, then Theorem~\ref{thm:intro_complexity_aligned} would imply 
\[\cl(x_0,x_\phi) + \cl(x_\phi, \phi(x_0))\le (\ent{S}+ \tfrac{2}{\wfal}) d_{\T(S)}(x_0,\phi(x_0))+K\le (\ent{S}+\delta)R+K,\]
provided $\wfal$ is chosen sufficiently large. By symmetry, the two terms on the left are equal and thus each at most $((\ent{S}+\delta)R + K)/2$. Theorem~\ref{thm:intro_complexity_count} thus implies that for large $R$  there are at most $\exp((\ent{S}+2\delta)\tfrac{R}{2})$ such fixed points $x_\phi$. 
Since the multiplicity of the assignment $\phi\mapsto x_\phi$ is bounded by the uniform finiteness of the stabilizer of $x_\phi$, this would give the desired upper bound on $\abs{\Lambda_{\mathrm{fo}}(x_0,x_0,R)}$.

Complexity length thus enables the heuristic argument of \S\ref{sec:heuristics} to work regardless of 
the distance to the closest fixed point $x_\phi$,
provided $(x_0,x_\phi, \phi(x_0))$ is strongly aligned. 
In the reducible case the requirement would instead roughly be that there is a point $x_\phi$ so that  $(x_0,x_\phi, \phi(x_\phi),\phi(x_0))$ is strongly aligned. 

While this strong alignment condition need not hold, we circumvent it by a pair of constructions.  First we construct a good point $x_\phi$ (Definition~\ref{def:the_good poitns}) that minimizes backtracking; in the finite-order case $x_\phi$ is a fixed point and in the reducible case it is a point where the reducing curves of $\phi$ are short.  Then by utilizing a sort of barycenter construction for the triples $(x_0, x_\phi, \phi(x_0))$ and $(x_0, x_\phi,\phi(x_0))$, we construct a pair $a_\phi,b_\phi$ of points so that the tuple $(x_0, a_\phi,b_\phi,\phi(x_0))$ is strongly aligned and satisfies a number of other technical properties (Proposition~\ref{prop:good-point-features}).

We then count the number of such pairs $a_\phi,b_\phi$ by the complexity argument described above. Finally, we carry out a reconstruction argument (Theorem~\ref{thm:bound_multiplicity}) to show any pair  $a_\phi,b_\phi$ arises for at most polynomially (in $R)$ many elements $\phi$ satisfying $d_{\T(S)}(x_0, \phi(x_0))\le R$. Multiplied together these counts yield the upper bounds on $\abs{\Lambda_{\mathrm{fo}}(x_0,x_0,R)}$ and $\abs{\Lambda_{\mathrm{red}}(x,y,R)}$,
where we incorporate the polynomial factor into the exponential by increasing the exponent slightly.

\subsection{Questions}
There are several natural questions prompted by this work. 
The most obvious is whether the $\delta$ in the upper bound of Theorem~\ref{thm:main} can be removed. It arises from various technical considerations that manifest in the additive $\frac{n}{\wfal}$ term in Theorem~\ref{thm:intro_complexity_aligned}. This is the result of a phenomenon that we term ``badness'' (\S\ref{sec:bad_sets}) having to do with the fact that the witness family $\Omega$ may have pairs of nested subsurfaces $V_1\esub V_2$ with $\ent{V_1}+\ent{V_2} > \ent{S}$ for which the distances $d_{\T(V_i)}(\res{x}{\Omega}{V_i}, \res{y}{\Omega}{V_i})$ appearing in complexity length correspond to a region during which the main Teichm\"uller geodesic $[x_0,x_n]$ is simultaneously moving through the Teichm\"uller spaces $\T(V_1)$ and $\T(V_2)$. 
Our construction endeavors to minimizes badness (\S\ref{sec:fixing_badness}),
but it would be nice to find a solution eliminating it entirely. Even if the $\frac{n}{\wfal}$ term from Theorem~\ref{thm:intro_complexity_aligned} could be removed, there are still two polynomial factors, coming from Theorems~\ref{thm:intro_complexity_count} and \ref{thm:bound_multiplicity}, that enlarge the upper bound but are currently absorbed into the $e^{\delta R}$ factor in the statement of Theorem~\ref{thm:main}.

A related question is whether complexity length itself satisfies a reverse triangle inequality $\cl(a,b) + \cl(b,c) \le \cl(a,c)+K$ for strongly aligned triples, rather than the hybrid formulation concerning both $\cl$ and $d_{\T(S)}$ in Theorem~\ref{thm:intro_complexity_aligned}. While this may likely be the case, proving such a statement appears to be quite difficult.

A third question concerns the fact that our main theorem only provides coarse  bounds with multiplicative error, rather than precise asymptotics as in Theorem~\ref{thm:ABEM}. One reason Theorem~\ref{thm:main} is so difficult is because it concerns intrinsically {nongeneric} phenomenon. In contrast to \cite{ABEM,Maher-lattices}, where dynamics and ergodic theory are the main tools,
our lattice points arise with vanishingly small probability that is undetectable by these tools.
We must instead rely on coarse geometric arguments that lead to coarser bounds.  Calculating the precise asymptotic growth of $\abs{\Lambda_{\mathrm{fo}}(x,y,R)}$ and $\abs{\Lambda_{\mathrm{red}}(x,y,R)}$ will  require completely different techniques.

A fourth question, discussed in Remark~\ref{rem:lower-bound-only-closed}, concerns the lower bound in the finite-order case:

\begin{question}
\label{ques:finite-order-lower-bound}
Does the lower bound $\abs{\Lambda_{\mathrm{fo}}(x,y,R)}\ge K\exp(\tfrac{\ent{S}}{2}R)$ from Theorem~\ref{thm:main} hold for arbitrary punctured surfaces $S$ with $\plex{S}>0$?
\end{question}

Finally, as suggested in \S\ref{sec:heuristics}, another class of natural subsets of $\Mod(S)$ are the individual conjugacy classes.
Thus, for any element $\psi\in \Mod(S)$ and points $x,y\in \T(S)$ one may consider the growth of the set
\[\Lambda_\psi(x,y,R) = \left\{\psi'\in [\psi] \mid d_{\T(S)}(\psi'(x),y)\le R\right\}\] consisting of the mapping classes
conjugate to $\psi$.  Recent work of Han \cite{han-growthPseudoAnosovs,han-growthDehnTwists} shows that for every Dehn twist $\psi$, as well as most multitwists and most pseudo-Anosovs, the quantity $\abs{\Lambda_\psi(x,y,R)}$ grows coarsely like $e^{\ent{S}R/2}$. 
In light of this and Theorem~\ref{thm:main}, we propose the following:

\begin{question}
\label{ques:conj_class_growth}
For any nontrivial element $\psi\in \Mod(S)$ and points $x,y\in \T(S)$, do there exist $K_1,K_2>0$ such that $K_1 e^{\ent{S}R/2} \le \abs{\Lambda_\psi(x,y,R)}\le K_2 e^{\ent{S}R/2}$ for all large $R$? Furthermore, does $\abs{\Lambda_\psi(x,y,R)}/e^{\ent{S}R/2}$ converge as $R\to \infty$ and, if so, to what?
\end{question}

Of course, a positive answer to Question~\ref{ques:conj_class_growth} for arbitrary surfaces would also imply a positive answer to Question~\ref{ques:finite-order-lower-bound}.

\begin{remark}We anticipate that the techniques of this paper could be applied to give, for each $\psi\in \Mod(S)$ and $\delta>0$, an upper bound $\abs{\Lambda_\psi(x,y,R)}\le \exp((\tfrac{\ent{S}}{2}+\delta)R)$ for all large $R$. The point being that when restricted to a single conjugacy class, projection distances between $a_\phi$ and $b_\phi$ can be uniformly bounded to the effect that the $(\ent{S_p} + \frac{3}{\wfal})d_{\T(S)}(x_0,\phi(x_0))$ term in Lemma~\ref{lem:technical_savings_bound} can be ignored, which would in turn allow one to ignore the $+m$ term in the exponent of Theorem~\ref{thm:general_upper_bound}. However, a lower bound for $\abs{\Lambda_{\psi}(x,y,R)}$ appears to be more elusive, since when $\psi$ has an infinite centralizer the simple argument from \S\ref{Sec:lowerbound} does not apply.
\end{remark}

Just as Theorem~\ref{thm:ABEM} parallels Margulis's result \cite{Marguli-thesis} that the fundamental group $\pi_1(M)$ of a compact negatively curved manifold grows at the rate of $e^{hR}$, a positive answer to Question~\ref{ques:conj_class_growth} would parallel work of Parkkonen--Paulin \cite{ParkkonenPaulin-hyperbolic_conj_counting} showing that
 each nontrivial conjugacy class in $\pi_1(M$) essentially grows at the rate of $e^{hR/2}$.

\subsection{Outline}
The paper is organized as follows. The lower bound in \S\ref{Sec:lowerbound}  is handled by constructing explicit examples.
In \S\ref{Sec:example} we provide an example of a finite-order element $\phi_0$ whose conjugates all have closest fixed points that are both far away from $x_0$ and exhibit backtracking. In \S\ref{Sec:Background} we collect the needed background material, surveying many of the ideas in the subject of the mapping class group, Teichm\"uller geometry, and the curve graph  of a surface.
Section~\ref{Sec:Antichains and branch points} proves preliminary technical results that are needed in the sequel. This includes bounds on antichains that are used repeatedly in the construction of complexity length to control the size of witness families; an explanation of how alignment can be promoted to strong alignment, and an application of Gromov hyperbolicity of curve complexes to construct branch points with various properties.

Then  \S\S\ref{sec:witness-families}--\ref{Sec:countingcomplexity}  are devoted to developing the theory of complexity length.
In \S\ref{sec:witness-families} we introduce the notion of {\em witness families} $\Omega$ with additional structures (wideness, insulation, subordering, and completeness) that will be needed in our constructions. Section~\ref{Sec:complexity} then defines the complexity $\cl(\Omega)$ of a witness family, associated to a strongly aligned tuple $(x_0,\dots,x_n)$, by constructing \emph{resolution points} $\res{x_i}{\Omega}{V}$ in the Teichm\"uller spaces $\T(V)$ of each subsurface $V\in \Omega$. Section \ref{sec:bound_contrib_witness} then bounds each summand $\ent{V} d_{\T(V)}(\res{x_{i-1}}{\Omega}{V},\res{x_{i}}{\Omega}{V})$ of complexity length in terms of the Lebesgue measure of a certain \emph{contribution set} $\contr^\Omega_V$ along the main Teichm\"uller geodesic $[x_0,x_n]$. This is perhaps the most intricate part of the argument and is accomplished by judicious use of Minsky's product regions (Theorem~\ref{thm:product_regions}).
Section~\ref{sec:bad_sets} then introduces the notion of badness, which can lead to complexity length being strictly larger than $\ent{S}d_{\T(S)}(x_0,x_n)$, and explains an iterative procedure for refining witness families and minimizing badness. It is here where we specify the various parameters of the definition and construct (in Proposition~\ref{prop:saturated_constants}) witness families that simultaneously have controlled badness and uniformly bounded cardinality. In \S\ref{sec:complexity_length} we then define the complexity length $\cl(x_0,\dots,x_n)$ of a strongly aligned tuple as an infimum of complexities $\cl(\Omega)$ of witness families satisfying certain properties. This culminates in Theorem~\ref{thm:main_complexity_length_bound} (c.f.\ \ref{thm:intro_complexity_aligned} above) proving the key properties that complexity length satisfies a reverse triangle inequality and is bounded in terms of Teichm\"uller distance. Finally, in \S\ref{Sec:countingcomplexity}, we come back to the application of counting and bound, in Theorem~\ref{thm:countcomplexity} (c.f.\ \ref{thm:intro_complexity_count} above) the number of net points within a given complexity length of $x$.

Returning to upper bound of Theorem~\ref{thm:main} and the main goal of the paper,  in \S\ref{sec:finding-good-points} we construct a good point $x_\phi$ for each element $\phi\in \Mod(S)$ along with a pair of branch points $a_\phi,b_\phi$ for which the tuple $(x_0,a_\phi,b_\phi,\phi(x_0))$ is strongly aligned. 
The key properties enjoyed by these points are collected in the technical Proposition~\ref{prop:good-point-features}. 
Next, in \S\ref{Sec:Bounding the multiplicity of branch points},  for each possible pair $a,b$,  we count the number of elements $\phi$ with $d_{\T(S)}(x_0,\phi(x_0))\le R$ whose branch points $a_\phi,b_\phi$ are $a,b$.  
The count turns out to be polynomial in $R$, which is proven in Theorem~\ref{thm:bound_multiplicity}.

The proof of the general upper bound Theorem~\ref{thm:general_upper_bound} is finally given in \S\ref{sec:finish}.  At this point, with all the tools ready, the argument is not too difficult and roughly follows the sketch in \S\ref{sec:summary_of_proof} above. However, there are additional complications involving thin annuli and a resulting \emph{savings} in complexity length that is needed to avoid over counting the number of pairs $a_\phi,b_\phi$ of branch points.

\subsection{Acknowledgments}
This paper owes an intellectual debt to Maryam Mirzakhani, who posed the question of counting finite-order elements back in 2006. We dedicate this paper to her memory in hopes that she would have found our answer interesting. We also thank Alex Eskin for numerous helpful conversations and for answering many questions related to the work in \cite{ABEM,EM-countingClosedGeods}.  The authors would also like to thank Kasra Rafi for many helpful conversations.
We gratefully acknowledge the support of the Simons Laufer Mathematical Sciences Institute, formerly Mathematical Sciences Research Institute, and its generous hospitality during the 2016 program on Geometric group theory, when the framework of complexity length was formalized, and during the 2022 Complementary Program, when the final components of the project were completed. The first named author was partially supported by NSF grants DMS-1711089, DMS-2005368, and DMS-2405061.

\section{Lower Bounds}
\label{Sec:lowerbound}
In this section we prove the lower bounds in Theorem~\ref{thm:main}. We first note:

\begin{observation}
\label{obs:can_use_any_points}
The triangle inequality and the fact that $\Mod(S)$ acts isometrically on $\T(S)$ imply  that for any points $x,y,x',y'\in T(S)$ we have
\[\Lambda(x,y,R) \subset \Lambda\left(x',y',  R + d_{\T(S)}(x,x') + d_{\T(S)}(y,y')\right).\]
Therefore it suffices to prove Theorem~\ref{thm:main} for a single pair $x'=y'=x_0$, as the bounds will then hold for any arbitrary pair $x,y$ with the constants $K_1,K_2$ adjusted by a factor of $\exp(d_{\T(S)}(x,x_0) + d_{\T(S)}(y,x_0))$.
\end{observation}

The lower bounds are obtained by constructing examples. We thank Dan Margalit for suggesting the following finite-order example:

\begin{example}
\label{example:finite-order}
The closed surface $S_g$ of genus $g\ge 2$ admits a finite-order element $\phi_0\in \Mod(S_g)$ with centralizer the finite group generated by $\phi_0$. To see this, realize $S_g$ as a $4g+2$ regular polygon $P$ with opposite sides identified. Let $\phi_0$ be the map induced by the order $4g+2$ rotation of $P$. The quotient of $S_g$ by $\phi_0$ is a sphere with $3$ marked points corresponding to the center of $P$, the identified vertices and the center of the edges.  A sphere with $3$ marked  points has trivial mapping class group and therefore the centralizer of $\phi_0$ is just the group generated by $\phi_0$. 
\end{example}

With this example in hand, it is a simple matter to prove:

\begin{theorem}
\label{thm:finite-order-lower-bound}
For a closed surface $S$ of genus $g\ge 2$ and any $x,y\in \T(S)$, there exist $K,R_0> 0$ so that for all $R\ge R_0$ one has $\abs{\Lambda_{\mathrm{fo}}(x,y,R)}\ge K\exp(\tfrac{\ent{S}}{2}R)$. 
\end{theorem}
\begin{proof}
Let $\phi_0$ be as in Example~\ref{example:finite-order}, and let $x_0\in T(S)$ be fixed point for $\phi_0$. 
By Observation~\ref{obs:can_use_any_points} we can assume $x=y=x_0$.
By Theorem~\ref{thm:ABEM} \cite{ABEM} there exists $K'>0$ such that for all sufficiently large $R$ there are at least $K'\exp(\ent{S}\tfrac{R}{2})$ elements $w\in \Mod(S)$
so that $d(x_0,w(x_0))\le\tfrac{R}{2}$. The point $w(x_0)$ is fixed by the finite-order element  $\phi_w=w  \phi_0  w^{-1}$, 
 and, since $\phi_w$ is an isometry, 
\[d(x_0,\phi_w(x_0))\leq d(x_0,w(x_0))+d(w(x_0),\phi_w(x_0))=2d(x_0,w(x_0))\leq R.\]
 To conclude the lower bound, it is enough to show that the assignment $w\mapsto \phi_w$ is $4g+2$ to $1$.  Notice that  if  $\phi_{w_1}=\phi_{w_2}$, then $w_1  \phi_0  w_1^{-1}=w_2 \phi_0 w_2^{-1}$. Hence $w_1^{-1}w_2$ is in the centralizer of $\phi_0$, which means $w_1\phi_0^j=w_2$ for some $j$.
\end{proof}

The lower bound in the reducible case is in a similar spirit:

\begin{theorem}
\label{thm:reducible-lower-bound}
For any surface $S$ with $\plex{S}>0$ and any $x,y\in \T(S)$, there exist $K,R_0>0$ so that for all $R\ge R_0$ one has $\abs{\Lambda_{\mathrm{red}}(x,y,R)}\ge K \exp((\ent{S}-1)R)$.
\end{theorem}
\begin{proof}
Suppose $S= S_{g,p}$ has genus $g$ and $p$ punctures. Fix a simple closed curve $\alpha$ on $S$. If $\alpha$ is nonseparating, the complement $S\setminus \alpha$ is a surface of genus $g-1$ with $p+2$ punctures. Otherwise $\alpha$ is separating and $S\setminus \alpha$ consists of two surfaces $S_{g_1,p_1}$ and $S_{g_2,p_2}$ where $g_1+g_2 = g$ and $p_1+p_2 = p + 2$. In either case we have $\plex{S\setminus \alpha} = \plex{S}-1$, where in the separating case $\plex{S\setminus \alpha}$ means $\plex{S_{g_1,p_1}}+ \plex{S_{g_2,p_2}}$. The stabilizer $\Gamma_\alpha = \{\phi\in \Mod(S) \mid \phi(\alpha) = \alpha\}$ consists of reducible elements and, by \cite[Proposition 3.20]{FM-Primer}, fits into a short exact sequence
\[1\to \langle D_\alpha\rangle \to \Gamma_\alpha\to \Mod(S\setminus \alpha)\to 1\]
where $\twist{\alpha}$ is the Dehn twist about $\alpha$ (see \S\ref{sec:mapping_class_group}).
In particular, for each $\psi\in \Mod(S\setminus \alpha)$ we may fix a lift $\tilde{\psi}\in \Mod(S)$ and consider all the elements $\phi = \tilde{\psi}D_\alpha^n\in \Gamma_\alpha$.

Fenchel--Nielsen coordinates give a homeomorphism $\productmap{\alpha}\colon \T(S)\xrightarrow{\cong} \T(S\setminus \alpha)\times \H^2$, where $\H^2$ is the hyperbolic plane and $\T(S\setminus \alpha)$ is interpreted as a product when $S\setminus \alpha$ is disconnected (see \S\ref{sec:product_regions}).
 We equip $\T(S\setminus\alpha)\times \H^2$ with the sup metric $\rho$. 
 Minsky's \cite{Minsky-product} product regions theorem (Theorem~\ref{thm:product_regions} below), says there exists $\epsilon,\minsky>0$ such $d_{\T(S)}(x,y)$ and $\rho(\productmap{\alpha}(x),\productmap{\alpha}(y))$ differ by at most $\minsky$ for all point $x,y\in \T(S)$ at which the length of $\alpha$ (see \S\ref{sec:teichmuller_space}) is at most $\epsilon$. 

Again, by Observation~\ref{obs:can_use_any_points}, we may assume  $x =y =x_0\in \T(S)$ is a point at which $\ell_\alpha(x_0) = \epsilon$. Write $\productmap{\alpha}(x_0) = (z,w)\in \T(S\setminus \alpha)\times \H^2$. By Theorem~\ref{thm:ABEM} \cite{ABEM} there is a constant $K'$ so that for sufficiently large $R$ there are at least
\[K' \exp(\ent{S\setminus \alpha} R) = K' \exp((\ent{S}-2)R)\]
elements $\psi \in \Mod(S\setminus\alpha)$ such that $d_{\T(S_\alpha)}(\psi(z),z)\le R-\minsky$. (In the case that $S\setminus \alpha = S_1\sqcup S_2$ is disconnected, we get $K_i \exp(\ent{S_i}R)$ elements $\psi_i$ in each factor $\Mod(S_i)$. Since $\rho$ is a sup metric, the $K_1K_2\exp((\ent{S_1}+\ent{S_2})R)$ many products $\psi = \psi_1\psi_2$ all move $z$ at most $R$.)
Note that in Fenchel--Nielsen coordinates the Dehn twist acts by $\productmap{\alpha}(\twist{\alpha}(x)) = (z, w + \ell_\alpha(x))$, where here we are using the upper half-plane model $\H^2\subset \C$.
Further, by basic hyperbolic geometry, for the given values $w$ and $\ell_\alpha(X)$ there exists $K''>0$ such that there are at least $K'' \exp(R)$ integers $n\in \Z$ for which $d_{\H^2}(w, w+n \ell_\alpha(x))\le R-\minsky$. 
For each such $\psi\in \Mod(S\setminus\alpha)$ and $n\in \Z$ we obtain a reducible element $\phi = \tilde{\psi}\twist{\alpha}^n\in \Gamma_\alpha$ for which
\begin{align*}
d_{\T(S)}(\phi(x_0),x_0) &\le \rho\big(\productmap{\alpha}(\phi(x_0)),\productmap{\alpha}(x_0)\big) + \minsky\\
&= \rho\big((\psi(z), w + n \ell_{x}(\alpha)), (z, w)\big) + \minsky \le R.
\end{align*}
Thus $\Lambda_{\mathrm{red}}(x_0,x_0,R)$ contains at least $K'\exp((\ent{S}-2)R)K''\exp(R)$ elements.
\end{proof}

\section{Example of a finite-order element with distant fixed points}
\label{Sec:example}
 We now provide  the example promised in the introduction 
of a finite-order mapping class $\phi_0$ with  fixed point at some $x_0$, a conjugate $\phi=w\circ \phi_0\circ w^{-1}$  with closest fixed point $w(x_0)$ such that 
\begin{itemize}
\item up to additive error $d_{\T(S)}(x_0,\phi(x_0))=d_{\T(S)}(x_0, w(x_0))$.
\item there is backtracking in the Teichm\"uller space of a subsurface
\end{itemize}
Readers may wish to consult \S\ref{Sec:Background} for any unfamiliar background material.
    
 Let $S$ be the genus $7$ surface, which we cut (along a multicurve) into $7$ disjoint once-punctured tori $\Sigma_1,\ldots, \Sigma_7$ and a seven-times punctured sphere $Z$. Let $\phi_0$ be the order $7$ rotation of $S$ which preserves $Z$ and cyclically permutes the tori sending $\Sigma_j$ to $\Sigma_{j+1}$ (indices taken modulo $7$). The homeomorphisms $\phi_0^{j-i}\colon \Sigma_i\to \Sigma_j$ give us a fixed identification between $\T(\Sigma_i)$ and $\T(\Sigma_j$) for all $i,j$.  These are copies of the hyperbolic plane $\mathbb{H}^2$ with the hyperbolic metric. 
We may choose a fixed point $x_0\in \T(S)$ for $\phi_0$ so that the lengths of the curves $\alpha_j = \partial \Sigma_j$ are short and so that, under the Fenchel--Nielsen homeomorphism $\T(S)\cong \T(Z)\times \prod_{j} \T(\Sigma_j) \times \prod_j \H^2$ (see \S\ref{sec:product_regions}), $x_0$ corresponds to the point $i$ in each factor $\T(\Sigma_j)=\H^2$. Also let $\mu_j$ be the corresponding marking of $x_0$ in $\Sigma_j$ (see \S\ref{sec:markings}).

Now let $f\in \Mod(S)$ be a map that is supported on $\Sigma_1$ and whose restriction $f_0 = f\vert_{\Sigma_1}$ gives an Anosov map of the once-punctured torus. Under the identification $\T(\Sigma_1) = \H^2$ with the upper-half plane we assume $f_0$ is of the form $z\mapsto \lambda z$; thus the imaginary axis is the Teichm\"uller axis of $f_0$. For positive integers $j \le n \le m$ we define
\[w = (f^n) (\phi_0^3 f^m \phi_0^{-3}) (\phi_0^5 f^{-m} \phi_0^{-5}) (\phi_0^7 f^j  \phi_0^{-7}).\]
Thus $w$ is the identity on $\Sigma_2\cup \Sigma_4\cup\Sigma_6\cup \Sigma_7$ and, using our identification $\Sigma_i\cong \Sigma_j$, acts like $f_0^n$ on $\Sigma_1$, like $f_0^m$ on $\Sigma_3$, like $f_0^{-m}$ on $\Sigma_5$, and like $f_0^{j}$ on $\Sigma_7$.

Next set $\phi=w \phi_0 w^{-1}$; this is a finite-order element fixing $w(x_0)$.
By the Minsky product formula \cite{Minsky-product} (Theorem~\ref{thm:product_regions}) it is easy to check from the definition of $w$ that up to additive constants independent of $j,n,m$ 
$$d_{\T(S)}(x_0,w(x_0))=d_{\T(S)}(x_0,\phi(x_0))=m\log\lambda.
$$  
We now sketch the argument that any other fixed point $u$ of $\phi$, up to additive constant, satisfies $d_{\T(S)}(u,x_0)\geq m\log\lambda$.
Let  $\nu_i$  the Bers marking  at $u$ projected to $\Sigma_i$. Since $\phi(u)=u$ we have $\nu_4=f^{(-m)}(\nu_3)$, $\nu_5=f^{(-m)}(\nu_4)$, where again the rotation $\phi_0$ allows us to identify markings in different $\Sigma_i$.  Thus $\nu_3=f^{(2m)}(\nu_5)$.
This implies that the projection of $u$ to the curve complex on $\Sigma_3$ and $\Sigma_5$ differ by $2m\log \lambda$ and so by the triangle inequality at least one of the two projections differs from the projection $u_3 = u_5$ of $x_0$ by at least $m\log\lambda$. Without loss of generality assume it is $\Sigma_3$.  There  must be interval an interval  $\mathcal{I} = \interval_{\Sigma_3}$ (called an active interval in this paper, \S\ref{sec:active_intervals})   along $[u,x_0]$ where the boundary loop $\alpha_3$ has length at most $\epsilon_0$   and outside $\mathcal{I}$ the projections to the curve complex of $\Sigma_3$ only changes by an additive constant.  
This means that we may consider $\mathcal{I}$ as lying in the product region $\mathbb{H}^2\times \T(S\setminus \Sigma_3)$.    Again by \cite{Minsky-product} we have up to an additive constant $\abs{\mathcal{I}}\geq m\log \lambda$. Therefore up to an additive constant the distance between $u$ and $x_0$ must be at least $m\log\lambda$  and so $w(x_0)$ is the closest fixed point up to additive error. 

Next we note that in the curve complex of $\Sigma_1$, in going from $x_0$ to $w(x_0)$ we travel distance $n\log\lambda$ and then from $w(x_0)$ to $\phi(x_0)$ we backtrack  distance  $j\log \lambda$. 
We remark that the fixed point  $w(x_0)$ will be called a good fixed point since there is not backtracking in {\em some} subsurface in the orbit of $\Sigma_1$ (see Definition~\ref{def:the_good poitns} and Lemma~\ref{lem:strongly_contracting}). In general we will find a good fixed point for any finite-order $\phi$.
A second observation is that if we let $y$ be a point with the same markings as $w(x_0)$ except in $\Sigma_1$ where we set the marking to agree with that of $\phi(x_0)$,  then the three points $x_0,y,\phi(x_0)$ 
are aligned as moving from $x_0$ to $y$ and then to $\phi(x_0)$ there is no backtracking in any domain. Another major goal will be to produce such points in general which we call good branch points (see \S\ref{sec:good_branch_pointsS} and Proposition~\ref{prop:good-point-features}).

\section{Background}
\label{Sec:Background}
Throughout, the term \define{surface} will indicate an oriented surface $\Sigma$ homeomorphic to a closed surface minus a (possibly empty) finite set of points. The missing points are called \define{punctures} and are in bijective correspondence with the ends of $\Sigma$.
We write $S_{g,p}$ for the connected genus $g\ge 0$ surface with $p\ge 0$ punctures, and define its \define{complexity} to be $\plex{S_{g,p}} \colonequals 3g-3+p$. 
In general, the complexity of a surface $\Sigma$ is the sum $\plex{\Sigma} \colonequals \sum_i \plex{\Sigma_i}$ over the connected components $\Sigma_i$ of $\Sigma$.

An \define{annulus} is a connected surface $\Sigma$ of complexity $-1$ (i.e., $\Sigma = S_{0,2}$). Annuli are exceptional in several respects, and must be handled with care throughout our discussion.

The \define{entropy} 
of a connected surface $\Sigma$ is defined to be $\ent{\Sigma} \colonequals 2\abs{\plex{\Sigma}}$ provided that $\plex{\Sigma}\ge -1$ and is defined to be zero otherwise, so that spheres, tori, once-punctured spheres, and thrice-punctured spheres have entropy equal to $0$, and annuli have entropy $2$.  As for complexity, the entropy of a disconnected surface $\Sigma$ is defined as the sum $\ent{\Sigma} = \sum_i \ent{\Sigma_i}$ of the entropies of its connected components $\Sigma_i$.

\begin{convention}
\label{conv:fixed_surface}
For the duration of this paper, we henceforth fix a connected surface $S$ with $\plex{S}>0$.
\end{convention}

\begin{notation}
\label{notation:additive_error}
We use the notation $A\ladd B$ or $B\gadd A$ to mean that there is a universal constant $c$, depending only on the topology of $S$, such that $A \le B + c$. The notation $A\eadd B$ means that $A\ladd B$ and $A\gadd B$ both hold. We instead use the notation $\ladd_*$, $\gadd_*$ or $\eadd_*$ to indicate that the implied constant $c$ depends only on $S$ and the object $*$.
\end{notation}

\subsection{Curves}

The term ``curve'' will always refer to an isotopy class of essential simple closed curves. More precisely, a \define{curve} in a non-annular surface $\Sigma$ is (the orientation-reversing-and-isotopy class of) an embedding $\thecircle\to \Sigma$ of the circle that is neither nullhomotopic nor homotopic into an end of $\Sigma$. A \define{curve in an annulus} is rather (the isotopy class of) an embedding of $\thecircle$ whose complement consists of two annuli. Notice that each annulus $Y$ has a unique curve; we call this the \define{core} of the annulus and denote it by $\partial Y$. The set of curves in $\Sigma$ will be denoted by $\curves{\Sigma}$. 

The \define{intersection number} of a pair of curves $\alpha,\beta\in \curves{\Sigma}$ is the minimum $\intnum(\alpha,\beta)$ of the quantity $\abs{a(\thecircle)\cap b(\thecircle)}$ over all representative embeddings $a,b\colon \thecircle\to \Sigma$. The curves $\alpha,\beta$ are \define{disjoint}, written $\alpha\disj\beta$, if they admit disjoint representatives, that is, if $i(\alpha,\beta) = 0$. 
Otherwise the curves \define{cut} each other and we write $\alpha\cut\beta$. A \define{curve system} or \define{multicurve} on $\Sigma$ is a nonempty set of pairwise disjoint curves in $\Sigma$; being a set of curves, note that the elements of a curve system are necessarily distinct and thus nonisotopic. Curve systems $\alpha$ and $\beta$ \define{cut} each other, denoted $\alpha\cut \beta$, if $\alpha_0\cut\beta_0$ for some $\alpha_0\in \alpha$ and $\beta_0\in \beta$ (that is, if their union is not a curve system). Otherwise the curve systems are disjoint and we write $\alpha\disj\beta$.

\subsection{Subsurfaces}
Following \cite[\S2.1]{BKMM-rigidity}, an \define{(essential) subsurface} of $\Sigma$ is a subset $Y\subset  \Sigma$ that is itself a surface and has the following structure:
\begin{itemize}
\item $Y$ is a union of (not necessarily all) complementary components of a (possibly empty) embedding $C\colon (\sqcup_{i=1}^k \thecircle)\to \Sigma$ whose components $C_1,\dotsc,C_k$ each define curves in $\Sigma$. The components of $C\cap \overline{Y}$ are then pairwise disjoint or isotopic curves in $\Sigma$ and so determine a curve system $\partial Y$ on $\Sigma$; these are the \define{boundary curves} of $Y$.
\item 
No two components of $Y$ are isotopic (equivalently, no two annuli components are isotopic).
\item No component of $Y$ is a thrice-punctured sphere.
\end{itemize}
Subsurfaces of $\Sigma$ are identified when they are isotopic in $\Sigma$. We reserve the term \define{domain} for a connected subsurface of $\Sigma$. Note that $\Sigma$ is itself a subsurface of $\Sigma$ and has trivial boundary $\partial \Sigma =\emptyset$
\begin{remark}
\label{rem:surfaces_are_open}
We remark that by convention all of our surfaces and subsurfaces are \emph{open} with empty boundary.
\end{remark}
Given a subsurface $Y$ of $\Sigma$, the inclusion induces an injection $\curves{Y}\hookrightarrow \curves{\Sigma}$ that allows us to identify $\curves{Y}$ with the set of curves in $\Sigma$ that are essential in $Y$.
Note that a curve of $\partial Y$ lies in $\curves{Y}\subset \curves{\Sigma}$ if and only if it is the core of an annulus component of $Y$. 
For any subsurface $Y$ of $S$ we use the notation $\bcurves{Y} = \partial Y\cup \curves{Y}$ for the set of curves that are essential in $Y$ or boundary curves.

On the set of subsurfaces of $\Sigma$, define a relation $Z\esub Y$  to mean $\curves{Z}\subset\curves{Y}$ as subsets of $\curves{\Sigma}$. In this case, one may adjust $Z$ by an isotopy so that it is a bona fide subsurface of $Y$. The relation $\esub$ may thus be safely read as ``is an (essential) subsurface of,'' and we accordingly use the notation $Y\esub \Sigma$ to mean that $Y$ is a subsurface of $\Sigma$. 
We also use  $Z\esubneq Y$ to mean $Z\esub Y$ and $Z\ne Y$.

\begin{lemma}[Behrstock--Kleiner--Minsky--Mosher {\cite[Lemma~2.1]{BKMM-rigidity}}]
\label{lem:BKMM-subsurface-operations}
The set of subsurfaces of $\Sigma$ (including $\emptyset$) is a lattice with partial order $\esub$ and meet/join operations (termed ``essential union/intersection'') denoted by $\ecup$ and $\ecap$:
\begin{itemize}
\item Subsurfaces $Z$ and $Y$ are isotopic if and only if $\curves{Z} = \curves{Y}$.
\item $Y\ecup Z$ is the unique $\esub$--minimal $W$ so that $Y\esub W$ and $Z\esub W$,
\item $Y\ecap Z$ is the unique $\esub$--maximal $W$ so that $W\esub Y$ and $W\esub Z$.
\end{itemize}
\end{lemma}

\subsection{Cutting}
We extend the binary relation $\cut$ to the set of all curve systems and subsurfaces of $\Sigma$ as follows: We have already defined $\alpha\pitchfork \beta$ when $\alpha,\beta$ are curves or curve systems on $\Sigma$. A curve system $\alpha$ is said to be \define{disjoint} from a subsurface $Y\esub \Sigma$, denoted by $\alpha \disj Y$, if their isotopy classes have disjoint representatives; otherwise they are said to \define{cut} and we write $\alpha\cut Y$. 

Two subsurfaces $Y$ and $Z$ of $\Sigma$ are said to be \define{disjoint}, denoted $Y\disj Z$, if they have disjoint representatives. They are \define{nested} if $Y\esub Z$ or $Z\esub Y$. Otherwise, $Y$ and $Z$ are said to \define{cut} (or \emph{intersect transversely}), denoted $Y\pitchfork Z$. Observe that $Y\pitchfork Z$ is equivalent to $Y\pitchfork \partial Z$ and $Z\pitchfork \partial Y$. 
We say that $Y$ and $Z$ \define{intersect} if they are not disjoint. 
We will also need a relative form of cutting:

\begin{definition}[Relative cutting]
\label{def:sever_domain}
Given a subsurface $V\esub \Sigma$, we say that two subsurfaces $Y_1,Y_2$ of $\Sigma$ \define{cut relative to $V$} if $Y_1'\cut Y_2'$ for all subsurfaces $Y_i'\esub Y_i$ that intersect $V$. 
In this case we write $Y_1\cut_V Y_2$.
\end{definition}

\begin{remark}
Note that $Y_1\cut_V Y_2$ vacuously holds if either $Y_1$ or $Y_2$ is disjoint from $V$; thus $Y_1 \cut_V Y_2$ does not necessarily imply $Y_1\cut Y_2$. However, if $Y_1$ and $Y_2$ both intersect $V$, then $Y_1\cut_V Y_2 \implies Y_1\cut Y_2$.
\end{remark}
Relative cutting has the useful property, over cutting, in that it passes to subsurfaces. That is, 
$Y\cut_V Z$ implies $Y'\cut_V Z'$ for all subsurfaces $Y'\esub Y$ and $Z'\esub Z$.

\subsection{Mapping class group}
\label{sec:mapping_class_group}

The \define{mapping class group} of a surface $\Sigma$ is the quotient \[\Mod(\Sigma)\colonequals \mathrm{Homeo}^+(\Sigma)/\mathrm{Homeo}_0(\Sigma)\]
of the group of orientation-preserving homeomorphisms of $\Sigma$ by the normal subgroup $\mathrm{Homeo}_0(\Sigma)$ of homeomorphisms that are isotopic to the identity. 
Observe that $\Mod(\Sigma)$ acts on the sets of curves, curve systems, 
and subsurfaces of $\Sigma$ preserving the relations $\pitchfork$, $\disj$, and $\esub$. We write $\twist{\alpha}\in \Mod(\Sigma)$ for the (left) \define{Dehn twist} about a curve $\alpha$ in $\Sigma$; see \cite[Chapter 3]{FM-Primer}.

\begin{def-thm}[Nielsen--Thurston Classification \cite{Thurston-classification}]
\label{def:NT-classification}

Let $\Sigma$ be a 
surface with $\chi(\Sigma) < 0$. An element $\phi\in \Mod(\Sigma)$ is said to be:
\begin{itemize} 
\item \define{finite-order} if there is an integer $k\ge 1$ such that $\phi^k$ is trivial in $\Mod(\Sigma)$;
\item \define{reducible} if it is not finite-order and 
$\phi(C) = C$ for some curve system $C$.
\item \define{pseudo-Anosov} there are a pair $\mathcal{F}_\pm$ of (singular) transverse measured foliations on $\Sigma$ and a number $\lambda >1$ so that $\phi(\mathcal{F}_\pm) = \lambda^{\pm1}\mathcal{F}_\pm$. 
\end{itemize}
Moreover, each $\phi\in\Mod(\Sigma)$ falls into \emph{exactly one} of these categories. Thus pseudo-Anosov is equivalent to each curve $\alpha$ having an infinite orbit $\{\phi^k(\alpha)\mid k\in \Z\}$.
\end{def-thm}

In general, a curve system $C$ so that $\phi(C) = C$ is called a \define{reducing system} for $\phi$. Each element $\phi\in \Mod(\Sigma)$ in fact has a \define{canonical reducing system} $\sigma(\phi)$ defined as the set of curves $\alpha$ that have a finite orbit $\{\phi^k(\alpha) \mid k\in \Z\}$ and which are disjoint from all curves with $\beta$ with a finite orbit. This canonical set of curves is clearly a reducing system for $\phi$ and has the property that every curve $\beta$ that cuts $\sigma(\phi)$ satisfies $\phi^k(\beta) \ne \beta$ for all $k\in \Z$; see, e.g., \cite[\S2]{HandelThurston}.

 Notice that $\sigma(\phi)$ is nonempty when $\phi$ is reducible and is the empty multicurve precisely when $\phi$ is finite-order or pseudo-Anosov. For each component $Y$ of $\Sigma\setminus \sigma(\phi)$, there is a smallest integer $k\ge 1$ for which $\phi(Y) = Y$. The restriction $\phi^k\vert_Y\in \Mod(Y)$ is then necessarily finite-order or pseudo-Anosov, since its canonical reduction system $\sigma(\phi^k\vert_Y)$ is evidently empty. Accordingly, we refer to $Y$ as either a \define{finite-order} or \define{pseudo-Anosov component} of $\Sigma\setminus\sigma(\phi)$.

\begin{definition}[Nielsen--Thurston Decomposition]
\label{def:nielsen-thurston-decomp}
For each $\phi\in \Mod(\Sigma)$, we decompose $\Sigma$ into three (possibly disconnected or empty) disjoint subsurfaces as follows: $\Sigma_p$ is the union of the pseudo-Anosov components of $\Sigma\setminus\sigma(\phi)$; similarly $\Sigma_f$ is the union of finite-order components; and $\Sigma_A$ is the union of annuli about the curves of the canonical reducing system $\sigma(\phi)$.
Thus we have $\Sigma_p = \Sigma$ with $\Sigma_f = \Sigma_A = \emptyset$ when $\phi$ is pseudo-Anosovs, and $\Sigma_f = \Sigma$ with $\Sigma_p = \Sigma_A = \emptyset$ when $\phi$ is finite-order.
\end{definition}

An element $\phi\in \Mod(\Sigma)$ is said to be in \define{Thurston normal form} if it preserves each component $\alpha$ of $\sigma(\phi)$ and each component $Y$ of $\Sigma\setminus\sigma(\phi)$ with restriction $\phi\vert_Y$ that is either pseudo-Anosov or the identity. In this case $\phi$ may be expressed as a product $\phi = \phi_1\circ\dotsb\circ \phi_k$ where each $\phi_i$ is either a nontrivial power $D_\alpha^m$ of a Dehn twist about a component $\alpha$ of $\sigma(\phi)$, or a pseudo-Anosov on a component $Y$ of $\Sigma\setminus\sigma(\phi)$.
For each $\phi$ there is a smallest integer $N = N(\phi)\ge 1$ so that $\phi^N$ is in Thurston normal form; we call this the Thurston normal form of $\phi$. (Note that $N(\phi)$ is bounded in terms of the topology of $\Sigma$, and hence there is a single integer $N$ so that $\phi^N$ is in Thurston-normal form for every $\phi\in \Mod(\Sigma)$).

\begin{lemma}
\label{lem:twist}
If $Y_1,Y_2$ are components of $\Sigma_f$ whose boundaries $\partial Y_1,\partial Y_2$ both contain the same curve $\alpha\in \sigma(\phi)$ (we allow the possibility $Y_1 = Y_2$), then the Thurston normal form for $\phi$ contains a nontrivial Dehn twist about $\alpha$.
\end{lemma}
\begin{proof}
Choose $N$ so that $\phi^{N}$ is in normal form. If the normal form does not contain a Dehn twist about $\alpha$, then we can find a curve $\beta$ in $Y_1 \cup\alpha\cup  Y_2$ cutting $\alpha$ such that $\phi^N(\beta)=\beta$, which contradicts $\alpha$ being in the canonical reducing system. 
\end{proof}

For any $\Omega$ we may replace the boundary curves of $\Omega$ with punctures. In the case of $\phi^N$ restricted to $\Omega$ being pseudo-Anosov, there are a pair of  projective foliations $F_u,F_s$ of $\Omega$ invariant under $\phi^N$. 
For any $V\subset \Omega$ there is a coarsely defined quasi-geodesic  $\beta_V\subset \mathcal{C}(V)$  joining  $\pi_V(F_u)$ and $\pi_V(F_s)$.

\subsection{Curve complexes}

Let $\Sigma$ be a connected surface with $\plex{\Sigma} \ge 1$. The \define{curve graph} (or \define{curve complex}) of $\Sigma$ is the simplicial graph $\cc(\Sigma)$ with vertex set $\cc_0(\Sigma) = \curves{\Sigma}$ and edges defined as follows: If $\plex{\Sigma}\ge 2$, then two curves $\alpha,\beta$ in $\Sigma$ are joined by an edge in $\cc(\Sigma)$ if they are disjoint. If $\plex{\Sigma} = 1$, then $\Sigma = S_{1,1}$ or $\Sigma = S_{0,4}$ and two curves $\alpha,\beta$ are joined by an edge if they intersect once in the case of $S_{1,1}$ or twice in the case of $S_{0,4}$; in either case $\cc(\Sigma)$ is the well-known Farey graph.

The curve graph $\cc(Y)$ of an annular subsurface $Y\esub \Sigma$ is defined following \cite{MasurMinsky-heirarchy}: The annular cover $A_Y\to \Sigma$ to which $Y$ lifts homeomorphically admits a natural compactification $\bar{A}_Y$ coming from the usual compactification of $\tilde{\Sigma}\cong\H^2$. An \define{(essential) arc} in $\bar{A}_Y$ is an isotopy class, rel endpoints, of properly embedded arcs whose endpoints lie on opposite components of $\partial \bar{A}_Y$. Two such arcs are said to be \define{disjoint} if they have representatives with disjoint interiors. We may then define the \define{curve graph} of $Y$ to be the simplicial graph $\cc(Y)$ whose vertices are arcs in $\bar{A}_Y$ with edges joining pairs of disjoint arcs.

Every curve graph is given the path metric in which each edge has length $1$. A geodesic metric space is said to be \define{$\delta$--hyperbolic} if each side of every geodesic triangle is contained in the $\delta$--neighborhood of the union of the other two sides.

\begin{theorem}[Masur--Minsky \cite{MasurMinsky-hyp}]
\label{thm:hyp_curve_cplx}
There is an integer $\delta>0$ depending only on $S$ such that the curve graph $\cc(\Sigma)$ of every domain $\Sigma \esub S$ is $\delta$--hyperbolic.
\end{theorem}

The number $\delta$ may in fact be chosen independently of $S$ \cite{Aougab-hyp,Bowditch-uniformHyp,ClayRafiSchleimer-unifHyp,HenselPrizytyckiWebb}; e.g.,  $\delta = 17$ works. We emphasize that by our convention \emph{$\delta$ is an integer}. 

\subsection{Markings}
\label{sec:markings}
Following \cite{MasurMinsky-heirarchy}, a \define{(clean) marking} $\mu$ on a connected surface $\Sigma$ with $\plex{\Sigma}=k\ge 1$ is a pair $\mu = (\base(\mu),t)$, where $\base(\mu) = (\beta_1,\dotsc,\beta_k)$ is a maximal curve system on $\Sigma$ (a so-called \define{pants decomposition}) and $t = (t_1,\dotsc,t_k)$ is a tuple of curves on $\Sigma$ such that $\beta_i$ and $t_i$ intersect minimally (either once or twice) subject to the condition that $\beta_i$ and $t_j$ are disjoint for all $i\ne j$. A \define{marking on an annulus} $Y\esub \Sigma$ is a pair $\mu = (\base(\mu),t)$ where $\base(\mu) = \beta$ is the core of $Y$ and $t\in \cc_0(Y)$ is a vertex of the curve complex of $Y$. The curve system $\base(\mu)$ and tuple $t$ are respectively called the \define{base} and \define{transversal} of the marking $\mu$. (Our definition of marking is more restrictive than that used in \cite{MasurMinsky-heirarchy} but shall suffice for our purposes).

A marking on a disconnected subsurface $\Sigma \esub S$ is simply a choice of marking for each component of $\Sigma$. The set of markings on a surface $\Sigma$ will be denoted $\markings{\Sigma}$.

\subsection{Subsurface projections}
\label{sec:subsurface_projections}
Let $\Sigma$ be a surface and $Y\esub \Sigma$ a domain. We define a map $\pi_Y\colon \curves{\Sigma}\to \mathcal{P}(\cc_0(Y))$, with codomain the set of subsets of $\cc_0(Y)$, as follows. For concreteness, fix a complete hyperbolic metric on $\Sigma$ and realize $Y$ such that $\partial Y$ is a union of geodesics. 
First suppose $\plex{Y}\ge 1$ so that $Y$ is not an annulus. If $\alpha\in \curves{\Sigma}$ is disjoint from $Y$ we set $\pi_Y(\alpha) = \emptyset$ and if $\alpha\in \curves{Y}$ we set $\pi_Y(\alpha) = \alpha$. Otherwise $\alpha\pitchfork Y$ and the geodesic representative of $\alpha$ intersects $Y$ in a collection of proper arcs. For each component $a_i $ of $\alpha\cap Y$, the boundary of a regular neighborhood of $a_i\cup \partial Y$ in $Y$ determines one or two curves in $Y$; the set of all curves obtained in this way is $\pi_Y(\alpha)$. 
For an annulus $Y$ and a curve $\alpha\in \curves{\Sigma}$, we still set $\pi_Y(\alpha) = \emptyset$ when $\alpha$ is disjoint from $Y$. Otherwise $\alpha\pitchfork Y$ and we let $\pi_Y(\alpha)$ be the set of lifts of $\alpha$ that give essential arcs in the compactified annular cover $\bar{A}_Y$.

We extend the domain of $\pi_Y$ to \emph{sets} of curves by adopting the convention that $\pi_Y(A) = \bigcup_{\alpha\in A} \pi_Y(\alpha)$ for any set $A$ of curves on $\Sigma$. We also define the notation
\[\diam_Y(A) \colonequals \diam_{\cc(Y)}(\pi_Y(A))\quad\text{and}\quad d_Y(A,B) \colonequals \diam_{\cc(Y)}(\pi_Y(A)\cup\pi_Y(B))\]
for any sets or elements $A$ and $B$ in $\curves{\Sigma}$. It is easy to see (e.g., \cite[Lemma 2.3]{MasurMinsky-heirarchy}) that $\diam_Y(\alpha)\le 2$ for any curve system $\alpha$ on $\Sigma$. In particular, if $\alpha_0,\dotsc,\alpha_n$ is any path in $\cc(\Sigma)$ such that $\alpha_i\pitchfork Y$ for all $i$, then $d_Y(\alpha_0,\alpha_n)\le 2n$. There is a much stronger result for geodesics in $\cc(\Sigma)$:

\begin{theorem}[Bounded geodesic image \cite{MasurMinsky-heirarchy}]
\label{thm:bounded_geodesic_image}
There is a constant $\bgit>0$ satisfying the following:
Let $\Sigma$ be a domain in $S$ with $\plex{\Sigma}\ge 1$ and let $Y\esubn \Sigma$ be a proper subdomain. If $g$ is a geodesic segment, ray, or bi-infinite line in $\cc(\Sigma)$ with $\pi_Y(\alpha)\ne \emptyset$ for all $\alpha\in g$, then $\diam_Y(g) \le \bgit$.
\end{theorem}

If $\mu=(\base(\mu),t)$ is a marking in $\Sigma$, we define $\pi_Y(\mu) = \pi_Y(\base(\mu))\cup\pi_Y(t)$; that is, $\pi_Y(\mu)$ is obtained by viewing $\mu$ as the set $\{\beta_1,\dotsc,\beta_k,t_1,\dotsc,t_k\}$ of base curves and transversals and projecting all these to $Y$. 
For any marking $\mu\in \markings{\Sigma}$ and domain $Y\esub \Sigma$, one easily finds that
\begin{equation*}
\pi_Y(\mu)\neq\emptyset \quad\text{with}\quad \diam_Y(\mu) = \diam_{\cc(Y)}(\pi_Y(\mu)) \le 6.
\end{equation*}

We shall need the following elementary facts.

\begin{lemma}[{\cite[Lemma 2.12]{BKMM-rigidity}}]
\label{lem:composing_projections}
There is a constant $\nestprojscommute$ so that the following holds for any nested domains $V\esub W\esub \Sigma$ in $S$. Let $\alpha$ denote any curve system or marking on $\Sigma$. Then
 $\pi_V(\alpha)$ is nonempty if and only if $\pi_V(\pi_W(\alpha))$ is nonempty and, moreover, $\diam_{\cc(V)}\big(\pi_V(\alpha)\cup \pi_V(\pi_W(\alpha))\big)\le \nestprojscommute$.
\end{lemma}

\begin{lemma}
\label{lem:build_marking}
For each $C\ge 1$ there exists $C'\ge 1$ with the following property. If $\mu\in \markings{\Sigma}$ is a marking of a subsurface $\Sigma$ in $S$ and $\alpha$ is a curve system on $\Sigma$ with $d_V(\mu,\alpha)\le C$ for all domains $V\esub \Sigma$, then there is a marking $\mu'\in \markings{\Sigma}$ with $\alpha\subset\base(\mu')$ and $d_V(\mu',\mu)\le C'$ for all domains $V\esub \Sigma$.
\end{lemma}
\begin{proof}
We follow the procedure, on page 798, of \cite[\S2.2]{BKMM-rigidity} to project the marking $\mu$ of $\Sigma$ to a new marking $\mu'$ of $\Sigma$. Since $\pi_V(\mu)$ and $\pi_V(\alpha)$ coarsely agree for all domains $V\esub \Sigma$, we may choose the components of $\alpha$ to be base curves in this inductive construction. This amounts to the following: For each curve $a\in \alpha$, let $A$ denote the annulus with $\partial A = a$ and choose a transversal $t_a\in \pi_A(\mu)\subset \cc(A)$. Then for each complementary component $V$ of $\Sigma\setminus\alpha$, project $\mu$ to a marking $\mu_V$ of $V$ as in \cite[\S2.2]{BKMM-rigidity}. These fit together to give a full marking (in the sense of \cite{BKMM-rigidity}) $\mu'$ of $\Sigma$ that we may straighten to be a (clean) marking in our sense. The uniform bound on $d_V(\mu,\mu')$ is then a consequence of \cite[Lemma 2.10]{BKMM-rigidity}
\end{proof}

\subsection{Teichm\"uller space}
\label{sec:teichmuller_space}

The \define{Teichm\"uller space} $\T(\Sigma)$ of a surface $\Sigma$ is the set of isotopy classes of marked hyperbolic structures on $\Sigma$. More precisely, $\T(\Sigma)$ is the space of pairs $(X,f)$, where $f\colon \Sigma\to X$ is a homeomorphism between $\Sigma$ and a complete, finite-area hyperbolic surface $X$, up to the equivalence relation $(X,f)\sim (Y,g)$ if there is an isometry $\Psi\colon X\to Y$ with $\Psi\circ f$ isotopic to $g$. Observe that the mapping class group $\Mod(\Sigma)$ naturally acts on $\T(\Sigma)$ by changing the marking: $\phi\cdot(X,f) = (X,f\circ \phi\inv)$. This action is properly discontinuous and isometric with respect to the \define{Teichm\"uller metric} given by
\[d_{\T(\Sigma)}((X,f),(Y,g)) \colonequals \inf\left\{\tfrac{1}{2} \log \mathrm{QC}(\Phi) \mid {\Phi\sim g\circ f\inv} \right\},\]
where the infimum is over all quasiconformal maps homotopic to $g\circ f\inv$ and $\mathrm{QC}(\Phi)$ denotes the maximum dilatation of $\Phi$. It is known that $d_{\T(\Sigma)}$ is a complete metric on $\T(\Sigma)$ and that the induced topology is homeomorphic to $\R^{\ent{\Sigma}}$ (e.g., see \cite{FM-Primer}).

For a point $x\in \T(\Sigma)$ and curve $\alpha\in \curves{\Sigma}$, we write $\ell_x(\alpha)$ for the length of the geodesic representative of $\alpha$ in the hyperbolic metric $x$. For any given $\epsilon > 0$, the \define{$\epsilon$--thick part} of Teichm\"uller space is defined to be the subset
\[\T_\epsilon(\Sigma) = \{x\in \T(\Sigma) \mid \ell_x(\alpha) \ge \epsilon\text{ for all }\alpha\in \curves{\Sigma}\}\]
consisting of those metrics for which all curves have length at least $\epsilon$. It is known that $\Mod(\Sigma)$ acts cocompactly on $\T_\epsilon(\Sigma)$ for every $\epsilon  > 0$.

Observe that the Teichm\"uller space of a disconnected surface is formally the product of Teichm\"uller spaces of its components, and that the space is empty unless each component has negative Euler characteristic. We extend the definition to allow for annular components as follows:

The conformal structure of an annulus is determined by its modulus, and the usual notion of the Teichm\"uller space of the annulus is accordingly  $\R_+$. 
In this paper we also include curve complex distance and so formally define the \define{Teichm\"uller space of an annulus} $A\esub S$ with core $\alpha = \partial A$ to be the upper half plane
\[\T(A) = \T(\alpha) = \H^2 = \{x+iy \mid y >0\}\]
equipped with the hyperbolic metric $ds^2 = \frac{dx^2 + dy^2}{4y^2}$ of curvature $-4$. In this definition $\nicefrac{1}{y}$ represents the hyperbolic length of the core $\partial A$ and $x$ measures twisting 
about $A$; so the thick part in this case is $\T_\epsilon(A) = \{x+iy \mid 0 < y \le \frac{1}{\epsilon}\}$. However, due to their role in Minsky's product regions (Theorem~\ref{thm:product_regions} below), for annuli we are more concerned about the \define{thin part} $\horoball_\epsilon(\alpha) = \{x+iy \mid y \ge\tfrac{1}{\epsilon}\}$. Note that $\T(A)$ is homeomorphic to $\R^{\ent{A}} = \R^2$.

It is known that $\T(\Sigma)$ is a \emph{unique} geodesic space when $\Sigma$ connected: every pair of points $x,y\in \T(\Sigma)$ are connected by a unique geodesic segment, denoted $[x,y]$. However, $\T(\Sigma)$ is not uniquely geodesic when $\Sigma$ is disconnected, since the Teichm\"uller metric is the supremum of the metrics for the individual factors. In this case, we use $[x,y]$ to denote the product in $\T(\Sigma)$ of the unique geodesics in each factor; note that while this is only a subset, it is possible to build a parameterized geodesic $\gamma\colon [0,d_\T(x,y)]\to \T(\Sigma)$ whose image is contained in $[x,y]$.

\subsection{Product regions}
\label{sec:product_regions}
Given any $\epsilon > 0$ and curve system $\alpha$ on $\Sigma$, write $\mathcal{H}_{\epsilon,\alpha}(\Sigma)= \{x\in \T(\Sigma) \mid \ell_x(a) < \epsilon,\, \forall a\in \alpha\}$ for the set of hyperbolic metrics in which the curves in $\alpha$ are all shorter than $\epsilon$. 
Let $\productregion{\Sigma}{\alpha}$ denote the product space
\[\productregion{\Sigma}{\alpha} := \prod_{V \subset \Sigma\setminus\alpha}\T(V) \times \prod_{a\in\alpha}\T(\alpha),\]
where the first product is over the connected components $V$ of $\Sigma\setminus\alpha$, equipped with the \define{sup metric}  $d_{\productregion{\Sigma}{\alpha}} = \max_{V,a}\{d_{\T(V)},d_{\T(\alpha)}\}$.

By using Fenchel--Nielsen coordinates adapted to the curve system $\alpha$ (see \cite{Minsky-product}), one obtains a natural homeomorphism
\[\productmap{\alpha}\colon \T(\Sigma) \to \productregion{\Sigma}{\alpha}\]
under which the $\T(\alpha)$ component of $\productmap{\alpha}(w)$ is $\tau_\alpha(w) + \frac{i}{\ell_\alpha(w)}$, where $\tau_\alpha$ is the FN twist parameter of $w$ for the curve $\alpha$.
The following foundational result of Minsky says that, for sufficiently small $\epsilon$, the restriction of $\productmap{\alpha}$ to $\mathcal{H}_{\epsilon,\alpha}(\Sigma)$ distorts distances by a bounded \emph{additive} amount:

\begin{theorem}[Minsky, \cite{Minsky-product}]
\label{thm:product_regions}
There exists $\minsky\ge 0$, depending only on $S$, such that the following holds for all sufficiently small $\epsilon>0$: For any domain $\Sigma\esub S$
and any curve system $\alpha$ on $\Sigma$, all pairs of points  $x,y\in \mathcal{H}_{\epsilon,\alpha}(\Sigma)$ satisfy
\[\abs{d_{\T(\Sigma)}(x,y) - d_{\productregion{\Sigma}{\alpha}}(\productmap{\alpha}(x),\productmap{\alpha}(y)) } \le \minsky.\]

Moreover for every component $V\subset \Sigma\setminus \alpha$ and essential curve $\gamma\in \curves{V}$,  the length of $\gamma$ in $x$ and in the $\T(V)$--component of $\productmap{\alpha}(x)$ have ratio in $[\minsky\inv,\minsky]$.
\end{theorem}

\subsection{Volumes and nets}
\label{sec:volumes_and_nets}
The Teichm\"uller space $\T(\Sigma)$ admits a \define{holonomy measure} $\holmeas$ defined as the push forward of the Masur--Veech measure on the space of unit area quadratic differentials; see \cite{Masur-IntervalExchange,Veech-TeichGeodFlow}. Let us write $\ball{x}{r}\subset \T(\Sigma)$ for the metric ball of radius $r >0$ centered at $x\in \T(\Sigma)$. We shall need two facts about $\holmeas$.
Firstly, Eskin and Mirzakhani \cite[Lemma 3.1]{EM-countingClosedGeods} have shown that there exists a constant $\netsep$ such all balls of radius $\netsep \le r \le 2\netsep$ have volume $\holmeas(\ball{x}{r})$ uniformly bounded above and below.
Secondly, Athreya, Bufetov, Eskin and Mirzakhani have calculated the volumes of large balls centered in the thick part:

\begin{theorem}[{\cite[Theorem 1.3]{ABEM}}]
\label{thm:ABEM-volumes}
There exists  a constant $C_1 \ge 1$ such that for every domain $\Sigma\esub S$, thick point $x\in \T_\thin(\Sigma)$ and radius $r > 0$, one has $\holmeas(\ball{x}{r}) \le C_1 e^{\ent{\Sigma}r}$.
\end{theorem}

Note that these results in \cite{ABEM} and \cite{EM-countingClosedGeods} are stated for closed surfaces of genus at least $2$, but the proofs in fact hold for surfaces with $\plex{\Sigma}\ge 2$ (see, e.g., \cite{Maher-lattices}). When $\plex{\Sigma} = \pm 1$ and thus $\ent{\Sigma} =2$, for $\Sigma = S_{0,4}$ or $S_{1,1}$ or an annulus $S_{0,2}$, the Teichm\"uller space $\T(\Sigma)$ is isometric to the hyperbolic plane with constant curvature $-4$ and these calculations are elementary. For annuli $A\esub S$ we will primarily be concerned with thin regions $\horoball_\epsilon(A)$ consisting of points $x$ with $\ell_x(\partial A) \le \epsilon$, whose volumes may be estimated as follows:

\begin{lemma}
\label{lem:thin_annular_horoball_meas}
There  exists $C_2$ so that if $A$ is an annulus and $x\in \T_\thin(A)$ is thick, then the $\epsilon = e^{4\netsep}\thin$ thin region within $r>0$ of $x$ has  
$\holmeas\big(\ball{x}{r}\cap \horoball_{e^{4\netsep}\thin}(A)\big) \le C_2 e^{r}$.
\end{lemma}
\begin{proof}
We may assume that $\ell_x(\partial A) = e^{4\netsep}\thin$ so that $x$ lies on the boundary of the thin region $\horoball_{e^{4\netsep}\thin}(A)$. Indeed. if $\ell_x(\partial A)$ is larger the $\ball{x}{r}$ has strictly smaller intersection with the thin part, and if $\thin \le \ell_x(\partial A) \le e^{4\netsep}\thin$ then we may move $x$ distance at most $2\netsep$ to achieve this.
We use the disc model of hyperbolic space with the metric $\frac{|dz|^2}{(1-|z|^2)^2}$ and take $x$ to be the origin. In polar coordinates $\rho e^{i\phi}$, the ball $\ball{x}{r}$ corresponds to the region $\rho \le \tanh(r)$ and has area $\pi \sinh^2(r)\approx \tfrac{\pi}{4}e^{2r}$. We may take the thin region to be the horoball $H$ enclosed by the circle $\abs{z-\frac{1}{2}}\le \frac{1}{2}$ centered at $\frac{1}{2}$ of Euclidean radius $\frac{1}{2}$.  We want to bound the hyperbolic area of the intersection of $H$ with the disc of Euclidean radius $\rho =\tanh(r)$ centered at $0$.

For a given $\rho$, let find the values of $\theta$ so that $\abs{\tfrac{1}{2} + \tfrac{1}{2}e^{i\theta}} = \rho$. We find
\[\cos \theta = 2 \rho^2 - 1.\]
We may fix $r_0$ sufficiently large so that if $r \ge r_0$, and accordingly $\rho \ge \tanh(r_0)$, we will have $\abs{\theta}\le \tfrac{1}{2}$. Using Taylor's theorem with remainder, we then find that
\[2 \rho^2 - 1 = \cos \theta \le 1 - \tfrac{\theta^2}{2} + \tfrac{\theta^4}{4!} \le  1 - \tfrac{\theta^2}{2} + \tfrac{\theta^2}{4!\cdot 4} \le 1 -\tfrac{\theta^2}{4},\]
which gives $\theta^2 \le 8(1-\rho^2)$. Hence, if $\rho \ge \tanh(r_0)$ and $\rho e^{i\phi}\in H$ we see that $\abs{\phi} \le \abs{\theta}  \le 4 \sqrt{1-\rho^2}$. Thus we may bound the area of the intersection by
\begin{align*}
\mathrm{Area}&\Big(\{\rho\le \tanh(r_0)\}\Big) + \int_{\tanh(r_0)}^{\tanh(r)} \int_{-4\sqrt{1-\rho^2}}^{4\sqrt{1-\rho^2}} \frac{\rho}{(1-\rho^2)^2}d\phi d\rho\\
&= \pi \sinh^2(r_0) + \left.\frac{8}{\sqrt{1-\rho^2}}\right\vert_{\tanh(r_0)}^{\tanh(r)}
= \pi\sinh^2(r_0) - 8\cosh(r_0) + 8\cosh(r),
\end{align*}
which is at most $C_2 e^r$ for some constant $C_2$ depending only on $r_0$.
\end{proof}

\begin{definition}[Fixed Nets]
\label{def:fixed_nets}
For each domain $\Sigma\esub S$, we henceforth fix
a $(\netsep,2\netsep)$--net $\tnet(\Sigma)$ in $\T(\Sigma)$; this is a subset such that the $\netsep$--balls centered on $\tnet(\Sigma)$ are all disjoint but the $2\netsep$--balls cover $\T(\Sigma)$. We additionally write $\tnet_\thin(\Sigma) = \tnet(\Sigma)\cap \T_\thin(\Sigma)$ for the set of thick net points and, for annuli $A\esub S$, write $\horoballnet_\epsilon(A) = \tnet(A)\cap \horoball_\epsilon(A)$ for the set of thin net pints.
\end{definition}

Since $\netsep$ and $2\netsep$ balls have uniformly controlled volumes, one finds, as in equation (17) of \cite{EM-countingClosedGeods}, that the volume of any ball is comparable to the number of net points it contains. Thus Theorem~\ref{thm:ABEM-volumes} and Lemma~\ref{lem:thin_annular_horoball_meas} immediately give

\begin{lemma}
\label{lem:net-point-count}
There is a uniform constant $\netcount > 0$ such that for any domain $\Sigma\esub S$, thick point $x\in \T_\thin(\Sigma)$, and radius $r > 0$, one has $\#\abs{\tnet(\Sigma)\cap \ball{x}{r}} \le \netcount e^{\ent{\Sigma}r}$.
Furthermore, if $\Sigma$ is an annulus, then $\#\abs{\horoballnet_{e^{4\netsep}\thin}(A)\cap \ball{x}{r}} \le \netcount e^{r}$.
\end{lemma}

\subsection{Projecting to curve complexes}
\label{sec:projection_to_curve_complexes}
Every point $x\in \T(\Sigma)$ admits a \define{Bers marking} $\mu_x \in \markings{\Sigma}$ constructed as follows: Greedily choose a shortest pants decomposition $\base(\mu_x)$ on the hyperbolic surface $x$, then choose a shortest-possible transversal $t_i$ for each base curve $\beta_i\in \base(\mu_x)$. There is a universal \define{Bers constant} $\bers > 0$, depending only on $S$, such that any Bers marking $\mu_x$ of any point $x\in \T(\Sigma)$ in the Teichm\"uller space of any non-annular domain $\Sigma\esub S$ satisfies $\ell_x(\beta)\le \bers$ for all $\beta\in \base(\mu_x)$. For an annulus, the Bers marking of a point $x\in \T(A)$ is just 
a pair $(\partial A, t)$ where the transversal $t\in \cc_0(A)$ records the horizontal component of $x$.

Given any domain $V\esub \Sigma$, the \define{projection} of $x\in \T(\Sigma)$ to $\cc(V)$ is defined as 
\[\pi_V(x)\colonequals \bigcup \big\{\pi_V(\mu_x) \mid \text{$\mu_x$ is a Bers marking on $x$}\big\}.\]
The \define{projection distance} in $V$ of a pair of points $x,y\in \T(\Sigma)$ is then defined to be
\[d_V(x,y) \colonequals \diam_{\cc(V)}\big(\pi_V(x)\cup\pi_V(y)\big).\]

This projection is coarsely Lipschitz for nonannuli: There is a constant $\lipconst \ge 1$, depending only on $S$, such that for all domains $\Sigma\esub S$ and $x,y\in \T(\Sigma)$ one has
\begin{equation}
\label{eqn:lipschitz_proj}
d_V(x,y) \le \lipconst d_{\T(\Sigma)}(x,y) + \lipconst\quad\text{for all nonannular subdomains $V\esub \Sigma$}.
\end{equation}
We caution that  $\pi_A\colon \T(\Sigma)\to \cc(A)$ is {\sc not} coarsely Lipschitz for annuli $A\esubn \Sigma$, as is evident from the Distance Formula Theorem~\ref{thm:distance_formula}. However, we at least have the following coarse continuity for \emph{every} domain $V\esub \Sigma$ and point $x\in \T(\Sigma)$:
\begin{equation}
\label{eqn:continuous_proj}
\diam_{\cc(V)}\left(\bigcup_{y\in U} \pi_V(y)\right) \le \lipconst \quad\text{for some neighborhood $U\subset \T(\Sigma)$ of $x$}
\end{equation}
In particular $\diam_{\cc(V)}(\pi_V(x)) = d_V(x,x)\le \lipconst$ for any $x\in \T(\Sigma)$, meaning that all potential Bers markings on $x$ have coarsely the same projection to $\cc(V)$.
For any set $F$ of curves on $\Sigma$ and any $x\in \T(\Sigma)$, we also adopt the notation
\[d_V(F,x) = \diam_{\cc(V)}\big(\pi_V(F)\cup \pi_V(x)\big).\]

\subsection{Alignment}
\label{sec:alignment}
We say that points in Teichm\"uller space are aligned if they do not backtrack when projected to curve complexes of subdomains. More precisely, for $\theta\ge 0$
an $n$--tuple $(x_0,\dots,x_n)$ in $\T(\Sigma)$ is said to be \define{$\theta$--aligned in a domain} $V\esub \Sigma$ if for all 
indices $0 \le i \le j \le k \le n$ we have
\[d_V(x_i,x_j) + d_V(x_j, x_k) \le d_V(x_i, x_k) + \theta.\]
The tuple is moreover \define{$\theta$--aligned} if it is $\theta$--aligned in every subdomain of $\Sigma$.
This leads to the following easy consequence of hyperbolicity:

\begin{lemma}
\label{lem:aligned_proj_to_quasigeod}
For any domains $V\esub \Sigma\esub S$, if a triple $(x,z,y)$ in $\T(\Sigma)$ is $\theta$--aligned in $V$, then $\pi_V(z)$ is contained in the $(\theta/2+4\delta+\lipconst)$--neighborhood of any $\cc(V)$--geodesic connecting $\pi_V(x)$ to $\pi_V(y)$.
\end{lemma}
\begin{proof}
Let $g = (\gamma_0,\dotsc,\gamma_k)$ be any geodesic  with $\gamma_0\in \pi_V(x)$ and $\gamma_k\in \pi_V(y)$. We will show that $\pi_V(z)$ lies within $\theta' = \theta/2 + 4\delta + \lipconst$  of $g$. To this end, choose any $\beta\in \pi_V(z)$ and let $\alpha$ be a closest point to $\beta$ along $g$. If $d_V(\beta,\alpha)\le 2\delta$ we are done. Otherwise, we may choose $\alpha'$ on a geodesic from $\beta$ to $\alpha$ with $d_V(\alpha',\alpha) = 2\delta$. Since $\alpha'$ is not within $\delta$ of $g$, applying $\delta$--hyperbolicity to the geodesic triangles $\triangle(\gamma_0,\alpha,\beta)$ and $\triangle(\gamma_k,\alpha,\beta)$ ensures there are points $\beta_x$ and $\beta_y$ on geodesics $[\gamma_0,\beta]$ and $[\beta,\gamma_k]$, respectively, such that $d_V(\beta_x,\alpha'),d_V(\beta_y,\alpha')\le \delta$. Whence
\begin{align*}
d_V(x,y) &\le 2\lipconst + d_V(\gamma_0,\gamma_k)\\
&\le 2\lipconst + d_V(\gamma_0,\beta_x) + d_V(\beta_x,\beta_y) + d_V(\beta_y,\gamma_k)\\
&\le 2\lipconst + d_V(\gamma_0,\beta) + d_V(\beta,\gamma_k) + 2\delta - (d_V(\beta_x,\beta) + d_V(\beta,\beta_y))\\
&\le 2\lipconst + d_V(x,y) + \theta + 2\delta - (d_V(\beta_x,\beta) + d_V(\beta,\beta_y))
\end{align*}
by (\ref{eqn:lipschitz_proj}) and alignment. Since $d_V(\beta,\alpha) \le d_V(\beta,\beta_x) + 3\delta$ and similarly for $\beta_y$, the claimed inequality now follows:
\[2d_V(\beta,\alpha)  \le d_V(\beta_x,\beta) + d_V(\beta,\beta_y) + 6\delta \le (2\lipconst + \theta + 2\delta) + 6\delta = 2\theta'.\qedhere\]
\end{proof}

The following result of Rafi says that Teichm\"uller geodesics are uniformly aligned:
\begin{theorem}[Rafi, \cite{Rafi-hyp}]
\label{thm:back}
There is a constant $\back$ such that for any domain $\Sigma\esub S$ with $\plex{\Sigma}\ge 1$, any consecutive triple of points $a,b,c$ along a geodesic in $\T(\Sigma)$ is $\back$--aligned, that is: $d_V(a,b) + d_V(b,c) \le d_V(a,c) + \back$ for all $V\esub \Sigma$.
\end{theorem}

This easily implies that Teichm\"uller geodesics project to within bounded Hausdorff distance of geodesics in curve graphs:

\begin{lemma}
\label{lem:qi-geod-projection}
For any subdomains $V\esub \Sigma \esub S$ and any geodesic $[x,y]\subset \T(\Sigma)$, the projection $\pi_V([x,y])\subset \cc(V)$ has Hausdorff distance at most $\back+8\delta+3\lipconst$ from any geodesic $g$  connecting $\pi_V(x)$ to $\pi_V(y)$.
\end{lemma}
\begin{proof}
Let $g = (\gamma_0,\dotsc,\gamma_k)$ be any geodesic in $\cc(V)$ with $\gamma_0\in \pi_V(x)$ and $\gamma_k\in \pi_V(y)$. Theorem~\ref{thm:back} and Lemma~\ref{lem:aligned_proj_to_quasigeod} show that $\pi_V([x,y])$ lies within $B_0 = \back/2 + 4\delta+\lipconst$ of $g$. 
Conversely, the fact that $\pi_V$ is coarsely continuous \eqref{eqn:continuous_proj}  implies there is a subset $\{p_0,p_1,\dotsc,p_n\}\subset \pi_V([x,y])$ with $p_0 = \gamma_0$, $p_n = \gamma_k$ and $d_V(p_i,p_{i+1}) \le 2\lipconst$ for all $0\le i < n$. Let $q_i\in g$ be a closest point to $p_i$. Then $d_V(p_i,q_i)\le B_0$ by the above, and hence $d_V(q_i,q_{i+1}) \le 2B_0 + 2\lipconst$. Since $q_0 = \gamma_0$ and $q_n = \gamma_k$, we see that $g$ lies in the $B_0 + \lipconst$ neighborhood of the set $\{q_0,\dots,q_n\}$. The claim now follows.
\end{proof}

\begin{corollary}
\label{cor:aligned_near_geod}
If a triple $(x,z,y)$ in $\T(\Sigma)$ is $\theta$--aligned in a domain $V\esub \Sigma$, then there exists $z'\in [x,y]$ so that $d_V(z,z')\le \theta/2+\back+12\delta+4\lipconst$.
\end{corollary}

\subsection{Strong alignment}

While our above notion of alignment (\S\ref{sec:alignment}) only concerns curve complex data, the construction of complexity length in \S\ref{sec:Complexity length} will require aligned tuples in which lengths of short curves vary roughly convexly. The precise condition is as follows:

\begin{definition}
\label{def:strong_alignment}
We say that a tuple $(x_0,\dots,x_n)$ in $\T(\Sigma)$ is \define{strongly $\theta$--aligned}, where $\theta\ge 1$, if it is $\theta$--aligned and for every domain $V\esub  \Sigma$, there exist points $x_0^V,\dots,x_n^V$ appearing in order along $[x_0,x_n]$ such that for each $0\le i \le n$ we have 
\begin{itemize}
\item $d_V(x_i, x_i^V) \le \theta$, and 
\item if $V$ is an annulus then $\min\{\thin, \ell_{x_i}(\partial V)\} / \min\{\thin, \ell_{x_i^V}(\partial V)\}$ lies in $[\frac{1}{\theta},\theta]$.
\end{itemize}
\end{definition}

This is in fact only a minor strengthening of alignment, in that any aligned tuple may be superficially modified to achieve it:

\begin{strong-align-lem}
For any $\theta\ge 1$ there exists $\theta' \ge \theta$ with the following property:
For any $\theta$--aligned tuple $(x_0,\dots,x_n)$ in $\T(\Sigma)$, there is a strongly $\theta'$--aligned tuple $(y_0,\dots,y_n)$ such that $x_0 = y_0$, $x_n = y_n$, and such that for all $0 \le i \le n$ we have
\[d_V(x_i, y_i) \ladd_\theta 0\quad\text{for every domain $V\esub \Sigma$}.\]
\end{strong-align-lem}

While the proof is not difficult, the formulation of strong alignment---involving separate tuples for all domain---leads to a somewhat technical and involved argument. As such, we defer the proof until \S\ref{sec:promoting_strong_alignmet} below. 
We note that, along the way, Lemma~\ref{lem:ordered_aligned_comp_pts} will show the first bulleted condition of strong alignment is redundant, in that any aligned tuple automatically satisfies it for some larger constant.

\subsection{Barycenters}
\label{sec:barycenters}

We call $b\in \T(\Sigma)$ a \define{$\theta$--barycenter} for a triple $y,x,z\in \T(\Sigma)$ and domain $V\esub \Sigma$ if each triple $(y,b,x)$, $(x,b,z)$ and $(z,b,y)$ is $\theta$--aligned in $V$. If this holds for every domain, we simply call $b$ a $\theta$--barycenter for the triple. This concept is closely related to the \define{Gromov product} in the domain $V$:
\begin{equation}
\label{eqn:gromov_product}
(y\vert z)_x^V = \frac{1}{2}\left(d_V(x,y) + d_V(x,z) - d_V(y,z)\right) \le \min\{d_V(x,y), d_V(x,z)\}.\end{equation}

Indeed, we record the following simple relationship.

\begin{lemma}
\label{lem:barycenter_vs_gromov_product}
If $b$ is a $\theta$--barycenter for a triple $y,x,z$ in a domain $V\esub S$, then $\abs{d_V(x,b) - (y\vert z)_x^V}\le \theta$ (and the same for each permutation of $x,y,z$). If $b'$ is also a $\theta$--barycenter for a triple $y, x', z$, then $d_V(b,b') \le 4\theta+16\delta+5\lipconst+d_V(x,x')$.
\end{lemma}
\begin{proof}
For the first claim, simply observe that 
\begin{align*}
d_V(x,b) = \frac{1}{2}(d_V(x,b) + d_V(b,y) + d_V(x,b) + d_V(b,z) - d_V(y,b) - d_V(b,z)).
\end{align*}
By $\theta$--alignment of $(x,b,y)$ and $(x,z,b)$,
the right hand side is at most $(y\vert z)_x^V+\theta$. Similarly,  by $\theta$--alignment of $(y,b,z)$, it is at least $(y\vert z)_x^V - \tfrac{\theta}{2}$.

For the second claim, it follows directly from the definition~(\ref{eqn:gromov_product}) and the triangle inequality that $\abs{(x\vert z)_y^V - (x'\vert z)_{y}^V} \le d_V(x,x')$. Hence, by the first claim, $\abs{d_V(y,b) - d_V(y,b')}$ is at most $2\theta+d_V(x,x')$. 
On the other hand Lemma~\ref{lem:aligned_proj_to_quasigeod} implies $\pi_V(b)$ and $\pi_V(b')$ both lie within $\tfrac{\theta}{2}+4\delta+\lipconst$ of a fixed $\cc(V)$--geodesic from $\pi_V(y)$ to $\pi_V(z)$. Since their distances to the endpoint $\pi_V(y)$ differ by at most $2\theta+d_V(x,x')$, it follows that $d_V(b, b')$ is at most 
$4(\tfrac{\theta}{2}+4\delta+\lipconst)+\lipconst+2\theta+d_V(x,x')$.
\end{proof}

\subsection{Active intervals}
\label{sec:active_intervals}

The subsurface projection values $d_V(x,y)$ are closely  related to the manner in which the Teichm\"uller geodesic $[x,y]$ interacts with product regions. This relationship is conveyed by the following results of Rafi.

In \cite[Theorem 3.1]{Rafi-combinatorial}, Rafi proves that for all sufficiently small $\epsilon > 0$ there exists $\epsilon > \epsilon' > 0$ such that for any curve $\alpha\in \curves{\Sigma}$, every Teichm\"uller geodesic $[x,y]$ in $\T(\Sigma)$ has a possibly empty subinterval $I$ such that $\ell_z(\alpha)\le\epsilon$ for all $z\in I$ and $\ell_z(\alpha)\ge \epsilon'$ for all $z\in [x,y]\setminus I$. For a subsurface $Y\esub \Sigma$, intersecting the intervals for the components of $\partial Y$ (see \cite[Corollary 3.4]{Rafi-combinatorial}) produces a corresponding interval for $Y$ that enjoys the following property:

\begin{theorem}[Rafi {\cite[Proposition 3.7]{Rafi-combinatorial}}]
\label{thm:rafi_active}
For each sufficiently small $\epsilon > 0$, there exist constants $0 < \epsilon' < \epsilon$ and $\rafi_\epsilon\ge 0$ such that for any domain $\Sigma\esub S$, any Teichm\"uller geodesic $[x,y]\in \T(\Sigma)$, and any subdomain $V\esub \Sigma$, there is a (possibly empty) connected interval $\tilde\interval^\epsilon_V\subset [x,y]$ such that
\begin{enumerate}
\item\label{interval-thin} $\ell_z(\alpha) < \epsilon$ for all  $z\in \tilde\interval^\epsilon_V$ and all $\alpha\in \partial V$.
\item\label{interval-fat} for all $z\in[x,y]\setminus \tilde\interval_V^\epsilon$, some component $\beta$ of $\partial V$ has $\ell_z(\beta) \ge \epsilon'$.
\item\label{interval-projections} $d_V(w,z)\le \rafi_\epsilon$ for every subinterval $[w,z]\subset [x,y]$ with $[w,z]\cap \tilde\interval_V^\epsilon = \emptyset$.
\end{enumerate}
\end{theorem}

\begin{remark}
Notice that item (\ref{interval-thin}) implies $\tilde\interval^\epsilon_V$ lies in the thin region $\mathcal{H}_{\epsilon,\partial V}(\Sigma)$ and hence, via Minsky's theorem, projects to a path in the $\T(V)$--factor of the product region $\productregion{\Sigma}{\partial V}$.
In a later result \cite[Theorem 5.3]{Rafi-hyp}, Rafi produces for \emph{non-annular} domains a related interval, defined in terms of expanding annuli, that satisfies (\ref{interval-thin}) and (\ref{interval-projections}) and additionally fellow travels a unit-speed Teichm\"uller geodesic in $\T(V)$. Since we will not need this latter property and must deal with annuli, we work the more rudimentary intervals from \cite{Rafi-combinatorial} instead.
\end{remark}

Recall that there is a universal \define{Margulis constant} such that on any complete hyperbolic surface, every pair of curves with hyperbolic length at most this value are disjoint. Hence, for small $\epsilon$, if domains $U,V\esub \Sigma$ have $\partial V\cut \partial U$ then Theorem~\ref{thm:rafi_active}(\ref{interval-thin}) implies $\tilde \interval^\epsilon_V$ and $\tilde\interval^\epsilon_U$ are disjoint. But it may be that $\tilde\interval^\epsilon_V$ and $\tilde\interval^\epsilon_U$ overlap when $V\cut U$. To correct this, we will use a slight variation of $\tilde \interval^\epsilon_V$.

\begin{definition}[Uniform constants]
\label{def:uniform_constants}
Fix once and for all a constant $\thin > 0$ smaller than the Margulis constant and small enough for Theorems~\ref{thm:product_regions} and \ref{thm:rafi_active} to hold for $\thin$. Let $\thinprime<\thin$ be the companion constant in Theorem~\ref{thm:rafi_active}. 
Define
\[\rafi = 100(\rafi_{\thin}+\delta+\lipconst + \back+\bgit+\nestprojscommute +\consist),\]
where $\rafi_{\thin}$ is from Theorem~\ref{thm:rafi_active}, $\delta$ is from Theorem~\ref{thm:hyp_curve_cplx}, $\lipconst$ is from \eqref{eqn:lipschitz_proj}, $\back$ is from Theorem~\ref{thm:back}, $\bgit$ is from the Bounded Geodesic Image Theorem~\ref{thm:bounded_geodesic_image}, $\nestprojscommute$ is from Lemma~\ref{lem:composing_projections}, and $\consist$ is from the consistency Theorem~\ref{Thm:Consistency} below.
\end{definition}

\begin{definition}[Active Intervals]
\label{def:active_intervals}
For any geodesic $[x,y]\esub \T(\Sigma)$, where $\Sigma\esub S$ is a domain, define the \define{active interval} of a subdomain $V\esub \Sigma$ along $[x,y]$ as follows:
\begin{itemize}
\item For $V$ annular,  $\interval_V = \tilde \interval^{\thin}_V$.
\item For $V$  nonannular,  $\interval_V$ is the intersection of all subintervals $[a,b]\subset [x,y]$ satisfying both $d_V(x,a) \le 2\rafi_{\thin}+5\lipconst$ and $d_V(b,y) \le 2\rafi_{\thin}+5\lipconst$.
\end{itemize}
\end{definition}

Our active intervals have the following properties: 
\begin{lemma}
\label{lem:our_active_intervals}
Let $[x,y]$ be a geodesic in $\T(\Sigma)$ and $V\esub \Sigma$ a subdomain. Then
\begin{enumerate}
\item\label{ourinterval-nonempty}
If $d_V(x,y) \ge \rafi$, then $\interval_V$ is a nonempty, nondegenerate subinterval of $[x,y]$.
\item\label{ourinterval-thin} 
$\interval_V\subset \tilde\interval^{\thin}_V$, and for  all $z\in \interval_V$ we have that $\ell_z(\alpha) < \thin$ for each component $\alpha$ of $\partial V$ and that $\partial V \subset \base(\mu_z)$ for every Bers marking $\mu_z$ of $z$.
\item\label{ourinterval-smallproj}
If $[w,z]$ is a subinterval of $[x,y]\setminus \interval_V$, then $d_V(w,z)< \rafi/3$.
\item\label{ourinterval-disjoint}
\label{ourinterval-subdomain-disjoint}
If $Y\esub \Sigma$ is a domain with $Y\cut V$, then $\interval_Y\cap \interval_V = \emptyset$.
Moreover, if $\interval_Y \ne\emptyset$, then $\interval_U\cap \interval_V = \emptyset$ for every subdomain $U\esub Y$ with $\partial U\cut V$. 
\end{enumerate}
\end{lemma}
\begin{proof}
For (\ref{ourinterval-nonempty}), if $V$ is annular then $\interval_V = \tilde\interval^{\thin}_V$ by definition; hence $\interval_V$ is an interval, and it being either empty or degenerate would imply $d_V(x,y)\le 2\rafi_\thin+\lipconst < \rafi$ by Theorem~\ref{thm:rafi_active}(\ref{interval-projections}) and (\ref{eqn:continuous_proj}). 
Next suppose $V$ is nonannular. Since $d_V(x,y)\ge \rafi$ and projections change coarsely continuously (\ref{eqn:continuous_proj}), we may find a nondegenerate subinterval $[w,z]\subset[x,y]$ so that the distances $d_V(x,w)$ and $d_V(z,y)$ are both within $\lipconst$ of $2\rafi_\thin+7\lipconst+\back$. Then every point $u\in [w,z]$ satisfies $d_V(x,u) \ge d_V(x,w)-\back > 2\rafi_\thin+6\lipconst$ by Theorem~\ref{thm:back}, and similarly $d_V(u,y)\ge 2\rafi_\thin + 6\lipconst$. Thus each subinterval $[a,b]$ in Definition~\ref{def:active_intervals} contains $[w,z]$. As an intersection of intervals that contain $[w,z]$, $\interval_V$ is thus indeed a nonempty, nondegenerate interval.

For (\ref{ourinterval-thin}), first observe that $\interval_V\subset\tilde\interval^\thin_V$; for annuli this is by definition, and for nonannuli it holds since Theorem~\ref{thm:rafi_active}(\ref{interval-projections}) implies $\tilde\interval^\thin_V$ qualifies as one of the intervals in the intersection defining $\interval_V$. Now for $z\in \interval_V$, Theorem~\ref{thm:rafi_active}(\ref{interval-thin}) ensures that $\ell_z(\alpha) < \thin$ for each component $\alpha$ of $\partial V$. If $\mu_z$ is a Bers marking at $z$ whose pants decomposition $\base(\mu_z)$ fails to contain some component $\alpha\in \partial V$, then we must have $\alpha\cut\beta$ for some $\beta\in\base(\mu_z)$. Since $\base(\mu_z)$ is a \emph{shortest} pants decomposition on $z$, it must be that $\ell_z(\beta)\le \ell_z(\alpha) < \thin$. Since $\thin$ is smaller than the Margulis constant, this forces $\alpha$ and $\beta$ to be disjoint; a contradiction. Therefore (\ref{ourinterval-thin}) holds.

For annuli, (\ref{ourinterval-smallproj}) follows immediately from Theorem~\ref{thm:rafi_active}(\ref{interval-projections}). If $V$ is not an annulus, then $\interval_V\cap [w,z]=\emptyset$ can only occur if $[w,z]$ is disjoint from some interval $[a,b]$, as in Definition~\ref{def:active_intervals}, of the intersection yielding $\interval_V$. Without loss of generality, we may assume $[w,z]$ is contained in $[x,a]$. Two applications of Theorem~\ref{thm:back} then give $d_V(w,z)\le d_V(x,a)+2\back \le \rafi_\thin + 5\lipconst+2\back < \rafi/3$.

For (\ref{ourinterval-disjoint}), we may assume $\interval_Y$ is nonempty, for else the needed conclusions are immediate or vacuous.
Thus it suffices to prove the `moreover' conclusion for a subdomain $U\esub Y$ satisfying $\partial U\cut V$ (which could possibly be nested in $V$), since then we may apply it with $Y = U$.
If $V$ is an annulus, then the assumption $\partial U\cut V$ reduces to $\partial U\cut \partial V$. Similarly if $Y$ is an annulus then necessarily $U=Y$ and now $Y\cut V$ reduces to $\partial U\cut \partial V$. In either case, the needed conclusion $\interval_U\cap \interval_V=\emptyset$ would follow from (\ref{ourinterval-thin}) and $\thin$ being chosen smaller than the Margulis constant. 
Thus we may assume that neither $V$ nor $Y$ is an annulus.
Finally, we also assume $\interval_U$ is nonempty, for else there is nothing to prove, and write it as $\interval_U = [a,b]$.

Since $\partial V$ projects to $Y$ and all points of $\tilde\interval^\thin_V$ contain $\partial V$ in their Bers marking by Theorem~\ref{thm:rafi_active}(\ref{interval-thin}) (and the choice of $\thin$), we observe that 
\begin{equation*}
\label{eqn:bounded_needed_for_active_intrval_lem}
d_Y(w,z) \le 2\lipconst\quad\text{for all }w,z\in \tilde\interval^\thin_V.
\end{equation*}
If $\tilde\interval^\thin_Y$ were contained in $\tilde\interval^\thin_V$, then Theorem~\ref{thm:rafi_active}(\ref{interval-projections}), together with (\ref{eqn:continuous_proj}), would evidently imply $d_Y(x,y)\le 2\rafi_\thin+4\lipconst$. Thus $[x,x]$ and $[y,y]$ would both be valid intervals in the intersection from Definition~\ref{def:active_intervals} that, since $Y$ is nonannular, defines $\interval_Y$. As $\interval_Y$ is nonempty, $\tilde\interval^\thin_Y$ therefore cannot be contained in $\tilde\interval^\thin_V$. Hence we may choose some point $w\in \tilde\interval^\thin_Y$ outside of $\tilde\interval^\thin_V$ and suppose, without loss of generality, that $w$ lies in the same component of $[x,y]\setminus \tilde\interval^\thin_V$ as $x$, so that $d_V(x,w)\le\rafi_\thin$ by Theorem~\ref{thm:rafi_active}(\ref{interval-projections}). The points $b\in \interval_U$ and $w\in \tilde\interval^\thin_Y$ now respectively contain $\partial Y$ and $\partial U$ in their Bers markings by (\ref{ourinterval-thin}) and Theorem~\ref{thm:rafi_active}(\ref{interval-thin}). Since $\partial U$, $\partial Y$ are disjoint and both project to $V$, it follows using the triangle inequality that
\[d_V(x,b) \le d_V(x,w) + d_V(w,b) \le \rafi_\thin+ d_V(\partial Y,\partial U) \le \rafi_\thin +2.\]
By coarse continuity (\ref{eqn:continuous_proj}), this shows that $d_V(x,b')\le \rafi_\thin+3\lipconst$ for some point $b'$ immediately to the right of $b$. This means $[b',y]$ qualifies for the intersection defining $\interval_V$ (recall that we have supposed $V$ is nonannular). Therefore $\interval_V$ is contained in $[b',y]$ and hence disjoint from $[a,b] = \interval_Y$, as desired.
\end{proof}

This also leads to the following analog of the bounded geodesic image theorem, which roughly says that if Teichm\"uller points $x,y$ have a large projection to a domain $Z$, then in any other domain $V$ that cuts $\partial Z$, the $\cc(V)$--geodesic from $\pi_V(x)$ to $\pi_V(y)$ must pass near $\partial Z$:

\begin{corollary}[BGIT for Teichm\"uller space]
\label{cor:bgit_for_teich}
For all domains $Z,V\esub\Sigma\esub S$, if $Z$ has a nonempty active interval $\interval_Z$ along a geodesic $[x,y]$ in $\T(\Sigma)$, so in particular if $d_Z(x,y)\ge \rafi$, then $d_V(x,\partial Z) + d_V(\partial Z, y) \le d_V(x,y)+\rafi/3$.
\end{corollary}
\begin{proof}
This is contentless when $\partial Z$ is disjoint from $V$, since in that case $d_V(w,\partial Z) = \diam_{\cc(V)}(\pi_V(w))\le \lipconst$ for all $w$.
By Lemma~\ref{lem:our_active_intervals}(\ref{ourinterval-thin}), the hypothesis provides a point $z\in \interval_Z\subset[x,y]$ whose Bers markings all contains $\partial Z$. Hence Theorem~\ref{thm:back} gives
\[d_V(x,\partial Z) + d_V(\partial Z, y) \le d_V(x,z) + d_V(z,y) \le d_V(x,y)+\back< d_V(x,y)+\rafi/3.\qedhere\]
\end{proof}

\subsection{Time order}
The fact that cutting domains necessarily have disjoint active intervals (Lemma~\ref{lem:our_active_intervals}(\ref{ourinterval-disjoint})) allows for the following definition:

\begin{definition}[Time order]
\label{def:time_order}
Given a Teichm\"uller geodesic $[x,y]$ in $\T(\Sigma)$ and a pair $U,V\esub \Sigma$ of domains, we write $U \tol V$ or $V \tog U$,
and say  \define{$U$ is time-ordered before $V$ along $[x,y]$}, to mean that $U\cut V$ and that $\interval_U$ and $\interval_V$ are nonempty along $[x,y]$ with $\interval_U$ occurring before $\interval_V$ when traveling from $x$ to $y$.
\end{definition}

While the geodesic in question and its orientation are both omitted from our notation, we will strive to make these clear from context so that the meaning of $U \tol V$ is unambiguous in our discussion of time-ordering. The following characterization of time-ordering follows immediately from Lemma~\ref{lem:our_active_intervals}:

\begin{lemma}[Characterizing time-order]
\label{lem:characterize_timeorder}
Let $[x,y]$ be a geodesic in $\T(\Sigma)$ and let $U,V\esub \Sigma$ be subdomains with $U\cut V$. Then $U \tol V$ implies $d_V(x, \partial U) < \rafi/3$ and $d_U(y,\partial V) < \rafi/3$. Accordingly, if $d_U(x,y)$ and $d_V(x,y)$ are at least $\rafi$, then the following are equivalent:\\
\begin{tabular}{lll}
\emph{(1)} $U \tol V$, & 
\emph{(2)} $d_V(x, \partial U) < \rafi/3$, & 
\emph{(4)} $d_U(x,\partial V) \ge 2\rafi/3$, \\ 
 & 
\emph{(3)} $d_V(\partial U, y) \ge 2\rafi/3$,&
\emph{(5)} $d_U(\partial V, y) < \rafi/3$.
\end{tabular}
\end{lemma}

\begin{corollary}[Triple time-order and relative cutting]
\label{cor:time-order_and_severing}
If $U,V,W\esub \Sigma$ are subdomains with $d_V(x,y)\ge \rafi$ and $U \tol V \tol W$ along $[x,y]$, then $U\cut_V W$.
\end{corollary}
\begin{proof}
If the conclusion fails, we may find subdomains $U'\esub U$ and $W'\esub W$ that intersect $V$ but have $\partial U'\disj \partial W'$. In particular, $\pi_V(\partial U')\ne\emptyset\ne\pi_V(\partial W')$.  Since every curve system projects to a set of diameter at most $2$ in $\cc(V)$, this implies
\[d_V(\partial U, \partial W) \le d_V(\partial U, \partial U') + d_V(\partial U', \partial W') + d_V(\partial W', \partial W) \le 6 < \rafi/3.\]
With Lemma~\ref{lem:characterize_timeorder}, this gives $d_V(x,y) \le 2\rafi/3 + 6$, contradicting $d_V(x,y)\ge \rafi$.
\end{proof}

\begin{corollary}[Time-ordering subsurfaces]
\label{cor:time-order_subsurf}
Suppose  $U,V\esub \Sigma$ satisfy  $U \tol V$ along a geodesic $[x,y]$ in $\T(\Sigma)$. If $d_V(x,y)\ge \rafi$, 
then $U' \tol V$ for all  $U'\esub U$  with $\interval_{U'}\neq\emptyset$ and $U'\cut V$. Symmetrically, $d_U(x,y)\ge \rafi$ implies $U \tol V'$ for all $V'\esub V$ with $\interval_{V'}\neq\emptyset$ and $U\cut V'$.
\end{corollary}
\begin{proof}
Let $U'\esub U$ be a subdomain with $\interval_{U'}\ne\emptyset$ and $U'\cut V$. If $U \tol V \tol U'$, then Corollary~\ref{cor:time-order_and_severing} would give $U\cut_V U'$. As this is clearly false, we must have $U' \tol V$ as claimed. The proof for $V'\esub V$ is similar.
\end{proof}

We also have the following basic observation.

\begin{lemma}
\label{lem:time-order-boundary_distance}
Let $U,V\esubn Y \esub \Sigma$ be domains with $U\tol V$ along a geodesic $[x,y]\in \T(\Sigma)$. Then $d_Y(x,\partial U) \le d_Y(x,\partial V) + \rafi/3$.
\end{lemma}
\begin{proof}
Choose points $u\in \interval_U$ and $v\in \interval_V$ and note these respectively contain $\partial U$, $\partial V$ in their Bers markings. Since $U\tol V$ forces $u\in [x,v]$, Theorem~\ref{thm:back} implies
\[d_Y(x,\partial U) \le d_Y(x,u) \le d_Y(x,v)+\back \le d_Y(x,\partial V) + \lipconst + \back.\qedhere\]
\end{proof}

\subsection{Distance estimates}
The following  distance formula of  Rafi \cite{Rafi-combinatorial} says that Teichm\"uller distance can be estimated, with controlled multiplicative and additive error, in terms of projection distances:

\begin{theorem}[Distance Formula; {\cite{Rafi-combinatorial}}] 
\label{thm:distance_formula}
For each sufficiently large threshold $T$, there exists $K\ge 1$ such that for every domain $\Sigma\esub S$ and all $x,y\in \T_\thin(\Sigma)$,
\[\frac{1}{K}d_{\T(\Sigma)}(x,y) - K \le \sum_{V} \left[d_V(x,y)\right]_T + \sum_{A} \left[\log d_A(x,y)\right]_T \le K d_{\T(\Sigma)}(x,y) + K,\]
where the first summand is over all non-annular domains $V\esub \Sigma$, the second over all annular domains $A\esub S$, and where $\left[w\right]_T$ equals $0$ when $w< T$ and otherwise equals $w$. Moreover, the rightmost inequality above holds for all $x,y\in \T(\Sigma)$. 
\end{theorem}

While the full strength of this result is generally off limits to us, as we cannot afford multiplicative errors, we do make frequent use of the following consequence:
\begin{equation}
\label{eqn:bdd_projections_implies_bdd_Teich}
x,y\in \T_\thin(\Sigma), d_V(x,y)\le T \text{ for all domains $V\esub S$} \implies d_{\T(\Sigma)}(x,y) \ladd_T 0.
\end{equation}
In the case that $x$ and $y$ have short curves but are close in all curve complexes, we have the following variation based on Minsky's product regions.

\begin{lemma}
\label{lem:thin_with_bdd_projections}
Given $w,z\in \T(\Sigma)$, let $\sigma_w$ (resp. $\sigma_z$) denote the multicurve consisting of all curves on $w$ (resp. $z$) of length at most $\thin$. 
Partition $\sigma_w = \delta_w\sqcup \gamma_w$ into those curves $\delta_w$ that are disjoint from $\sigma_z$ and those $\gamma_w$ that cut $\sigma_z$. Partition $\sigma_z = \delta_z\sqcup \gamma_z$ similarly.  
Set $R = \max\{R_1,R_2\}$ where
\begin{align*}
R_1 = \max_{\alpha\in \delta_w\cup\delta_z} \frac{1}{2}\abs{\log \frac{\min\{\ell_w(\alpha),\thin\}}{\min\{\ell_z(\alpha),\thin\}}},\;\;
R_2=\max_{\substack{\alpha\in \gamma_w,\beta\in \gamma_z \\ \alpha \cut \beta}} \frac{1}{2}\left(\log\frac{\thin}{\ell_w(\alpha)}+\log\frac{\thin}{\ell_z(\beta)}\right).
\end{align*}
\begin{enumerate}
\item If $k\ge 1$ is such that $\frac{1}{k}\le \frac{\ell_w(\alpha)}{\ell_z(\alpha)}\le k$ for all $\alpha\in \sigma_w\cup\sigma_z$, then $R\le \log k$.
\item  $d_{\T(\Sigma)}(w,z)\gadd R$
\item If $d_V(x,y)\leq T$ for all domains $V\esub \Sigma$ then $d_{\T(\Sigma)}(w,z))\ladd_T R$.
\end{enumerate}
\end{lemma}
\begin{proof}
Suppose the hypothesis of (1) holds. For every $\alpha\in \delta_w\cup\delta_z$ we then have $\min\{\ell_w(\alpha),\thin\} \le \min\{k\ell_z(\alpha),\thin\} \le k \min\{\ell_z(\alpha),\thin\}$, so that the $\alpha$--term in the max defining $R_1$ is at most $\frac{1}{2}\log k$. Hence $R_1\le \log k$. Now consider $\alpha\in \gamma_w$ and $\beta\in \gamma_z$ with $\alpha\cut\beta$. Since $\alpha$ and $\beta$ cannot both be short on $z$, we have $\ell_z(\alpha)\ge\thin$ and the hypothesis gives $\ell_w(\alpha)\ge \thin/k$. Similarly $\ell_z(\beta)\ge \thin/k$. Hence $R_2 \le \log k$.

For (2), if $x,y\in \T(\Sigma)$ are two points for which $\ell_x(\alpha),\ell_y(\alpha)\le\thin$, then Minsky's Product Regions Theorem~\ref{thm:product_regions} implies 
\[d_{\T(\Sigma)}(x,y) +\minsky \ge \frac{1}{2}\abs{\log\frac{\ell_x(\alpha)}{\ell_y(\alpha)}}\]
since $\productmap{\alpha}(x),\productmap{\alpha}(y)\in \productregion{\Sigma}{\alpha}$ have at least this distance in the $\T(\alpha)$--factor of the product. Now let $\alpha\in \delta_w\cup\delta_z$ realize the max in $R_1$. If $\alpha\in \sigma_w\cap \sigma_z$ then the above inequality implies $d_{\T(\Sigma)}(w,z)\ge R_1 - \minsky$. If $\alpha\in \sigma_w$ but $\alpha\notin \sigma_z$, then we may choose a point $y\in [w,z]$ with $\ell_y(\alpha) = \thin$ and apply the above to get $d_{\T(\Sigma)}(w,z)\ge d_{\T(\Sigma)}(w,y) \ge R_1- \minsky$. The symmetric reasoning applies if $\alpha\in \sigma_z\setminus \sigma_w$. Now consider any $\alpha\in \gamma_w$ and $\beta\in \gamma_z$ with $\alpha\cut \beta$. The geodesic $[w,z]$ has to lengthen $\alpha$ to at least $\thin$ before the intersecting curve $\beta$ can become short; hence there are ordered points $w',z'$ 
 along $[w,z]$ with $\ell_{w'}(\alpha) = \ell_{z'}(\beta) = \thin$. Now the above observation bounds $d_{\T(\Sigma)}(w,w')+\minsky$  below by $\frac{1}{2}\log\frac{\thin}{\ell_w(\alpha)}$ and symmetrically for $d_{\T(\Sigma)}(z',z)$. Adding these together and taking a max over all such $\alpha,\beta$ proves $d_{\T(\Sigma)}(w,z)\ge R_2 - 2\minsky$.

For (3), 
we first claim there is a constant $k$ so that $\ell_z(\alpha)\le k$ for all $\alpha \in \delta_w\setminus \sigma_z$. To see this, let $z'$ be the thick point obtained by lengthening every curve in $\sigma_z$ to have length $\thin$ (this can be done in Fenchel--Nielsen coordinates by, for example, adjusting the vertical component of $\productmap{\sigma_z}(z)$ in each $\H^2$ factor of the product region $\productregion{\Sigma}{\sigma_z}$). This adjustment changes neither subsurface projections nor the lengths of curves in $\delta_w\setminus \sigma_z$, since any such $\alpha$ is disjoint from and unequal to the curves in $\sigma_z$ whose lengths are modified in the adjustment. Do the same to build $w'$. Then $d_V(w',z')\ladd T$ for all $V$ so \eqref{eqn:bdd_projections_implies_bdd_Teich} implies $d_{\T(\Sigma)}(w',z')\ladd_T 0$. 
Now for $\alpha\in \delta_w\setminus \sigma_z$ we have $\ell_z(\alpha)$, $\ell_{z'}(\alpha)$, and $\ell_{w'}(\alpha) = \thin$ coarsely agree by construction of $z'$ and the fact that $w',z'$ have bounded distance. Hence $\ell_z(\alpha)$ is bounded, proving the claim.

By moving $w$ and $z$ a bounded distance, we may now assume all curves $\alpha\in \delta_w\setminus \sigma_z$ have $\ell_z(\alpha) = \thin$ and similarly that all $\beta \in \delta_z\setminus \sigma_w$ have $\ell_w(\beta) = \thin$. Therefore, letting $\mathcal{Y}$ be the multicurve $\delta_w\cup \delta_z$, we have $\ell_w(\alpha),\ell_z(\alpha)\le \thin$ for all $\alpha \in \mathcal{Y}$. Hence by Minsky's theorem, $d_{\T(\Sigma)}(w,z)$ agrees up to additive error with the distance between $\productmap{\mathcal{Y}}(w),\productmap{\mathcal{Y}}(z)$ in $\productregion{\Sigma}{\mathcal{Y}} = \T(\Sigma\setminus \mathcal{Y})\times \prod_{\mathcal{Y}}\T(\alpha)$. Since the projections $d_\alpha(w,z)$ are bounded, the distance in the $\prod_{\mathcal{Y}}\T(\alpha)$ factor is by definition coarsely given by $R_1$.

We compute the distance in $\T(\Sigma\setminus \mathcal{Y})$. Using the identification of this factor with the product $\productregion{\Sigma\setminus\mathcal{Y}}{\gamma_w}$, simultaneously lengthen all the curves in $\gamma_w$ until they achieve length $\thin$; for each $\alpha\in \gamma_w$ this takes distance $\frac{1}{2}\log\frac{\thin}{\ell_w(\alpha)}$. Parameterizing this path as $x(t)\in \productregion{\Sigma}{\mathcal{Y}}$ for $t>0$, let us write $\gamma_w^t\subset \gamma_w$ for the curves that are still shorter than $\thin$ at time $t$. The same argument as in the claim above shows that if $\beta\in \gamma_z$ is disjoint from $\gamma_w^t$, then $\beta$ has uniformly bounded length at $x(t)$. Hence as soon as $\beta\in \gamma_z$ becomes disjoint from $\gamma_w^t$ we may begin shortening $\beta$ until it has length $\ell_z(\beta)$, which takes time $\frac{1}{2}\log \frac{\ell_{x(t)}(\beta)}{\ell_z(\beta)} \ladd \frac{1}{2}\log\frac{\thin}{\ell_z(\beta)}$. In this way we build a path of length, up to additive error, $R_2$ from $\productmap{\mathcal{Y}}(w)$ to a new point $z'$ so that $\ell_{z'}(\beta) = \ell_z(\beta)$ for all $\beta\in \gamma_z$. As this procedure does not change subsurface projections, we see that $z',\productmap{\mathcal{Y}}(z)\in \productregion{\Sigma}{\mathcal{Y}}$ have bounded distance in the $\T(\Sigma\setminus\mathcal{Y})$--factor. (To see this, look in the product $\productregion{\Sigma\setminus\mathcal{Y}}{\gamma_z}$ and note that for each component $V$ of $\Sigma\setminus(\mathcal{Y}\cup\gamma_z)$ these points are thick in $\T(V)$ and hence close by (\ref{eqn:bdd_projections_implies_bdd_Teich}); further they are close in $\T(\beta)$ for $\beta\in \gamma_z$ by construction). Thus the distance from $\productmap{\mathcal{Y}}(w)$ to $\productmap{\mathcal{Y}}(z)$ in the $\T(\Sigma\setminus\mathcal{Y})$--factor is equal up to additive error to $R_2$.
\end{proof}

\subsection{Consistency}
By ``undoing'' the projection maps $\pi_V$, one may use curve complex data of subsurfaces to build points in  Teichm\"uller space. This is accomplished by the work of Behrstock--Kleiner--Minsky--Mosher \cite{BKMM-rigidity} on consistency.

\begin{definition}[Consistency]
\label{def:consistency}
Given a number $\theta\ge 1$ and a connected surface $\Sigma$, we say that a tuple $(z_V)\in \prod_{V\esub \Sigma} \cc(V)$ is \define{$\theta$--consistent} if the following holds for all pairs of domains $U,V\esub \Sigma$:
\begin{center}\begin{tabular}{rcll}
(1) & $U\cut V$ & $\implies$& $\min\big\{ d_U(z_U,\partial V),\, d_V(z_V,\partial U)\big\}\le \theta,\text{ and}$
\vspace{5pt}\\
(2) & $U\esub V$ & $\implies$ & $\min\big\{d_U(z_U, \pi_U(z_V)),\, d_V(z_V, \partial U)\big\} \le \theta$.
\end{tabular}\end{center}
(Observe that if $\pi_U(z_V) = \emptyset$ in (2), then $z_V$ is disjoint from $U$ and $\partial U$ so that $d_V(z_V, \partial U) \le 1\le \theta$ is automatic).
\end{definition}

The following result says that, up to bounded error, the consistent tuples in $\prod_{V\esub \Sigma} \cc(V)$ are exactly those obtained by projecting points in the Teichm\"uller space $\T(\Sigma)$. It was proven for the case of markings as Lemmas 4.1--4.2 and Theorem 4.3 of \cite{BKMM-rigidity}. However, since every marking $\mu\in \markings{\Sigma}$ may be realized as the Bers marking of some thick point, the result holds for Teichm\"uller space as well:

\begin{theorem}[Consistency and Realization \cite{BKMM-rigidity}]
\label{Thm:Consistency} 
There is a constant $\consist\ge 1$ and function $\csistfun\colon \R_+\to \R_+$ so that the following holds for every domain $\Sigma\esub S$: 
\begin{itemize}
\item For every $x\in \T(\Sigma)$, the projection tuple $(\pi_V(x))_{V\esub \Sigma}$ is $\consist$--consistent.
\item Conversely, every $\theta$--consistent tuple $(z_V)\in \prod_{V\esub \Sigma} \cc(V)$ has a \define{realization} point $z\in \T(\Sigma)$ with $d_V(\pi_V(z),z_V)\le \csistfun(\theta)$ for all domains $V\esub \Sigma$. In fact we may assume $z\in \tnet_\thin(\Sigma)$ is a thick net point. 
\end{itemize}
\end{theorem}

Using consistency, one can easily see that the length of a curve $\alpha\in \curves{\Sigma}$ at a point $x\in \T(\Sigma)$ is related to the projection distances $d_V(x,\alpha)$ for domains $V\esub \Sigma$:

\begin{lemma}
\label{lem:projection_bound_yields_length_bound}
For any non-annular domain $\Sigma\esub S$, curve $\alpha\in \curves{\Sigma}$, and point $x\in \T_\thin(\Sigma)$, if $d_V(x,\alpha)\le k$ for every domain $V\esub \Sigma$, then $\ell_x(\alpha)\ladd_k 0$.
\end{lemma}
\begin{proof}
Let $Y\esub \Sigma$ be a component of $\Sigma\setminus \alpha$. Define a tuple $(z_V)\in \prod_{V\esub Y} \cc(V)$ by $z_V = \pi_V(x)$ for each $V\esub Y$. This tuple $(z_V)$ is $\consist$--consistent by Theorem~\ref{Thm:Consistency}, and hence is realized by a point $x_Y\in \T_\thin(Y)$. Do this for each component of $\Sigma\setminus \alpha$, and choose a point $x_\alpha\in \T(\alpha)$ on the horocycle $y = \nicefrac{1}{\thin}$ and with twist parameter so that $\pi_\alpha(x_\alpha) = \pi_\alpha(x)$. These choices define a point in the product $z' \in \productregion{\Sigma}{\alpha}$, and we let $z = \productmap{\alpha}\inv(z')\in \T(\Sigma)$ be the corresponding point under Minsky's homeomorphism. Note that $z$ is thick by construction and has $\ell_z(\alpha) = \thin$.

We claim $d_V(x,z)\ladd_k 0$ for every domain $V\esub \Sigma$. Indeed, if $V\disj \alpha$ then either $V$ is the annulus with core $\alpha$ or $V\esub Y$ for some component $Y$ of $\Sigma\setminus \alpha$; in either case $\pi_V(z) = \pi_V(x_Y)$ coarsely agrees with $\pi_V(x)$ by construction. If instead $V\cut \alpha$, then the fact that $\alpha$ is in every Bers marking at $z$ implies $d_V(z,\alpha)\le\lipconst$. Thus $d_V(x,z) \le d_V(x,\alpha) + d_V(\alpha, z) \le k + \lipconst$. This proves the claim and accordingly bounds $d_{\T(\Sigma)}(x,z)$ by (\ref{eqn:bdd_projections_implies_bdd_Teich}) since $x$ and $z$ are thick. Since $\alpha$ has length $\thin$ at $z$, it follows that $\ell_x(\alpha)$ is bounded.
\end{proof}

\begin{corollary}
\label{cor:bounded_domains_with_bounded_combinatorics}
For any $k\ge 1$, domain $\Sigma\esub S$, and point $x\in \T_\thin(\Sigma)$, we have
\[\#\big\{Z\esub \Sigma \mid d_V(x,\partial Z)\le k\text{ for all domains }V\esub \Sigma\big\} \ladd_k 0\]
\end{corollary}
\begin{proof}
By Lemma~\ref{lem:projection_bound_yields_length_bound}, there is a number $L$ such that if $Z$ satisfies $d_V(x,\partial Z)\le k$ for all $V\esub \Sigma$, then each component $\alpha$ of $\partial Z$ has $\ell_x(\alpha)\le L$. On any hyperbolic surface $y$, there are only finitely many  curves of length at most $L$. Varying $y$ over a compact fundamental domain for the action of $\Mod(\Sigma)$ on $\T_\thin(\Sigma)$, we obtain a number $F = F(k,\thin)$ so that every point $x\in \T_\thin(\Sigma)$ has at most $F$ curves of length at most $L$. Thus the number of subsurfaces $Z\esub \Sigma$ whose boundary curves have length at most $L$ on $x$ is bounded in terms of $F$.
\end{proof}

We also need the following result of Brock--Bromberg--Canary--Minsky \cite{BBCM} (which in fact can be obtained using consistency and product regions).

\begin{theorem}(\cite{BBCM})
\label{thm:BCCM}
There exists a constant $C$ (such as $\csistfun(\consist)$) such that the following holds for every domain $\Sigma\esub S$ and multicurve $\alpha$ on $\Sigma$: For each $x\in \T(\Sigma)$ there is a point $\hat{x}\in \T(\Sigma)$ so that:
$\ell_{\hat{x}}(a) \le \thin$ for all $a\in \alpha$, and
 $d_V(x,\hat{x})\le C$ for every domain $V\esub \Sigma$ that is disjoint from $\alpha$; in particular, this applies if $V$ is an annulus whose core is an element of $\alpha$.
\end{theorem}

\section{Preliminaries --  antichains, strong alignment, and branch points}
\label{Sec:Antichains and branch points}

\subsection{Antichains}
\label{sec:antichains}
If $\Sigma$ is a domain in $S$ and $\Omega$ is a collection of subdomains of $\Sigma$, we typically write $\underline{\Omega}$ for the set of topologically maximal domains in $\Omega$ (that is, maximal with respect to the partial order $\esub$ on $\Omega$). Taking active intervals into account, for each geodesic $[x,y]$ in $\T(\Sigma)$ we may also consider  the partial order $\prec_{[x,y]}$ on domains in $\Sigma$ defined by
\[V \prec_{[x,y]} W \iff V \esub W\text{ and }\interval_V\subset\interval_W\text{ along }[x,y].\]
We then write $\underline{\Omega}_x^y$ for the set of domains in $\Omega$ that are maximal with respect to $\prec_{[x,y]}$. Since the order $\prec_{[x,y]}$ is more restrictive than $\esub$, we note that $\underline{\Omega}\subset\underline{\Omega}_x^y$.

For any collection $\Omega$ of domains and given integer $i$, we additionally write
\[\abs{\Omega}_i = \#\{V\in \Omega \mid \plex{V} = i \}\]
for the number of domains in $\Omega$ with complexity $i$. 
The following is a variation of Rafi and Schleimer's bound on the cardinality of an antichain \cite[Lemma 5.1]{RafiSchleimer}. As the statement we need does not follow directly from the result in \cite{RafiSchleimer}, we include a  proof in the same spirit as their argument.

\begin{lemma}
\label{lem:threshholds}
Consider a domain $\Sigma\esub S$ and a sequence of thresholds $T_{j}$, for $j\in \{-1,\dotsc, \plex{\Sigma}+1\}$, satisfying $\rafi\le T_{\plex{\Sigma}+1}\le T_{\plex{\Sigma}}\le \dotsb \le T_{-1}$.
 For any geodesic $[x,y]$ in $\T(\Sigma)$ and any domain $W\esub \Sigma$, the collections $\underline{\mathcal{P}}(W)$ and $\underline{\mathcal{P}}_x^y(W)$ of topologically maximal and $\prec_{[x,y]}$--maximal domains in the set 
\[\mathcal{P}(W) = \{V\esub W \mid d_V(x,y) \ge T_{\plex{V}}\}\]
satisfy  $\abs{\underline{\mathcal{P}}(W)}_{j}\le\abs{\underline{\mathcal{P}}_x^y(W)}_{j}\le (3T_{j+1})^{\plex{W}-j}$ for each $j\in \{-1,\dotsc,\plex{\Sigma}\}$.
\end{lemma}
\begin{proof}
Fix some $j\in \{-1,\dotsc, \plex{\Sigma}\}$. It suffices to prove the bound on $\abs{\underline{\mathcal{P}}_x^y(W)}_j$. Since $\mathcal{P}(W)$ contains at most one domain (namely $W$) of complexity $\plex{W}$ or greater, it is clear that $\abs{\underline{\mathcal{P}}_x^y(W)}_j = 0$ when $j > \plex{W}$ and that $\abs{\underline{\mathcal{P}}_x^y(W)}_j \le 1 = (3T_{j+1})^0$ when $j = \plex{W}$. We may therefore assume $j < \plex{\Sigma}$ and restrict to domains $W\esub \Sigma$ with $\plex{W} > j$.

Given any domain $W\esub \Sigma$, write $\underline{\Omega}_x^y(W)$ for the set of $\prec_{[x,y]}$--maximal domains in the collection
\[\Omega(W) = \{V\esub W \mid d_V(x,y) \ge T_{j+1}\}.\]
We claim that $\underline{\Omega}_x^y(W)$ contains every domain of $\underline{\mathcal{P}}_x^y(W)$ of complexity $j$. Indeed, if $V\in \underline{\mathcal{P}}_x^y(W)$ has $\plex{V} = j$, then we must have $V\in \Omega(W)$ since $V \esub W$ and $d_V(x,y) \ge T_{j} \ge T_{j+1}$ by definition of $\mathcal{P}(W)$. Thus if $V\notin \underline{\Omega}_x^y(W)$, it must be that $V\prec_{[x,y]} Z$ for some distinct $Z\in \Omega(W)$. This implies $V\esubn Z$, and hence $\plex{Z} \ge j+1$ and  $T_{\plex{Z}}\le T_{j+1}$ by monotonicity of the thresholds. Therefore, the fact $Z\in \Omega(W)$ gives $Z\esub W$ and $d_Z(x,y) \ge T_{j+1} \ge T_{\plex{Z}}$. But this implies $Z\in \mathcal{P}(W)$, contradicting the maximality of $V$ in $\mathcal{P}(W)$. This proves $\abs{\underline{\mathcal{P}}_x^y(W)}_j\le \abs{\underline{\Omega}_x^y(W)}_j$. It thus suffices to prove $\abs{\underline{\Omega}_x^y(W)}_j \le (3T_{j+1})^{\plex{W}-j}$ for every geodesic $[x,y]$ and domain $W\esub \Sigma$ with $\plex{W} > j$.

The proof of this proceeds by induction on the complexity $k = \plex{W}$ of the domain $W$: For each $k = j,\dotsc,\plex{\Sigma}$, we will prove that 
$\abs{\underline{\Omega}_x^y(W)}_j \le (3T_{j+1})^{k - j}$ for every domain $W$ with $\plex{W} \le k$. We have already seen that this bound is immediate for $k = j$.  

Let us therefore fix $k > j$ and assume that $\abs{\underline{\Omega}_x^y(Z)}_j \le (3T_{j+1})^{k-1-j}$ whenever $\plex{Z} < k$. Let $W\esub \Sigma$ be any domain with $\plex{W}\le k$ and choose curves $\alpha\in \pi_W(x)$ and $\beta\in \pi_W(y)$ realizing the distance $d_{\cc(W)}(\alpha,\beta) = d_W(x,y)$. Fix also a geodesic $\alpha = \gamma_0,\dotsc,\gamma_m = \beta$ joining $\alpha$ to $\beta$ in $\cc(W)$. 

We claim that every domain $V\in {\Omega}(W)$ with $V\ne W$ is disjoint from one of the curves $\gamma_i$. This is immediate if $\alpha$ or $\beta$ is disjoint from $V$.
 Otherwise, letting $\mu_x$ be a Bers marking on $x$ with $\alpha\in \pi_W(\mu_x)$, Lemma~\ref{lem:composing_projections} implies that $\pi_V(\alpha)\subset \pi_V(\pi_W(\mu_x))$ lies within $\nestprojscommute$ of $\pi_V(\mu_x)$. Hence $d_V(\alpha,x) \le \nestprojscommute+2\lipconst$ and similarly for  $d_V(\beta,y)$. Thus
\[d_V(\alpha,\beta) \ge d_V(x,y) - 2\nestprojscommute - 4\lipconst \ge T_{j+1} \ge \rafi-2\nestprojscommute - 4\lipconst > \bgit\]
by the specification of $\rafi$ in Definition~\ref{def:uniform_constants}. We may now invoke the Bounded Geodesic Image Theorem~\ref{thm:bounded_geodesic_image} to conclude $\pi_V(\gamma_i) = \emptyset$ for some $i$, as claimed. 

We moreover claim that if $V\in \underline{\Omega}_x^y(W) \subset\Omega(W)$ is disjoint from $\gamma_i$, then
\[i < \tfrac{1}{2}T_{j+1}\qquad\text{or}\qquad m-i < \tfrac{1}{2}T_{j+1}.\]
Indeed, if this is not the case then necessarily $d_W(x,y) = m \ge T_{j+1}$. Therefore $W\in \Omega(W)$ by definition. For any point $v\in \interval_V$, the Bers marking $\mu_v$ contains $\partial V$, which is disjoint from $\gamma_i$. Hence $d_W(v,\gamma_i)\le \lipconst +2\le 3\lipconst$. On the other hand
\[d_W(x,\gamma_i) \ge d_W(\alpha,\gamma_i) = i\qquad\text{and}\qquad d_W(y,\gamma_i) \ge d_W(\beta,\gamma_i) = m-i.\]
As these quantities are both at least $\tfrac{1}{2}T_{j+1}$, we therefore see that
\[d_W(x,v),d_W(y,v)  \ge \tfrac{1}{2}T_{j+1} - 3\lipconst \ge \tfrac{1}{2}\rafi - \tfrac{1}{7}\rafi > \tfrac{1}{3}\rafi.\]
But by Lemma~\ref{lem:our_active_intervals}, this is only possible if $v\in \interval_W$. Thus we evidently have $\interval_V\subset \interval_W$, contradicting the $\prec_{[x,y]}$--maximality of $V$ in $\Omega(W)$.

Therefore every $V\in \underline{\Omega}_x^y(W)$ satisfies $V = W$ or else $V\esub Z$ for some $Z$ in the set
\[\mathcal{Z} = \{ Z \mid \text{$Z$ is a component of $W\setminus \gamma_i$ for some $i$ with $\max\{i,m-i\} < \tfrac{1}{2}T_{j+1}$}\}.\]
Further, since $\Omega(Z)\subset\Omega(W)$, every $V\esub Z$ that is $\prec_{[x,y]}$--maximal in $\Omega(W)$ is also $\prec_{[x,y]}$--maximal in $\Omega(Z)$. Thus we have
\[\underline{\Omega}_x^y(W)\subset \big\{W\big\}\cup \bigcup_{Z\in \mathcal{Z}}\underline{\Omega}_x^y(Z).\]
In particular, $\abs{\underline{\Omega}_x^y(W)}_j \le \sum_{Z\in \mathcal{Z}}\abs{\underline{\Omega}_x^y(Z)}_j \le \abs{\mathcal{Z}}(3T_{j+1})^{k-1-j} \le (3T_{j+1})^{k-j}$ by our induction hypothesis and the fact that $\abs{\mathcal{Z}}\le 4(1 + \tfrac{1}{2}T_{j+1})\le 3T_{j+1}$. This concludes the induction and the proof of the lemma.
\end{proof}

\subsection{Promoting alignment to strong alignment}
\label{sec:promoting_strong_alignmet}
Here we prove Lemma~\ref{lem:building_strong_alignment} and show that aligned tuples may be transformed into strongly aligned ones by adjusting lengths of certain curves and without affecting curve complex projections.

Throughout this subsection we fix a $\theta$--aligned tuple $(x_0,\dots,x_n)$ in $\T(\Sigma)$. 
We first determine which curves require adjusted lengths. To this end, we say $x_i$ is \define{$c$--leftward} for an annulus $A$ if $\ell_{x_0}(\partial A) \ge \thinprime$ and  $d_A(x_0,x_i)\le c$. Symmetrically, $x_i$ is \define{$c$--rightward} for $A$ if $\ell_{x_n}(\partial A) \ge \thinprime$ and $d_A(x_i,x_n)\le c$. If $x_i$ is neither $c$--leftward nor $c$--rightward, we say it is \define{$c$--central} for $A$.

\begin{lemma}
\label{lem:central-implies-interval}
\label{lem:length_constraints_disjoint}
Suppose $c\ge \theta+2\rafi$ and that $x_i$ is $c$--central for annuli $A$ and $B$. Then $\partial A$ and $\partial B$ are disjoint.
\end{lemma}
\begin{proof}
First observe that $A$ has a nonempty active interval $\interval_A = \tilde\interval_A^\thin$ along $[x_0,x_n]$.
If $\ell_{x_0}(\partial A)$ or $\ell_{x_n}(\partial A)$ is less than $\thinprime$, then $\tilde\interval_A^\thin$ is nonempty by Definition~\ref{def:active_intervals} and Theorem~\ref{thm:rafi_active}. Otherwise $d_A(x_i,x_0), d_A(x_i,x_n)> c$ so that alignment implies $d_A(x_0,x_n) \ge 2\rafi+\theta$, showing that $\interval_A$ is nonempty by Lemma~\ref{lem:our_active_intervals}. Similarly $\interval_B$ is nonempty. If $\partial A \cut \partial B$, then they are time-ordered and without loss of generality we may assume $A \tol B$. Thus $d_A(\partial B, x_n), d_B(x_0, \partial A) \le \rafi/3$ by 
Lemma~\ref{lem:characterize_timeorder}. Evidently $\ell_{x_n}(\partial A) \ge \thinprime$, since $x_n\notin \tilde\interval_A^\thin$; hence $c$--centrality forces $d_A(x_i, x_n) > c$. Similarly $d_B(x_0, x_i) > c$. Since $c\ge \theta+2\rafi$ 
it follows that
\[\left.\begin{array}{c}
d_A(x_i,\partial B) \ge d_A(x_i,x_n) - d_A(\partial B,x_n)\\
d_B(\partial A, x_i) \ge d_B(x_0, x_i) - d_B(\partial A, x_0)
\end{array}\right\} \ge \theta+2\rafi - \rafi/3 \ge \rafi > \consist.\]
But this contradicts Theorem~\ref{Thm:Consistency} and Definition~\ref{def:uniform_constants}.
\end{proof}

Lemma~\ref{lem:length_constraints_disjoint} implies that the core curves of annuli for which $x_i$ is $c$--central form a (possibly empty) multicurve. The next lemma says this multicurve is close to the Bers marking $\mu_{x_i}$ in all subsurfaces.

\begin{lemma}
\label{lem:length_constraints_nearby}
If $x_i$ is $(\theta+2\rafi)$--central for $A$, then $d_V(x_i,\partial A) \ladd_\theta 0$ for every $V\esub \Sigma$.
\end{lemma}
\begin{proof}
Let $V\esub \Sigma$ be arbitrary. It suffices to suppose $\partial A$ projects to $V$, for else $d_V(x_i,\partial A) = \diam_{\cc(V)}(\pi_V(x_i))\ladd 0$. Since $A$ has a nonempty thin interval along, we may choose a point $y\in [x_0,x_n]$ such that $\ell_y(\partial A) <\thin$. Thus $\partial A$ is contained in every Bers marking at $y$. If $d_V(x_0,x_n)\le \rafi$, then two applications of the triangle inequality followed by $\theta$--alignment and Theorem~\ref{thm:back} imply
\begin{align*}
2d_V(\partial A, x_i) &\le 2d_V(y,x_i) \le d_V(x_0, y) + d_V(y, x_n) + d_V(x_0, x_i) + d_V(x_i, x_n)\\
 &\le 2 d_V(x_0, x_n) + \back + \theta \le 2\rafi + \back + \theta,
\end{align*}
as desired. Hence it remains to suppose $d_V(x_0,x_n) > \rafi$, which ensures $V$ has a nonempty active interval $\interval_V$ along $[x_0,x_n]$.
 
First suppose $A\cut V$ and, by symmetry, that $A \tol V$ along $[x_0,x_n]$. Then evidently $x_n\notin \interval_A$ and so (as in Lemma~\ref{lem:length_constraints_disjoint}) the centrality hypothesis implies $d_A(x_i,x_n)\ge \theta+2\rafi$. Time order also gives $d_A(\partial V, x_n)\le \rafi/3$. Therefore we have $d_A(x_i, \partial V) \ge \theta+2\rafi-\rafi/3 > \consist$. Consistency of the point $x_i$ (Theorem~\ref{Thm:Consistency}) now implies the desired bound $d_V(x_i,\partial A) \le \consist$. 

Otherwise we necessarily have $A\esub V$. By Corollary~\ref{cor:aligned_near_geod}, we may choose a point $z'\in [x_0,x_n]$ with $d_V(x_i,z')\le \theta+\rafi$.
Since $\partial A \in \pi_V(w)$ for all $w\in \interval_A$, the entire interval projects to a set of diameter at most $2\lipconst$ in $\cc(V)$. In particular, we cannot have $\interval_A = [x_0,x_n]$, as that would put us in the case  $d_V(x_0,x_n)\le 2\lipconst < \rafi$ dispensed with above. Thus we may by (\ref{eqn:continuous_proj}) choose some $z\in [x_0,x_n]\setminus \interval_A$ with $d_V(z,z') \le 3\lipconst$. 
The point $z$ either lies before or after $\interval_A$, let us suppose it is the former (the opposite possibility being symmetric). Then Lemma~\ref{lem:our_active_intervals} gives $d_A(x_0,z)\le \rafi/3$. This also implies $x_0\notin \tilde\interval_A^\thin=\interval_A$ (since $\interval_A$ connected) and hence $d_A(x_i, x_0)\ge \theta+2\rafi$ since $x_i$ is not $(\theta+\rafi)$--leftward for $A$.  Combining these we find $d_A(x_i,z) >\rafi$. 
By Corollary~\ref{cor:bgit_for_teich}, we now conclude the desired bound
\[d_V(x_i,\partial A) 
\le d_V(x_i,z)+\rafi/3 \le d_V(x_i,z')+3\lipconst + \rafi/3  \le \theta+2\rafi+3\lipconst.\qedhere\]
\end{proof}

\begin{definition}
By a \define{partition} for a domain $V\esub \Sigma$, we mean a partition $\mathcal{Q}$ of the set $\{1,\dots,n-1\}$.
We consider the following conditions on a partition:
 \begin{itemize}
\item \define{$c$--small}: $d_V(x_i,x_j)\le c$ for each $P\in \mathcal{Q}$ and all $i,j\in P$.
\item \define{order--respecting}: $i\le j \le k$ and $i,k\in P$ for some $P\in \mathcal{Q}$, then $j\in P$ also;
\begin{itemize}
\item in this case, there is a well-defined total order on $\mathcal{Q}$ whereby we have $P\le P'$ for $P,P'\in \mathcal{Q}$ if and only if $i\le j$ for all $i\in P$ and $j\in P'$.
\end{itemize}
\end{itemize}
Another partition $\mathcal{Q}'$ is said to be \define{coarser} than $\mathcal{Q}$ if $\mathcal{Q}$ restricts to a partition of each element of $\mathcal{Q}'$; that is, if for each $P\in \mathcal{Q}$ there exists $P'\in \mathcal{Q}'$ so that $P\subset P'$.
\end{definition}

\begin{lemma}
\label{lem:small-partition-order-preserving}
Let $\mathcal{Q}$ be a $c$--small partition for an domain $V\esub \Sigma$. 
If $c\ge \theta$, then $\mathcal{Q}$ may be coarsened to an order-respecting $4^nc$--small partition $\mathcal{Q}'$ for $V$.
\end{lemma}
\begin{proof}
If $\mathcal{Q}$ is not order-respecting, then there are elements $P_1,P_2\in \mathcal{Q}$ of the partition and indices $i < j < k$ so that $i,k\in P_1$ but $j\in P_2$. Now, by $\theta$--alignment and $c$--smallness, we see that for any indices $\ell\in P_1$ and $m\in P_2$ we have
\[d_A(x_\ell,x_m) \le d_A(x_i,x_j) + 2c \le d_A(x_i,x_k)+2c+\theta\le 3c+\theta \le 4c.\]
Thus if we let $\mathcal{Q}_1$ be the coarsened partition consisting of the set $P_1\cup P_2$ and all other elements of $\mathcal{Q}$, we see that $\mathcal{Q}_1$ is $4c$--small. Repeating this procedure at most $n$ times, we obtain a partition for $A$ that is both order-preserving and $4^nc$--small. 
\end{proof}

\begin{lemma}
\label{lem:partition_balanacing_function}
Let $\mathcal{Q}$ be an order-respecting, $c$--small partition for an annulus $A\esub \Sigma$. Then there exists a function $\beta\colon \mathcal{Q}\to \{-1,0,1\}$ such that for all $P,P'\in \mathcal{Q}$
\begin{itemize}
\item if $P\le P'$ then $\beta(P)\le \beta(P')$ (i.e.~$\beta$ is order-preserving);
\item if $\beta(P)= -1$, then all points $x_j$ with $j\in P$ are $2c$--leftward for $A$;
\item if $\beta(P) = 0$, then all points $x_j$ with $j\in P$ are $c$--central for $A$;
\item if $\beta(P)= 1$, then all points $x_j$ with $j\in P$ are $2c$--rightward for $A$.
\end{itemize}
\end{lemma}
\begin{proof}
Let $i\in \{1,\dots, n-1\}$ be the largest index (if any) so that $x_i$ is $c$--leftward for $A$, and let $P_i\in \mathcal{Q}$ be the element of the partition with $i\in P_i$. If $P_i$ is not the largest element of the ordered set $\mathcal{Q}$, then the largest index $i_0$ in $P_i$  satisfies $i_0 < n-1$ and we may let $k\in \{i_0+1,\dots,n-1\}$ be the smallest index (if any) so that $x_k$ is $c$--rightward for $A$. Similarly let $P_k\in \mathcal{Q}$ be the element containing $k$.
For $P\in \mathcal{Q}$ we now define $\beta(P)$ as follows: If $P\le P_i$, then set $\beta(P) = -1$. If $P_k\le P$, then set $\beta(P) = 1$. Otherwise set $\beta(P) = 0$. 

We verify that $\beta$ satisfies the desired properties:
Since $P_i < P_k$ by construction, this function $\beta\colon\mathcal{Q}\to \{-1,0,1\}$ is order-respecting. For the remaining conditions, fix $P\in \mathcal{Q}$ and any index $j\in P$. 
If $\beta(P) = 0$, then we evidently have $i < j < k$. By construction of $i$ and $k$, it follows that $x_j$ is neither $c$--leftward nor $c$--rightward;  hence $x_j$ is $c$--central for $A$, as required.
Next suppose $\beta(P) = -1$, so that $P\le P_i$. This ensures the index $i$ is defined and that $x_i$ is $c$--leftward for $A$. In particular, $d_A(x_0,x_i)\le c$ and  $\ell_{x_0}(\partial A)\ge \thinprime$. So, to prove $x_j$ is $2c$--leftward it remains to show $d_A(x_0,x_j)\le 2c$. We either have $P = P_i$ or $P < P_i$. 
In the former case $P = P_i$, the fact that $\mathcal{Q}$ is $c$--small implies $d_A(x_j,x_i)\le c$; hence $d_A(x_0,x_j)\le 2c$ by the triangle inequality. And in the latter $P < P_i$, the fact that $\mathcal{Q}$ is order-respecting implies $j  < i$. Thus, by $\theta$--alignment, we have the needed bound.
\[\d_A(x_0,x_j) \le d_A(x_0,x_j) + d_A(x_j,x_i) \le d_A(x_0,x_i) + \theta \le c+\theta\le 2c\]
Finally, when $\beta(P) = 1$, a symmetric argument shows $x_j$ is $c$--rightward.
\end{proof}

Lastly, we show how to produce the ordered points $x_i^V$ along the geodesic $[x_0,x_n]$ satisfying the first condition in the Definition~\ref{def:strong_alignment} of strong alignment.

\begin{lemma}
\label{lem:ordered_aligned_comp_pts}
For each domain $V\esub \Sigma$, there are ordered points $y_0,\dots,y_n$ along $[x_0,x_n]$ so that $d_V(x_i,y_i)\le 4(\theta+\rafi)$ for each $0\le i \le n$.
\end{lemma}
\begin{proof}
Set $w_0 = x_0$ and $w_n = x_n$. Then for each $0<i<n$ use Corollary~\ref{cor:aligned_near_geod} to find a point $w_i\in [x_0,x_n]$ with $d_V(x_i,w_i)\le T$, where $T =\tfrac{1}{2}(\theta+\rafi)$. Now let $\sigma$ be a permutation of $\{0,\dots,n\}$ such that the points $w_{\sigma(0)},\dots,w_{\sigma(n)}$ appear in order along $[x_0,x_n]$ and  set $y_i = w_{\sigma_i}$.  Observe that these satisfy $y_0 = x_0$ and $y_n = x_n$. 

 We fix $0 \le j \le n$ and bound $d_V(x_j, y_j)$. 
By the pigeonhole principle, we may pick some $i \le j$ so that $j \le \sigma\inv(i)$. Similarly, there is some $k\ge j$ so that $\sigma\inv(k) \le j$. Thus the points $w_k = y_{\sigma\inv(k)}$, $y_j$, and $y_{\sigma\inv(i)} = w_i$ appear in order along $[x_0, x_n]$. Since $w_k$ appears before $w_i$ when traveling from $x_0$ to $x_n$, two applications of Theorem~\ref{thm:back} then give
\begin{align*}
d_V(x_0,w_k)& + 3 d_V(w_k,w_i) + d_V(w_i,x_n)\\
&\le d_V(x_0,w_i) + d_V(w_k,w_i) + d_V(w_k,x_n) + 2\back\\
&\le d_V(x_0,x_i) + d_V(x_i,x_k) + d_V(x_k,x_n) + 4T + 2\back,
\end{align*}
where for the last inequality we have used the assumption that $d_V(x_i,w_i)$ and $d_V(x_k,w_k)$ are both at most $T$. By $\theta$--alignment and the triangle inequality, the right hand side above is at most
\begin{align*}
d_V(x_0,x_n) &+ 2\theta + 4T + 2\back \\
& \le d_V(x_0,w_k) + d_V(w_k,w_i) + d_V(w_i,x_n) + 2\theta + 4T + 2\back.
\end{align*}
Subtracting the beginning and end of this string of inequalities now gives
\[d_V(w_k, w_i) \le \theta + \back + 2T.\]
On the other hand, using $\theta$--alignment and Theorem~\ref{thm:back}, we have that
\begin{align*}
2d_V(y_j, x_j)
&\le \big( d_V(y_j, w_i) + d_V(w_i, x_i) + d_V(x_i,x_j)\big)\\ &\;\;\;\; + \big(d_V(y_j, w_k) + d_V(w_k, x_k) + d_V(x_k,x_j)\big)\\
&\le \big(d_V(w_k, y_j) + d_V(y_j, w_i)\big) + 2T + \big(d_V(x_i,x_j) + d_V(x_j, x_k)\big)\\
&\le d_V(w_k, w_i) + \back + 2T + d_V(x_i,x_k) + \theta\\
&\le 2d_V(w_k, w_i) + \back + 4T + \theta
\end{align*}
Combining this with the above yields $d_V(y_j, x_j) \le 4T + \frac{3(\theta+\back)}{2}\le 4(\theta+\rafi)$.
\end{proof}

Using these results, we can now transform any aligned tuple into a strongly aligned one by modifying the lengths of core curves of certain annuli.

\begin{lemma}
\label{lem:building_strong_alignment}
For any $\theta\ge 1$ there exists $\theta' \ge \theta$ with the following property:
For any $\theta$--aligned tuple $(x_0,\dots,x_n)$ in $\T(\Sigma)$, there is a strongly $\theta'$--aligned tuple $(y_0,\dots,y_n)$ in $\tnet(\Sigma)$ such that $x_0 = y_0$, $x_n = y_n$, and such that:
\begin{itemize}
\item $d_V(x_i, y_i) \ladd_\theta 0$ for every domain $V\esub \Sigma$ and $0\le i \le n$.
\item For each annulus $A\esub \Sigma$ and $0 < i < n$ we have $\ell_{y_i}(\partial A) \ge \thin$ unless each $z\in \{x_0,x_n\}$ satisfies the condition: $d_A(x_i,z)\ge 4(\theta+\rafi)$ or $\ell_z(\partial A) < \thin'$.
\end{itemize}
Moreover, if we are given a $\theta$--small partition $\mathcal{Q}_0^A$ for each annulus $A$, then, after increasing $\theta'$ in a way that only depends on $n$, the points $y_i$ may be chosen so that the ratio of $\min\{\thin,\ell_{y_i}(\partial A)\}$ and $\min\{\thin,\ell_{y_j}(\partial A)\}$ lies in $[\tfrac{1}{\theta'},\theta']$ whenever $i,j\in P$ and $P\in \mathcal{Q}_0^A$.
\end{lemma}
\begin{proof}
We may assume $\theta \ge \rafi$. 
Suppose that for each annulus $A\esub \Sigma$ we are given a $\theta$--small partition $\mathcal{Q}_0^A$ for $A$. For example, we may always take the trivial partition of $\{1,\dots,n-1\}$ into singletons, which is automatically $\lipconst$--small and order-respecting. Now apply Lemma~\ref{lem:small-partition-order-preserving} to coarsen to an order-respecting partition $\mathcal{Q}^A$ that is $c$--small, where $c = 4^n\theta + 2\rafi$. Let $\beta^A\colon \mathcal{Q}^A\to \{-1,0,1\}$ be the function provided by Lemma~\ref{lem:partition_balanacing_function}, and let $\bar{\beta}^A\colon \{1,\dots,n-1\}\to \{-1,0,1\}$ be the associated function defined by $\bar{\beta}^A(i) = \beta^A(P)$ where $P\in \mathcal{Q}^A$ is partition element with $i\in P$.
Finally, use Lemma~\ref{lem:ordered_aligned_comp_pts} to obtain ordered points $x^A_0,x_1^A,\dots, x_n^A$ along $[x_0,x_n]$ so that $d_V(x_i,x_i^A)\le 4(\theta+\rafi)\le 8\theta$.

We now define points $y_i^A$ along $[x_0,x_n]$ as follows: Set $y_0^A = x_0$ and $y_n^A = x_n$.  
For each $0 < i < n$, if $\bar{\beta}^A(i) = -1$, then set $y_i^A = x_0$, and if $\bar{\beta}^A(i) = 1$ set $y_i^A= x_n$.
It remains to define $y_i^A$ when $\bar{\beta}^A(i) = 0$, meaning $i$ lies in an element $P\in \mathcal{Q}^A$ with $\beta^A(P) = 0$, For each such $P$, we fix some index $k_P\in P$ and set $y_i^A = x_{k_P}^A$ for all $i\in P$.
Observe that necessarily $\ell_{y_i^A}(\partial A) < \thinprime$. If not, then $y_i^A =x_{k_P}^A$ lies outside the thin interval $\tilde{\interval}_A^\thin = \interval_A$ along $[x_0,x_n]$. Hence we may choose an endpoint $z = x_0$ or $z = x_n$ so that $z$ and $x_{k_P}^A$ lie in the same component of $[x_0,x_n]\setminus \tilde{\interval}_A^\thin$. This implies  $\ell_{z}(\partial A) \ge \thinprime$ (since $z\notin \tilde{\interval}_A^\thin)$ and $d_A(x_{k_P}^A,z) \le \rafi/3 < \theta \le c$. By definition, this means $x_{k_P}$ is either $c$--leftward or $c$--rightward, contradicting that $x_{k_P}$ is $c$--central.

In this way obtain a sequence of points $y_0^A,\dots,y_n^A$ along $[x_0,x_n]$ for each annulus $A\esub \Sigma$ (a separate sequence for each annulus). Let us record the following features:
1) The points $y_0^A,\dots,y_n^A$ appear in order because $\mathcal{Q}^A$ is order respecting.
2) If $i,j\in P$ for some $P\in \mathcal{Q}^A$, then $y_i^A = y_j^A$. 
3) If $\bar{\beta}^A(i) = 0$ then $x_i$ is $c$--central for $A$ and $\ell_{y_i^A}(\partial A) < \thinprime$.
4) We have $d_A(x_i,y_i^A)\le 10c$. This is immediate for $i=0,n$. If $\bar{\beta}^A(i) = 1$, then the fact that $x_i$ is $2c$--rightward implies $d_A(x_i,y_i^A) = d_A(x_i,x_n)\le 2c$. We symmetrically conclude $d_A(x_i,y_i^A)\le 2c$ when $\bar{\beta}^A(i) = -1$. Finally, if $\bar{\beta}^A(i) = 0$, then $d_A(x_i,x_{k_P})\le c$ where $i\in P\in \mathcal{Q}^A$ (since $\mathcal{Q}^A$ is $c$--small) and $d_A(x_{k_P},y_i^A)\le 8\theta\le 8c$. Thus $d_A(x_i,y_i^A)\le 10c$ by the triangle inequality.

For $0 < i < n$, fix a Bers marking $\mu_i$ at $x_i$. Also, let $\alpha_i$ be the set of core curves of annuli $A\esub \Sigma$ such that $\bar{\beta}^A(i)= 0$. Since $x_i$ is $c$--central for each such annulus $A$, Lemma~\ref{lem:length_constraints_disjoint} implies that $\alpha_i$ is a multicurve in $\Sigma$, and Lemma~\ref{lem:length_constraints_nearby} implies that $d_V(\mu_i,\alpha_i)\ladd_c 0$ for every $V\esub \Sigma$. Lemma~\ref{lem:build_marking} therefore provides a new marking $\nu_i$ of $\Sigma$ such that $\alpha_i\subset \base(\nu_i)$ and $d_V(\mu_i,\nu_i)\ladd_c 0$ for every domain $V\esub \Sigma$.

We now construct points in $\T(\Sigma)$ by picking lengths for the curves of $\base(\nu_i)$. Specifically, use Fenchel--Nielsen coordinates to build a point $y_i\in \T(\Sigma)$ such that $\nu_i$ is a Bers marking at $y_i$ and such that for each curve $\gamma\in \base(\nu_i)$: if $\gamma\in \alpha_i$ then declare $\ell_{y_i}(\gamma) = \ell_{y_i^A}(\partial A) < \thinprime$ where $A$ is the annulus with $\partial A = \gamma$, and if not then declare $\ell_{y_i}(\gamma) = \thin$. We note the following property of $y_i$: A curve $\gamma$ satisfies $\ell_{y_i}(\gamma) < \thinprime$ if and only if $\gamma\in \alpha_i$. This is immediate for $\gamma\in \base(\nu_i)$, and otherwise $\ell_{y_i}(\gamma)\ge\thin$ by the Margulis lemma, since $\gamma$ necessarily intersects a curve of $\base(\nu_i)$.
To complete the notation, also set $y_0 = x_0$ and $y_n =x_n$. 

We now show these points $y_0,\dots,y_n$ satisfy the desired conclusions of the lemma. We may then replace $y_1,\dots,y_{n-1}$ with nearby net points to obtain the desired sequence in $\tnet(\Sigma)$ that satisfies the same properties with slightly larger constants. We check the conditions in reverse order: For the ``moreover'' statement,  consider a pair of indices $i,j\in P_0\in \mathcal{Q}_0^A$ for some annulus $A$. Since $\mathcal{Q}^A$ coarsens $\mathcal{Q}_0^A$, it follows that $i,j\in P$ for some $P\in \mathcal{Q}^A$ and hence $y_i^A = y_j^A$ and $\bar{\beta}^A(i) = \bar{\beta}^A(j)$. If this number is $0$, then by construction $\ell_{y_i}(\partial A) = \ell_{y_i^A}(\partial A)$ and $\ell_{y_j}(\partial A) = \ell_{y_j^A}(\partial A)$ are equal. And if this number is $\pm1$ then $\partial A\notin \alpha_i$ and $\partial A\notin \alpha_j$ so that $\ell_{y_i}(\partial A) ,\ell_{y_j}(\partial A) \ge \thin$. 
For the second bullet of the lemma, we have seen that for any annulus $A$ and $0 < i < n$ we have $\ell_{y_i}(\partial A) \ge \thin$ unless $\partial A \in \alpha_i$. By definition of $\alpha_i$, this means $\bar{\beta}^A(i) = 0$ and hence that $x_i$ is $c$--central for $A$. Thus $x_i$ is not $c$--leftward for $A$, meaning that $z = x_0$ satisfies either $\ell_{z}(\partial A) <\thin'$ or $d_A(x_i,z)\ge c > 4(\theta + \rafi)$. Since $x_i$ is also not $c$--rightward, the same conclusion holds for $z = x_n$.
For the first bullet of the lemma, since $\nu_i$ is a Bers marking at $y_i$, we have $d_v(\nu_i,y_i)\ladd 0$ for every domain $V$. Since we also have $d_V(\nu_i,x_i)\ladd_c 0$, it follows that there is a constant $K$ depending only on $c = 4^n\theta+\rafi$ so that $d_V(x_i,y_i)\le K$ for all $i$. 

It remains to show $(y_0,\dots,y_n)$ is indeed strongly aligned. 
Since $(x_0,\dots,x_n)$ is $\theta$--aligned, the triangle inequality now implies $(y_0,\dots,y_n)$ is $(6K+\theta)$--aligned, which is the first part of the Definition~\ref{def:strong_alignment} of strong alignment. By Lemma~\ref{lem:ordered_aligned_comp_pts}, it immediate that each domain $V\esub \Sigma$ has ordered points $z_0^V,\dots,z_n^V$ along $[x_0,x_n]$ so that $d_V(x_i,z_i^V)\le 4(6K+\theta+\rafi)$, which is the next requirement for strong alignment. For the final condition of strong alignment, concerning annuli, we use the ordered points $y_0^A,\dots,y_n^A$. Since $d_A(x_i,y_i^A)\le 10c$ we see that these satisfy $d_A(y_i,y_i^A) \le 10 c + 6K + \theta$. For $i\in \{0,n\}$ we have $y_i= y_i^A$ and thus $\ell_{y_i}(\partial A) =  \ell_{y_i^A}(\partial A)$. And for $0 < i < n$, if $\bar{\beta}^A(i) = 0$, then by construction $\ell_{y_i}(\partial A) = \ell_{y_i^A}(\partial A)$. 
 If instead $\bar{\beta}^A(i) = -1$, then $x_i$ is $2c$--leftward and by construction we have $\ell_{y_i}(\partial A) \ge \thin$ and $y_i^A = x_0$. By definition of $2c$--leftward, this means $\ell_{x_0}(\partial A) \ge \thinprime$. Hence $\min\{\thin,\ell_{y_i^A}(\partial A)\}$ and $\min\{\thin,\ell_{y_i}(\partial A)\} = \thin$ both lie in the interval $[\thinprime,\thin]$, so their ratio is uniformly bounded. The same conclusion holds when $\bar{\beta}^A(i) = 1$.
\end{proof}

\subsection{Branch points}
\label{sec:branch_points}
Fix some subsurface $\Sigma\esub S$. For any triple of points $y,x,z\in \T(\Sigma)$ and domain $V\esub \Sigma$, hyperbolicity of $\cc(V)$ implies that geodesics from $\pi_V(x)$ to $\pi_V(y)$ and $\pi_V(z)$ fellow travel for distance roughly equal to the Gromov product (\ref{eqn:gromov_product})
More precisely, since $\pi_V(\ast)$ has diameter at most $\lipconst$, it is a basic exercise in hyperbolic geometry (using, e.g. \cite[Proposition 2.1]{ABC-notes}) to prove the following:

\begin{lemma}
\label{lem:gromov_product_fellow_travel}
If $\gamma_y$ and $\gamma_z$ are any $\cc(V)$--geodesics from $\pi_V(x)$ to $\pi_V(y)$ and $\pi_V(z)$, then each  $p\in \gamma_y$ with $d_V(\pi_V(x),p)\le (y \vert z)_x^V+c$ lies within $8\delta+4\lipconst+c$ of $\gamma_z$. \qed
\end{lemma}

Hyperbolicity of $\cc(V)$ (again, see \cite{ABC-notes}) additionally implies that
\begin{equation}
\label{eqn:gromov_product_triangle_ineq}
(y\vert z)_x^V \ge \min\{(y\vert w)_x^V, (w\vert z)_x^V\} - 5\delta -3\lipconst\quad\text{for all }x,y,z,w\in \T(\Sigma), V\esub \Sigma.
\end{equation}
These observations allows us to perform the following construction to any collection of geodesics in $\T(\Sigma)$ with a common endpoint:

\begin{lemma}
\label{lem:branch_point_for_many_geodesics}
Consider any points $x,y_1,\dots,y_n\in \T(\Sigma)$. For each domain $V\esub S$, there is a ``branch point'' $\zeta_V\in \cc(V)$ so that $\zeta_V$ lies within $8(\delta+\lipconst)$ of any geodesic from $\pi_V(x)$ to $\pi_V(y_i)$. Further, if $(y_j\vert y_k)_x^V\le \min_{l,m}(y_l\vert y_m)_x^V + c$, then $\zeta_V$ lies within $24(\delta+\lipconst)+c$ of any geodesic from $\pi_V(y_j)$ to $\pi_V(y_k)$.
\end{lemma}
\begin{proof}
For $\beta,\gamma,\mu\in \cc(V)$ let us write $2(\beta\vert \gamma)_\mu = d_V(\mu,\beta) + d_V(\mu,\gamma) - d_V(\beta,\gamma)$.
Fix curves $\alpha\in \pi_V(x)$ and $\alpha_i\in \pi_V(y_i)$ for each $i=1,\dots n$. Now let $G = \min(y_l\vert y_m)_x^V$ and fix indices $l,m$ achieving this minimum. 
Let $\zeta_V\in \cc(V)$ be the point on a geodesic from $\alpha$ to $\alpha_l$ with
\[d(\alpha,\zeta_V) = (\alpha_l\vert \alpha_m)_\alpha \le (y_l\vert y_m)_x^V+\lipconst = G + \lipconst.\]
Therefore $d_V(x,\zeta_V)-2\lipconst \le G$ is smaller than any Gromov product $(y_l\vert y_i)_x^V$ and so Lemma~\ref{lem:gromov_product_fellow_travel} implies $\zeta_V$ lies within $8\delta+6\lipconst$ of any geodesic from $\pi_V(x)$ to $\pi_V(y_i)$. 

Now suppose $(y_j\vert y_k)_x^V\le G+c$ and, by the above, pick a point $\beta$ on a geodesic from $\alpha$ to $\alpha_j$ with $d_V(\zeta_V,\beta)\le 8\delta+6\lipconst$. Then
\[d_V(x,\beta) + 8\delta+7\lipconst\ge d_V(\alpha,\zeta_V) = (\alpha_l\vert \alpha_m)_\alpha \ge G-2\lipconst \ge (y_j\vert y_k)_x^V-2\lipconst-c.\]
Since $d_V(x,y_j) = (y_j\vert y_k)_x^V + (x\vert y_k)_{y_j}^V$ and $\beta$ lies on a geodesic from $\pi_V(x)$ to $\pi_V(y_j)$,  the above implies that
\[d_V(y_j,\beta)
\le d_V(x,y_j) - d_V(x,\beta)+2\lipconst
\le (x\vert y_k)_{y_j}^V +8\delta + 11\lipconst + c.
\]
Thus Lemma~\ref{lem:gromov_product_fellow_travel} implies $\beta$ lies within $16\delta+15\lipconst+c$ of any geodesic from $\pi_V(y_j)$ to $\pi_V(y_k)$. Since $d_V(\zeta_V,\beta)\le 8\delta+6\lipconst$, the claim follows.
\end{proof}

The next step is to show that the tuple $(\zeta_V)$ of branch points is consistent:

\begin{lemma}
\label{lem:branch_tuple_is_consistent}
The tuple $(\zeta_V)\in \prod_{V\esub \Sigma}\cc(V)$ is $7\rafi$--consistent.
\end{lemma}
\begin{proof}
Let us call $V$ ``large'' if $d_V(x,y_i)\ge \rafi$ for all $i=1,\dots,n$ and call $V$ ``small'' otherwise. 
Note that  $d_V(x,\zeta_V) \le 2\rafi$ whenever $V$ is small; this is because
$\zeta_V$ lies within $8(\delta+\lipconst)\le\rafi$ of any geodesic from $\pi_V(x)$ to any $\pi_V(y_i)$ and therefore satisfies $d_V(x,\zeta_V) \le d_V(x,y_i)+\rafi$ for some $i=1,\dots,n$.

Fix two domains $V,W\esub \Sigma$. We must establish the inequalities in Definition~\ref{def:consistency} for the constant $7\rafi$. If $V$ and $W$ are both small, then $d_V(x,\zeta_V),d_W(x,\zeta_W)\le 2\rafi$ and the claim follows from the fact that $(\pi_Z(x))_{Z\esub \Sigma}$ is $\consist$--consistent (Theorem~\ref{Thm:Consistency}). Hence we may assume one of $V,W$ is large. 

First suppose $W\cut V$ with $V$ and $W$ both large, then they are time-ordered along each geodesic $[x,y_i]$. The characterization (Lemma~\ref{lem:characterize_timeorder}) implies that $W \tol V$ along $[x,y_i]$ iff $d_V(x,\partial W) <\rafi/3$. Hence we may suppose $W \tol V$ along each geodesic $[x,y_i]$ (the alternate possibility $V \tol W$ along each $[x,y_i]$ being symmetric). Now time-ordering implies $d_W(y_i,\partial V)\le \rafi/3$ for each $i$. The fact that $\zeta_W$ lies within $\rafi$ of some geodesic from $\pi_W(y_j)$ to $\pi_W(y_k)$, which evidently has length at most $2\rafi/3$, thus implies $d_W(\zeta_W, y_j)\le 2\rafi$ for some $j$. Therefore $d_W(\zeta_W,\partial V)\le d_W(\zeta_W,y_j) + d_W(y_j,\partial V) \le 3\rafi$ and we are done in this case.

Next suppose that $V$ is small and $W$ large with $\partial W$ projecting to $\cc(V)$.
In this case $d_V(x,\zeta_V)\le 2\rafi$ and we may pick $i$ so that $d_V(x,y_i)<\rafi$. 
Since $d_W(x,y_i)\ge \rafi$ mean, Corollary~\ref{cor:bgit_for_teich} implies that
\begin{align*}
d_V(\zeta_V,\partial W) 
\le 2 \rafi + d_V(x, \partial W) + d_V(\partial W, y_i) 
\le d_V(x,y_i) + 3\rafi
< 4\rafi.
\end{align*}
This establishes consistency when $W\esub V$ with $V$ small and $W$ large, and (by symmetry) when $W\cut V$ with at least one domain large.

Finally suppose $W\esub V$ with $V$ large. If $d_V(\partial W, \zeta_V)\le 7\rafi$ we satisfy consistency, so it suffices to assume $d_V(\partial W, \zeta_V) > 7\rafi$. 
Fix curves $\alpha\in \pi_V(x)$ and $\alpha_i\in \pi_V(y_i)$ that each project to $W$ (which we can do by Lemma~\ref{lem:composing_projections}). For each $i$, additionally fix a geodesic $g_i$ in $\cc(V)$ from $\alpha$ to $\alpha_i$ and a point $\beta_i\in g_i$ with $d_V(\zeta_V,\beta_i)\le \rafi$. Fixing indices $j,k$ realizing $\min (y_j\vert y_k)_x^V$, we also take a geodesic $g$ from $\alpha_j$ to $\alpha_k$ that contains a curve $\beta$ with $d_V(\zeta_V,\beta)\le \rafi$. Note that every curve within $6\rafi$ of $\zeta_V$ cuts $W$ (since $\partial W$ is too far away); hence all curves within $5\rafi$ of $\beta$ or any $\beta_i$ also cut $W$. In particular,  Theorem~\ref{thm:bounded_geodesic_image} gives $d_W(\zeta_V,\beta_i),d_W(\zeta_V,\beta) \le \bgit$.

We consider two subcases: Firstly suppose $d_W(\zeta_V,\alpha) > 2\bgit$. Then for each $i$ we have $d_W(\beta_i,\alpha) > \bgit$, which by Theorem~\ref{thm:bounded_geodesic_image} implies that some curve along $g_i$ between $\alpha$ and $\beta_i$ is disjoint from $W$. This means every curve along $g_i$ from $\beta_i$ to $\alpha_i$ cuts $W$; indeed, the curves missing $W$ have diameter $2$ in $\cc(V)$ and all lie distance at least $5\rafi$ from $\beta_i$, thus such curves cannot occur along $g_i$ both between $\alpha$ and $\beta_i$ and between $\beta_i$ and $\alpha_i$. Therefore the Bounded Geodesic Image Theorem gives $d_W(\beta_i,\alpha_i)\le \bgit$. Since $d_W(\alpha_i, y_i)\le \nestprojscommute + \lipconst$ by Lemma~\ref{lem:composing_projections}, we conclude
\[d_W(y_i, \zeta_V) \le d_W(y_i, \alpha_i) + d_W(\alpha_i, \beta_i) + d_W(\beta_i,\zeta_V) \le 2\bgit+\nestprojscommute+\lipconst < \rafi\]
for each $i$. The fact that $\zeta_W$ lies within $\rafi$ of some geodesic from $\pi_W(y_m)$ to $\pi_W(y_l)$, and that this geodesic evidently has length at most $2\rafi$, implies $d_W(\zeta_W,y_m)\le 3\rafi$. Therefore the triangle inequality gives $d_W(\zeta_W,\zeta_V)\le 4\rafi$ as needed.

The final subcase is $d_W(\zeta_V,\alpha) \le 2\bgit$. First observe that either $d_W(\alpha_j,\beta)\le \bgit$ or $d_W(\alpha_k,\beta)\le \bgit$, since otherwise the Bounded Geodesic Image Theorem would imply that $g$ has curves missing $W$ both between $\alpha_j$ and $\beta$ and between $\beta$ and $\alpha_k$. By symmetry let us suppose $d_W(\alpha_j,\beta)\le \bgit$ so that the triangle inequality gives $d_W(\alpha,\alpha_j)\le 4\bgit$. Since $d_W(\alpha,x),d_W(\alpha_j,y_j)\le \nestprojscommute+\lipconst$ by Lemma~\ref{lem:composing_projections}, this gives $d_W(x,y_j)\le \rafi$ and implies $d_W(\zeta_W,x) \le \min_i d_W(x,y_i)+\rafi \le 2\rafi$. Therefore
\[d_W(\zeta_W,\zeta_V)\le d_W(\zeta_W,x)+d_W(x,\alpha) + d_W(\alpha,\zeta_V) \le 2\rafi+\nestprojscommute+\lipconst + 2\bgit \le 3\rafi,\]
which concludes the proof of the Lemma.
\end{proof}

\begin{lemma}[Barycenters]
\label{cor:barycenter}
\label{cor:adjusted barycenter}
There is a constant $\branchal\ge 2\rafi$ such that for any domain $\Sigma\esub S$, 
every ordered triple $y,x,z\in \T(\Sigma)$ has a \define{$\branchal$--barycenter} $b\in \tnet_\thin(\Sigma)$.
Additionally, for any $\theta\ge 2\rafi$ there exist \define{annular-split barycenters} $y',z'\in \tnet(\Sigma)$ such that
$(y, y', x)$, $(x,z',z)$, and $(y,y',z',z)$ are each $\branchal$--aligned, and so that for every domain $V\esub \Sigma$:
\begin{itemize}
\item If $V$ is an annulus and $(y\vert z)_x^V >\theta$, then 
\[d_V(y,y')\le \branchal \quad\text{and}\quad d_V(z',z)\le\branchal.\]
\item Otherwise $\diam_{\cc(V)}\pi_V(\{b,y',z'\})\le \branchal$ with both  $(y,z',x)$ and $(x,y',z)$ $\branchal$--aligned in $V$.
\end{itemize}
\end{lemma}
Thus  $y'$ and $z'$ coarsely agree with the branch point $b$ in all domains, except for certain annuli for which $y'$ and $z'$ instead agree with $y$ and $z$. 
\begin{proof}
Let $(\beta_V)\in \prod_{V\esub \Sigma}\cc(V)$ be the $7\rafi$--consistent branch tuple from Lemmas~\ref{lem:branch_point_for_many_geodesics}--\ref{lem:branch_tuple_is_consistent}. Theorem~\ref{Thm:Consistency} then gives $b\in \tnet_\thin(\Sigma)$ so that $d_V(b,\beta_V)\le \csistfun(7\rafi)$ for every $V\esub \Sigma$. 
Since $(y\vert z)_x^V \le \min\{ (y\vert y)_x^V, (z\vert z)_x^V\} + \rafi/2$ by (\ref{eqn:gromov_product_triangle_ineq}), Lemma~\ref{lem:branch_point_for_many_geodesics} ensures that $\beta_V$ lies within $\rafi$ of any geodesic joining a pair of $\pi_V(x)$, $\pi_V(y)$ and $\pi_V(z)$. These observations imply the triples $(x,b,y)$, $(x,b,z)$, and $(y,b,z)$ are, for example, each $2(\csistfun(7\rafi)+\rafi)$--aligned.  

Now define $(\xi_V),(\zeta_V)\in \prod_{V\esub \Sigma} \cc(V)$ 
so that $\xi_V = \zeta_V= \beta_V$ except for annuli $A$ with $(y\vert z)_x^A> \theta$, in which case instead set $\xi_A = \pi_A(y)$ and $\zeta_A= \pi_A(z)$.

\begin{claim}
$(\xi_V)$ and $(\zeta_V)$ are $R$--consistent, for $R = 2\csistfun(7\rafi)+11\rafi$.
\end{claim}
\begin{proof}
We prove the claim for $(\xi_V)$, as the proof for $(\zeta_V)$ is symmetric.
Since $(\beta_V)$ is $7\rafi$--consistent, it suffices to check consistency for pairs $A,V$ involving a domain $A$ with $d_A(\xi_A,\beta_A)>4\rafi$. In this case $A$ must be an annulus
with $(y\vert z)_x^A > \theta$. 
Since $\beta_A$ lies within $\rafi$ of any geodesic from $\pi_A(y)$ to $\pi_A(z)$, we note that 
\[d_A(\xi_A,\beta_A)=d_A(y,\beta_A) \le d_A(y,\beta_A) + d_A(\beta_A,z) \le d_A(y,z)+2\rafi.\]
Thus $d_A(y,z) > 2\rafi$, and additionally $\min\{d_A(y,x),d_A(x,z)\} > (y\vert z)_x^A >\theta \ge 2\rafi$ by (\ref{eqn:gromov_product}). It follows that at least two of the quantities $d_A(y,b)$, $d_A(x,b)$, $d_A(z,b)$ must be larger than $\rafi$, since otherwise the triangle inequality would bound the minimum of $d_A(y,x)$, $d_A(x,z)$, $d_A(y,z)$ by $2\rafi$. Without loss of generality, let suppose $d_A(y,b),d_A(z,b)> \rafi$, 
in which case Corollary~\ref{cor:bgit_for_teich} implies
\begin{align*}
d_V(y,z) +2 d_V(b,\partial A) 
&\le d_V(y,\partial A) + d_V(\partial A, b) +  d_V(b, \partial A)+ d_V(\partial A, z) \\
&\le d_V(y,b) + d_V(b,z) + 2\rafi/3\\
&\le d_V(y,z) + 2(\csistfun(7\rafi) + \rafi) + 2\rafi/3.
\end{align*}
Thus $d_V(\beta_V, \partial A) \le d_V(b,\partial A) + \csistfun(7\rafi) \le 2(\csistfun(7\rafi)+\rafi)$. If $d_V(\xi_V,\beta_V)\le 4\rafi$, this bounds $d_V(\xi_V,\partial A)$ and proves the claimed consistency.
Otherwise $d_V(\xi_V,\beta_V) >4\rafi$ which again means $V$ is an annulus with $\xi_V = \pi_V(y)$. Therefore
\[\min\{d_A(\xi_A,\partial V), d_V(\xi_V,\partial A)\} = \min\{d_A(\pi_A(y),\partial V), d_V(\pi_V(y),\partial A)\}\le \consist \le \rafi\]
by Theorem~\ref{Thm:Consistency}, and the required inequality is satisfied.
\end{proof}

Let $y',z'\in \tnet_\thin(\Sigma)$ be the thick net points provided by Theorem~\ref{Thm:Consistency} realizing the $R$--consistent tuples $(\xi_V),(\zeta_V)$.
Then for annuli $A$ with $(y\vert z)_x^A>\theta$ the claim implies $d_A(y,y'),d_A(z,z')\le \csistfun(R)$ which immediately gives $2\csistfun(R)$--alignment of $(y,y',z',z)$, $(y,y',x)$ and $(x,z',z)$ in $A$. For all other domains $V$,  $d_V(y',\beta_V)$ and $d_V(\beta_V,z')$ are at most $\csistfun(R)$, which gives  $d_V(y',z')\le 2\csistfun(R)$. Since $\beta_V$ is a $\rafi$--barycenter of $\{\pi_V(x),\pi_V(y),\pi_V(z)\}$, it also implies the three tuples above and $(y,z',x)$, $(x, y',z)$ are all $4(\csistfun(R)+\rafi)$--aligned in $V$, as desired.
\end{proof}

\section{Witness families}
\label{sec:witness-families}

We now lay the foundation for our main technical construction of the ``complexity length'' between points of $\T(\Sigma)$. This notion of length is defined in terms of curve complex data of subsurfaces of $\Sigma$ and relies on building and manipulating collections of subsurfaces with certain properties. In this section we introduce our terminology and establish basic operations and results about such families. The definition of complexity length will then be given in \S\ref{sec:Complexity length}.

Given any parameter $\wfal\ge 2\rafi$, 
we once and for all fix a sequence of constants
\begin{equation}
\label{eqn:indexed_thresholds}
\plex{S} + 30\wfal \frac{\thin}{\thin'} = \indrafi{\plex{S}+1} \le \indrafi{\plex{S}}\le \indrafi{\plex{S}-1}\le \dotsc\le  \indrafi{0}\le \indrafi{-1} = \irafi,
\end{equation}
where $\rafi$ and $\thin > \thin'$ are from Definition~\ref{def:uniform_constants}. 
Note that $1 = \plex{S_{0,4}}$ and $-1 = \plex{\text{annulus}}$. The exact value of these constants (along with other related constants) will be specified later in a recursive manner (Proposition~\ref{prop:saturated_constants}), but we stress that they depend only on $\pfback$ and $S$. 
By abuse of notation, for any domain $V$ of $S$ we set $\indrafi{V}\colonequals \indrafi{\plex{V}}$ and emphasize that these depend only on the integer $\plex{V}$ and not on the domain $V$.

\begin{definition}
\label{def:upsilons}
For $\Sigma$ a domain in $S$ and $x,y\in \T(\Sigma)$ we consider the following collections of subdomains:
\begin{align*}
\Upsilon^c(x,y) &= \big\{V\esub \Sigma \mid d_V(x,y)\ge \indrafi{V}\big\},
\quad\text{and}\\
\Upsilon^\ell(x,y) &= \left\{A \esub \Sigma \mid A\text{ an annulus with }\min\{\ell_x(\partial A), \ell_y(\partial A)\} < \thin/{\indrafi{A}} < \thin'\right\}.
\end{align*}
Thus $\Upsilon^c$ consists of those domains with big \textbf{c}urve complex distance, and $\Upsilon^\ell(x,y)$ consists of those curves (really annuli) that are short at $x$ or $y$
Note that $\abs{\Upsilon^\ell(x,y)}\le 2\plex{\Sigma}$, since $x$ and $y$ each have at most $\plex{\Sigma}$ curves of length smaller than $\thinprime$.
We then define
\[\Upsilon(x,y) = \Upsilon^c(x,y) \cup \Upsilon^\ell(x,y)\]
and, when the points $x,y$ are understood, abbreviate these simply as $\Upsilon, \Upsilon^c, \Upsilon^\ell$.
\end{definition}

\begin{remark}
\label{Upsilon_implies_active_interval}
Observe that every $V\in \Upsilon$ has a nonempty active interval $\interval_V$ along $[x,y]$. If $V\in \Upsilon^c$ this follows form Lemma~\ref{lem:our_active_intervals}(\ref{ourinterval-nonempty}). If instead $V\in \Upsilon^\ell$ then $V$ is an annulus with at least one of $\ell_x(\partial V)$ or $\ell_y(\partial V)$ smaller than $\frac{1}{\irafi_V}\thin \le \frac{\thin'}{30\wfal\thin}\thin < \thin'$. Hence $\interval_V = \tilde\interval_V^\thin$ is nonempty by Definition~\ref{def:active_intervals} and Theorem~\ref{thm:rafi_active}(\ref{interval-fat}).

\end{remark}

The distance formula (Theorem~\ref{thm:distance_formula}) basically says the domains $V$ of $\Upsilon(x,y)$ (along with the projections $\pi_V(\ast)$ and lengths $\ell_\ast(\partial V)$ when $V$ is an annulus) account for all of the data needed to estimate $d_{\T(\Sigma)}(x,y)$. However, the \emph{multiplicative} error in the distance formula is unacceptable in our application because $d_{\T(\Sigma)}(x,y)$ goes into the exponent and the whole point is to calculate its coefficient.

Morally, this multiplicative error stems from the fact that the collection $\Upsilon(x,y)$ can be arbitrarily large. The point of witness families, defined next, is to partition $\Upsilon(x,y)$ into \emph{uniformly boundedly many} subcollections (Definition~\ref{def:contribute}). The data for each subcollection will then be recombined into a Teichm\"uller distance (Definition~\ref{def:resolution}). If everything is done carefully, the weighted sum of these distances may be related to $d_{\T(\Sigma)}(x,y)$ with \emph{only additive error} (Theorem~\ref{thm:main_complexity_length_bound}).

\begin{definition}[Witness family]
\label{def:projection_family}
Let $\Sigma$ be a domain in $S$. 
A collection $\Omega$ of domains of $\Sigma$ is called a \define{witness family} for a geodesic segment $[x,y]$ in $\T(\Sigma)$ if:
\begin{enumerate}[label=(WF\arabic*)]
\item\label{projfam-big}
Every $V\in \Omega$ satisfies $V\in \Upsilon(x,y)$; that is $\Omega \subset \Upsilon(x,y)$.
\item\label{projfam-see_everyon}
Every $Z\esubset \Sigma$ with $Z\in \Upsilon(x,y)$ satisfies $Z\esubset V$ for some $V\in \Omega$.
\item\label{projfam-maximal}
If $Z \esubset W$ are such that $Z\in \Omega$ and $W\in \Upsilon(x,y)$, then either $W\in \Omega$ or else $W\pitchfork Z'$ for some $Z'\in \Omega$ with $Z\esubset Z'$.
\end{enumerate}
\end{definition}

\subsection{Supremums for witness families}
\label{sec:supremums_witness}
We will have need to discuss the minimal subsurfaces in a collection that contain a given subsurface:

\begin{definition}[Minimal containment]\label{def:supremum}
If  $\Omega$ is a collection of subsurfaces of $\Sigma$ and $Z$ is an arbitrary subsurface of $\Sigma$, we use the notation $Z\clld{\Omega}W$ to mean that $W$ is a minimal (with respect to inclusion) subsurface of $\Sigma$ satisfying the conditions $Z\esubset W$ and $W\in \Omega$.
If $W$ is moreover the \emph{unique} element of $\Omega$ such that $Z\clld{\Omega}W$, we write $W = \colsup{Z}{\Omega}$ and call $W$ the \define{$\Omega$--supremum} of $Z$. 
\end{definition}

\begin{lemma}
\label{lem:characterize_supremum}
Let $\Omega$ be a witness family for $[x,y]$ and suppose that $W\esubset \Sigma$ has an $\Omega$--supremum $\colsup{W}{\Omega} \in \Omega$. Then $\colsup{W}{\Omega} \esubset V$ for every $V\in \Omega$ with $W\esubset V$.
\end{lemma}
\begin{proof}
Let $V\in \Omega$ be such that $W\esubset V$. Then the family $\{V'\in \Omega \mid W\esubset V'\esubset V\}$ is nonempty and so contains a topologically minimal element $V_0$. By minimality, it must be that $W\clld{\Omega} V_0$. The assumed uniqueness of the $\Omega$--supremum now implies that $\colsup{W}{\Omega} = V_0 \esubset V$, as claimed.
\end{proof}

\subsection{Complete witness families}

In general, a domain $Z$ could satisfy $Z\clld{\Omega}W$ for multiple elements $W$ of a witness family $\Omega$. Our construction of Teichm\"uller resolutions below will utilize witness families for which every large-projection domain has a $\Omega$--supremum:

\begin{definition}[Completeness]\label{def:completeness}
A witness family $\Omega$ for a geodesic segment $[x,y]$ in $\T(\Sigma)$ is said to be \define{complete} if every domain $Z\esub \Sigma$ with $Z\in \Upsilon(x,y)$ has an $\Omega$--supremum $\colsup{Z}{\Omega}\in \Omega$.
\end{definition}

We next describe a criterion for completeness. Suppose that $A,B\esubset \Sigma$ are two domains with $A\pitchfork B$. Define $\bigd{A,B}$ to be the collection
\[\bigd{A,B}\colonequals \{Z\esubset \Sigma \mid Z\esubset A, Z\esubset B\text{ and } Z\in \Upsilon(x,y)\}\]
and let $\maxbigd{A,B}$ be the subcollection of topologically maximal surfaces in $\bigd{A,B}$.

\begin{lemma}
\label{lem:bounded_refinement}
For every geodesic $[x,y]$ in $\T(\Sigma)$ and pair of domains $A,B\esub\Sigma$ with $A\pitchfork B$, one has $\abs{\maxbigd{A,B}}_j\le (2\indrafi{j+1})^{\plex{\Sigma}+2}$ for every $-1\le j \le \plex{\Sigma}$.
\end{lemma}
\begin{proof}
Consider $\mathcal{F}^c(A,B) = \{Z\esubset \Sigma \mid Z\esubset A, Z\esubset B\text{ and } Z\in \Upsilon^c(x,y)\}$.
Applying Lemma~\ref{lem:threshholds} to the essential intersection $A\ecap B$ (Lemma~\ref{lem:BKMM-subsurface-operations}) immediately yields $\abs{\underline{\mathcal{F}}^c(A,B)}_j \le (2\indrafi{j+1})^{\plex{\Sigma}+1}$ for every $-1\le j \le \plex{\Sigma}$. Since $\bigd{A,B}\subset \mathcal{F}^c(A,B)\cup \Upsilon^\ell(x,y)$ and $\Upsilon^\ell(x,y)$ consists of annuli, we see that  $\maxbigd{A,B}\subset \underline{\mathcal{F}}^c(A,B)\cup \Upsilon^\ell(x,y)$ and that the claimed bound follows for all $0\le j \le \plex{\Sigma}$. For $j=-1$, the fact $\abs{\Upsilon^\ell(x,y)} \le 2\plex{\Sigma} \le \indrafi{0}$ now gives $\abs{\maxbigd{A,B}}_{-1} \le (2\indrafi{0})^{\plex{\Sigma}+2}$.
\end{proof}

Let $\Omega$ be a witness family for $[x,y]$ in $\T(\Sigma)$. A \define{cutting pair in $\Omega$} is a pair of subsurfaces $A,B\in \Omega$ such that $A\pitchfork B$. We say that the cutting pair $(A,B)$ is \define{filled} in $\Omega$ if $\maxbigd{A,B}\subset \Omega$ and that the pair is \define{unfilled} otherwise.

\begin{lemma}[Filled to completeness]
\label{lem:filled_complete}
A witness family $\Omega$ for $[x,y]$ in $\T(\Sigma)$ is complete if and only if every cutting pair $(A,B)$ in $\Omega$ is filled.
\end{lemma}
\begin{proof}
First suppose every cutting pair in $\Omega$ is filled. Let $Z\esubset \Sigma$ be any domain with $Z\in \Upsilon(x,y)$. Condition~\ref{projfam-see_everyon} ensures there exists some $A\in \Omega$ with $Z\clld{\Omega}A$. Thus if $Z$ fails to have an $\Omega$--supremum, there is a second domain $B\in \Omega$ with $B\ne A$ and $Z\clld{\Omega} B$. By minimality, $A$ and $B$ cannot be nested and so we have $A\pitchfork B$. It follows that $(A,B)$ is a cutting pair and that $Z\in\bigd{A,B}$. Therefore, by definition of $\maxbigd{A,B}$, we have $Z \esubset Z_0$ for some $Z_0\in \maxbigd{A,B}$. The hypothesis that $(A,B)$ is filled now implies that $Z_0\in \Omega$. But since $Z\esubset Z_0 \esubset A,B$, this contradicts the minimality of $Z\clld{\Omega}A$ and $Z\clld{\Omega}B$. Therefore $\colsup{Z}{\Omega}$ exists.

Next suppose $\Omega$ is complete. Let $(A,B)$ be a cutting pair in $\Omega$ and choose any $Z\in \maxbigd{A,B}$. We must show $Z\in \Omega$. Since $Z\esubset A$, the $\Omega$--supremum necessarily satisfies $\colsup{Z}{\Omega}\esubset A$ by Lemma~\ref{lem:characterize_supremum}. Similarly $\colsup{Z}{\Omega}\esubset B$. Therefore $\colsup{Z}{\Omega}\in \bigd{A,B}$ by definition. Since $Z$ is a topologically maximal element of $\bigd{A,B}$ and $Z\esubset \colsup{Z}{\Omega}$ by definition, it follows that $Z = \colsup{Z}{\Omega} \in \Omega$. 
\end{proof}

\subsection{Insulation}
\label{sec:insulation}

Let $[x,y]$ be a Teichm\"uller geodesic in $\T(\Sigma)$. 
If $Z,V\esubset \Sigma$ are two domains with $Z\esubset V$ and $Z\ne V$, define
\[\cc(V\vert_Z) = \{\alpha \in \cc(V) \mid \text{$\alpha$ is essential or peripheral in $Z$}\} = \Gamma(Z) \cup \partial Z.\]
Observe that $\cc(V\vert_Z)$ has diameter at most $2$ in $\cc(V)$. For any domain $E\esubset \Sigma$ and parameter $0\le t \le d_E(x,y)$, we then define $\maxlset{t}{E}$ and $\maxrset{t}{E}$ to be the topologically maximal domains in the respective collections
\begin{align*}
\lset{t}{E}\colonequals \big\{Z\esubsetneq E &\mid Z\in \Upsilon(x,y)\text{ and }\exists \alpha\in \cc(E\vert_Z) : d_E(\alpha,x)\in [t-9\wfal,t+9\wfal]\big\}\\
\rset{t}{E}\colonequals \big\{Z\esubsetneq E &\mid Z\in \Upsilon(x,y)\text{ and }\exists \alpha\in \cc(E\vert_Z) : d_E(\alpha,y)\in [t-9\wfal,t+9\wfal]\big\},
\end{align*}
where here $d_E(\alpha,x) = \diam_{\cc(E)}(\pi_E(x)\cup\{\alpha\})$ and similarly for $d_E(\alpha,y)$. Note that by construction $\lset{t}{A} = \emptyset = \rset{t}{A}$ for any annulus $A$.

\begin{lemma}
\label{lem:bounded_pads}
For every geodesic $[x,y]$ in $\T(\Sigma)$, domain $E\esub \Sigma$, and parameter $t$ with $0 \le t \le d_E(x,y)\ge \indrafi{E}$, one has $\abs{\maxlset{t}{E}}_j,\abs{\maxrset{t}{E}}_j \le (2\indrafi{j+1})^{\plex{\Sigma}+3}$ for every $-1\le j \le \plex{\Sigma}$.
\end{lemma}
\begin{proof}
We give the proof for $\lset{t}{E}$: As in the proof of Lemma~\ref{lem:bounded_refinement}, let
\[\mathcal{L}_t^c(E) = \big\{Z\esubsetneq E \mid Z\in \Upsilon^c(x,y)\text{ and }\exists \alpha\in \cc(E\vert_Z) : d_E(\alpha,x)\in [t-9\wfal,t+9\wfal]\big\}\]
and observe that $\maxlset{t}{E} \subset \underline{\mathcal{L}}_t^c(E)\cup \Upsilon^\ell(x,y)$. Since $\abs{\Upsilon^\ell(x,y)}\le 2\plex{\Sigma}\le \indrafi{j+1}$, it therefore suffices to prove that $\abs{\underline{\mathcal{L}}_t^c(E)}_j \le (2\indrafi{j+1})^{\plex{\Sigma}+2}$ for all $-1 \le j \le \plex{\Sigma}$.

Choose $\alpha\in \pi_E(x)$ and $\beta\in \pi_E(y)$ realizing the distance $d_E(x,y) = d_{\cc(E)}(\alpha,\beta)$ and fix a geodesic $\alpha = \gamma_0,\dotsc,\gamma_m=\beta$ in $\cc(E)$. As in the proof of Lemma~\ref{lem:threshholds}, we claim that every $Z\esubn E$ with $d_Z(x,y)\ge \indrafi{Z}$ is disjoint from some curve $\gamma_i$. Indeed, this is immediate if $\alpha$ or $\beta$ misses $Z$, and otherwise Lemma~\ref{lem:composing_projections} ensures $d_{\cc(Z)}(\pi_Z(\alpha),\pi_Z(\beta))\ge d_Z(x,y) - 2\nestprojscommute \ge \rafi-2\nestprojscommute$ so that the Bounded Geodesic Image Theorem implies $\pi_Z(\gamma_i) = \emptyset$ for some $i$. 

Thus every $Z\in \mathcal{L}^c_t(E)$ is disjoint from some $\gamma_i$. Further, from the definition of $\mathcal{L}^c_t(E)$, we see that this curve $\gamma_i$ must satisfy $d_E(\gamma_i,x)\in [t-9\wfal-1,t+9\wfal+1]$. Since $\diam_{\cc(E)}(\pi_E(x)) \le \lipconst$, this implies $d_{\cc(E)}(\alpha,\gamma_i)\in [t-9\wfal-2\lipconst,t+9\wfal+2\lipconst]$. Letting $\mathcal{W}$ denote the set of all components of $E\setminus \gamma_i$ obtained as $i$ ranges between $\max\{0,t-9\wfal-2\lipconst\}$ and $\min\{m,t+9\wfal+2\lipconst\}$, it follows that
\[\mathcal{L}^c_t(E) \subset \bigcup_{W\in\mathcal{W}} \mathcal{P}(W),\]
where $\mathcal{P}(W)$ is as in Lemma~\ref{lem:threshholds}. Now let $Z\in \underline{\mathcal{L}}^c_t(E)$ be a topologically maximal element of $\mathcal{L}^c_t(E)$ and choose $W\in\mathcal{W}$ such that $Z\in \mathcal{P}(W)$. If $Z\esubn V$ for some $V\in \mathcal{P}(W)$, then the facts $d_V(x,y)\ge \indrafi{V}$ and $Z\esub V\esubn E$ with $Z\in \mathcal{L}^c_t(E)$ imply that $V\in \mathcal{L}^c_t(E)$ as well. But this contradicts the maximality of $Z$ in $\mathcal{L}^c_t(E)$. Therefore $Z$ is maximal in $\mathcal{P}(W)$ as well. This proves $\underline{\mathcal{L}}^c_t(E)\subset \bigcup_W \underline{\mathcal{P}}(W)$. We may now invoke Lemma~\ref{lem:threshholds} to conclude
\[\abs{\underline{\mathcal{L}}^c_t(E)}_j \le \sum_{W\in\mathcal{W}} \abs{\underline{P}(W)}_j \le \abs{\mathcal{W}}(2\indrafi{j+1})^{\plex{\Sigma}+1} \le 2(18\wfal+5\lipconst)(2\indrafi{j+1})^{\plex{\Sigma}+1}\]
for every $j$. Since $18\wfal+5\lipconst < 23\wfal < \indrafi{j+1}$, the lemma follows.
\end{proof}

\begin{definition}[Insulation]
\label{def:insulation}
If $\Omega$ is a witness family for $[x,y]$ in $\T(\Sigma)$, we say that $E\in \Omega$ is \define{insulated} in $\Omega$ if $\maxlset{0}{E}\cup\maxrset{0}{E}\subset \Omega$. The witness family $\Omega$ is said to be \define{insulated} if every $E\in \Omega$ is insulated in $\Omega$.
\end{definition}

The terminology stems from the observation that if $E$ is insulated in $\Omega$, then for every domain $Z\in \Upsilon(x,y)$ with $Z\clld{\Omega}E$, $\partial Z$ occurs towards the middle or ``interior'' of the geodesic $[\pi_E(x),\pi_E(y)]$ in $\cc(E)$, rather than near the endpoints. This has the following useful consequence:

\begin{lemma}
\label{lem:insular_properties}
Let $\Omega$ be witness family for $[x,y]$ in $\T(\Sigma)$.
\begin{enumerate}
\item\label{insular-nest} Suppose $V\in \Omega$ is insulated in $\Omega$. If $Z\esubset \Sigma$ is such that $Z\clld{\Omega}V$ with $Z\in \Upsilon(x,y)$, then the active intervals of $Z$ and $V$ satisfy $\interval_Z\subset \interval_V$.
\item\label{insular-cur} Suppose $V\in \Omega$ is insulated in $\Omega$. If $Z,W\esubset \Sigma$ are such that $Z\clld{\Omega}V$ with $Z\in \Upsilon(x,y)$ and $W,V$ time-ordered along $[x,y]$, then $W\pitchfork Z$.
\item\label{insular-subcut} Suppose $\Omega$ is insulated. If $V_1,V_2\in \Omega$ are such that $V_1\pitchfork V_2$ and $Z_1, Z_2\esubset \Sigma$ are such that $Z_i\in \Upsilon(x,y)$ and $Z_i\clld{\Omega}V_i$ for $i=1,2$, then $Z_1\pitchfork Z_2$.
\end{enumerate}
\end{lemma}
\begin{proof}
Suppose, contrary to (\ref{insular-nest}), that $z\notin\interval_V$ for some point $z\in \interval_Z$. Then $z$ lies in the same component of $[x,y]\setminus \interval_V$ as either $x$ or $y$. Without loss of generality, say $x$ an $z$ lie in the same component. Then $d_V(x,z)\le \rafi/3$ by Lemma~\ref{lem:our_active_intervals}(\ref{ourinterval-smallproj}). Since the Bers marking at $z$ contains $\partial Z$ by Lemma~\ref{lem:our_active_intervals}(\ref{ourinterval-thin}), it follows that $d_V(x,\partial Z)\le \rafi/3$ and thus that $Z\in \lset{0}{V}$ by definition. Hence $Z\esub Z'\esubn V$ for some $Z'\in \maxlset{0}{V}$. But then $Z'\in \Omega$ by the insulation of $V$, contradicting $Z\clld{\Omega}V$.

We next prove (\ref{insular-cur}): If $W$ and $Z$ do not cut, then $\partial W$ and $\partial Z$ are disjoint so that $d_V(\cc(V\vert_Z),\partial W)\le 2 + d_V(\partial W, \partial Z)\le 3< \rafi/2$. The fact that $W$ and $V$ are time-ordered implies $\min\{d_V(\partial W, x), d_V(\partial W,y)\} < \rafi/3$ by Lemma~\ref{lem:characterize_timeorder}. Therefore $Z\in\leftpad{V}\cup\rightpad{V}$ and hence $Z\esubset Z'$ for some $Z'\in \maxleftpad{V}\cup \maxrightpad{V}$. Since $V$ is insulated in $\Omega$, it follows that $Z'\in \Omega$, contradicting  $Z\clld{\Omega}V$.

Conclusion (\ref{insular-subcut}) now follows easily.
Since $V_1$ is insulated in $\Omega$ by hypothesis, (\ref{insular-cur}) implies $Z_1\pitchfork V_2$. Applying (\ref{insular-cur}) again to the insulated $V_2\in \Omega$, we conclude $Z_1\pitchfork Z_2$.
\end{proof}

We also note the following useful observation:

\begin{lemma}
\label{lem:insulated_contains_short_curves}
If $\Omega$ is an insulated witness family for $[x,y]$, then $\Upsilon^\ell(x,y)\subset \Omega$.
\end{lemma}
\begin{proof}
Consider any annulus $A\in \Upsilon^\ell(x,y)$. Then either $\ell_x(\partial A) < \thin$ or $\ell_y(\partial A) < \thin$; by symmetry, let us suppose it is the former.
By \ref{projfam-see_everyon}, there exists some $Z\in \Omega$ with $A\clld{\Omega}Z$. If $A = Z$ we are done. Otherwise $\partial A$ is an essential curve in $Z$, and the fact $\ell_x(\partial A) < \thin$ implies that $d_Z(x,\partial A)\le \lipconst \le \rafi$. Thus $A\in \lset{0}{Z}$ and hence $A \esub Z'$ for some $Z'\in \maxlset{0}{Z}$.  Since $\Omega$ is insulated, we have $Z'\in \Omega$. But now the containments $A\esub Z'\esubn Z$ contradict $A\clld{\Omega} Z$.
\end{proof}

\subsection{Subordered witness families}
\label{sec:subordering}
To construct our complexity length, we will work with witness families that come equipped with the following structure:

\begin{definition}[Subordering]
\label{def:subordering}
Let $\Omega$ be a witness family for the segment $[x,y]$ in $\T(\Sigma)$. A \define{subordering} on $\Omega$ is an ordering designation $Z\swarrow V$ exclusive or $V\searrow Z$ for every $Z,V\in \Omega$ with $Z\esubn V$. This ordering data must satisfy:
\begin{enumerate}[label={(SO\arabic*)}]
\item\label{subord-triplenest}
 If $Z,W,V\in \Omega$ are such that $Z\esubn W \esubn V$ then
\[Z\swarrow V \iff W\swarrow V\quad\text{(and so $V\searrow Z \iff V\searrow W$ also).}\]
\item\label{subord-subcutting}
If $Z,W,V\in \Omega$ are such that $Z \swarrow V \searrow W$, then $Z \cut_V W$ and $Z \tol W$.
\item\label{subord-supercutting}
If $Z,V\in \Omega$ and $W\esub \Sigma$ with $W\in \Upsilon(x,y)$ are such that $Z\swarrow V \tol W$ or $W \tol V \searrow Z$, then $Z\cut_V W$.
\item\label{subord-contribute}
If $Z,V\in \Omega$ are such that $Z\swarrow V$ (resp. $V \searrow Z$), then there does not exist any domain $W\in \Upsilon(x,y)$ with $\colsup{W}{\Omega} = V$, and $W \tol Z$ (resp. $Z \tol W$).
\end{enumerate}
\end{definition}

\begin{remark}
\label{rem:suborder_implications}
Condition~\ref{subord-subcutting} in fact implies condition~\ref{subord-triplenest}, as can be seen by noting that if $Z\esubn W \esubn V$, then $Z\cut_V W$ clearly fails so that $Z\swarrow V \searrow W$ and $W\swarrow V\searrow Z$ must both fail as well. 
\end{remark}

If $\Omega$ and $\Omega'$ are two subordered witness families with $\Omega \subset \Omega'$, we say that the subordering on $\Omega'$ \define{extends} the subordering on $\Omega$ if the ordering designations coming from $\Omega$ and $\Omega'$ agree on each pair $Z,V\in \Omega$ with $Z\esubn V$.
Subordered witness families enjoy the following property:

\begin{lemma}
\label{lem:timing_of_suborder}
Let $\Omega$ be a subordered witness family for a geodesic $[x,y]$ in $\T(\Sigma)$. Fix a domain  $V\in \Omega$ and let $W\esubn V$ be a domain with $W\in \Upsilon(x,y)$ and $\colsup{W}{\Omega} = V$. Then for any $Z\in \Omega$ with $Z\swarrow V$ (resp. $V\searrow Z$) we have that either $Z$ and $W$ are disjoint, or $\interval_Z$ occurs before (resp. after) $\interval_W$.
\end{lemma}
\begin{proof}
We only consider the case $Z\swarrow V$. If $Z$ and $W$ are disjoint, the lemma is satisfied. If $Z\cut W$, then \ref{subord-contribute} ensures we have the time ordering $Z \tol W$ so that $\interval_Z$ occurs before $\interval_W$ as required. The possibility $W\esub Z$ is ruled out by $\colsup{W}{\Omega} =V$, so it remains to consider the case $Z\esubn W$. 

Here, since $W\notin\Omega$, \ref{projfam-maximal} provides $Z'\in \Omega$ so that $Z\esub Z'$ and $Z'\cut W$. Either $Z'\cut V$, in which case \ref{subord-supercutting} (applied to $Z \swarrow V$ and $Z\esub Z'$) forces $Z' \tol V$ and hence $Z' \tol W$ by Corollary~\ref{cor:time-order_subsurf}. Otherwise $Z'\esub V$ and \ref{subord-triplenest} gives $Z' \swarrow V$ so that we may invoke \ref{subord-contribute} (using $Z'\cut W$) to again conclude $Z' \tol W$.

Since $\interval_{Z'}$ is nonempty with $Z'\cut W$, Lemma~\ref{lem:our_active_intervals}(\ref{ourinterval-subdomain-disjoint}) now implies $\interval_Z\cap\interval_W=\emptyset$. In fact, since $Z'\tol W$ it must be that $\interval_Z$ occurs before $\interval_W$ along $[x,y]$.
\end{proof}

\begin{definition}[Encroachment]
\label{def:encroachment}
Let $\Omega$ be a subordered witness family for $[x,y]$ in $\T(\Sigma)$. For $V\in \Omega$, the \define{left and right encroachments} of $V$ in $\Omega$ are defined as
\[\lenc_{\Omega}(V) \colonequals \sup_{Z\in \Omega, Z\swarrow V} d_V(x,\cc(V\vert_Z))\quad\text{and}\quad
\renc_{\Omega}(V) \colonequals  \sup_{Z\in \Omega, V \searrow Z} d_V(y,\cc(V\vert_Z)),\]
respectively (where here $d_V(w,\cc(V\vert_Z)) = \diam_{\cc(V)}(\pi_V(w)\cup \cc(V\vert Z))$). The \define{encroachment} of $V$ is then defined as $\enc_\Omega(V) = \max\{\lenc_\Omega(V),\renc_\Omega(V)\}$. To streamline notation, for each $W\notin\Omega$ these encroachments are set to  zero:
\[\lenc_\Omega(W) = \renc_\Omega(W) = \enc_\Omega(W) = 0\quad\text{when}\quad W\notin \Omega.\]
\end{definition}

\begin{definition}[Wide]
\label{def:wide}
A subordered witness family $\Omega$ is \define{wide} if $\enc_\Omega(V) \le \indrafi{V}/3$ for all $V\in \Omega$.
\end{definition}

We next describe two operations---refinement and augmentation---that may be used to enlarge a witness family and ultimately produce one that is both complete and insulated. For each operation, we must work to show that suborderings may be naturally extended to the new family.

\subsection{Refinement}
\label{sec:refinement}
Lemma~\ref{lem:filled_complete} suggests a means of making any witness family complete: simply add collections of the form $\maxbigd{A,B}$ until every cutting pair is filled. This motivates the following operation: 

\begin{definition}[Refinement]
Let $\Omega$ be a witness family for $[x,y]$ and let $(A,B)$ be a cutting pair in $\Omega$. The \define{refinement of $\Omega$ along $(A,B)$} is the collection
\[\underline{\Omega}(A,B) \colonequals \Omega \cup \maxbigd{A,B}.\]
\end{definition}

Thus $\underline{\Omega}(A,B) = \Omega$ if and only if $(A,B)$ is a filled cutting pair in $\Omega$. Fortunately refinement always produces a new witness family.

\begin{lemma}
\label{lem:refine}
Let $\Omega$ be a witness family for $[x,y]$ in $\T(S)$ and let $(A,B)$ be a cutting pair in $\Omega$. Then 
the refinement $\underline{\Omega}(A,B)$ is a witness family for $[x,y]$ and the pair $(A,B)$ is filled in $\underline{\Omega}(A,B)$.
\end{lemma}
\begin{proof}
It is obvious that $(A,B)$ is filled in $\underline{\Omega}(A,B)$, provided that $\underline{\Omega}(A,B)$ is a witness family. For this, conditions \ref{projfam-big} and \ref{projfam-see_everyon} are immediate; we verify  \ref{projfam-maximal}. 
Let $Z\esub W$ be such that $Z\in \underline{\Omega}(A,B)$ and $W\in \Upsilon(x,y)$. 
If $Z\in \Omega$, then $W$ satisfies condition \ref{projfam-maximal} because $\Omega\subset\underline{\Omega}(A,B)$ is a witness family. Therefore we may suppose $Z\notin \Omega$ so that $Z\in \maxbigd{A,B}$. If $W\esub A$ and $W\esub B$, then the fact $W\in \Upsilon(x,y)$ implies $W\in \bigd{A,B}$ so that $W = Z\in \underline{\Omega}(A,B)$ by maximality of $Z$. If $W$ cuts $A$ or $B$, then we have verified \ref{projfam-maximal} since $A$ and $B$ lie in $\underline{\Omega}(A,B)$ and contain $Z$. Therefore, let us suppose $W$ cuts neither $A$ nor $B$ and that $[W\esub A$ and  $W\esub B]$ fails.  In this case we must have $A,B\esub W$. But now \ref{projfam-maximal}, applied to $A\in \Omega$, implies that either $W\in \Omega\subset \underline{\Omega}(A,B)$ or else that $W\pitchfork Z'$ for some $Z'\in \Omega\subset\underline{\Omega}(A,B)$ with $Z'\esup A \esup Z$. This proves that $\underline{\Omega}(A,B)$ satisfies condition \ref{projfam-maximal} and establishes the lemma.
\end{proof}

Furthermore, suborderings always extend to refinements:

\begin{lemma}[Subordering refinements]
\label{lem:refined_suborder}
If $\Omega$ is a subordered witness family for $[x,y]$ in $\T(\Sigma)$ and $(A,B)$ is a cutting pair in $\Omega$, then there is a unique subordering on the refinement $\underline{\Omega}(A,B)$ that extends the subordering on $\Omega$.
\end{lemma}
\newcommand{\uswar}{\rotatebox[origin=c]{225}{$\rightarrowtail$}}
\newcommand{\usear}{\rotatebox[origin=c]{135}{$\leftarrowtail$}}
\begin{proof}
To simplify notation, write $\underline{\Omega} = \underline{\Omega}(A,B)$ and without loss of generality suppose $A \tol B$. 
We first show that conditions \ref{subord-triplenest}--\ref{subord-contribute} uniquely determine a well-defined ordering designation $\uswar$ or $\usear$ for each pair $Z,V\in \underline{\Omega}$ with $Z\esubn V$:
\begin{enumerate}[label=(\arabic*)]
\item\label{extend-bothin} If $Z,V\in \Omega$ we must use the ordering designation from $\Omega$: $Z \uswar V$ iff $Z\swarrow V$. 
\item\label{extend-bigin}
If $V\in \Omega$ and $Z\notin \Omega$, then $Z\in \maxbigd{A,B}$ so that \ref{subord-supercutting} forces us to set $A \usear Z \uswar B$ (since $Z\cut_A B$ and $A\cut_B Z$ both fail). We claim that exactly one of the following hold: (i) $A,B \esub V$, (ii) $A \tol V$, or (iii) $V \tol B$. Firstly, it is impossible to have $V \esub A,B$, as that would contradict the maximality of $Z$ in $\bigd{A,B}$. Triple nesting $A\esub V \esub B$ or $B\esub V\esub A$ is also ruled out by $A\pitchfork B$. Thus if (i) fails, then $V$ necessarily cuts $A$ or $B$. If $A \pitchfork V$ then we must have $A \tol V$ along $[x,y]$, for otherwise we have $V \tol A \tol B$ and (by Corollary~\ref{cor:time-order_and_severing}) $V\cut_A B$, contradicting $Z\esub V,A,B$. Similarly, if $V\pitchfork B$ then $V \tol B$ and we are in case (iii). To prove the claim, it remains to show (ii) and (iii) are mutually exclusive; but this is clear: since $A\cut_V B$ fails (as $Z\esub A,V,B$), Corollary~\ref{cor:time-order_and_severing} precludes $A \tol V \tol B$.  We now suborder $Z$ and $V$ in each case:
\begin{enumerate}
\item[(i)] If $A,B\esub V$, then \ref{subord-subcutting} implies $A\swarrow V\iff B\swarrow V$. In accordance with \ref{subord-triplenest}, we thus set $Z\uswar V$ in the case that $A\swarrow V$ and $B\swarrow V$ and set $V\usear Z$ in the case that $V \searrow A$ and $V\searrow B$.
\item[(ii)] If $A \tol V$  we declare $Z\uswar V$ in accordance with \ref{subord-supercutting}, since $\lnot (A\cut_V Z)$.
\item[(iii)] If $V \tol B$ we similarly declare $V\usear Z$.
\end{enumerate}
\item\label{extend-bigout} If $Z\in \Omega$ and $V\notin \Omega$, then $V\in \maxbigd{A,B}$. Here \ref{projfam-maximal} implies that $V\pitchfork Z'$ for some $Z'\in \Omega$ with $Z\esub Z'$. Since $\lnot(Z'\cut_V Z)$, we thus declare $Z\uswar V$ if $Z' \tol V$ and $V\usear Z$ if $V \tol Z'$ in accordance with \ref{subord-supercutting}. Note that this is well-defined: If $Z'_0\in \Omega$ is any other such domain, $Z' \tol V \tol Z'_0$ is ruled out by Corollary~\ref{cor:time-order_and_severing} and the fact $\lnot(Z'\cut_V Z'_0)$.
\item\label{extend-bothout} If $Z,V\not\in \Omega$ then $Z,V\in\maxbigd{A,B}$, contradicting the fact that no two surfaces in $\maxbigd{A,B}$ can be properly nested. Therefore this case does not occur.
\end{enumerate}
We have now established ordering designations on $\underline{\Omega}$ that extend those of $\Omega$. We henceforth use $\swarrow$ and $\searrow$ for these orderings in both $\Omega$ and $\underline{\Omega}$, as the meaning is unambiguous. It remains to show these in fact give a subordering on $\underline{\Omega}$:

\textbf{Condition \ref{subord-subcutting}}: Let $Z,W,V\in \underline{\Omega}$ be such that $Z \swarrow V \searrow W$. If all three domains are in $\Omega$ the condition is clear. Suppose $V\notin \Omega$. Then $Z,W\in \Omega$ since no pair of domains in $\maxbigd{A,B}$ are nested. By case \ref{extend-bigout} above, there must exist $Z',W'\in \Omega$ such that $Z\esub Z'$, $W\esub W'$, and $Z' \tol V \tol W'$. Corollaries~\ref{cor:time-order_and_severing}--\ref{cor:time-order_subsurf} now imply $Z'\cut_V W'$ and consequently $Z \cut_V W$ and $Z \tol W$.

Next suppose $V\in \Omega$. Observe that if $Z,W\notin \Omega$, then case \ref{extend-bigin} above dictates that orderings for $Z\esubn V$ and $W\esubn V$ are both determined solely by the relationship of $V$ to $A$ and $B$ so that in fact $Z\swarrow V \iff W\swarrow V$. As this is not the case, we see that at most one of $Z,W$ can lie outside of $\Omega$. By symmetry, let us suppose that $Z\notin \Omega$ (so that $Z\in \maxbigd{A,B}$) and $W\in \Omega$. Since $Z\swarrow V$, the definition in \ref{extend-bigin} dictates that either $A,B\swarrow V$, or else $A \tol V$. If $A,B\swarrow V$, then we have $B \swarrow V \searrow W$ with $B,W,V\in \Omega$ so that \ref{subord-subcutting} for $\Omega$ implies $B \tol W$ with $B\cut_V W$. If $A,B\swarrow V$ fails, then we have $A \tol V\searrow W$ so that \ref{subord-supercutting} for $\Omega$ implies $A\cut_V W$. Since $Z\esub A,B$, either of these outcomes implies $Z\cut_V W$. Using Corollary~\ref{cor:time-order_subsurf}, we may further conclude $Z \tol W$. Thus condition \ref{subord-subcutting} is satisfied in this case.

\textbf{Condition \ref{subord-triplenest}}: This follows from the above and Remark~\ref{rem:suborder_implications}.

\textbf{Condition \ref{subord-supercutting}}: Let $Z,V\in \underline{\Omega}$ and $W\in \Upsilon(x,y)$ be such that $Z\swarrow V \tol W$ (the case $W \tol V \searrow Z$ is similar). We must show $Z\cut_V W$. As above, this is clear if $Z,V\in \Omega$, and at most one of $Z$ or $V$ can lie outside of $\Omega$. Suppose first that $Z\notin \Omega$, so that $Z\in \maxbigd{A,B}$, and $\Omega\in V$. Let us first consider the case (2i) above in which  $A,B\esub V$. Since $Z\swarrow V$, the definition dictates that  $A\swarrow V$ as well. Therefore we may apply \ref{subord-supercutting} for $\Omega$ to $A \swarrow V \tol W$ and conclude that $A\cut_V W$. If we are not in case (2i), then (since $Z\swarrow V$) we must be in case (2ii) with $A \tol V$. Therefore we have $A \tol V \tol W$ and Corollary~\ref{cor:time-order_and_severing} gives $A\cut_V W$. In either case, since relative cutting descends to the subsurface $Z\esubset A$, we may conclude $Z\cut_V W$ as desired.

It remains to suppose that $Z\in \Omega$ and $V\notin \Omega$. Now case \ref{extend-bigout} dictates that there is a domain $Z'\in \Omega$ with $Z\esub Z'$ and $Z' \tol V$. Therefore we have $Z\esub Z' \tol V \tol W$ and may again conclude $Z'\cut_V W$ and consequently $Z\cut_V W$. This proves that $\underline{\Omega}$ satisfies condition \ref{subord-supercutting}.

\textbf{Condition \ref{subord-contribute}}: Let $Z,V\in \underline{\Omega}$ be such that $Z\swarrow V$ and suppose that $W\in \Upsilon(x,y)$ satisfies $\colsup{W}{\underline{\Omega}} = V$. We show that $W \not\tol Z$. (The case $V\searrow Z$ is similar). This is clear if $Z$ and $W$ are disjoint or nested, so we may assume $Z\pitchfork W$. Note that this gives $W\neq V$ and, consequently $W\notin\underline{\Omega}$. As before, it suffices to suppose that exactly one of $Z$ or $V$ lies in $\Omega$. 

First suppose $V\in \Omega$ and $Z\notin \Omega$. We claim that the facts $V = \colsup{W}{\underline{\Omega}}$ and $\Omega\subset \underline{\Omega}$ imply $V = \colsup{W}{\Omega}$ as well. Indeed, if $V_0\in \Omega$ is any domain with $W\esub V_0$, then $V_0\in \underline{\Omega}$ so that $V\esub V_0$ by Lemma~\ref{lem:characterize_supremum}. Whence $V = \colsup{W}{\Omega}$ as claimed.
Since $Z\notin \Omega$, we have $Z\in \maxbigd{A,B}$ with $Z\swarrow V$ so that, by the definition in \ref{extend-bigin}, either $A,B\esub V$ with $B\swarrow V$, or else $A\tol V$. First consider the former case $A,B\esub V$ with $B\swarrow V$. Since $W\clld{\Omega}V$ and $B\in \Omega$ with $B\esubn V$, it cannot be that $W\esub B$. Also we cannot have $B\esub W$ or $B$ disjoint from $W$ because $W\pitchfork Z$. Therefore $W\pitchfork B$. Since $B,V\in \Omega$ with $B\swarrow V$ and $\colsup{W}{\Omega} = V$, we can now invoke \ref{subord-contribute} for $\Omega$ to conclude $B \tol W$. The desired time-ordering $Z \tol W$ now follows from Corollary~\ref{cor:time-order_subsurf}. Next consider the latter case $A \tol V$.  Now, $A$ and $W$ cannot be disjoint nor can $A\esub W$ because $Z\pitchfork W$. If $W\pitchfork A$, then $A \tol V$ implies $A \tol W$ and $Z \tol W$ by Corollary~\ref{cor:time-order_subsurf}. So it remains to suppose $W\esub A\in \Omega$; but here Lemma~\ref{lem:characterize_supremum} implies $V=\colsup{W}{\Omega} \esub A$ contradicting $A \pitchfork V$.  This proves that \ref{subord-contribute} holds when $V\in \Omega$ and $Z\notin \Omega$.

Next suppose $V\notin \Omega$ and $Z\in \Omega$. Since $Z \swarrow V$, the definition in \ref{extend-bigout} provides some $Z'\in \Omega$ such that $Z \esub Z'$ and $Z' \tol V$. If $W\esub Z'$, then since $Z'\in \Omega\subset \underline{\Omega}$, Lemma~\ref{lem:characterize_supremum} implies that $V = \colsup{W}{\underline{\Omega}}  \esub Z'$, contradicting $Z'\pitchfork V$. Therefore $W\not\esub Z'$. Neither can we have $Z'\esub W$ or $Z'$ and $W$ disjoint (since $Z\pitchfork W$). Therefore $Z' \pitchfork W$ and we may invoke Corollary~\ref{cor:time-order_subsurf} to conclude $Z' \tol W$ and subsequently $Z \tol W$, as desired. This proves that \ref{subord-contribute} holds when $V\notin \Omega$ and $Z\in \Omega$ and completes the proof of Lemma~\ref{lem:refined_suborder}.
\end{proof}

One may now ask how encroachments in $\Omega$ and $\underline{\Omega}(A,B)$ are related:

\begin{lemma}[Refined encroachments]
\label{lem:encroach_refine}
Let $\Omega$ be a subordered witness family for $[x,y]$ in $\T(\Sigma)$ and let $\underline{\Omega} = \underline{\Omega}(A,B)$ be the subordered refinement along the cutting pair $(A,B)$. 
Then every domain $V\esub \Sigma$ satisfies
\[\lenc_{\underline{\Omega}}(V) \le \max\{\lenc_{\Omega}(V),\rafi\}\quad\text{and}\quad \renc_{\underline{\Omega}}(V) \le \max\{\renc_{\Omega}(V),\rafi\}.\]
\end{lemma}
\begin{proof}
First suppose $V\notin \Omega$. As the claim is immediate for $V\notin\underline{\Omega}$, we assume $V\in \underline{\Omega}\setminus\Omega$ so that $\enc_\Omega(V) = 0$. An examination of the proof of Lemma~\ref{lem:refined_suborder} shows that any $Z\in \underline{\Omega}$ with $Z\esubn V$ falls under case \ref{extend-bigout} and thus satisfies $Z\esub Z'$ for some $Z'\in \Omega$ with $Z'\pitchfork V$. If $Z' \tol V$, so that $Z\swarrow V$, it follows that  $d_V(x,\cc(V\vert_Z)) \le 1+ d_V(x,\partial Z') < \rafi$ by Lemma~\ref{lem:characterize_timeorder}. If instead $V \tol Z'$, so that $V\searrow Z$, we similarly have $d_V(y,\cc(V\vert_Z))< \rafi$. Thus $\enc_{\underline{\Omega}}(V)\le \rafi$ and the claim follows. 

Next suppose $V\in \Omega$. Now the proof of Lemma~\ref{lem:refined_suborder} shows that every $Z\in \underline{\Omega}$ with $Z\swarrow V$ either satisfies $Z\in \Omega$, or else falls under case (2i) with $Z\esub A \swarrow V$ or case (2ii) with $Z\esub A \tol V$. In the former case we have $d_V(x,\cc(V\vert_Z)) \le \lenc_{\Omega}(V)$ by definition, and in the latter case we have $d_V(x,\cc(V\vert_Z))< \rafi$ as in the previous paragraph. For the middle case, we simply note that $\cc(V\vert_Z)\subset \cc(V\vert_A)$ and thus that $d_V(x,\cc(V\vert_Z))\le d_V(x,\cc(V\vert_A))\le \lenc_\Omega(V)$ by definition. This proves $\lenc_{\underline{\Omega}}(V)\le \max\{\lenc_\Omega(V),\rafi\}$; the proof for $\renc_{\underline{\Omega}}(V)$ is similar.
\end{proof}

\subsection{Augmentation}
\label{sec:augmentation}
We will repeatedly need to enlarge witness families $\Omega$ by adding sets of the form $\maxlset{t}{E}$ or $\maxrset{t}{E}$ (see \S\ref{sec:insulation}) for $E\in \Omega$:

\begin{definition}[Augmentation]
\label{def:augment}
Let $\Omega$ be a witness family for $[x,y]\in \T(\Sigma)$. For any $E\in \Omega$ and $0 \le t \le d_E(x,y)$, the collections $\Omega\cup\maxlset{t}{E}$ and $\Omega\cup\maxrset{t}{E}$ are termed the \define{left} and \define{right augmentations} of $\Omega$ along $E$ with parameter $t$.
\end{definition}

\begin{lemma}
\label{lem:add_chunk}
If $\Omega$ is a witness family for $[x,y]$ in $\T(\Sigma)$ and $E\in \Omega$, then $\Omega\cup \maxlset{t}{E}$ and $\Omega\cup \maxrset{t}{E}$ are witness families for each  $0 \le t\le d_E(x,y)$.
\end{lemma}
\begin{proof}
We prove the claim for $\Omega' = \Omega \cup \maxlset{t}{E}$; the proof for $\Omega\cup \maxrset{t}{E}$ is identical. Conditions \ref{projfam-big} and \ref{projfam-see_everyon} are clear because $\Omega\subset \Omega'$ and each $Z\in \maxlset{t}{E}$ satisfies $Z\in \Upsilon(x,y)$ by definition. For condition \ref{projfam-maximal}, suppose $Z\esub W$ are such that $Z\in \Omega'$ and $W\in \Upsilon(x,y)$; we must show $W\in \Omega'$ or else $W\pitchfork Z'$ for some $Z'\in \Omega'$ with $Z\esub Z'$. If $Z\in \Omega$, this follows from the fact that $\Omega$ is a witness family. Otherwise we have $Z\in \maxlset{t}{E}$ so that $Z\esub E$. If $E\pitchfork W$ we have satisfied \ref{projfam-maximal}. If $E\esub W$, then we may apply \ref{projfam-maximal} to $E\in\Omega$ to obtain our conclusion. It thus remains to suppose $Z \esub W \esubn E$. But now  $W\in \Upsilon(x,y)$, $\cc(E\vert_Z)\subset \cc(E\vert_W)$ and $Z\in\lset{t}{E}$ together imply that $W\in \lset{t}{E}$. By maximality of $Z$, it follows that $W = Z\in \Omega'$. This establishes \ref{projfam-maximal} for $\Omega\cup \maxlset{t}{E}$ and proves the lemma.
\end{proof}

Extending suborderings to augmentations will require the following fact.

\begin{lemma}
\label{lem:time_order_and_dist_from_endpoint}
Let $[x,y]$ be a geodesic in $\T(\Sigma)$ and let $Z,W\esub E\esub \Sigma$ be domains such that $\{Z,W,E\}\subset \Upsilon(x,y)$. 
\begin{itemize}
\item If $W \tol Z$ along $[x,y]$, then $d_E(x,\cc(E\vert_W)) \le d_E(x,\cc(E\vert_Z))$.
\item If $d_E(x,\cc(E\vert_W)) \le d_E(x,\cc(E\vert_Z)) - 3$, then $W \tol Z$ along $[x,y]$.
\end{itemize}
The same conclusions of course hold with the roles of $x,y$ and $W,Z$ swapped.
\end{lemma}
\begin{proof}
Consider the first claim. If $W\in \Upsilon^\ell(x,y)$, then $W$ is an annulus and $\partial W$ is short at either $x$ or $y$. The time ordering $W \tol Z$ implies it must be that $\ell_x(\partial W) < \thin$. Hence $\cc(E\vert_W)$ consists of the single curve $\partial W$, which is an element of any Bers marking $\mu_x$ at $x$. Thus
\[d_E(x, \cc(E\vert_W)) = \diam_{\cc(E)}(\pi_E(x)) \le d_E(x, \cc(E\vert_Z)\]
and the first claim holds when $W\in \Upsilon^\ell(x,y)$.

If $W\notin \Upsilon^\ell(x,y)$, then necessarily $d_W(x,y)\ge \indrafi{W}$. Let us set $k_W = d_E(x,\cc(E\vert_W))$ and suppose on the contrary that $k_W > d_E(x,\cc(E\vert_Z))$. 
Since $E$ contains two subdomains that cut each other, $E$ cannot be an annulus. 
Recalling that $\pi_E(x)$ is the set of all essential simple closed curves in $\cc(E)$ achieved by projecting the curves of the Bers marking $\mu_x$ to $E$, it follows that $\pi_E(x)$ contains at least two distinct curves in $\cc(E)$.
In particular
\[k_W > d_E(x,\cc(E\vert_Z)) = \diam_{\cc(E)}(\pi_E(\mu_x)\cup \cc(E\vert_Z)) \ge \diam_{\cc(E)}\pi_E(\mu_x) \ge 1,\]
which gives $k_W \ge 2$. 

Choose curves $\gamma\in \pi_E(x)$ and $\nu\in \cc(E\vert_W)$ such that $d_E(\gamma,\nu) = k_W$. Choose also a curve  $\zeta\in \partial Z$ that cuts $W$, and a geodesic $(\alpha_0,\dotsc,\alpha_m)$ in $\cc(E)$ from $\alpha_0 = \gamma$ to $\alpha_m = \zeta$. The curve $\alpha_m$ cuts $W$ by construction. Thus if $m=0$, we trivially have $\pi_W(\alpha_i)\ne\emptyset$ for each $0 \le i \le m$. Otherwise $m\ge 1$ and for each $0 \le i < m < k_W$ the curve $\alpha_i$ necessarily intersects $\nu$ (and consequently cuts $W$) by the fact that
\[d_E(\nu,\alpha_i) \ge d_E(\nu,\gamma) - d_E(\gamma,\alpha_i) = k_W - i \ge 2.\]
In any case we find that $\pi_W(\alpha_i)\ne\emptyset$ for all $0\le i \le m$. 
It follows from the Bounded Geodesic Image Theorem (Theorem~\ref{thm:bounded_geodesic_image}), that  $d_W(\gamma,\zeta) \le \bgit$. Using $d_W(\partial Z,\zeta)\le 2$  and Lemma~\ref{lem:composing_projections} and recalling that $\rafi\ge 100(\nestprojscommute+\bgit+1)$ (Definition~\ref{def:uniform_constants}), this implies 
\[d_W(x,\partial Z) \le \nestprojscommute + 2 + d_W(\gamma,\zeta) \le \nestprojscommute+2 + \bgit \le \rafi/2.\]
However, the time ordering $W \tol Z$ implies $d_W(x,\partial Z) \ge 2\rafi/3$, a contradiction.

For the second claim, if $W$ and $Z$ were disjoint or nested, we would have $\diam_{\cc(E)}(\cc(E\vert_W)\cup \cc(E\vert_Z))\le 2$; this can be seen by choosing a curve in $\partial W\cup \partial Z$ that is disjoint  from every curve in $\cc(E\vert_W)\cup\cc(E\vert_Z)$. As this is incompatible with the hypothesis $d_E(x,\cc(E\vert_W))\le d_E(x,\cc(E\vert_Z)) - 3$, it must be that $W\pitchfork Z$. Thus either $Z \tol W$ or $W \tol Z$. But by the first part of the lemma, $W \tol Z$ is the only option compatible with the hypothesis.
\end{proof}

We may now extend suborderings to any augmentation in which the parameter and corresponding encroachment are controlled:

\begin{lemma}[Subordering augmentations]
\label{lem:pad_suborder}
Let $\Omega$ be a subordered witness family for $[x,y]$ in $\T(\Sigma)$ and suppose $E\in \Omega$ satisfies $\enc_{\Omega}(E)\le \indrafi{E}/3$. For each parameter $0 \le t \le \lenc_\Omega(E)$ (respectively, $0\le t\le \renc_\Omega(E)$) there is a natural subordering on $\Omega\cup\maxlset{t}{E}$ (respectively, $\Omega\cup \maxrset{t}{E}$) that extends the subordering on $\Omega$. 
\end{lemma}
\begin{proof}
We prove the lemma for $\underline{\Omega} = \Omega\cup \maxlset{t}{E}$; the proof for $\Omega\cup\maxrset{t}{E}$ is symmetric. We first show that conditions \ref{subord-triplenest}--\ref{subord-contribute} give rise to a natural ordering designation $\uswar$ or $\usear$ for each pair $Z,V\in \underline{\Omega}$ with $Z\esubn V$:
\begin{enumerate}[label=(\arabic*)]
\item\label{augment-bothin} If $Z,V\in \Omega$, we use the designation from $\Omega$ and set $Z\uswar V$ iff $Z\swarrow V$.
\item\label{augment-bigin} If $V\in \Omega$ and $Z\notin \Omega$, then $Z\in\maxlset{t}{E}$ and we proceed as follows:
\begin{enumerate}
\item[(i)] If $V = E$, we set $Z\uswar V$. (For the case of $\Omega\cup\maxrset{t}{E}$ with $Z\in \maxrset{t}{E}\setminus\Omega$ we instead set $E\usear Z$).
\item[(ii)] If $E\esubn V$, then we set $Z \uswar V \iff E\swarrow V$ and $V\usear Z \iff V\searrow E$ in accordance with \ref{subord-triplenest}.
\item[(iii)] If $V\esubn E$, then the fact $V\esup Z \in \lset{t}{E}$ implies $V\in \lset{t}{E}$, contradicting the maximality of $Z$ in $\lset{t}{E}$. Hence $V\esubn E$ cannot occur.
\item[(iv)] If $V\pitchfork E$, we claim it must be that $V \tol E$ and therefore set $V\usear Z$ in accordance with \ref{subord-supercutting}. (For the case of $\Omega\cup \maxrset{t}{E}$ with $Z\in \maxrset{t}{E}$, we instead have $E \tol V$ and accordingly set $Z\uswar V$.) To see this, note that $V\esup Z\in\maxlset{t}{E}$ implies there exists $\nu\in \cc(E\vert_Z)$ with $d_E(x,\nu)\le t + 9\wfal$. Since $t \le \lenc_\Omega(E)\le \indrafi{E}/3$, this gives $d_E(y,\nu) \ge \tfrac{2}{3}\indrafi{E}-9\wfal \ge 5\rafi$. Now since $\partial V$ and $\nu$ are disjoint, we have  $d_E(y,\partial V) \ge 4\rafi -2> 2\rafi/3$, showing that $V \tol E$ by Lemma~\ref{lem:characterize_timeorder}.
\end{enumerate}
\item\label{augment-bigout} If $V\notin \Omega$ and $Z\in \Omega$, then \ref{projfam-maximal} provides some $Z'\in \Omega$ with $Z\esub Z'$ and $Z'\pitchfork V$. Thus we set $Z\uswar V$ if $Z'\tol V$ and $V\usear Z$ if $V \tol Z'$ in accordance with \ref{subord-supercutting}. This is well-defined, as Corollary~\ref{cor:time-order_and_severing} ensures it is impossible to have the two such domains $Z'_1,Z'_2\in \Omega$ with the time ordering $Z'_1 \tol V \tol Z'_2$.
\item\label{augment-bothout} The case $Z,V\notin\Omega$ is ruled out by the fact domains in $\maxlset{t}{E}$ are not nested.
\end{enumerate}
The above establishes ordering designations on $\underline{\Omega}$ that extend those of $\Omega$, and so we henceforth use  $\swarrow$ and $\searrow$ for the designations in both $\Omega$ and $\underline{\Omega}$. To prove these give a subordering on $\underline{\Omega}$, it remains (by Remark~\ref{rem:suborder_implications}) to verify \ref{subord-subcutting}--\ref{subord-contribute}:

\textbf{Condition~\ref{subord-subcutting}:} Let $Z,W,V\in \underline{\Omega}$ be such that $Z\swarrow V \searrow W$. The condition is immediate if all three domains lie in $\Omega$. If $V\notin\Omega$, then $Z,W\in \Omega$ because domains in $\maxlset{t}{E}$ cannot be nested. As dictated by \ref{augment-bigout}  above, we may choose $Z',W'\in \Omega$ such that $Z\esub Z'$, $W\esub W'$, and $Z' \tol V \tol W'$. By Corollaries~\ref{cor:time-order_and_severing}--\ref{cor:time-order_subsurf}, this implies $Z'\cut_V W'$ and $Z\cut_V W$ with $Z \tol W$, as desired.

Next suppose $V\in \Omega$. If $Z,W\notin \Omega$, then an examination of case \ref{augment-bigin} above shows that $Z\swarrow V \iff W\swarrow V$ since the ordering designations for $Z\esubn V$ and $W\esubn V$ are both determined by the relationship of $V$ to $E$. As this is not the case, at most one of $Z$ or $W$ can lie outside of $\Omega$. Let us first suppose $Z\in \Omega$ and $W\notin\Omega$, so that $W\in \maxlset{t}{E}$. Now case \ref{augment-bigin} above dictates that the designation $V\searrow W$ must fall under (2ii) with $V\searrow E$ or else (2iv) in which $V \tol E$. In the first case $V\searrow E$ we have $Z\swarrow V \searrow E$ so that \ref{subord-subcutting} for $\Omega$ ensures $Z\cut_V E$ with $Z \tol E$, and in the second case $V \tol E$ we have $Z\swarrow V \tol E$ so that  $Z\cut_V E$ with $Z \tol E$ by \ref{subord-supercutting} for $\Omega$.  In either case, we may conclude $Z\cut_V W$ with $Z \tol W$, as desired. 

It remains to suppose $V,W\in \Omega$ and $Z\notin\Omega$. By case \ref{augment-bigin} above, the designation $Z\swarrow V$ must fall under (2i) with $V = E$, or else (2ii) with $E\swarrow V$. In the latter case $E\swarrow V$ we have $E\swarrow V \searrow W$ so that we may use \ref{subord-subcutting} to conclude $Z\cut_V W$ and $Z\tol W$ as above. So let us restrict our attention to the case $V = E$. Since $E\searrow W$ and $Z\in \maxlset{t}{E}$, we find that $d_E(y,\cc(E\vert_W))\le \renc_\Omega(E)$ and that 
\[d_E(x,\cc(E\vert_Z)) \le t + 9\wfal \le \lenc_\Omega(E) + 9\wfal.\]
Since $\enc_\Omega(E)\le \indrafi{E}/3$ and $d_E(x,y)\ge \indrafi{E}> 30\wfal$, the triangle inequality gives $d_E(x,\cc(E\vert_Z)) \le d_E(x,\cc(E\vert_W)) - \wfal$ so that we may conclude $Z \tol W$ by Lemma~\ref{lem:time_order_and_dist_from_endpoint}. The fact $d_E(\partial Z,\partial W) \ge \rafi > 3$ further ensures that $Z\cut_V W$, as desired.

\textbf{Condition~\ref{subord-supercutting}:} Let $Z,V\in \underline{\Omega}$ and $W\in \Upsilon(x,y)$ be such that $Z \swarrow V \tol W$ or $W \tol V \searrow Z$. We may assume exactly one of $Z$ or $V$ lies outside of $\Omega$. First suppose $V\in \Omega$ and $Z\notin\Omega$ so that $Z\in \maxlset{t}{E}$ and the designation for $Z\esubn V$ is dictated by case \ref{augment-bigin} above. Let us examine these possibilities in turn: If $Z\esubn V$ falls under (2i), then $V = E$ and we must have $Z\swarrow V \tol W$. Here we find that $d_V(y,\partial W)\le \rafi/3$ and that $d_V(x,\partial Z) \le t+9\wfal \le \lenc_\Omega(V) + 9\wfal$. Since $d_V(x,y)\ge \indrafi{V}\ge 30\wfal$ and $\enc_\Omega(V)\le\indrafi{V}/3$, this together with the triangle inequality implies $d_V(\partial Z,\partial W) \ge \wfal$. We may thus conclude the desired $Z\cut_V W$, for otherwise we would find that $d_V(\partial Z, \partial W) \le \rafi/3$ exactly as in the proof of Corollary~\ref{cor:time-order_and_severing}. If $Z\esubn V$ falls under (2ii), then (2ii) dictates $E\swarrow V \tol W$ in the case that $Z\swarrow V$ and  instead dictates $W \tol V \searrow E$ in the case $V\searrow Z$. Either way, we may invoke \ref{subord-supercutting} to conclude $E\cut_V W$ and consequently $Z\cut_V W$. Finally, if $Z\esubn V$ falls under (2iv), then it must be that $V \tol E$ and $V\searrow Z$. Hence we are in the case $W \tol V \searrow Z$ and in fact have $W \tol V \tol E$. Therefore $W\cut_V E$ by Corollary~\ref{cor:time-order_and_severing} and we may conclude $W\cut_V Z$ as desired.

Next suppose $V\notin \Omega$ and $Z\in\Omega$. Then \ref{projfam-maximal} allows us to choose $Z'\in \Omega$ such that $Z\esub Z'$. If $Z\swarrow V \tol W$ then \ref{augment-bigout} dictates that $Z' \tol V \tol W$, and if $W \tol V\searrow Z$ then \ref{augment-bigout} instead dictates $W \tol V \tol Z'$.  Either way, we may invoke Corollary~\ref{cor:time-order_and_severing} to conclude $Z\cut_V W$.

\textbf{Condition~\ref{subord-contribute}:} Let $Z,V\in \underline{\Omega}$ and $W\in \Upsilon(x,y)$ be such that $Z\esubn V$, $Z\pitchfork W$, and $\colsup{W}{\underline{\Omega}} = V$. Observe that these facts imply $W\notin\underline{\Omega}$. We must show that $Z\swarrow V \implies Z \tol W$ and $V\searrow Z \implies W \tol Z$.

First suppose $V\notin\Omega$. Then we must have $Z\in \Omega$ because domains in $\maxlset{t}{E}$ cannot be nested. Now \ref{projfam-maximal} provides a domain $Z'\in \Omega$ such that $Z'\pitchfork V$ and $Z\esub Z'$. Since $Z\pitchfork W$, it cannot be that $Z'$ and $W$ are disjoint. Nor can we have $Z'\esub W$. If $W \esub Z'\in \underline{\Omega}$, then the assumption $\colsup{W}{\underline{\Omega}} = V$ implies $V\esub Z'$ by Lemma~\ref{lem:characterize_supremum}. As this contradicts $Z'\pitchfork V$, we must have $Z'\pitchfork W$. Now, if $Z\swarrow V$, then \ref{augment-bigout} dictates that $Z' \tol V$ so that we find $Z \tol W$ by Corollary~\ref{cor:time-order_subsurf}. If instead $V\searrow Z$, then \ref{augment-bigout} dictates $V \tol Z'$ and we similarly deduce $W \tol Z$. This establishes \ref{subord-contribute} when $V\notin \Omega$.

Next suppose $V\in \Omega$. Then the facts $V\in \Omega\subset\underline{\Omega}$ and $\colsup{W}{\underline{\Omega}}= V$ imply $\colsup{W}{\Omega} = V$ as well (since Lemma~\ref{lem:characterize_supremum} implies $V \esub V_0$ for any $V_0\in \Omega$ with $W\esub V_0$). Thus if $Z\in \Omega$ as well we may invoke \ref{subord-contribute} to prove the claim. It therefore suffices to suppose $Z\notin\Omega$ so that $Z\in\maxlset{t}{E}$. Let us examine the subcases of \ref{augment-bigin} in turn: 

If $V = E$, then (2i) dictates $Z\swarrow V$ and we must show $Z \tol W$. If instead $W \tol Z$, then Lemma~\ref{lem:time_order_and_dist_from_endpoint} implies that $d_E(x,\cc(E\vert_W)) \le d_E(x,\cc(E\vert_Z)) \le t + 9\wfal$. Since $\colsup{W}{\underline{\Omega}} = E$ it cannot be that $W\in\lset{t}{E}$, for then we would have $W\esub W'\esubn E$ for some $W'\in \maxlset{t}{E}\subset \underline{\Omega}$. Thus $d_E(x,\cc(E\vert_W))\notin [t-9\wfal,t+9\wfal]$. Together, these inequalities give
\[d_E(x,\cc(E\vert_W))  < t-9\wfal \le \lenc_\Omega(E) - 9\wfal.\]
By definition of encroachment, we may choose a domain $U\in \Omega$ such that $U\swarrow E$ and $d_E(x,\cc(E\vert_U)) = \lenc_\Omega(V)$. By Lemma~\ref{lem:time_order_and_dist_from_endpoint}, the above inequality implies $W \tol U$ along $[x,y]$. We now have $U,V\in \Omega$ with $U\swarrow V$ and $W \tol U$. Since we have seen $\colsup{W}{\Omega} =V$, this contradicts the fact that $\Omega$ satisfies \ref{subord-contribute}. Therefore, $W \tol Z$ leads to a contradiction, and we may conclude the desired ordering $Z \tol W$.

If $E\esubn V$, then (2ii) dictates $Z\swarrow V \iff E\swarrow V$. Observe that $W\clld{\Omega}V$ precludes $W\esub E$. We also cannot have $E\esub W$ nor $E$ and $W$ disjoint (since $Z\pitchfork W$). Therefore $E\pitchfork W$, and we may apply \ref{subord-contribute} to $E\esubn V$ in $\Omega$ to conclude $Z\swarrow V \implies E\swarrow V \implies E \tol W$, which in turn implies $Z \tol W$ as desired. If instead $V\searrow Z$, we similarly conclude $W \tol E$ and consequently $W \tol Z$. 

If $V\pitchfork E$, then (2iv) dictates that $V \tol E$ and $V\searrow Z$. We cannot have $E\esub W$ nor $E$ and $W$ disjoint (because $Z\pitchfork E$), and if $W \esub E$, then Lemma~\ref{lem:characterize_supremum} would imply $V = \colsup{W}{\Omega} \esub E$, contradicting $V\pitchfork E$. Therefore it must be that $W\pitchfork E$. We may now apply Corollary~\ref{cor:time-order_subsurf} to $V \tol E$ to conclude $W \tol E$ and $W \tol Z$, as desired. This verifies  \ref{subord-contribute} when $V\in\Omega$ and completes the proof of the lemma.
\end{proof}

As for refinements, we shall need to bound each augmentation's encroachments.

\begin{lemma}[Augmented encroachments]
\label{lem:encroach_augment}
Let $\Omega$ be a subordered witness family for $[x,y]$ in $\T(\Sigma)$, let $E\in \Omega$ satisfy $\enc_\Omega(E)\le \indrafi{E}/3$, and let $\underline{\Omega} = \Omega\cup\maxlset{t}{E}$ with $0 \le t \le\lenc_\Omega(E)$ or $\underline{\Omega} = \Omega\cup\maxrset{t}{E}$ with $0\le t \le \renc_\Omega(E)$ be the augmentation of $\Omega$ along $E$ with parameter $t$. Then every domain $V\esub \Sigma$ satisfies
\[\lenc_{\underline{\Omega}}(V) \le \begin{cases}
\max\{\lenc_\Omega(E),t+9\wfal\}, & \text{if $V = E$ and $\underline{\Omega}=\Omega\cup \maxlset{t}{E}$}\\
\max\{\lenc_{\Omega}(V),\rafi\}, & \text{else}
\end{cases}\]
and
\[\renc_{\underline{\Omega}}(V) \le \begin{cases}
\max\{\renc_\Omega(E),t+9\wfal\}, & \text{if $V = E$ and $\underline{\Omega}=\Omega\cup \maxrset{t}{E}$}\\
\max\{\renc_{\Omega}(V),\rafi\}, & \text{else}
\end{cases}.\]
\end{lemma}
\begin{proof}
We prove the lemma for $\Omega\cup\maxlset{t}{E}$; the proof for $\Omega\cup\maxrset{t}{E}$ is symmetric. First suppose $V\notin\Omega$. The claim is immediate for $V\notin\underline{\Omega}$ (since then $\enc_{\underline{\Omega}}(V) = 0$), so we assume $V\in \underline{\Omega}\setminus \Omega$. The proof of Lemma~\ref{lem:pad_suborder} shows that any $Z\in \underline{\Omega}$ with $Z\esubn V$ falls under case \ref{augment-bigout} and satisfies $Z\esub Z'$ for some $Z'\in \Omega$ with $Z'\pitchfork V$. 
Therefore, as in the proof of Lemma~\ref{lem:encroach_refine}, $Z\swarrow V$ implies $d_V(x,\cc(V\vert_Z)) \le \rafi$ and $V\searrow Z$ implies $d_V(y,\cc(V\vert_Z))\le \rafi$. 
Thus the lemma holds for $V\notin\Omega$.

Next suppose $V\in \Omega$ and consider any $Z\in\underline{\Omega}$ with $Z\esubn V$. If $Z\swarrow V$, then the proof of Lemma~\ref{lem:pad_suborder} shows that either (1) $Z\in \Omega$, (2) $Z\in \maxlset{t}{E}$ with $V = E$, or (3) $Z\esub E \swarrow V$ so that $\cc(V\vert_Z)\esub \cc(V\vert_E)$. In the former and latter cases we conclude the desired bound $d_V(x,\cc(V\vert_Z))\le \lenc_\Omega(V)$. In the middle case, we have $V = E$ and instead find $d_V(x,\cc(V\vert_Z))\le t + 9\wfal$ by the definition of $\maxlset{t}{E}$. Thus we conclude the stated bound on $\lenc_{\underline{\Omega}}(V)$. If instead $V\searrow Z$, the proof of Lemma~\ref{lem:pad_suborder} now shows that either (1) $Z\in \Omega$, (2) $V\searrow E \esup Z$ so that $\cc(V\vert_Z)\subset \cc(V\vert_E)$, or (3) $V \tol E \esup Z$. In these three cases we may respectively bound $d_V(y,\cc(V\vert_Z))$ by $\renc_\Omega(V)$, $\renc_\Omega(V)$, and $\rafi$. Therefore $\renc_{\underline{\Omega}}(V) \le \max\{\renc_\Omega(V),\rafi\}$ as claimed.
\end{proof}

\subsection{Completion}

We may now extend any witness family to a complete and insulated one:

\begin{definition}[Insulated completion]
\label{def:completion}
If $\Omega$ is a witness family for $[x,y]$ in $\T(\Sigma)$, define $\check{\Omega}$ to be
\[\check{\Omega} \colonequals \Omega \cup \left(\bigcup_{E\in \Omega} \maxleftpad{E}\cup \maxrightpad{E}\right) \cup \left(\bigcup_{(A,B)\text{ a cutting pair in }\Omega} \maxbigd{A,B}\right).\]
The \define{insulated completion} of $\Omega$ is then defined to be $\overline{\Omega} = \cup_{i\in \mathbb{N}} \Omega_i$, where $\Omega_0,\Omega_1,\dotsc$ is the sequence recursively defined by $\Omega_0 = \Omega$ and $\Omega_{i+1} = \check{\Omega}_i$.
\end{definition}

\begin{lemma}
\label{lem:completion_properties}
Let $\Omega$ be a witness family for $[x,y]$ in $\T(\Sigma)$ and let $\Omega_0,\Omega_1,\dotsc$ be the sequence recursively defined by $\Omega_0 = \Omega$ and $\Omega_{i+1} = \check{\Omega}_i$. Then 
\begin{enumerate}
\item $\check{\Omega}$ is a witness family,
\item $\overline{\Omega} = \Omega_{\plex{\Sigma}+1}$, and 
\item $\overline{\Omega}$ is a complete and insulated witness family.
\item Any subordering on $\Omega$ extends to natural suborderings on $\check{\Omega}$ and $\overline{\Omega}$ whose encroachments, for each $V\esubset \Sigma$, satisfy
\[\lenc_{\check{\Omega}}(V),\lenc_{\overline{\Omega}}(V) \le \max\{\lenc_\Omega(V),9\wfal\}\qquad\text{and}\qquad
\renc_{\check{\Omega}}(V),\renc_{\overline{\Omega}}(V) \le \max\{\renc_\Omega(V),9\wfal\}.\]
\end{enumerate}
\end{lemma}
\begin{proof}
For (1), first observe that $\Omega$ is finite. This is because there are only finitely many subsurfaces $Z\esub \Sigma$ with $d_Z(x,y)\ge \rafi$ and only finitely many annuli $A$ with $\ell_x(\partial A)< \thin$ or $\ell_y(\partial A) < \thin$. The family $\check{\Omega}$ may be constructed by adding the finitely many families $\maxleftpad{E}$, $\maxrightpad{E}$, and $\maxbigd{A,B}$ one at a time. Lemmas \ref{lem:refine} and \ref{lem:add_chunk} show that each addition results in another witness family. Therefore the output $\check{\Omega}$ of those finitely many additions is a witness family.

For (2), let $k_0 = \max\{\plex{V} \mid V\in \Omega_0\}$, and for each $i > 0$ let $k_i = \max\{\plex{V} \mid V\in \Omega_i \setminus \Omega_{i-1}\}$ be the maximal complexity of any domain that was added during the $i$th iteration, with the convention that $k_i = -\infty$ if $\Omega_i = \Omega_{i-1}$ (in which case $\overline{\Omega} = \Omega_{i-1}$). We claim that $k_{i-1} \ge 1 + k_i$ for each $i > 0$. The result $\overline{\Omega} = \Omega_{\plex{\Sigma}+1}$  will then follow from the observation $k_0 \le \plex{\Sigma}$. To see this, suppose $V\in \Omega_{i}\setminus\Omega_{i-1}$. Then either $V\in \maxleftpad{E}\cup \maxrightpad{E}$ for some $E\in \Omega_{i-1}$, or else $V\in \maxbigd{A,B}$ for some cutting pair $(A,B)$ in $\Omega_{i-1}$. If $i = 1$, this shows that $V$ is a proper subsurface of domain in $\Omega_0$ and hence that $k_0 \ge 1 + \plex{V}$. Next suppose $i > 1$. If $V\in\maxleftpad{E}\cup \maxrightpad{E}$ then we must have  $E\notin\Omega_{i-2}$, for otherwise $V\in \Omega_{i-1} = \check{\Omega}_{i-2}$ by construction. Similarly, if $V\in \maxbigd{A,B}$, then either $A\notin\Omega_{i-2}$ or $B\notin\Omega_{i-2}$ by the same reasoning. Therefore $V$ is a proper subsurface of a domain in $\Omega_{i-1}\setminus\Omega_{i-2}$ and we may conclude the claimed inequality $k_i \le k_{i-1} -1$. 

Combining (1) and (2), we see that $\overline{\Omega}$ is a witness family and that $\maxbigd{A,B}\cup\maxleftpad{E}\cup \maxrightpad{E}\subset \overline{\Omega}$ for all $E\in \overline{\Omega}$ and all cutting pairs $(A,B)$ in $\overline{\Omega}$. Therefore $\overline{\Omega}$ is complete by Lemma~\ref{lem:filled_complete} and insulated by Definition~\ref{def:insulation}, which proves (3). Finally, (4) follows from Lemmas~\ref{lem:refined_suborder}, \ref{lem:encroach_refine},  \ref{lem:pad_suborder}, and \ref{lem:encroach_augment}.
\end{proof}

We may also use Lemmas~\ref{lem:bounded_refinement} and \ref{lem:bounded_pads} to control the cardinality of $\overline{\Omega}$.

\begin{lemma}
\label{lem:completion_bound}
For each $-1 \le j \le \plex{\Sigma}$, there exists a computable function $G_j \colon \mathbb{N}^{\plex{\Sigma} - j} \to \mathbb{N}$ depending only $\indrafi{\plex{\Sigma}},\dotsc,\indrafi{j+1}$ with the following property. If $\Omega$ is any witness family for any geodesic $[x,y]$ in $\T(\Sigma)$, then 
\[\abs{\overline{\Omega}}_j - \abs{\Omega}_j\le G_j(K_{\plex{\Sigma}}, \dotsc, K_{j+1})\]
for any tuple $(K_{\plex{\Sigma}},\dotsc, K_{j+1})$ satisfying $\abs{\Omega}_i \le K_i$ for each $\plex{\Sigma}\ge i > j$. In particular, there exists a computable function $G\colon \mathbb{N}\to \mathbb{N}$ such that $\abs{\overline{\Omega}} \le G(\abs{\Omega})$ for every witness family $\Omega$.
\end{lemma}
\begin{proof}
For $j = \plex{\Sigma}$, we may take the constant function $G_{\plex{\Sigma}}\equiv 1$ since $\overline{\Omega}$ can contain at most one domain of this complexity. For the remaining $j$, we proceed inductively: Fix an integer $-1 \le j < \plex{\Sigma}$ and suppose that the stipulated functions $G_{\plex{\Sigma}},\dotsc G_{j+1}$ have been constructed. Let us count the domains $V\in \overline{\Omega}\setminus \Omega$ of complexity $\plex{V}=j$. 
Any such $V$ satisfies $V\in \maxleftpad{E}\cup\maxrightpad{E}$ for some $E\in \overline{\Omega}$ with $\plex{E} > j$, or else $V\in \maxbigd{A,B}$ for some $A,B\in \overline{\Omega}$ with $\plex{A},\plex{B} > j$. Letting $J$ denote the number of domains of $\overline{\Omega}$ of complexity at least $j+1$, it follows from Lemmas~\ref{lem:bounded_refinement} and \ref{lem:bounded_pads} that 
\[\abs{\overline{\Omega}}_j - \abs{\Omega}_j = 
\abs{\overline{\Omega}\setminus\Omega}_j \le \binom{J}{2}(2\indrafi{j+1})^{\plex{\Sigma}+2}+ 2J(2\indrafi{j+1})^{\plex{\Sigma}+3} .\]
However, our induction hypothesis gives
\[J \le \sum_{i=j+1}^{\plex{\Sigma}} \abs{\overline{\Omega}}_i \le \big(G_{\plex{\Sigma}} + K_{\plex{\Sigma}}\big) +  \dotsb + \big(G_{j+1}(K_{\plex{\Sigma}},\dotsc K_{j+2}) + K_{j+1}\big).\] 
Combining these inequalities shows that $\abs{\overline{\Omega}}_j-\abs{\Omega}_j$ is bounded above by function of $(K_{\ent{\Sigma}},\dotsc, K_{j+1})$ that depends only on the thresholds $\indrafi{\plex{\Sigma}},\dotsc,\indrafi{j+1}$, as claimed. The final assertion of the lemma then holds for the function $G\colon \mathbb{N}\to \mathbb{N}$ defined as
\[G(x) = x + G_{\plex{\Sigma}} + \dotsb + G_{-1}(x,\dotsc,x).\qedhere\]
\end{proof}

\section{Complexity of witness families}
\label{Sec:complexity}
We next explain how the structure of a witness family organizes curve complex projection data into a quantity that we call complexity.
Let us first designate an acronym combining the many types of witness families 
that have been introduced in Definitions~\ref{def:completeness}, \ref{def:insulation}, \ref{def:subordering}, and \ref{def:wide}.
\begin{terminology}[WISC]
\label{def:cisb}
A witness family is \define{WISC} if it is wide, insulated, subordered, and complete.
\end{terminology}

The starting point is the following notion that is suggested by completeness:

\begin{definition}[Contribute]
\label{def:contribute}
If $\Omega$ is a complete witness family for a geodesic segment $[x,y]\in \T(\Sigma)$, we say that a domain $Z\esub \Sigma$ \define{contributes to $V\in \Omega$} if $Z\in \Upsilon(x,y)$ and $V = \colsup{Z}{\Omega}$.
\end{definition}

Since every domain $Z\in \Upsilon(x,y)$ has a unique $\Omega$--supremum, we may partition the domains of $\Upsilon(x,y)$ according to the elements of $\Omega$ they contribute to. We would like to somehow combine the data $\{(Z,d_Z(x,y)) \mid Z\text{ contributes to } V\}$ into a notion of ``distance in $V$'' that, when summed over all $V\in \Omega$, can be used for counting problems and is moreover related to the total Teichm\"uller distance $d_{\T(\Sigma)}(x,y)$. A subordering on $\Omega$ allows us to accomplish this by \emph{resolving} $x$ and $y$ into points in the Teichm\"uller space $\T(V)$ for each $V\in \Omega$. 
In fact, we can resolve any point coarsely aligned between $x$ and $y$.

\subsection{Teichm\"uller resolutions}
\label{sec:teich_resolutions}

Recall the constant $\pfback\ge 0$ specified at the start of \S\ref{sec:witness-families} (which determines the $\indrafi{i}$).

\begin{definition}[Projection tuple]
\label{def:projection_tuple}
Let $\Omega$ be a WISC witness family for a geodesic  $[x,y]$ in $\T(\Sigma)$. For each domain $V\in \Omega$ and point $w\in \T(\Sigma)$ satisfying
\[d_Z(x,w) + d_Z(w,y)\le d_Z(x,y) + 9\pfback \quad\text{for all }Z\esub V,\]
define its \define{projection tuple} to be the tuple $(\tilde{w}_Z)\in \prod_{Z\esub V}\cc(Z)$ given by:
\begin{align*}
\tilde{w}_Z\colonequals&\begin{cases}
\pi_Z(y), & \text{if } Z \in \Upsilon(x,y)\text{ and }\colsup{Z}{\Omega} \swarrow V\\
\pi_Z(x), & \text{if } Z\in \Upsilon(x,y)\text{ and }V\searrow \colsup{Z}{\Omega}\\
\pi_Z(w), &\text{else}\end{cases}.
\end{align*}
In particular, $\tilde{w}_Z = \pi_Z(w)$ whenever $Z\esub\Sigma$ contributes to $V$. 
\end{definition}

\begin{proposition}
\label{prop:consistent_resolution}
With the notation from Definition~\ref{def:projection_tuple}, the projection tuple $(\tilde{w}_Z)\in \prod_{Z\esub V}\cc(Z)$ is $k$--consistent for some constant $k$ depending only on $\pfback$.
\end{proposition}
\begin{proof}

Let $U,Z \esub V$ be arbitrary subdomains. We must show that:
\begin{equation}\label{eqn:condition}
\begin{array}{rcl}
U\pitchfork Z & \implies & \min\{d_U(\tilde{w}_U,\partial Z), d_Z(\tilde{w}_Z,\partial U)\} \le k\\
U\esub Z & \implies & \min\{d_U(\tilde{w}_U,\pi_U(\tilde{w}_Z)), d_Z(\tilde{w}_Z,\partial U)\} \le k
\end{array}
\end{equation}
Note that for each $p\in \{x,w,y\}$ the pair $(\pi_U(p),\pi_Z(p))$ in $\cc(U)\times \cc(Z)$ satisfies these conditions with constant $\consist$ by Theorem~\ref{Thm:Consistency}. 
Thus we may assume $\tilde{w}_U \ne \pi_U(w)$ or $\tilde{w}_Z \ne \pi_Z(w)$.
We may additionally assume $d_U(x,y) \ge \indrafi{U}$ and $d_Z(x,y)\ge \indrafi{Z}$. Indeed, if say $d_Z(x,y) < \indrafi{Z}$, then by coarse alignment $\pi_Z(x)\cup\pi_Z(w)\cup\pi_Z(y)$ has diameter at most $9\pfback+\indrafi
{Z}\le 2\irafi$. Hence, regardless of whether $\tilde{w}_U$ is defined as $\pi_U(x)$, $\pi_U(w)$, or $\pi_U(y)$ we may move $\tilde{w}_Z$ by distance at most $2\irafi$ to arrive at a $\consist$--consistent pair $(\pi_U(p),\pi_Z(p))$ as above. In particular, $U,Z\in \Upsilon(x,y)$ and $\colsup{U}{\Omega}$ and $\colsup{Z}{\Omega}$ both exist by the completeness of $\Omega$.

Suppose first $U\pitchfork Z$. By symmetry, we may suppose $\tilde{w}_Z \ne \pi_Z(w)$ so that $\colsup{Z}{\Omega}$ is a proper subsurface of $V$. We only consider the case $\colsup{Z}{\Omega} \swarrow V$ as the opposite case $V\searrow \colsup{Z}{\Omega}$ is symmetric. In this case $\tilde{w}_Z = \pi_Z(y)$ by definition.
If we also have $\colsup{U}{\Omega} \swarrow V$, then $\tilde{w}_U = \pi_U(y)$ and the pair $(\tilde{w}_U, \tilde{w}_Z)$ is $\consist$--consistent. If instead $V \searrow \colsup{U}{\Omega}$, then \ref{subord-subcutting} and Corollary~\ref{cor:time-order_subsurf} imply that $Z \tol U$ along $[x,y]$. Therefore $d_Z(\tilde{w}_Z,\partial U) = d_Z(y,\partial U) < \rafi/3$ by Lemma~\ref{lem:characterize_timeorder}, and \eqref{eqn:condition} is satisfied. The only remaining possibility is $\colsup{U}{\Omega} = V$. In this case we necessarily have $U\pitchfork (\colsup{Z}{\Omega})$ ($U$ is not disjoint from $\colsup{Z}{\Omega}$ since $U\pitchfork Z\esub \colsup{Z}{\Omega}$,  $U$ is not contained in $\colsup{Z}{\Omega}$ as that would give $\colsup{U}{\Omega} \esub \colsup{Z}{\Omega}$, and $\colsup{Z}{\Omega}$ is not contained in $U$ since that would give $Z\esub U$). Thus  $U$ and $\colsup{Z}{\Omega}$ are time-ordered. Since $\colsup{Z}{\Omega} \swarrow V$ and $\colsup{U}{\Omega} = V$, condition \ref{subord-contribute} forces $\colsup{Z}{\Omega} \tol U$ which in turn implies $Z \tol U$. Therefore $d_Z(\tilde{w}_Z,\partial U) = d_Z(y,\partial U) < \rafi/3$, as above, and we have verified \eqref{eqn:condition}  when $U\pitchfork Z$.

Next let us suppose that $U \esub Z$. Then $\colsup{U}{\Omega} \esub \colsup{Z}{\Omega}$. If $\colsup{U}{\Omega} = \colsup{Z}{\Omega}$, then consistency is automatically satisfied by definition of $\tilde{w}_U,\tilde{w}_Z$ and Theorem~\ref{Thm:Consistency}. So suppose $\colsup{U}{\Omega} \esubn \colsup{Z}{\Omega}$. If $\colsup{Z}{\Omega} \esubn V$, then condition \ref{subord-triplenest} ensures that  $\colsup{U}{\Omega} \swarrow V$ iff $\colsup{Z}{\Omega} \swarrow V$ and we again have consistency by Theorem~\ref{Thm:Consistency}. The only remaining possibility is $\colsup{U}{\Omega} \esubn \colsup{Z}{\Omega} = V$. 
Let us consider the case $\colsup{U}{\Omega} \swarrow V$ (the other case $V \searrow \colsup{U}{\Omega}$ being similar).
In this case we have $\tilde{w}_Z = \pi_Z(w)$ and $\tilde{w}_U = \pi_U(y)$  by definition.

\begin{claim}
\label{claim:proving_consist_of_resolution_point}
 $d_Z(x,\partial U) < \enc_\Omega(V) + \rafi \le \irafi/2$.
\end{claim}
\begin{proof}
We clearly cannot have $Z\esub \colsup{U}{\Omega}$. If $Z\pitchfork \colsup{U}{\Omega}$ then \ref{subord-contribute} implies we must have $\colsup{U}{\Omega} \tol Z$ and consequently $d_Z(x,\partial(\colsup{U}{\Omega})) \le \rafi/3$. Since $\partial U \cup  \partial(\colsup{U}{\Omega})$ is a curve system on $\Sigma$ we have $d_Z(\partial U, \partial(\colsup{U}{\Omega})) \le 2$. Thus $d_Z(x,\partial U) \le 2+\rafi/3 < \rafi.$

Since $\colsup{U}{\Omega}$ and $Z$ cannot be disjoint, it remains to suppose $\colsup{U}{\Omega} \esub Z$. There are two possibilities: Firstly, if $Z = V$, then
\[d_Z(x,\partial U) = d_V(x,\partial U) \le d_V(x,\cc(V\vert_{\colsup{U}{\Omega}})) \le \enc_{\Omega}(V)\le \irafi/3\]
by the fact $\colsup{U}{\Omega} \swarrow V$. Secondly, if $Z \esubn V$, then  $\colsup{Z}{\Omega} = V$ implies $Z\notin \Omega$ so that \ref{projfam-maximal} provides some $Z'\in \Omega$ with  $Z\pitchfork Z'$ and  $\colsup{U}{\Omega} \esub Z'$. Note that we must have $Z'\ne V$. If $Z'\esubn V$, then \ref{subord-triplenest} implies $Z'\swarrow V$ so that \ref{subord-contribute} forces $Z' \tol Z$. Otherwise we have $Z'\cut V$ so that \ref{subord-supercutting} (using $\colsup{U}{\Omega} \swarrow V$ and $\colsup{U}{\Omega} \esub Z'$) forces $Z' \tol V$ and we may again conclude $Z' \tol Z$ by Corollary~\ref{cor:time-order_subsurf}. Therefore $d_Z(x,\partial Z')\le \rafi/3$ so that we may use $U \esub Z'$ to conclude $d_Z(x,\partial U) \le \rafi/3+2 <  \rafi$ as above. 
\end{proof}

Since $\diam_{\cc(U)}(\pi_U(w), \pi_U(\pi_Z(w)))$ is bounded by Lemma~\ref{lem:composing_projections}, verifying (\ref{eqn:condition}) amounts to bounding $\min\{d_U(y,w), d_Z(w,\partial U)\}$.  
Thus if $d_U(w,y) \le \irafi$ we are done. Otherwise $d_U(w,y) > \irafi$, and applying Corollary~\ref{cor:bgit_for_teich} and  Claim~\ref{claim:proving_consist_of_resolution_point} gives
\[d_Z(w,x) + d_Z(x,y) \le d_Z(w,\partial U) + d_Z(\partial U,y) + \irafi \le d_Z(w,y) + 2\irafi.\]
On the other hand, the coarse alignment hypothesis on $w$ gives
\[d_Z(x,w) + d_Z(w,y) \le d_Z(x,y) + 9\wfal.\]
Combining these inequalities yields
\[2d_Z(x,w) + d_Z(w,y) -9\wfal \le d_Z(w,y) + 2\irafi,\]
or equivalently $d_Z(w,x) \le (2\irafi + 9\wfal)/2 \le 3\irafi/2$. Thus the triangle inequality and Claim~\ref{claim:proving_consist_of_resolution_point} now give the desired bound $d_Z(w,\partial U)\le 2\irafi$.
\end{proof}

Combining Proposition~\ref{prop:consistent_resolution} with Theorem~\ref{Thm:Consistency} and Lemma~\ref{lem:build_marking}, we are now able to \emph{resolve} $w$ into the Teichm\"uller space of any $V\in \Omega$:

\begin{definition}[Resolution point]
\label{thm:resolve}
\label{def:resolution}
Let $\Omega$ be a WISC witness family for $[x,y]$ in $\T(\Sigma)$. For any $V\in \Omega$, there are coarsely well-defined \define{resolution points} $\res{x}{\Omega}{V},\res{y}{\Omega}{V}\in \T(V)$ constructed as follows: Let $w\in \{x,y\}$. If $V$ is nonannular, then $\res{w}{\Omega}{V}\in \T_\thin(V)$ is a thick point realizing the consistent tuple $(\tilde{w}_Z)_{Z\esub V}$ from Definition~\ref{def:projection_tuple}. If $V$ is an annulus, then the tuple$(\tilde{w}_Z)_{Z\esub V}$ is a singleton $\tilde{w}_V\in \cc(V)$, and we define $\res{w}{\Omega}{V}$ to be the point in $\T(V) = \H^2$  whose twist coordinate is given by $\tilde{w}_V = \pi_V(w)$, and whose length coordinate is $\frac{1}{\min\{\thin, \ell_w(\partial V)\}}$. 
\end{definition}

\subsection{Complexity via Teichm\"uller distance}
\label{sec:Complexity length}

Given a WISC witness family $\Omega$ for a Teichm\"uller geodesic $[x,y]$ in $\T(\Sigma)$, Proposition~\ref{prop:consistent_resolution} provides a pair $\res{x}{\Omega}{V},\res{y}{\Omega}{V}$ of 
resolutions
for each $V\in \Omega$. We now combine these into the following quantity:

\begin{definition}[Complexity]
\label{def:complexity}
The \define{complexity} of a WISC witness family $\Omega$ for a geodesic $[x,y]$ in $\T(\Sigma)$ is the weighted sum
\[\cl(\Omega) = \sum_{V\in \Omega} \entstar{V} d_{\T(V)}(\res{x}{\Omega}{V},\res{y}{\Omega}{V})\]
where $\entstar{V} = \ent{V}$ for every nonannular domain, and for annuli $A$ we set $\entstar{A} = 1$ in the case that $\res{x}{\Omega}{A},\res{y}{\Omega}{A}$ are both $\thin$--thick, and otherwise set $\entstar{A} = 2 = \ent{A}$.
\end{definition}

\begin{remark}
\label{rem:complexity_features}
Let us highlight three features of this definition. 
\begin{enumerate}
\item
\label{item:complexity_feature-captures_data}
 The resolution points $\res{x}{\Omega}{V},\res{y}{\Omega}{V}$ coarsely encode all the projection data of $x,y$, with the result that it is possible to reconstruct the original points from their resolutions.
This allows one to relate complexity $\cl(\Omega)$ to counting problems, as we do in \S\ref{Sec:countingcomplexity} below.

\item 
\label{item:complexity_feature-compare_dist_formula}
It is helpful to compare this definition of $\cl(\Omega)$ to the distance formula Theorem~\ref{thm:distance_formula}. Indeed, if one applies the distance formula to each term $d_{\T(V)}(\res{x}{\Omega}{V},\res{y}{\Omega}{V})$,
the result is a weighted (by $\entstar{V}$ and the multiplicative errors)
sum of curve complex distances $d_Z(x,y)$ for all $Z\in \Upsilon(x,y)$. Thus $\cl(\Omega)$ is coarsely equivalent to $d_{\T(\Sigma)}(x,y)$ with some bounded but \emph{unknown} multiplicative and additive error. The purpose of \S\S\ref{sec:bound_contrib_witness}--\ref{sec:bad_sets} below to show that one can choose $\Omega$ carefully so that, up to only additive error, $\cl(\Omega)$ is bounded above by the \emph{explicit} multiple $\ent{\Sigma}d_{\T(\Sigma)}(x,y)$ (Theorem~\ref{thm:main_complexity_length_bound}).
This multiplicative control is crucial in our counting applications (Theorem~\ref{thm:countcomplexity}) since the quantity $\cl(\Omega)$ appears  in the exponent.

\item 
\label{item:complexity_feature-alignment}
The final and perhaps least apparent feature is that by decomposing the Teichm\"uller distance into separate subsurfaces, the quantity $\cl(\Omega)$ is able to tap into the hyperbolicity of curve complexes and promote alignment in curve complexes to a sort of alignment for complexity. That is, the definition is constructed with the heuristic that if $(x,y,z)$ is an aligned triple in $\T(\Sigma)$ with associated WISC witness families $\Omega_x^y, \Omega_y^z, \Omega_x^z$, then one should morally expect $\cl(\Omega_x^y) + \cl(\Omega_y^z) \le \cl(\Omega_x^z)$ up to additive error. 
To achieve this precise statement seems to be quite difficult. 
However, we will show in Theorem~\ref{thm:main_complexity_length_bound} that the witness families can be chosen so that, up to only additive error, $\cl(\Omega_x^y) + \cl(\Omega_y^z)$ is bounded above by $\ent{\Sigma}d_{\T(\Sigma)}(x,z)$
provided $(x,y,z)$ is strongly aligned. This feature together with the above-mentioned Theorem~\ref{thm:countcomplexity} make complexity a useful tool for counting orbit points of finite-order and reducible mapping classes.
\end{enumerate}
\end{remark}

\subsection{Complexity of tuples}
\label{sec:complexity_of_tuples}

To obtain the features indicated in Remark~\ref{rem:complexity_features}, we will work in a more general setting of tuples of witness families.
Recall the parameter $\wfal\ge 2\rafi$ from \S\ref{sec:witness-families} that determines the constants $\indrafi{i}$ and satisfies \eqref{eqn:indexed_thresholds}.

\begin{definition}
\label{def:tuple-proj-fam}
A \define{witness family for a strongly $\wfal$--aligned tuple} $(x_0,\dots,x_n)$ in $\T(\Sigma)$ is a tuple $\Omega = (\Omega_1,\dots,\Omega_n)$ where each $\Omega_i$ is a witness family for $[x_{i-1},x_i]$.
\end{definition}
All of the notation and terminology from \S\ref{sec:witness-families}---such as subordering, refinement, augmentation, completion---are extended \emph{componentwise} to the setting of witness families for tuples.
Thus $\Omega$ has a given property provided it holds for each $\Omega_i$. In particular, a subordering on $\Omega$ is a subordering on each $\Omega_i$, and we will write $\indsw{i}$ and $\indse{i}$ for the subordering designations on $\Omega_i$. We additionally define encroachments as $\enc_{\Omega}(V) = \max_i \enc_{\Omega_i}(V)$ and similarly for $\lenc_\Omega(V)$ and $\renc_{\Omega}(V)$, and define the insulated completion of $\Omega$ to be $\overline{\Omega} = (\overline{\Omega_1},\dots,\overline{\Omega_n})$.

\begin{notation}
When a strongly $\wfal$--aligned tuple $(x_0,\dots,x_n)$ has been specified, we will use the shorthand $\Upsilon_i = \Upsilon(x_{i-1},x_i)$ and similarly $\Upsilon_i^\ell = \Upsilon_i^\ell(x_{i-1},x_i)$ and $\Upsilon_i^c = \Upsilon^c(x_{i-1},x_i)$. Similarly, if $\Omega = (\Omega_1,\dots,\Omega_n)$ is a witness family for $(x_0,\dots,x_n)$, we will by abuse of notation write  $V\in \Omega$ to mean that $V \in \cup_i \Omega_i$. 
\end{notation}

\begin{definition}
\label{def:complexity_for_tuple}
A witness family $\Omega$ for a strongly $\wfal$--aligned tuple $(x_0,\dots,x_n)$ is WISC if each $\Omega_i$ is WISC,
and in this case the \define{complexity} of $\Omega$ is defined as
\[\cl(\Omega) = \sum_{i=1}^n \cl(\Omega_i) = \sum_{i=1}^n \sum_{V\in \Omega_i} \entstar{V}d_{\T(V)}(\res{x_{i-1}}{\Omega_i}{V}, \res{x_i}{\Omega_i}{V}).\]
\end{definition}
In order to account for those annuli where we use $\entstar{V} =1$ instead of $\entstar{V} =2$ in the above formula, we also introduce the following:

\begin{definition}
\label{def:savings_in_complexity}
The \define{savings} of a WISC witness family $\Omega = (\Omega_1,\dots,\Omega_n)$ is
\[\mathfrak{S}(\Omega) = \sum_{i=1}^n \sum_{V\in \Omega_i}(\ent{V}-\entstar{V})d_{\T(V)}(\res{x_{i-1}}{\Omega_i}{V}, \res{x_i}{\Omega_i}{V}).\]
\end{definition}

\section{Bounding the contribution of a witness}
\label{sec:bound_contrib_witness}
We recall from the introduction that the reason witness families were introduced and the goal of the whole second half of the paper are Theorem~\ref{thm:main_complexity_length_bound} and Theorem~\ref{thm:countcomplexity}  
The first bounds the complexity of a collection of witness families defined by a strongly aligned set of points in terms of Teichm\"uller distance. The second counts net points in terms of complexity. Together 
they will give the desired count of net points in terms of Teichm\"uller distance. 
As a major first step towards proving Theorem~\ref{thm:main_complexity_length_bound}  in this section we  bound the distances $d_{\T(V)}(\res{x_{i-1}}{\Omega}{V},\res{x_i}{\Omega}{V})$ 
contributed by each individual witness; this is the content of Theorem~\ref{thm:bound_single_resolution_dist}.
Throughout this section, we fix a WISC witness family $\Omega = (\Omega_1,\dotsc,\Omega_n)$ for a strongly $\wfal$--aligned tuple $(x_0,\dots,x_n)$ in $\T(\Sigma)$. 
For each domain $V\esub \Sigma$, we let $x_0^V,\dots,x_n^V\in [x_0,x_n]$ denote the points provided by Definition~\ref{def:strong_alignment} (strong alignment) that appear in order along $[x_0,x_n]$ and satisfy $d_V(x_i, x_i^V)\le \wfal$. In the case of an annulus, we furthermore assume the ratio of $\min\{\thin, \ell_{x_i}(\partial A)\}$ and $\min\{\thin, \ell_{x_i^V}(\partial A)\}$ is at most $\wfal$.
We also remind the reader that the collections $\Upsilon, \Upsilon^c,\Upsilon^\ell$ were introduced in Definition~\ref{def:upsilons}.

\subsection{Contribution sets}
\label{sec:contribution_set}

Estimating $\cl(\Omega)$ will involve a careful analysis of active intervals along the main Teichm\"uller geodesic $[x_0,x_n]$. 
To this end, we have the following basic observations.

\begin{lemma}
\label{lem:pftuple_domains_have_intervals}
If $V\in \Upsilon(x_{i-1},x_i)$, for some $1\le i \le n$, then $V$ has a nonempty active interval $\interval_V$ along $[x_0, x_n]$. In particular, this holds for each $V\in \cup_i \Omega_i$.
\end{lemma}
\begin{proof}
Assume first $V\in \Upsilon^c(x_{i-1},x_i)$.  Then  $d_V(x_{i-1},x_i)\ge \indrafi{V}$. 
Hence by $\wfal$--alignment of $(x_0,\dots,x_n)$ we have
\[d_{V}(x_0,x_n) \ge d_V(x_0,x_{i-1}) + d_V(x_{i-1},x_i) + d_V(x_i,x_n) -2\wfal \ge \indrafi{V}-2\wfal > \rafi.\]
Hence $\interval_V\ne \emptyset$ by Lemma~\ref{lem:our_active_intervals}. 

Otherwise $V\in \Upsilon^\ell(x_{i-1},x_i)$ and $V$ is an annulus with at least one of 
 $\ell_{x_{i-1}}(\partial V)$ and $\ell_{x_i}(\partial V)$ smaller than $\thin/\indrafi{V}$
Without loss of generality, we may therefore suppose $\ell_{x_i}(\partial V) \le \thin/\indrafi{V}$. By strong $\wfal$--alignment and the choice of $\indrafi{V}$ (\ref{eqn:indexed_thresholds}), this gives $\ell_{x_i^V}(\partial V) \le \wfal\thin/\indrafi{V} < \thin'$. Therefore, $V$ has a nonempty active interval $\interval_V = \tilde{\interval}_V^\thin$ along $[x_0,x_n]$  by Theorem~\ref{thm:rafi_active}(\ref{interval-fat}) and Definition~\ref{def:active_intervals}.
\end{proof}

\begin{lemma}
\label{lem:compatible-time-order}
If $Z\in \Upsilon^c(x_{i-1},x_i)$ 
and $W\in \Upsilon(x_{i-1},x_i)$ satisfy $Z\cut W$, then $Z$ and $W$ are time-ordered compatibly along $[x_0,x_n]$ and $[x_{i-1},x_i]$.
\end{lemma}
\begin{proof}
We know from Lemma~\ref{lem:pftuple_domains_have_intervals} and Remark~\ref{Upsilon_implies_active_interval} that $Z$ and $W$ have nonempty active intervals along both $[x_{i-1},x_i]$ and $[x_0,x_n]$. 
Let us suppose that $Z \tol W$ along $[x_{i-1},x_i]$ (the reverse possibility being handled similarly), and by contradiction that $W \tol Z$ along $[x_0,x_n]$. Then by time-ordering, $d_Z(\partial W,x_i)$ and $d_Z(x_0,\partial W)$ are at most $\rafi/3$. Hence
\[d_Z(x_0, x_i) \le d_Z(x_0,\partial W) + d_Z(\partial W,x_i) \le 2\rafi/3.\]
Since $Z\in \Upsilon^c_i$, alignment now gives the contradictory inequality 
\[d_Z(x_0,x_i) \ge d_Z(x_0, x_{i-1}) + d_Z(x_{i-1},x_i) - \wfal \ge \indrafi{Z} - \pfback \ge \rafi.\qedhere\]
\end{proof}

Recall (Definition~\ref{def:contribute}) that $Z\esub V$ contributes to $V\in \Omega_i$ if $Z\in  \Upsilon_i$ and $V=\bar Z^{\Omega_i} $.
For each  $V\in \Omega$, we will now define a ``contribution set'' for $V$ along $[x_0,x_n]$ by starting with the active interval $\interval_V$, then removing the active interval $\interval_W$ for any domain $W\in \Omega$ with $W\esubn V$, and finally adding the active intervals $\interval_Z$ of any domain $Z$ that contributes to $V$ in some $\Omega_i$. More precisely, for each $V\in \Omega$, we use Lemma~\ref{lem:pftuple_domains_have_intervals} to define
\begin{align*}
M(V) &= \bigcup\left\{\interval_W \mid W\in \Omega \text{ with }W\esubn V\right\} \subset [x_0,x_n],\quad\text{and}\\
C(V) &= \bigcup_{1\le i \le n} C_i(V),
\end{align*}
where for each index $1\le i \le n$ we define $C_i(V) = \emptyset$ if $V\notin \Omega_i$ and otherwise define
\begin{align*}
C_i(V) = \bigcup \left\{\interval_Z \mid Z\esubn V\text{ contributes to $V$ in }\Omega_i\right\} \subset [x_0,x_n].
\end{align*}

\begin{definition}[Contribution set]
\label{def:contribution_set}
The \define{contribution set} of  $V\in \cup_i \Omega_i$ is
\[\contr^\Omega_V = \big( \interval_V \setminus M(V)\big)\cup C(V)\subset [x_0,x_n].\]
We stress that all active intervals here are taken along the main geodesic $[x_0,x_n]$.
\end{definition}

The following result is the heart of proving Theorem~\ref{thm:main_complexity_length_bound}. It bounds Teichm\"uller distances in terms of size of active intervals of contribution sets. 

\begin{theorem}
\label{thm:bound_single_resolution_dist}
\label{thm:resolution<active}
If $V\in \Omega_i$, then $\displaystyle d_{\T(V)}(\res{x_{i-1}}{\Omega_i}{V}, \res{x_i}{\Omega_i}{V}) \ladd_\irafi \int_{x_{i-1}^V}^{x_i^V} \charfunc_{\contr^\Omega_V}$.
\end{theorem}

We remark that the term on the left and the integrand both depend on the witness family, while the limits of integration just depend on the points $x_{i-1}$ and $x_i$. 

\subsection{Proving Theorem~\ref{thm:bound_single_resolution_dist} for annuli}
We maintain the notation $\Omega$, $x_i$, and $x_i^V$ from the start of \S\ref{sec:bound_contrib_witness}.
Fix some index $1 \le i \le n$ and an annular domain $V\in \Omega_i$. So by strong alignment we have, in particular:
\begin{equation}
\label{eqn:annuli-comp_point_properties}
d_V(x_{j}, x_j^V)\le \wfal\quad\text{and}\quad \frac{1}{\wfal}\le \frac{\min\{\thin,\ell_{x_j}(\partial V)\}}{\min\{\thin,\ell_{x_{j}^V}(\partial V)\}} \le \wfal \quad\text{for $j={i-1},i$}.
\end{equation}
To ease notation, set $\hat{x}_j = \res{x_j}{\Omega_i}{V}\in \T(V) =\H^2_{\partial V}$ for $j\in \{i-1,i\}$, and recall that by definition these resolution points satisfy
\begin{equation}
\label{eqn:annuli-res_point_properties}
d_V(\hat{x}_j, x_j) \ladd_\wfal 0\quad\text{and}\quad \ell_{\hat{x}_j}(\partial V) = \min\{\thin, \ell_{x_j}(\partial V)\}\quad\text{for $j={i-1},i$}.
\end{equation}
The proof will follow easily from these facts:

\begin{proof}[Proof of Theorem~\ref{thm:bound_single_resolution_dist}--Annular case]
Consider the active interval $\interval_V$ of $V$ along $[x_0,x_n]$. For each point $w\in \interval_V$ we have $\ell_w(\partial V) <\thin$, and we write $w\vert_V$ for the $T(V)$--component of the point $\productmap{\partial V}(w)$ in the product region $\productregion{\Sigma}{\partial V}$. 
Since $V$ is an annulus, there are no proper subdomains of $V$; hence by definition  the contribution set is simply $\contr_V^\Omega = \interval_V$. 

First suppose that  $\contr_V^\Omega\cap [x_{i-1}^V, x_i^V]$ is empty. Then $d_V(x_{i-1}^V,x_i^V) \le \rafi$ and $\ell_{x_{i-1}^V}(\partial V), \ell_{x_i^V}(\partial V) \ge \thin'$. Therefore equations (\ref{eqn:annuli-comp_point_properties}) and (\ref{eqn:annuli-res_point_properties}) above imply that
\[d_V(\hat{x}_{i-1}, \hat{x}_i)\ladd_\wfal 0 \qquad\text{and}\qquad \frac{\thin'}{\wfal} \le \ell_{\hat{x}_{i-1}}(\partial V), \ell_{\hat{x}_i}(\partial V) \le \thin,\]
which together uniformly bound $d_{\T(V)}(\hat{x}_{i-1},\hat{x}_i)$ in terms of $\wfal$.

If $\contr_V^\Omega\cap [x_{i-1}^V, x_i^V]$ is nonempty, then (being the intersection of intervals) it is necessarily an interval and we may write it as $[y,z] \subset [x_{i-1}^V,x_i^V]$. 
We claim that
\begin{equation}
\label{eqn:annuli-res-points_and_endpoints_active_interval}
d_{\T(V)}(\hat{x}_{i-1}, y\vert_V)\ladd_\wfal 0 \qquad\text{and}\qquad d_{\T(V)}(\hat{x}_i, z\vert_V)\ladd_\wfal 0.
\end{equation}
By symmetry, let us just consider $d_{\T(V)}(\hat{x}_i, z\vert_V)$. To see this, note that if $z = x_{i}^V$ then obviously $\ell_{z}(\partial V) = \ell_{x_i^V}(\partial V)$, and otherwise we have both $\ell_{x_i^V}(\partial V)\ge \thin'$ and $\thin' \le \ell_{z}(\partial V) \le \thin$. Thus in either case equations (\ref{eqn:annuli-comp_point_properties})--(\ref{eqn:annuli-res_point_properties}) imply
\[
\frac{\thin'}{\wfal\thin} 
\le \frac{\min\{\thin, \ell_{x_i^V}(\partial V)\}}{\wfal \ell_{z}(\partial V)} 
\le \frac{ \ell_{\hat{x}_i}(\partial V)}{\ell_{z}(\partial V)} 
\le \wfal\frac{\min\{\thin, \ell_{x_i^V}(\partial V)\}}{\min \{\thin, \ell_{z}(\partial V)\}} 
\le \wfal\frac{\thin}{\thin'}.
\]
Furthermore, since $[z,x_{i}^V]$ is disjoint from the interior of $\interval_V$, Lemma~\ref{lem:our_active_intervals} gives  $d_V(z,x_i^V)\le \rafi$. Combining with (\ref{eqn:annuli-comp_point_properties})--(\ref{eqn:annuli-res_point_properties}) we therefore have $d_V(\hat{x}_i, z)\ladd_\wfal 0$. This proves $\hat{x}_i$ and $z\vert_V$ coarsely have the same horizontal coordinate in $\H^2_{\partial V}$, and the above bounds on $\ell_{\hat{x}_i}(\partial V)/\ell_z(\partial V)$ show they coarsely have the same vertical component. Therefore $d_{\T(V)}(\hat{x}_i,z\vert_V)$ is indeed bounded as claimed.

To conclude the argument, since $y$ and $z$ both lie in the thin region for $\partial V$, Minsky's product regions Theorem~\ref{thm:product_regions} implies that
\[d_{\T(V)}(y\vert_V, z\vert_V) \ladd d_{\T(\Sigma)}(y,z) = \int_y^z 1 \le \int_{x_{i-1}^V}^{x_i^V} \charfunc_{\contr_V^\Omega}.\]
Combining this with \eqref{eqn:annuli-res-points_and_endpoints_active_interval} and the triangle inequality proves the proposition.
\end{proof}

\subsection{Proving Theorem~\ref{thm:bound_single_resolution_dist} for nonannuli}
We maintain the notation $\Omega$, $x_i$, and $x_i^V$ fixed at the start of \S\ref{sec:bound_contrib_witness}.
We also fix an index $1\le i \le n$ and a nonannular domain $V\in \Omega_i$. Note that in this case $V\in \Upsilon^c(x_{i-1},x_i)$ so that $d_V(x_{i-1},x_i)\ge \indrafi{V}$.

\subsubsection{Setup}
 We begin by identifying a subinterval of $[x_{i-1}^V,x_i^V]$ on which we have better control of resolution points.

\begin{lemma}
\label{lem:good_subinterval}
There is a nonempty subinterval $J = [y,z]\subset [x_{i-1}^V, x_i^V]$ such that
\begin{itemize}
\item $d_V(x_{i-1},y)$ and $d_V(x_i,z)$ are both at most $7\wfal$.
\item For all $w\in J$ the distances $d_V(x_{i-1},w)$ and $d_V(w,x_i)$ are both at least $3\wfal$.
\end{itemize}
Furthermore $J$ is contained in the active interval $\interval_V$ of $V$ along $[x_0, x_n]$.
\end{lemma}
\begin{proof}
Recall $\wfal \ge \rafi \ge \lipconst$. We know $d_V(x_i,x_i^V),d_V(x_{i-1},x_{i-1}^V) \le \wfal$. Therefore
\[d_V(x_{i-1}^V,x_i^V) \ge d_V(x_{i-1},x_i) - 2\wfal \ge \indrafi{V} - 2\wfal \ge 28\wfal.\]
Since $\pi_V\colon \T(\Sigma)\to \cc(V)$ is coarsely $\lipconst$--Lipschitz and $\lipconst \le \wfal$, there must exists points $y,z\in [x_{i-1}^V, x_i^V]$ such that 
\[5 \wfal \le d_V(x_{i-1}^V, y) \le 6\wfal\qquad\text{and}\qquad 5\wfal \le  d_V(x_i^V,z) \le 6\wfal.\]
Observe that necessarily $y$ and $z$ appear in order along $[x_{i-1}^V,x_i^V]$ for otherwise $y\in [z,x_i^V]$ and we may apply Theorem~\ref{thm:back} (no backtracking) to conclude
\begin{align*}
d_V(x_{i-1}^V, x_i^V) &\le d_V(x_{i-1}^V, y) + d_V(y, x_i^V)
\le 6\wfal + d_V(z, y) + d_V(y, x_i^V)\\
&\le 6\wfal + d_V(z,x_i^V) + \back 
\le 6\wfal + 6\wfal + \back < 13 \wfal,
\end{align*}
which we have seen is false. By the triangle inequality, we also clearly have
\[d_V(x_{i-1},y) \le 7\wfal\quad\text{and}\quad d_V(x_i, z) \le 7\wfal.\]
Finally, for any $w\in [x_0,z]$ Theorem~\ref{thm:back} additionally gives
\begin{equation}
\label{eqn:J_far_from_endpoints}
d_V(w,x_i^V) \ge d_V(w,z) + d_V(z,x_i^V) - \back \ge 5\wfal - \back  \ge 4\wfal
\end{equation}
so that $d_V(w,x_i) \ge 3\wfal$ by the triangle inequality. Similarly $d_V(x_{i-1},w)\ge 3\wfal$ for all $w\in [y,x_n]$. This proves all $w\in J$ satisfy the second bullet point.

Finally, we know from Lemma~\ref{lem:pftuple_domains_have_intervals} that $V$ has a nonempty active interval along $[x_0,x_n]$. If $\interval_V$ were disjoint from $[z,x_i^V]$, then we would have $d_V(z,x_i^V) \le \rafi/3$ by Lemma~\ref{lem:our_active_intervals}(\ref{ourinterval-smallproj}). But this contradicts the implication $d_V(z,x_i^V)\ge 4\wfal$ of Equation (\ref{eqn:J_far_from_endpoints}). Thus $\interval_V$ necessarily intersects $[z,x_i^V]$ and, similarly, $[x_{i-1}^V,y]$. Since $\interval_V$ is an interval, the containment $J = [y,z]\subset \interval_V$ follows.
\end{proof}

The interval $J$ moreover contains the active interval of each domain contributing to $V$ in $\Omega_i$; this is a variation of Lemma~\ref{lem:insular_properties}(\ref{insular-nest}) for this more general context of witness families for aligned tuples:

\begin{lemma}
\label{lem:contribution_in_reference_interval}
If $Z\esubn V$ contributes to $V$ in $\Omega_i$, then its active interval along $[x_0,x_n]$ lies in the interior of $J$. Further, $d_Z(x_{i-1},x_{i-1}^V)\le \rafi$ and $d_Z(x_i,x_i^V)\le \rafi$.
\end{lemma}
\begin{proof}
The fact that $Z$ contributes to $V$ implies $Z\in \Upsilon(x_{i-1},x_i)$ but that $Z\notin \Omega_i$. 
Recall from Lemma~\ref{lem:insulated_contains_short_curves} that $\Omega_i \supset \Upsilon^\ell(x_{i-1},x_i)$; hence in fact $Z\in \Upsilon^c(x_{i-1},x_i)$. If $d_V(x_i, \partial Z)\le 9\wfal$, then by definition we would have $Z\in \rset{0}{V}$ for $\Omega_i$ and hence $Z\esub Z'$ for some $Z'\in \maxrset{0}{V}$. But since $\Omega_i$ is insulated, this would imply $Z'\in \Omega_i$ and contradict $\colsup{Z}{\Omega_i} = V$. Therefore
\[d_V(x_i, \partial Z) > 9\wfal > \wfal + \rafi \ge d_V(x_i,x_{i}^V) + \rafi.\]
Corollary~\ref{cor:bgit_for_teich} therefore implies $d_Z(x_i,x_i^V) <  \rafi$. Similarly $d_Z(x_{i-1},x_{i-1}^V) < \rafi$. Also observe that for all $w\in [z,x_i^V]$ we have
\[d_V(x_{i}^V, w) \le d_V(x_i^V, w) + d_V(w,z) \le d_V(x_i^V,z)  + \back \le 6\wfal + \back\]
and therefore $d_V(x_i,w) \le 8\wfal$.  Similarly $d_V(x_{i-1},w)\le 8\wfal$ for all $w\in [x_{i-1}^V,y]$.

We know from Lemma~\ref{lem:pftuple_domains_have_intervals} that $Z$ has a nonempty active interval $\interval_Z$ along $[x_0,x_n]$. We claim that $\interval_Z$ is disjoint from $[z,x_i^V]$. Indeed, otherwise we would have $w\in [z,x_i^V]\cap \interval_Z$ with $\partial Z\subset \mu_w$ and hence $d_V(x_i, \partial Z)\le d_V(x_i,w) \le 8\wfal$, contradicting the above  lower bound \[d_V(x_i, \partial Z) > 9\wfal.\] Similarly $\interval_Z$ must be disjoint from $[x_{i-1}^V,y]$.

Therefore, if $\interval_Z$ is not contained in the interior of $J = [y,z]$, it is necessarily disjoint from $[x_{i-1}^V,x_i^V]$. This gives $d_Z(x_{i-1}^V,x_i^V)\le \rafi/3$ and thus by the triangle inequality
\begin{align*}
d_Z(x_{i-1},x_i) &\le d_Z(x_{i-1},x_{i-1}^V) + d_Z(x_{i-1}^V,x_i^V) + d_Z(x_i^V,x_i)\\
&\le \rafi + \rafi/3 +\rafi < \indrafi{Z}.
\end{align*}
But this contradicts the fact, observed above, that $Z\in \Upsilon^c(x_{i-1},x_i)$.
\end{proof}

The following observation will also be useful.

\begin{lemma}
\label{lem:J_is_in_the_middle}
Suppose $W\esubn V$ has a nonempty active interval along $[x_0,x_n]$. 
If $\interval_W$ intersects $[x_0,z]$ (resp. $[y,x_n])$ then $d_W(x_j,x_j^V)\le \rafi$ for all $j\ge i$ (resp. all $j\le i-1$). In particular, if $\interval_W$ intersects $J$ (as holds for every $Z$ that contributes to $V$ in $\Omega_i$ by Lemma~\ref{lem:contribution_in_reference_interval}), then $d_W(x_j, x_j^V)\le \rafi$ for all $0\le j \le n$. 
\end{lemma}
\begin{proof}
We suppose $\interval_W\cap [x_0,z]\ne\emptyset$, the alternate hypothesis $\interval_W\cap [y,x_n]\ne\emptyset$ being handled symmetrically. 
Fix any $j\ge i$.  Pick some point $w\in \interval_W\cap [x_0,z]$, so that $\partial W\subset \mu_w$. We then have $[z,x_i^V]\subset [w,x_{j}^V]$ and therefore (by Theorem~\ref{thm:back})
\[d_V(w,x_j^V) \ge d_V(w,z) + d_V(z,x_i^V) + d_V(x_i^V,x_j^V) - 2\back \ge d_V(z,x_i^V) - 2\back \ge 4\wfal.\]
It follows that
\[d_V(\partial W, x_j^V) \ge d_V(w,x_j^V) - \lipconst \ge 3\wfal > d_V(x_j,x_j^V) + \rafi.\]
Thus Corollary~\ref{cor:bgit_for_teich} gives the desired bound $d_W(x_j,x_j^V)\le \rafi$.
\end{proof}

\begin{corollary}
\label{cor:other_intervals_miss_J}
Suppose that $W\esubn V$ satisfies $W\in \Upsilon(x_{j-1},x_j)$ for some $j\ne i$.  Then $\interval_W\cap J = \emptyset$.
\end{corollary}
\begin{proof}
We assume $W\in \Upsilon(x_{j-1},x_j)$ for $j > i$, the alternate case $j < i$ being symmetric. We know (Lemma~\ref{lem:pftuple_domains_have_intervals}) that $W$ has a nonempty active interval $\interval_W$ along $[x_0,x_n]$. To derive a contradiction, let us suppose there is a point $w\in \interval_W\cap J$. It cannot be that $x_i^V\in \interval_W$, since that would imply $[z,x_i^V]\subset \interval_W$ and hence
\[d_V(z,x_i^V) \le d_V(z,\partial W) + d_V(\partial W, x_i^V) \le 2\lipconst < 5\wfal,\]
violating the choice of $z$ in Lemma~\ref{lem:good_subinterval}. Since $\interval_W$ is an interval, we find that $[x_i^V,x_n]\supset [x_{j-1}^V,x_j^V]$ misses $\interval_W$.  Lemma~\ref{lem:our_active_intervals}(\ref{ourinterval-smallproj}) and Lemma~\ref{lem:J_is_in_the_middle} now give
\[d_W(x_{j-1},x_j) \le d_W(x_{j-1},x_{j-1}^V) + d_W(x_{j-1}^V, x_j^V) + d_W(x_j^V, x_j) \le 3\rafi < \indrafi{W}.\]
This shows $W\notin \Upsilon^c(x_{j-1},x_j)$. Thus we must have $W\in \Upsilon^\ell(x_{j-1},x_j)$. Choose $k\in \{j-1,j\}$ so that $\ell_{x_k}(\partial W) < \thin/\indrafi{W}$, and note that $k \ge i$. Since the curve $\partial W$ is short at $x_k$, we evidently have $d_V(\partial W, x_k) \le \lipconst$. Since $\partial W$ is also short at the chosen point $w\in \interval_W \cap J$, this shows
\[d_V(w,x_k^V) \le d_V(w, \partial W) + d_V(\partial W, x_k) + d_V(x_k, x_k^V) \le 2\lipconst + \wfal.\]
On the other hand the fact that $[z,x_i^V]\subset [w,x_k^V]$ gives (via Theorem~\ref{thm:back})
\[d_V(w,x_k^V) \ge d_V(w,z) + d_V(z,x_i^V) + d_V(x_i^V,x_k^V) - 2\back \ge 5\wfal - 2\back.\]
As these inequalities are incompatible, we have derived our contradiction.
\end{proof}

The following property of the interval $J$ will play a key role in our argument.

\begin{lemma}
\label{lem:resolving_points_exist}
If $w\in J$, then every domain $Z\esub V$ satisfies
\[d_Z(x_{i-1},w) + d_Z(w,x_i) \le d_Z(x_{i-1},x_i) + 9\wfal.\]
\end{lemma}
\begin{proof}
Fix any domain $Z\esub V$. First suppose that $d_Z(x_{i-1},x_{i-1}^V)$ and $d_Z(x_i,x_i^V)$ are both at most $2\wfal$ (as is the case for $Z = V$). Then since $J\subset [x_{i-1}^V, x_i^V]$, Theorem~\ref{thm:back} and the triangle inequality give
\begin{align*}
d_Z(x_{i-1},w) + d_Z(w,x_i) 
&\le d_Z(x_{i-1}^V,w) + d_Z(w,x_i^V) + 4\wfal\\
&\le d_Z(x_{i-1}^V,x_i^V) + 4\wfal + \back
\le d_Z(x_{i-1},x_i) + 9\wfal.
\end{align*}
So it suffices to assume at least one of the quantities is larger that $2\wfal$. Suppose then that $d_Z(x_i^V,x_i) >2\wfal>\rafi$ (the other possibility is handled similarly). 
Then 
\[d_V(x_i, \partial Z) \le d_V(x_i, \partial Z) + d_V(\partial Z, x_i^V) \le d_V(x_i,x_i^V) + \rafi/3 \le 2\wfal\]
by Corollary~\ref{cor:bgit_for_teich}. The triangle inequality therefore gives
\[d_V(x_{i-1},\partial Z) \ge d_V(x_{i-1},x_i) - d_V(x_i,\partial Z) \ge \indrafi{V} - 2\wfal \ge 28\wfal.\]
In particular, it must be the case that $d_Z(x_{i-1},x_{i-1}^V) \le 2\wfal$ (since otherwise the above argument would force $d_V(x_{i-1},\partial Z) \le 2\wfal$, which is false).

We next show that $d_Z(x_{i-1}^V,w)\le \rafi$. Indeed, otherwise $d_Z(x_{i-1}^V,w)>\rafi$ and $Z$ must have an active interval along $[x_{i-1}^V,w]$. Thus there is some point $u\in [x_{i-1}^V,w]$ that contains $\partial Z$ in its Bers marking. Thus $d_V(x_i,u) \le d_V(x_i,\partial Z) + 1 \le 2\wfal  + 1$. 
On the other hand equation~(\ref{eqn:J_far_from_endpoints}) (in the proof of Lemma~\ref{lem:good_subinterval}) gives $d_V(u,x_i^V) \ge 4\wfal$, which implies $d_V(u, x_i) \ge 3\wfal$; a contradiction.

We now know both $d_Z(x_{i-1},x_{i-1}^V) \le 2\wfal$ and $d_Z(x_{i-1}^V,w)\le \rafi$. Combining these gives $d_Z(x_{i-1},w) \le 3\wfal$. It is now easy to conclude
\begin{align*}
d_Z(x_{i-1},w) + d_Z(w,x_i)
&\le d_Z(x_{i-1},w) + d_Z(w, x_{i-1}) + d_Z(x_{i-1},x_i)\\
&\le 3\wfal + 3\wfal + d_Z(x_{i-1},x_i).\qedhere
\end{align*}
\end{proof}

\subsubsection{Comparison points}
Lemma~\ref{lem:resolving_points_exist} and Proposition~\ref{prop:consistent_resolution} imply that for each $w\in J$, the projection tuple $(\tilde{w}_Z)\in \prod_{Z\esub V}\cc(Z)$ from Definition~\ref{def:projection_tuple} is $k$--consistent for some constant $k$ depending only on $\wfal$. We next use this fact together with the lengths of certain curves at $w$ to define a point $\hat w\in \T(V)$ as follows:

\begin{definition}[Comparison point]
\label{def:comparison_point}
For each point $w\in J$, consider the tuple $(\tilde{w}_Z)_{Z\esub V}$ from Definition~\ref{def:projection_tuple}. Let $\alpha_w$ be the multicurve consisting of those curves $\gamma\in \curves{V}$ which are essential in $V$, have $\ell_w(\gamma)< \thin$, and satisfy
\begin{equation}
\label{eqn:short_condition}
d_Z(\gamma,\tilde{w}_Z) = \diam_{\cc(Z)}(\pi_Z(\gamma)\cup \tilde{w}_Z) \le 2\rafi\text{ for every domain }Z \esub V.
\end{equation}
Using Proposition~\ref{prop:consistent_resolution}, Theorem~\ref{Thm:Consistency}, and Lemma~\ref{lem:build_marking}, we may then build a marking $\mu$ of $V$ that realizes the tuple $(\tilde{w}_Z)_Z$ and has $\alpha_w\subset \base(\mu)$. Working in Fenchel--Nielsen coordinates for the pants decomposition $\base(\mu)$, take $\hat w\in \T(V)$ to be the point whose Bers marking is $\mu$ and such that $\gamma\in \base(\mu)$ has 
\[\ell_{\hat w}(\gamma) = \begin{cases} \thin,& \text{if }\gamma\notin\alpha_w\\
\ell_w(\gamma),& \text{if }\gamma\in \alpha_w\end{cases}.\] 
This comparison point satisfies (and is coarsely characterized by):
\begin{enumerate}
\item\label{resolution-marking} $d_Z(\hat w,\tilde{w}_Z) \ladd_{\irafi} 0$ for every domain $Z\esub V$.
\item\label{resolution-curves} If $\gamma\in \curves{V}$ is an essential curve in $V$, then $\ell_{\hat w}(\gamma) < \thin$ if and only if $\gamma$ satisfies $\ell_w(\gamma)< \thin$ and (\ref{eqn:short_condition}). Further, in this case $\ell_{\hat w}(\gamma) = \ell_w(\gamma)$.
\end{enumerate}
\end{definition}

The next lemma shows that if $w\in \interval_Z$ for some domain $Z\esubn V$ that contributes to $V$ in $\Omega_i$, then $\partial Z\subset \alpha_w$ and hence, by construction, $\ell_{\hat{w}}(\gamma)< \thin$ for each component $\gamma$ of $\partial Z$ that is essential in $V$. Thus the points $w\in \T(\Sigma)$ and $\hat{w}\in \T(V)$ both live in product regions for $Z$, and we may compare them as follows:

\begin{lemma}[Comparisons in active intervals]
\label{lem:res_in_active_interval} 
Suppose $Z\esubn V$ contributes to $V$ in $\Omega_i$, For all $w\in \interval_Z$ with corresponding comparison $\hat{w}\in\T(V)$, the following hold:
\begin{enumerate}
\item\label{short_resolutions_in_active} $\ell_w(\gamma)< \thin$ and $\ell_{\hat{w}}(\gamma)<\thin$ for each component $\gamma\in \partial Z\cap \curves{V}$.
\item\label{compare_res_in_active} Writing $w\vert_Z$ and $\hat{w}\vert_Z$ for the $\T(Z)$--components of $\productmap{\partial Z}(w) \in \productregion{\Sigma}{\partial Z}$ and $\productmap{\partial Z}(\hat{w})\in \productregion{V}{\partial Z}$, respectively, we have $d_{\T(Z)}(w\vert_Z , \hat{w}\vert_Z) \ladd_{\irafi} 0$.
\end{enumerate}
\end{lemma}
\begin{proof}
We will need the following observation.

\begin{claim}
\label{claim:disjointness_for_contributing_and_subordered}
If $U\in \Upsilon(x_{i-1},x_i)$ satisfies $\colsup{U}{\Omega_i} \indsw{i} V$ (resp. $V \indse{i} \colsup{U}{\Omega_i}$), then either $Z$ and $U$ are disjoint, or else  $\interval_U$ occurs before (resp. after) $\interval_Z$ along $[x_0, x_n]$. 
\end{claim}
\begin{proof}
Set $U' = \colsup{U}{\Omega_i}$ and, by symmetry, suppose $U'\indsw{i} V$. We may assume $Z$ is not disjoint from $U$, and hence neither disjoint from $U'$. Note that we cannot have $Z\esub U$ or $Z\esub U'$, as that would imply $V = \colsup{Z}{\Omega_i} \esub U' \esubn V$ by Lemma~\ref{lem:characterize_supremum}. 

The fact that $Z\esubn V$ contributes to $V$ implies $Z\notin \Omega_i$. 
We claim there is some $W\in \Omega_i$ such that $U'\esub W$ and $W\cut Z$. If $U'\cut Z$ then we can simply take $W = U'$. Otherwise $U'\esub Z$ and  \ref{projfam-maximal} (applied to $U'\in \Omega_i$ and $Z\notin \Omega_i$) provides such a $W$.

Since $W\cut Z \esub V$, we see that both  $V\esub W$ and $V\disj W$ are impossible. If $W\cut V$, then \ref{subord-supercutting} (applied to $U'\indsw{i}V$ and $V \esup U' \esub W)$ forces $W \tol V$ and hence $W \tol Z$ along $[x_{i-1},x_i]$ by Corollary~\ref{cor:time-order_subsurf}. Otherwise $W\esubn V$ and \ref{subord-triplenest} (using $U'\indsw{i} V$) implies $W \indsw{i}V$ so that we may invoke \ref{subord-contribute} (using $W\cut Z$) to again conclude $W \tol Z$ along $[x_{i-1},x_i]$.
Note that the fact $Z\notin \Omega_i\supset \Upsilon^\ell(x_{i-1},x_i)$ ensures that $Z\in \Upsilon^c(x_{i-1},x_i)$. Hence Lemma~\ref{lem:compatible-time-order} implies we have the same time-ordering $W \tol Z$ along $[x_0,x_n]$. Since $U\esub W$, Lemma~\ref{lem:our_active_intervals}(\ref{ourinterval-subdomain-disjoint}) now implies the intervals $\interval_U$ and $\interval_Z$ along $[x_0,x_n]$ are disjoint, and in fact it must be that $\interval_U$ occurs before $\interval_Z$.
\end{proof}

Returning to the lemma: 
Since $w\in \interval_Z$, Lemma~\ref{lem:our_active_intervals} implies $\ell_w(\alpha) < \thin$ for every component $\alpha$ of $\partial Z$. Hence, (\ref{short_resolutions_in_active}) will follow from the following fact:

\begin{claim}
\label{claim:same_short_curves}
If $\gamma\in \cc(V\vert_Z)$ satisfies $\ell_{w}(\gamma) < \thin$, then $\ell_{\hat{w}}(\gamma) = \ell_{w}(\gamma) < \thin$.
\end{claim}
\begin{proof}[Proof of claim]
Given the hypotheses, by definition of $\hat{w}$ it suffices to show $\gamma$ satisfies (\ref{eqn:short_condition}). Let $U\esub V$ be any subdomain. If $\gamma$ is disjoint from $U$, we trivially have $d_U(\gamma,\tilde{w}_U) = \diam_{\cc(U)}(\tilde{w}_U)\le \rafi$. Observe also that
\[d_U(\gamma, w) = \diam_{\cc(U)}(\pi_U(\gamma)\cup \pi_U(w))\le \lipconst\le \rafi/2\]
owing to the fact that $\gamma$ is short at $w$. Thus (\ref{eqn:short_condition}) is immediate when $\tilde{w}_U = \pi_U(w)$; this takes care of the case that $U$ contributes to $V$. It remains to suppose, then, that $\gamma\cut U$ and $U\in \Upsilon(x_{i-1},x_i)$ with $\colsup{U}{\Omega_i} \ne V$. We write $U' = \colsup{U}{\Omega_i}$ and, by symmetry, assume $U' \indsw{i} V$. Then $\tilde{w}_U = \pi_U(x_i)$. As the curve $\gamma\in \cc(V\vert_Z)$ cuts $U$, it cannot be that $Z$ and $U$ are disjoint. Claim~\ref{claim:disjointness_for_contributing_and_subordered} thus ensures $\interval_U$ occurs before $\interval_Z$ along $[x_0,x_n]$.
Since $w\in\interval_Z\subset J = [y,z]\subset [x_{i-1}^V,x_i^V]$ by Lemma~\ref{lem:contribution_in_reference_interval}, 
we now see that $w$ and $x_i^V$ lie in the same component of $[x_0,x_n]\setminus \interval_U$. Whence $d_U(w,x_i^V)\le \rafi/3$ by Lemma~\ref{lem:our_active_intervals}(\ref{ourinterval-smallproj}).
We also see that $\interval_U$ intersects $[x_0,z]$ and hence that $d_U(x_i,x_i^V)\le \rafi$ by Lemma~\ref{lem:J_is_in_the_middle}.
Therefore $d_U(w, x_i)\le 4\rafi/3$. Since $\tilde{w}_U = \pi_U(x_i)$ and we have already observed $d_U(\gamma, w)\le \rafi/2$, we conclude that $d_U(\gamma, \tilde{w}_U)\le 2\rafi$
and condition (\ref{eqn:short_condition}) is verified.
\end{proof}

It now follows from (\ref{short_resolutions_in_active}) that $w$ and $\hat{w}$  lie in product regions for $\partial Z$, so we are justified in considering $w\vert_Z,\hat{w}\vert_Z\in \T(Z)$. By the distance formula \cite[Theorem 6.1]{Rafi-combinatorial}, to bound $d_{\T(Z)}(w\vert_Z,\hat{w}\vert_Z)$ it suffices to show that $w\vert_Z$ and $\hat{w}\vert_Z$ have the same short curves and the same curve complex projections to all subsurfaces of $Z$. 

First let $\beta\in \curves{Z}$ be an essential curve of $Z$. We claim that either $\ell_{w\vert_Z}(\beta)$ and $\ell_{\hat{w}\vert_Z}(\beta)$ are both at least $\thin$, or else $\ell_{w\vert Z}(\beta)$ and $\ell_{\hat{w}\vert_Z}(\beta)$ coarsely agree. Indeed, by nature of the homeomorphism $\productmap{\partial Z}$, the lengths $\ell_w(\beta)$ and $\ell_{w\vert_Z}(\beta)$ coarsely agree, as do $\ell_{\hat{w}}(\beta)$ and $\ell_{\hat{w}\vert_Z}(\beta)$. Thus it suffices to show either $\ell_w(\beta),\ell_{\hat{w}}(\beta)\ge \thin$ or else $\ell_w(\beta)$ and $\ell_{\hat{w}}(\beta)$ coarsely agree. But this follows from the construction of $\hat{w}$: if $\ell_w(\beta) < \thin$, then $\ell_{\hat{w}}(\beta) = \ell_w(\beta)$ by Claim~\ref{claim:same_short_curves}. Conversely, if $\ell_{\hat{w}}(\beta) < \thin$, then we must have $\ell_{\hat{w}}(\beta) = \ell_w(\beta)$ by item (\ref{resolution-curves}) of Definition~\ref{def:comparison_point}.

Next let $U\esub Z$ be any domain in $Z$. Since the curves of $\partial Z$ are all short at $w$, the Bers marking $\mu_w$ at $w$ has $\partial Z\subset \base(\mu_w)$. Therefore, taking the curves of $\mu_w$ that are essential in $Z$ defines a marking of $\mu'$ of $Z$, and in fact $\mu'$ is a Bers marking $\mu_{w\vert_Z}$ of $w\vert_Z$. Since $U\esub Z$, we have $\pi_U(\mu_w) = \pi_U(\mu') = \pi_U(\mu_{w\vert_Z})$. Thus $d_U(w,w\vert_Z)\ladd 0$. Similarly $d_U(\hat{w}, \hat{w}\vert_Z) \ladd 0$. It therefore suffices to bound $d_U(w,\hat{w})$. By construction (Definition~\ref{def:comparison_point}(\ref{resolution-marking})) $d_U(\hat{w},\tilde{w}_U)\ladd_{\irafi} 0$ for $\tilde{w}_U$ as in Proposition~\ref{prop:consistent_resolution}. Thus we must bound $d_U(w,\tilde{w}_U)$. We consider the three possibilities of $\tilde{w}_U$: if $\tilde{w}_U = \pi_U(w)$ this is immediate. If not then $U\in \Upsilon(x_{i-1},x_i)$ and $\colsup{U}{\Omega_i} \ne V$. Since $U$ and $Z$ are evidently not disjoint, if $\colsup{U}{\Omega_i} \indsw{i} V$ then Claim~\ref{claim:disjointness_for_contributing_and_subordered} implies that $\interval_U$ occurs before $\interval_Z$ along $[x_0,x_n]$. As above, (using Lemmas~\ref{lem:our_active_intervals}(\ref{ourinterval-smallproj}) and \ref{lem:J_is_in_the_middle}) it follows that $d_U(w,x_i^V)\le \rafi/3$ and $d_U(x_i^V,x_i)\le \rafi$ so that $d_U(w,\tilde{w}_U) = d_U(w,x_i) \le 2\rafi$. 
If instead $V\indse{i}\colsup{U}{\Omega_i}$, we similarly obtain $d_U(w,\tilde{w}_U) = d_U(w,x_{i-1}) \le 2\rafi$ and thereby establish (2).
\end{proof}

\subsubsection{The main argument}

With the requisite notation and setup established, we now work in earnest towards the proof of Theorem~\ref{thm:bound_single_resolution_dist}. 

\begin{definition}[The point $\bar{w}$]
Since $J\subset \interval_V$ (Lemma~\ref{lem:contribution_in_reference_interval}), each point $w\in J$ lies in the thin region for the multicurve $\partial V$; accordingly we let $\bar{w}$ denote the $\T(V)$--component of product region point $\productmap{\partial V}(w)\in \productregion{\Sigma}{\partial V}$. 
\end{definition}

Our proof relies on comparing the points $\bar{w},\hat{w}\in \T(V)$ for carefully chosen $w\in J$. 
To streamline notation, and mimic that used in Definition~\ref{def:comparison_point}, we will set $\hat{x}_{i-1} = \res{x_{i-1}}{\Omega_i}{V}$ and $\hat{x}_i = \res{x_i}{\Omega_i}{V}$; however we stress that $\hat{x}_{i-1}$ and $\hat{x}_i$ are defined by Definition~\ref{def:resolution} and are necessarily thick, whereas points $\hat w$ for $w\in J$ (from Definition~\ref{def:comparison_point}) may be thin.

\begin{remark}
The fact that points $\hat w$, for $w\in J$, are allowed to be thin causes technical complications in the proof. However, allowing thinness is necessary in order for the crucial ingredient Lemma~\ref{lem:res_in_active_interval}(\ref{compare_res_in_active}) to hold.
\end{remark}

\begin{strategy}
\label{strategy:plan_for_bounding_complexity_length}
The goal is to show that $d_{\T(V)}(\hat{x}_{i-1}, \hat{x}_i)$ is bounded, up to additive error, by $\int_{x_{i-1}^V}^{x_i^V} \charfunc_{\contr_V^\Omega}$. Since $J = [y,z]\subset [x_{i-1}^V, x_i^V]$, it suffices to instead work with $\int_y^z \charfunc_{\contr^\Omega_V}$. That is, we are concerned with the Lebesgue measure of $J\cap \contr^\Omega_V$. 

We will construct a piecewise geodesic path in $\T(\Sigma)$ from $x_{i-1}$ to $x_i$ with the property that each segment $[p,q]$ satisfies either $d_{\T(V)}(\hat{p},\hat{q})\ladd_\wfal d_{\T(\Sigma)}(p,q)$, or else $d_Z(\hat{p},\hat{q})\ladd_\wfal 0$ for every domain $Z\esub V$; these two properties  will be established in Lemmas~\ref{lem:paralleling_subintervals} and \ref{lem:bound_for_noncontributing_intervals} below. 
The piecewise path will consist of boundedly many segments---each of which is either $[x_{i-1},y]$, $[z,x_i]$, or a subintervals of $J$---and will be constructed using breakpoints provided (essentially) by Lemma~\ref{lem:maximal_DL_bound}.

Furthermore, the segments $[p,q]$ with $d_{\T(V)}(\hat{p},\hat{q})\ladd_\wfal d_{\T(\Sigma)}(p,q)$ will have total length at most $\int_y^z \charfunc_{\contr_V^\Omega}$. The triangle inequality thus implies $d_{\T(V)}(\hat{x}_{i-1},\hat{x}_i)$ is at most $\int_y^z \charfunc_{\contr_V^\Omega}$ plus the sum of the lengths $d_{\T(V)}(\hat{p},\hat{q})$ for the other segments $[p,q]$ with $d_Z(\hat{p},\hat{q})$ bounded for all $Z\esub V$. To complete the proof, we will use Minsky's product regions Theorem~\ref{thm:product_regions} to show these latter segments can be ignored.
\end{strategy}

To begin, let $\mathcal{D}$ denote the set of domains $Z\esub \Sigma$ such that $Z\in \Upsilon(x_{i-1},x_i)$ with $Z\esubn V$ and $\colsup{Z}{\Omega_i} = V$. Thus $\mathcal{D}$ consists of all domains contributing to $V$ in $\Omega_i$ except for $V$ itself, and hence $C_i(V) = \cup_{Z\in \mathcal{D}}\interval_Z$. 

\begin{definition}
We say a subinterval $[p,q]$ of $J$ is \define{squarely covered by $\mathcal{D}$} if:
\begin{itemize}
\item the open interval $(p,q)$ intersects $C_i(V)$, and
\item  whenever the open interval $(p,q)$ intersects $\interval_Z$ for some $Z\in \mathcal{D}$, then we have $[p,q]\subset \interval_Y$ for some $Y\in \mathcal{D}$ with $\interval_Z\subset \interval_Y$ and $Z\esub Y$.
\end{itemize}
\end{definition}

\begin{lemma}
\label{lem:squarely_covered_interval}
If $[p,q]\subset J$ is squarely covered by $\mathcal{D}$, then $d_{\T(V)}(\hat{p}, \hat{q}) \ladd_\wfal d_{\T(\Sigma)}(p,q)$.
\end{lemma}
\begin{proof}
By hypothesis there exists $Z\in \mathcal{D}$ with $(p,q)\cap \interval_Z \neq \emptyset$. If $Z\in\mathcal{D}$ is any such domain, then square covering further implies $[p,q]\subset \interval_Y$ for some $Y\in \mathcal{D}$ with $\interval_Z\subset \interval_Y$ and $Z\esub Y$. Let $\mathcal{Y}$ denote the set of topologically maximal domains in the collection
\[\{Y\in \mathcal{D} \mid [p,q]\subset\interval_Y\}.\]
It follows from the above that $\mathcal{Y}$ is nonempty and moreover that if $Z\in \mathcal{D}$ satisfies $\interval_Z\cap (p,q)\ne \emptyset$, then $Z\esub Y$ for some $Y \in \mathcal{Y}$. 

The domains in $\mathcal{Y}$ are evidently pairwise disjoint, since they cannot be nested and their active intervals overlap. Therefore $\partial\mathcal{Y} = \cup_{Y\in \mathcal{Y}} \partial Y$ defines a multicurve in $V$ with the property that every element of $\mathcal{Y}$ is a component of $V\setminus\partial \mathcal{Y}$. By Lemma~\ref{lem:res_in_active_interval}, each component $\gamma$ of $\partial\mathcal{Y}$ satisfies $\ell_w(\gamma)<\thin$ and $\ell_{\hat{w}}(\gamma) < \thin$ for all $w\in [p,q]$. 
Consider the product regions map $\productmap{\partial\mathcal{Y}}\colon \T(V)\to \productregion{V}{\partial\mathcal{Y}}$. For each component $Z$ of $V\setminus\partial \mathcal{Y}$ and each point $w\in [p,q]$, we may consider the projection $\hat{w}\vert_Z$ of $\productmap{\partial\mathcal{Y}}(\hat{w})$ to $\T(Z)$. 

Recall that $d_{\productregion{V}{\partial \mathcal{Y}}}(\productmap{\partial \mathcal{Y}}(\hat p), \productmap{\partial \mathcal{Y}}(\hat q))$ is the supremum of $d_{\T(Y)}(\hat p \vert_Y, \hat q\vert_Y)$ over all factors $\T(Y)$ of the product $\productregion{V}{\partial \mathcal{Y}}$, that is, over all components $Y$ of $V\setminus \partial \mathcal{Y}$. Note that the components of the multicurve $\partial \mathcal{Y}$ count as annular components of $V\setminus \partial \mathcal{Y}$. Let $Y$ be the component of $V\setminus \partial \mathcal{Y}$ maximizing this supremum. By Minsky's Theorem~\ref{thm:product_regions} we thus have
\[d_{\T(V)}(\hat p, \hat q) \ladd d_{\productregion{V}{\partial \mathcal{Y}}}(\productmap{\partial \mathcal{Y}}(\hat p), \productmap{\partial \mathcal{Y}}(\hat q)) = d_{\T(Y)}(\hat p\vert_Y, \hat q\vert_Y).\]

First suppose $Y$ is not an element of $\mathcal{Y}$. We claim that $d_W(\hat{p}\vert_Y, \hat{q}\vert_Y)$ is uniformly bounded for all domains $W\esub Y$. Note that by definition of product region factors, we have $d_W(\hat{p}\vert_Y, \hat{q}\vert_Y) \ladd d_W(\hat{p},\hat{q})$. 
Clearly $Y$ is the only component of $V\setminus\partial\mathcal{Y}$ containing $W$;
since elements of $\mathcal{Y}$ \emph{are} components of $V\setminus\partial\mathcal{Y}$ and yet $Y\notin \mathcal{Y}$, it follows that $W$ cannot be contained in any element of $\mathcal{Y}$.
If $W\in \mathcal{D}$, it follows that $(p,q)$ is disjoint from $\interval_W$, since otherwise $W$ would be contained in an element of $\mathcal{Y}$ by construction.
 Hence in this case
\[d_W(\hat{p}, \hat{q}) \ladd_\wfal  d_W(p,q) \le \rafi/3.\]
If $W\notin \mathcal{D}$ but $W\in \Upsilon(x_{i-1},x_i)$, then evidently $\colsup{W}{\Omega_i} \ne V$ and therefore the points $\tilde{p}_W$ and $\tilde{q}_W$ in the projection tuple (Definition~\ref{def:projection_tuple}) are equal (either $\pi_W(x_{i-1})$ or $\pi_W(x_i)$). Hence $d_W(\hat{p}, \hat{q}) \ladd_\wfal 0$ in this case as well. In the remaining case $W\notin \Upsilon(x_{i-1},x_i)$ we have $d_W(x_{i-1},x_i) < \indrafi{W}$ and therefore $d_W(p,x_i),d_W(q,x_i)\ladd_\wfal \indrafi{W}$ by Lemma~\ref{lem:resolving_points_exist}. Consequently $d_W(\hat{p},\hat{q})\ladd_\wfal d_W(p,q)\ladd_\wfal 0$ as before. Thus we have shown $d_W(\hat{p}\vert_Y, \hat{q}\vert_Y) \ladd_\wfal 0$ for every domain $W\esub Y$.

Now let $R$ denote the quantity from Lemma~\ref{lem:thin_with_bdd_projections} 
for the pair $p,q$, and let $\hat{R}\vert_Y$ denote the corresponding quantity for the pair $\hat{p}\vert_Y,\hat{q}\vert_Y$. The lengths $\ell_{\hat{p}}(\gamma)$ and $\ell_{\hat{p}\vert_Y}(\gamma)$ are comparable for every essential curve $\gamma$ in $Y$. Further, by construction, if $\ell_{\hat{p}}(\gamma) < \thin$, then $\ell_p(\gamma) = \ell_{\hat{p}}(\gamma)< \thin$. Thus every short curve at $\hat{p}\vert_Y$ is also short, with a comparable length, at $p$. The same holds for the points $\hat{q}\vert_Y$ and $q$. Therefore we evidently have $\hat{R}\vert_Y \ladd R$. Applying Lemma~\ref{lem:thin_with_bdd_projections}, and using our bound $ d_W(\hat{p}\vert_Y, \hat{q}\vert_Y) \ladd_\wfal 0$ for all $W\esub Y$, we now conclude
\[d_{\T(V)}(\hat{p},\hat{q}) \ladd d_{\T(Y)}(\hat{p}\vert_Y, \hat{q}\vert_Y) \ladd_\wfal \hat{R}\vert_Y \ladd R \le d_{\T(\Sigma)}(p,q).\]

It remains to suppose that $Y$ is an element of $\mathcal{Y}$. Hence $[p,q]\subset \interval_Y$. Using the product regions map $\productmap{\partial\mathcal{Y}}\colon \T(\Sigma) \to \productregion{\Sigma}{\partial\mathcal{Y}}$ in the main Teichm\"uller space $\T(\Sigma)$, we may consider the $\T(Y)$--components $p\vert_Y$ and $q\vert_Y$ of $\productmap{\partial \mathcal{Y}}(p)$ and $\productmap{\partial Y}(q)$, respectively. We may now finally invoke Lemma~\ref{lem:res_in_active_interval}(\ref{compare_res_in_active}) to obtain
\[d_{\T(Y)}(\hat{p}\vert_Y,\hat{q}\vert_Y)\ladd_\wfal d_{\T(Y)}(p\vert_Y, q\vert_Y).\]
Combining with the above estimate, and again using Theorem~\ref{thm:product_regions}, we conclude
\begin{align*}
d_{\T(V)}(\hat{p}, \hat{q}) 
&\ladd d_{\T(Y)}(\hat{p}\vert_Y , \hat{q} \vert_Y)
\ladd_\wfal d_{\T(Y)}(p\vert_Y, q\vert_Y)
\ladd d_{\T(\Sigma)}(p,q).\qedhere
\end{align*}
\end{proof}

Recall from Definition~\ref{def:contribution_set} that $\contr_V^\Omega = (\interval_V\setminus M(V))\cup C(V)$.  
Lemmas~\ref{lem:good_subinterval} and \ref{lem:contribution_in_reference_interval} together show that $C_i(V) \subset J\subset \interval_V$. If we define $M_j(V) = \{\interval_W \mid W\in \Omega_j \text{ with }W\esubn V\}$
then Lemma~\ref{cor:other_intervals_miss_J} furthermore shows that $M_j(V)\cap J = \emptyset$ and $C_j(V)\cap J = \emptyset$ for $j\ne i$; that is, we have $J\cap M(V) = J\cap M_i(V)$ and $J\cap C(V) = C_i(V)$. Combining these observations, we conclude that
\begin{equation}
\label{eqn:contribution_intersect_J}
\contr_V^{\Omega} \cap J = \Big(\big(\interval_V \setminus M(V)\big) \cup C(V)\Big) \cap J = \big(J \setminus M_i(V)\big) \cup C_i(V).
\end{equation}
Since $M_i(V)$ is the union of the active intervals $\interval_W$ of all domains $W\in \Omega_i$ with $W \indsw{i} V$ or $V\indse{i} W$, let us define
\begin{align*}
\mathcal{W}_- = \{W\in \Omega_i \mid W \indsw{i} V\}\qquad\text{and}\qquad
\mathcal{W}_+ = \{W\in \Omega_i \mid V \indse{i} W\}.
\end{align*}
Using these collections, we then define
\[y_1 = \sup\left(\{y\} \cup \bigcup_{W\in \mathcal{W}_-}\interval_W\right)
\qquad\text{and}\qquad
z_1 = \inf\left(\{z\} \cup \bigcup_{W\in \mathcal{W}_+}\interval_W\right),
\]
where here each interval $\interval_W$ is taken along $[x_0,x_n]$ and the supremum/infimum are taken with respect to the orientation of this interval from $x_0$ to $x_n$. Note that we have included $\{y\}$ and $\{z\}$ in the definition to both handle the case that $\mathcal{W}_\pm$ may be empty and to ensure $y_1,z_1\in [y,z] = J$.

\begin{lemma}
\label{lem:the_point_y1_z1}
The points $y_1,z_1$ satisfy the following:
\begin{enumerate}
\item $d_V(x_{i-1},y_1)$ and $d_V(z_1,x_i)$ are both at most $\indrafi{V}/3 + \lipconst$.
\item $y_1$ and $z_1$ lie in and occur in order along $J$.
\item The interval $[y_1,z_1]\subset J$ is contained in $J\setminus M_i(V)\subset J\cap  \contr_V^\Omega$.
\item Each point $w\in [z_1,z]$ satisfies $d_V(w,x_i)\ladd_\wfal 0$.
\item Each point $w\in [y,y_1]$ satisfies $d_V(x_{i-1},w)\ladd_\wfal 0$,
\end{enumerate}
\end{lemma}
\begin{proof}
For (1), let us only consider $d_V(z_1,x_i)$. The construction of $J$ (Lemma~\ref{lem:good_subinterval}) ensures $d_V(x_i,z)\le 7\wfal$. Hence the claim is immediate if $z_1 = z$. Otherwise, there is some $W\in \mathcal{W}_+$ so that $z_1\in \interval_W$. Thus $\partial W$ is contained in the Bers marking at $z_1$ so that $d_V(z_1,x_i) \le d_V(\partial W, x_i) + \lipconst$. Since $V \indse{i} W$, the definition of encroachment and the fact that $\Omega_i$ is wide now give
\[d_V(x_i, z_1) \le d_V(x_i, \partial W) + \lipconst \le \enc_{\Omega_i}(V) + \lipconst\le \indrafi{V}/3 + \lipconst.\]

For (2), since the pairs $y,y_1$ and $z_1,z$ occur in order by construction, it suffices to show $y_1,z_1$ occur in order along $[x_0,x_n]$ as this will force $[y_1,z_1]\subset [y,z] = J$. By means of contradiction, let us instead suppose $z_1,y_1$ occur in order.
First note that having $y_1\in [x_i^V,x_n]$ would imply (by Theorem~\ref{thm:back})
\[
d_V(x_{i-1}^V,y_1) \ge d_V(x_{i-1}^V,x_i^V) - \back \ge d_V(x_{i-1},x_i) - 2\wfal-\back
\ge \indrafi{V} - 3\wfal\]
and hence $d_V(x_{i-1},y_1) \ge \indrafi{V}-4\wfal$. Since $\indrafi{V} \ge 30\wfal$, this is incompatible with (1). Hence we must in fact have $y_1\in [z_1,x_i^V]$, in which case Theorem~\ref{thm:back} now gives
\[d_V(y_1,x_i^V) \le d_V(z_1,x_i^V) + \back \le d_V(z_1,x_i) + \wfal + \back \le \indrafi{V}/3 + 3\wfal,\]
where we have again utilized (1). 
Using $V\in \Upsilon^c(x_{i-1},x_i)$ together with one more application of (1), this now leads to the contradiction:
\begin{align*}
\indrafi{V} \le 
d_V(x_{i-1},x_i)
& \le d_V(x_{i-1},y_1) + d_V(y_1,x_i^V) + d_V(x_i^V,x_i)\\
& \le \indrafi{V}/3 + \lipconst + \indrafi{V}/3 + 3\wfal + \wfal 
< \indrafi{V}.
\end{align*}

Since $[y_1,z_1]\subset J$, the assertion $[y_1,z_1]\subset J\setminus M_i(V)$ of (3) is clear:  $M_i(V)$ is the union of intervals $\interval_W$ for $W\esubn V$ with $W\in \mathcal{W}_- \cup \mathcal{W}_+$. By definition of the points $y_1,z_1$, if $W\in \mathcal{W}_-$ then $\interval_W\subset [x_0, y_1]$ and if $W\in \mathcal{W}_+$ then $\interval_W\subset [z_1, x_n]$. Hence $M_i(V)$ is disjoint from $[y_1,z_1]$, which proves the claim.

For (4), if $w\in [z_1,z] \subset [z_1,x_i^V]$, then as above Theorem~\ref{thm:back} and (1) give
\begin{align*}
d_V(w, x_i) &\le d_V(z_1, w) + d_V(w, x_i^V) + d_V(x_i^V,x_i)\\
 &\le d_V(z_1,x_i^V) + \back + \wfal 
\le d_V(z_1,x_i) + 2\wfal + \back \ladd_\wfal 0.
\end{align*}
The argument for (5) is symmetric.
\end{proof}

The significance of the subinterval $[y_1,z_1]$ is highlighted by the next lemma.

\begin{lemma}
\label{lem:bar-w_and_hat-w_are_close}
Every point $w\in [y_1,z_1]$ satisfies $d_{\T(V)}(\hat w, \bar w)\ladd_\wfal 0$.
\end{lemma}

\begin{proof}
We use Lemma~\ref{lem:thin_with_bdd_projections} 
and show that $\hat{w}$ and $\bar{w}$ agree in all subsurfaces and have the same short curves with the same lengths. Consider any domain $Z\esub V$ and let $\tilde{w}_Z\subset\cc(Z)$ be as in Proposition~\ref{prop:consistent_resolution}. Then $d_Z(\hat{w},\tilde{w}_Z)\ladd_\wfal 0$ by construction. Since $\bar{w}$ is simply the $\T(V)$--component of $w$, we also observe that $d_Z(w,\bar{w})\ladd 0$. Thus to bound $d_Z(\hat{w},\bar{w})$ it suffices to bound $d_Z(w,\tilde{w}_Z)$. 

If $Z\notin \Upsilon(x_{i-1},x_i)$ or if $Z\in \Upsilon(x_{i-1},x_i)$ with $\colsup{Z}{\Omega_i} = V$, then $\tilde{w}_Z = \pi_Z(w)$ by definition and hence $d_Z(w, \tilde{w}_Z)\ladd 0$ is immediate. So suppose $Z\in \Upsilon(x_{i-1},x_i)$ with $\colsup{Z}{\Omega_i} = W\esubn V$. We consider the case $W \indsw{i} V$, the opposite possibility $V\indse{i}  W$ being similar. By definition we now have $\tilde{w}_Z = \pi_Z(x_i)$. On the other hand, the construction of $y_1$ implies $\interval_W\subset [x_0, y_1]$. As $\interval_Z \subset \interval_W$ by Lemmas~\ref{lem:good_subinterval} and \ref{lem:contribution_in_reference_interval} (applied with $W$ in place of $V$), it follows that $\interval_Z$ is contained in $[x_0, y_1]$ and that $w$, $x_i^V$ lie in the same component of $[x_0,x_n]\setminus \interval_Z$. Therefore $d_Z(x_i^V,x_i) \le \rafi$ by Lemma~\ref{lem:J_is_in_the_middle} and $d_Z(w,x_i^V) \le \rafi/3$ by Lemma~\ref{lem:our_active_intervals}. The triangle inequality thus gives $d_Z(w,\tilde{w}_Z) = d_Z(w, x_i)\ladd 0$ here as well.

By Lemma~\ref{lem:thin_with_bdd_projections} it remains to bound the quantity $R$ associated to the two points $\hat{w},\bar{w}\in \T(V)$. For this, it suffices to bound the ratio $\ell_{\bar{w}}(\gamma)/\ell_{\hat{w}}(\gamma)$, from above and below, for every curve $\gamma$ that is short on either $\bar{w}$ or $\hat{w}$. Note that $\ell_w(\gamma)$ and $\ell_{\bar{w}}(\gamma)$ agree up to bounded multiplicative error for all essential curves $\gamma$ in $V$, thus we may instead bound the ratio $\ell_{w}(\gamma)/\ell_{\hat{w}}(\gamma)$. Suppose now that $\gamma$ is an essential curve in $V$ with $\ell_{\hat{w}}(\gamma) < \thin$. Then by Definition~\ref{def:comparison_point}, $\ell_w(\gamma)/\ell_{\hat{w}}(\gamma)= 1$. Conversely, suppose $\gamma$ is a curve in $V$ with $\ell_{w}(\gamma) < \thin$. We show that $\gamma$ satisfies condition (\ref{eqn:short_condition}), it will then follow from the definition of $\hat{w}$ that $\ell_{\hat{w}}(\gamma) = \ell_w(\gamma)$. Let $Z\esub V$ be any domain. Since $\ell_w(\gamma)<\thin$, we have $\gamma\in \mu_w$. Thus $d_Z(\gamma,w)\le \lipconst$. Hence (\ref{eqn:short_condition}) is satisfied if $\tilde{w}_Z = \pi_Z(w)$. If $\tilde{w}_Z \ne \pi_Z(w)$, then $Z\in \Upsilon(x_{i-1},x_i)$ with $W = \colsup{Z}{\Omega_i} \ne V$. Let us suppose $W\indsw{i} V$ so that $\tilde{w}_Z = \pi_Z(x_i)$, the reverse possibility $V\indse{i} W$ being similar. As above, we have that $d_Z(w, x_i) \le 4\rafi/3$ and therefore conclude
\[d_Z(\gamma,\tilde{w}_Z) = d_Z(\gamma,x_i)\le d_Z(\gamma,w) + d_Z(w,x_i) \le \lipconst + 4\rafi/3 \le 2\rafi\]
as required. This establishes (\ref{eqn:short_condition}) for $\gamma$ and proves the claim.
\end{proof}

We now establish the properties mentioned in Strategy~\ref{strategy:plan_for_bounding_complexity_length} that will hold for the segments of the  yet-to-be-constructed piecewise geodesic path from $x_{i-1}$ to $x_i$.

\begin{lemma}
\label{lem:paralleling_subintervals}
Let $[p,q] \subset J$ be a subgeodesic satisfying either
\begin{itemize}
\item $[p,q]$ is squarely covered by $\mathcal{D}$, or
\item $[p,q]\subset [y_1,z_1]$.
\end{itemize}
Then $[p,q]$ is contained in $J\cap \contr_V^\Omega$ and $d_{\T(V)}(\hat p, \hat q) \ladd_\wfal d_{\T(\Sigma)}(p,q)$.
\end{lemma}
\begin{proof}
To see that $[p,q]\subset \contr_V^\Omega$, we simply note that $[y_1,z_1]$ is contained in $J\setminus M_i(V)$ by Lemma~\ref{lem:the_point_y1_z1}(3), and that each squarely covered interval is contained in $\interval_Y\subset C_i(V)$ for some $Y\in \mathcal{D}$. Thus clearly $[p,q]\subset (J\setminus M_i(V))\cup C_i(V) = J\cap \contr_V^\Omega$.

If $[p,q]$ is squarely covered by $\mathcal{D}$, the bound on $d_{\T(V)}(\hat{p},\hat{q})$ is simply Lemma~\ref{lem:squarely_covered_interval}. If instead $[p,q]\subset [y_1,z_1]$, then Lemma~\ref{lem:bar-w_and_hat-w_are_close} implies  $d_{\T(V)}(\bar{p},\hat{p})$ and $d_{\T(V)}(\bar{q},\hat{q})$ are both bounded in terms of $\wfal$. 
Therefore
\[d_{\T(V)}(\hat{p}, \hat{q}) \ladd_\wfal  d_{\T(V)}(\bar{p},\bar{q})\]
Since the metric in  $\productregion{\Sigma}{\partial V}$ is a sup metric, by Minsky's Theorem~\ref{thm:product_regions} we have
\[d_{\T(V)}(\bar{p},\bar{q}) \le d_{\productregion{\Sigma}{\partial V}}(\productmap{\partial V}(p), \productmap{\partial V}(q)) \le d_{\T(\Sigma)}(p,q) + \minsky.\]
Combining with the previous inequality thus proves the lemma in this case.
\end{proof}

In contrast to Lemma~\ref{lem:paralleling_subintervals}, we have the following for certain subintervals of $J$:

\begin{lemma}
\label{lem:bound_for_noncontributing_intervals}
Let $[p,q]$ be a geodesic segment in $\T(\Sigma)$ such that either
\begin{itemize}
\item $[p,q] = [x_{i-1},y]$, or $[p,q] = [z,x_i]$, or
\item $(p,q)$ is contained in $J$ and disjoint from $C_i(V)\cup[y_1,z_1]$.
\end{itemize}
Then $d_Z(\hat p, \hat q) \ladd_\wfal 0$ for every domain $Z\esub V$.
\end{lemma}
\begin{proof}
To ease notation, set $J' = \{x_{i-1}\}\cup J \cup \{x_i\}$ and note that $p,q\in J'$. By Lemma~\ref{lem:resolving_points_exist}, each point $w\in J'$  satisfies the condition of Definition~\ref{def:projection_tuple}
\[d_Z(x_{i-1},w) + d_Z(w,x_i) \le d_Z(x_{i-1},x_i) + 9\wfal \quad\text{ for all }Z\esub V\]
and so determines a consistent tuple $(\tilde{w}_Z)\in \prod_{Z\esub V}\cc(Z)$. Recall from Definitions~\ref{def:resolution} and \ref{def:comparison_point} that $\hat w\in \T(V)$ then satisfies $d_Z(\tilde{w}_Z,\hat w) \ladd_\wfal$ for any $Z\esub V$.

Let us now fix a domain $Z\esub V$ and bound $d_Z(\hat p, \hat q)$ provided any of the conditions hold. First suppose $Z\notin \Upsilon(x_{i-1},x_i)$, so that by Definition~\ref{def:projection_tuple} $\tilde{p}_Z = \pi_Z(p)$ and $\tilde{q}_Z = \pi_Z(q)$. In this case we have $d_Z(x_{i-1},x_i) < \indrafi{V}$, so the  above condition implies
\[d_Z(x_{i-1},w) + d_Z(w,x_i) \le \indrafi{V} + 9\wfal\]
for every point $w\in J'$. Therefore we conclude that
\[d_Z(\hat p,\hat q) \ladd_\wfal d_Z(\tilde{p}_Z, \tilde{q}_Z) = d_Z(p,q) \le d_Z(p,x_i) + d_Z(x_i, q) \ladd 2(\indrafi{V} + 9\wfal).\]

Next suppose $Z\in \Upsilon(x_{i-1},x_i)$, and set $W = \colsup{Z}{\Omega_i}\in \Omega_i$. If $W \indsw{i} V$, then by definition $\tilde{p}_Z = \pi_Z(x_i) = \tilde{q}_Z$, and we have
\[d_Z(\hat p, \hat q) \ladd_\wfal d_Z(\tilde{p}_Z, \tilde{q}_Z) = d_Z(x_i,x_i) \le \lipconst.\]
Similarly if $W \indse{i}V$, then $\tilde{p}_Z = \pi_Z(x_{i-1}) = \tilde{q}_Z$ and we again find $d_Z(\hat p, \hat q) \ladd_\wfal 0$. The remaining possibility is $W = V$, in which $Z$ contributes to $V$ in $\Omega_i$. In this case, $\tilde{p}_Z = \pi_Z(p)$ and $\tilde{q}_Z = \pi_Z(q)$, so that $d_Z(\hat p, \hat q) \ladd_\wfal d_Z(p,q)$. Hence it suffices to bound this latter quantity $d_Z(p,q)$. We consider two cases:

First, suppose $Z = V$ itself. If $[p,q] = [x_{i-1}, y]$ or if $[p,q] = [z,x_i]$, then Lemma~\ref{lem:good_subinterval} provides the desired bound:
\[d_Z(p,q) \in \Big\{d_V(x_{i-1},y), d_V(z,x_i)\Big\} \le 7\wfal.\]
Otherwise $(p,q)$ is contained in $J$ and  disjoint from $C_i(V)\cup[y_1,z_1]$.  It follows that $[p,q]$ is either contained in $[y,y_1]$ or $[z_1,z]$. In the latter case, Lemma~\ref{lem:the_point_y1_z1}(4) implies
\[d_V(p,q) \le d_V(p,x_i) + d_V(x_i, q) \ladd_\wfal 0,\]
and in the former case Lemma~\ref{lem:the_point_y1_z1}(5) similarly implies $d_V(p,q) \ladd_\wfal 0$.

Second, suppose $Z\esubn V$. In that case we know that $\interval_Z\subset J = [y,z]$ (Lemma~\ref{lem:contribution_in_reference_interval}) and that $d_Z(x_{i-1},x_{i-1}^V)$ and $d_Z(x_i,x_i^V)$ are both at most $\rafi$ (Lemma~\ref{lem:J_is_in_the_middle}). By Lemma~\ref{lem:our_active_intervals} and the triangle inequality, it follows that
\[
d_Z(z,x_i) \le d_Z(z,x_i^V) + d_Z(x_i^V, x_i) \le \rafi/3 + \rafi \le 2\rafi\]
and similarly that $d_Z(x_{i-1},y)\le 2\rafi$. This handles the case that $[p,q]$ equals $[x_{i-1},y]$ or $[z,x_i]$. If instead $(p,q)$ is contained in $J$ and disjoint from $C_i(V)\cup[y_1,z_1]$, then evidently  $(p,q)\cap \interval_Z = \emptyset$ due to the fact that $\interval_Z \subset C_i(V)$ by definition. Therefore $d_Z(p,q) \le \rafi/3$ by Lemma~\ref{lem:our_active_intervals} and the lemma is proven.
\end{proof}

In order to decompose $J$ into subsegments that satisfy either Lemma~\ref{lem:paralleling_subintervals} or \ref{lem:bound_for_noncontributing_intervals} above, we will use endpoints of active intervals $\interval_Z$ for domains $Z\in \mathcal{D}$. 
To this end,  let $\mathcal{D}_R$ denote the collection of all $Z\in \mathcal{D}$ such that $\interval_Z$ intersects $[z_1,z]$. Define $\mathcal{D}_L$ symmetrically. Observe that $z_1 = z$ forces $\mathcal{D}_R = \emptyset$ (since each $Z\in \mathcal{D}$ has $\interval_Z\subset (y,z)$ by Lemma~\ref{lem:contribution_in_reference_interval}) and similarly for $\mathcal{D}_L$. Using the notation of \S\ref{sec:antichains}, we write $\underline{\mathcal{D}_L}_{x_0}^{x_n}$ and $\underline{\mathcal{D}_R}_{x_0}^{x_n}$ for the set of domains in $\mathcal{D}_L$ and $\mathcal{D}_R$, respectively, that are maximal with respect to the order
\[Z \prec_{[x_0,x_n]} Y \iff Z\esub Y\text{ and }\interval_Z\subset\interval_Y\text{ along }[x_0,x_n].\]

\begin{lemma}
\label{lem:maximal_DL_bound}
The collections $\underline{\mathcal{D}_L}_{x_0}^{x_n}$ and $\underline{\mathcal{D}_R}_{x_0}^{x_n}$ have uniformly bounded  cardinality. That is $\abs{\underline{\mathcal{D}_L}_{x_0}^{x_n}},\abs{\underline{\mathcal{D}_R}_{x_0}^{x_n}}\ladd_\irafi 0$.
\end{lemma}
\begin{proof}
We only consider $\underline{\mathcal{D}_R}_{x_0}^{x_n}$. We may assume $z_1 \ne z$, for otherwise $\mathcal{D}_R = \emptyset$ and there is nothing to prove. Thus, by definition of $z_1$, we may choose $W\in \Omega_i$ such that $V \indse{i} W$ and so that $z_1$ is the left endpoint of $\interval_W$. Since $\Omega_i$ is assumed to be wide, we have that $d_V(\partial W, x_i) \le \enc_{\Omega_i}(V) \le \indrafi{V}/3$. 

We claim that for every $Z\in \mathcal{D}_R$, the multicurves $\partial Z$ and $\partial W$ are disjoint. Indeed, the definition of $\mathcal{D}_R$ ensures that $\interval_Z$ either intersects or occurs to the right of $\interval_W$. If $\partial Z \cut \partial W$, then $Z\cut W$ and hence $W$ is necessarily time-ordered before $Z$ along $[x_0,x_n]$. Lemma~\ref{lem:compatible-time-order} implies we also have the time ordering $W \tol Z$ along $[x_{i-1},x_i]$. But, since $Z$ contributes to $V$ in $\Omega_i$ and $V\indse{i}W$, this contradicts \ref{subord-contribute}.

Choose $\alpha\in \pi_V(x_{i-1})$ and $\beta\in \pi_V(x_i)$ realizing $d_V(x_{i-1},x_i)$ and fix a geodesic $\beta = \gamma_0, \dotsc, \gamma_m = \alpha$ in $\cc(V)$. Since $\Upsilon^\ell(x_{i-1},x_i)\subset \Omega_i$ by insulation, we have $\mathcal{D}\subset \Upsilon^c(x_{i-1},x_i)$; that is $d_Z(x_{i-1},x_i)\ge \indrafi{Z}$ for each $Z\in \mathcal{D}$. Exactly as in the proof of Lemma~\ref{lem:threshholds}, the bounded geodesic image theorem implies that each $Z\in \mathcal{D}$ is disjoint from one of the curves $\gamma_j$.
If we fix $Z\in \mathcal{D}_R$ and let $0\le j \le m$ be such that $Z$ is disjoint from $\gamma_j$, it follows that
\[j = d_V(\beta,\gamma_j) \le d_V(x_i,\gamma_j) \le d_V(x_i, \partial W) + d_V(\partial W, \partial Z) + d_V(\partial Z, \gamma_j) \le \enc_{\Omega_i}(V) + 2.\]
This proves that if we define
\[\mathcal{Y} = \{Y \mid Y\text{ is a connected component of }V\setminus \gamma_j\text{ for some } j\le \enc_{\Omega_i}(V) + 2\},\]
then each $Z\in \mathcal{D}_R$ satisfies $Z\esub Y$ for some $Y\in \mathcal{Y}$. Notice that, since $\Omega_i$ is wide, $\abs{\mathcal{Y}}\le 2(\enc_{\Omega_i}(V)+3) \le 2\indrafi{V}/3 + 6 \le \indrafi{V}$.

For each $Y\in \mathcal{Y}$ we consider the collection
\[\mathcal{P}(Y) = \{U\esub Y \mid d_U(x_{i-1},x_i)\ge \indrafi{U}\}.\]
Now choose any $Z\in \underline{\mathcal{D}_R}_{x_0}^{x_n}$, that is a maximal element of $\mathcal{D}_R$ with respect to the partial order $\prec_{[x_0,x_n]}$. Choose some $Y\in \mathcal{Y}$ so that $Z\esub Y$. Since $d_Z(x_{i-1},x_i)\ge \indrafi{Z}$, we have $Z\in \mathcal{P}(Y)$ as well. We claim that furthermore $Z\in \underline{\mathcal{P}}_{x_0}^{x_n}(Y)$. To see this, consider any $U\in \mathcal{P}(Y)$ with $Z\esub U$ and $\interval_Z\subset \interval_U$. Since $Z\esub U \esub Y\esubn V$, the fact $\colsup{Z}{\Omega_i} = V$ forces $\colsup{U}{\Omega_i} = V$ as well. As $U\esubn V$ and $U\in \Upsilon(x_{i-1},x_i)$, we see that $U$ contributes to $V$ and in fact that $U\in \mathcal{D}$. Finally, since $\interval_Z$ intersects $[z_1,z]$, the same holds for $\interval_U\supset \interval_Z$. Therefore $U\in \mathcal{D}_R$. Since $Z$ is $\prec_{[x_0,x_n]}$--maximal in $\mathcal{D}_R$, it follows that $U = Z$. Hence $Z\in \underline{\mathcal{P}}_{x_0}^{x_n}(Y)$ as claimed. This proves that each element of $\underline{\mathcal{D}_R}_{x_0}^{x_n}$ is contained in $\underline{\mathcal{P}}_{x_0}^{x_n}(Y)$ for some $Y\in \mathcal{Y}$. Thus we have
\[\underline{\mathcal{D}_R}_{x_0}^{x_n}\subset \bigcup_{Y\in \mathcal{Y}}\underline{\mathcal{P}}_{x_0}^{x_n}(Y).\]
Applying Lemma~\ref{lem:threshholds} with the thresholds $\indrafi{\plex{\Sigma}}\le \dotsb \le \indrafi{-1} = \irafi$ gives a bound $\abs{\underline{\mathcal{P}}_{x_0}^{x_n}(Y)} \ladd_\irafi 0$ for every $Y$. Since $\abs{\mathcal{Y}}\le \irafi$, we conclude $\abs{\underline{\mathcal{D}_L}_x^y} \ladd_\irafi 0 $, as desired.
\end{proof}

We are now finally ready to complete the proof of the Theorem:

\begin{proof}[Proof of Theorem~\ref{thm:bound_single_resolution_dist}--Nonannular case]
Let $E$ denote the union of $\{y,y_1,z_1,z\}$ with the set of all endpoints of active intervals $\interval_Z$ for $Z\in \underline{\mathcal{D}_L}_{x_0}^{x_n}$ or $Z\in \underline{\mathcal{D}_R}_{x_0}^{x_n}$. 
This nonempty set is contained in $J$ and has uniformly bounded cardinality by Lemma~\ref{lem:maximal_DL_bound}. Let us write $E = \{e_1,\dots,e_{k-1}\}$  ordered along $J$ as
\[y = e_1 < e_2 < \dots < e_{k-1} = z.\]
We also define $e_0 = x_{i-1}$ and $e_k = x_i$. The points $e_0,\dots,e_k$ therefore define a piecewise geodesic path in $\T(\Sigma)$ from $x_{i-1}$ to $x_i$:
\[[e_0,e_1][e_1,e_2]\dotsb[e_{k-1},e_k]\]

We claim that each segment $[p,q]$ of this concatenation satisfies the hypotheses of either Lemma~\ref{lem:bound_for_noncontributing_intervals} or Lemma~\ref{lem:paralleling_subintervals}. Indeed, the first and last segments $[x_{i-1},y]$ and $[z,x_i]$ satisfy Lemma~\ref{lem:bound_for_noncontributing_intervals} by fiat, and any subsegment of $[y_1,z_1]$ satisfies Lemma~\ref{lem:paralleling_subintervals}. If $[p,q]$ is not covered by the previous sentence, then $[p,q]$ is contained in $[y,y_1]$ or $[z_1,z]$. By symmetry, let us suppose it is the former. We may assume $(p,q)$ intersects $C_i(V)$, for otherwise it satisfies Lemma~\ref{lem:bound_for_noncontributing_intervals}. Now let $Z\in \mathcal{D}$ be any domain for which $\interval_Z$ intersects $(p,q)$. Since $[p,q]\subset [y,y_1]$ we evidently have $Z\in \mathcal{D}_L$ and may choose some $Y\in \underline{\mathcal{D}_L}_{x_0}^{x_n}$ with $Z\esub Y$ and $\interval_Z\subset \interval_Y$. It follows that $[p,q]$ intersects $\interval_Y$ as well. Since the points $p,q$ are consecutive in the set $E$, which by definition contains both endpoints of $\interval_Y$, it must be that $[p,q]\subset \interval_Y$. Therefore $[p,q]$ is squarely covered by $\mathcal{D}$ and satisfies Lemma~\ref{lem:paralleling_subintervals}.

Taking resolutions produces a sequence of points $\hat{x}_{i-1} = \hat{e}_0,\dots,\hat{e}_k = \hat{x}_i$ in $\T(V)$. 
Let $P\subset \{1,\dots,k\}$ be the set of indices $1\le j \le k$ such that the segment $[e_{j-1},e_j]$ satisfies Lemma~\ref{lem:paralleling_subintervals}. Since the intervals $[e_{j-1},e_j]$ with $j\in P$ have disjoint interiors and are each contained in $J\cap \contr_V^\Omega$, applying Lemma~\ref{lem:paralleling_subintervals} implies that
\begin{equation}
\label{eqn:sum_paralleling_intervals}
\sum_{j\in P} d_{\T(V)}(\hat{e}_{j-1},\hat{e}_j) \ladd_\wfal \sum_{j\in P} d_{T(\Sigma)}(e_{j-1},e_j) \le \int_y^z \charfunc_{\contr_V^\Omega}.
\end{equation}
Note that in the first inequality above we have used the fact that $k$ is uniformly bounded (Lemma~\ref{lem:maximal_DL_bound}) to combine the additive errors from each of the $\abs{P}\le k$ applications of Lemma~\ref{lem:paralleling_subintervals} into a single additive error depending only on $\wfal$.

Now let $Q = \{1,\dots,k\}\setminus P$ be the set of remaining indices. By the above, for each $j\in Q$ the segment $[e_{j-1},e_j]$ satisfies Lemma~\ref{lem:bound_for_noncontributing_intervals}; consequently we have
\begin{equation}
\label{eqn:bounded_projections_for_noncontributing}
d_Z(\hat{e}_{j-1},\hat{e}_j)\ladd_\wfal 0 \quad\text{for each $j\in Q$ and every domain $Z\esub V$}.
\end{equation}
For each $j\in Q$ let $\Gamma_j$ denote the  set of essential curves $\alpha$ in $V$ such that either $\ell_{\hat{e}_{j-1}}(\alpha) < \thin'/2$ or $\ell_{\hat{e}_j}(\alpha) < \thin'/2$. 
Since a point in $\T(V)$ can have at most $\plex{V}$ disjoint curves, we see that $\abs{\Gamma_j} \le2\plex{S}$. 
Setting $\Gamma = \cup_{j\in Q} \Gamma_j$ now gives a set of uniformly bounded cardinality.

Note that if $\Gamma_j = \emptyset$, then the quantity $\hat{R}_j$ in Lemma~\ref{lem:thin_with_bdd_projections} for the pair $\hat{e}_{j-1},\hat{e}_j$ is uniformly bounded and hence that lemma implies $d_{\T(V)}(\hat{e}_{j-1},\hat{e}_j) \ladd_\wfal 0$. Thus if $\Gamma$ were empty, combining the inequalities (\ref{eqn:sum_paralleling_intervals}) and (\ref{eqn:bounded_projections_for_noncontributing}) above would prove the proposition. However, since the points $\hat{w}$ for $w\in J$ are allowed to be thin (c.f. Definition~\ref{def:comparison_point}), $\Gamma$ may be nonempty and we must work a bit harder.

\begin{claim}
\label{claim:short_annuli_are_noncontributing}
If $A\esub V$ is an annulus with $\partial A = \alpha\in \Gamma$, then $A$ does not contribute to $V$ in $\Omega_i$. Therefore $\diam_{\cc(A)}(\pi_A(\hat{e}_0) \cup \dots \cup \pi_A(\hat{e}_k)) \ladd_\wfal$.
\end{claim}
\begin{proof}[Proof of claim]
By contradiction, suppose that $A$ contributes to $V$. The hypothesis implies there is some $0 \le j \le k$ such that $\ell_{\hat{e}_j}(\partial A) < \thin'/2$ and such that either $j\in Q$ or $j+1\in Q$. By construction, $\hat{e}_0 = \hat{x}_{i-1}$ and $\hat{e}_k = \hat{x}_i$ are both thick; therefore it must be that $0 < j < k$. The construction of $\hat{e}_j$ (Definition~\ref{def:comparison_point}) implies in this case that $\ell_{e_j}(\partial A) = \ell_{\hat{e}_j}(\partial A) < \thin'/2$. Therefore the point $e_j\in J$ evidently lies in the interior of the active interval $\interval_A$ of $A$. Since $\interval_A\subset J$ by Lemma~\ref{lem:contribution_in_reference_interval}, this rules out both possibilities $e_j = y$ and $e_j=z$; hence in fact $1 < j < k-1$. Since $e_j$ is in the \emph{interior} of $\interval_A$, we see that $\interval_A\subset C_i(V)$ intersects the interiors of both $[e_{j-1},e_j]$ and $[e_j,e_{j+1}]$. It follows neither $(e_{j-1},e_j)$ nor $(e_j,e_{j+1})$ is disjoint from $C_i(V)$, and thus that neither of these intervals satisfies Lemma~\ref{lem:bound_for_noncontributing_intervals}. But this contradicts the assumption that either $j\in Q$ or $j+1\in Q$. Hence $A$ cannot contribute to $V$.

For the second conclusion, if $A\in \Upsilon(x_{i-1},x_i)$ then the above implies that either $\colsup{A}{\Omega_i}\indsw{i}V$ or $V\indse{i} \colsup{A}{\Omega_i}$. In the former case we have $d_A(\hat{e_j}, x_i)\ladd_\wfal 0$ for all $0\le j \le k$, and in the latter case we have $d_A(\hat{e}_j,x_{i-1})\ladd_\wfal 0$. Otherwise $A\notin \Upsilon(x_{i-1},x_i)$ so that $d_A(x_{i-1},x_i)\le \indrafi{A}$. In this case for each $0\le j \le k$ we have $d_A(\hat{e}_j,e_j)\ladd_\wfal 0$ by construction and, by Lemma~\ref{lem:resolving_points_exist}, that
\[d_A(\hat{e}_j,x_i)\ladd_\wfal d_A(e_j, x_i) \le d_A(x_{i-1},e_j) + d_A(e_j, x_i) \ladd_\wfal d_A(x_{i-1},x_i) \le \indrafi{A}.\]
In any case, $\cup_j \pi_A(\hat{e}_j)$ lies within bounded distance of either $\pi_A(x_{i-1})$ or $\pi_A(x_i)$.
\end{proof}

For any essential curve $\alpha$ on $V$, we now define a transformation $f_\alpha$ of $\T(V)$ to itself by utilizing the product region $\productregion{V}{\alpha} = \T(V\setminus \alpha) \times \T(\alpha)$.  Recalling that $\T(\alpha) = \H^2$, let $h_\alpha\colon \T(\alpha)\to \T(\alpha)$ be the map that pushes points vertically down to below the horizontal line $1/\thin'$; that is, $h_\alpha(x,y) = (x,\min\{y,\frac{1}{\thin'}\})$ for $(x,y)\in \H^2$. Conjugating with $\productmap{\alpha}$ then gives a transformation $f_\alpha = \productmap{\alpha}\inv\circ (\mathrm{id}\times  h_\alpha)\circ \productmap{\alpha}$ from $\T(V)$ to itself. Observe that $f_\alpha$ is the identity on the complement of the thin region $\mathcal{H}_{\thin',\alpha}$, and therefore fixes every point of $w\in\T(V)$ with $\ell_w(\alpha) \ge \thin'$. 
The fact that $f_\alpha$ only makes $\alpha$ longer and does not affect twisting leads easily to the following:
\begin{claim}
\label{claim:length_lenthening_is_nice}
For every point $w\in \T(V)$ we have:
\begin{itemize}
\item $d_Z(w,f_\alpha(w))\ladd 0$ for every domain $Z\esub V$.
\item $\log(\ell_{f_\alpha(w)}(\gamma))\gadd \log(\min\{\ell_w(\gamma), \thin'\})$ for every essential curve $\gamma$ on $V$.
\end{itemize}
\end{claim}
\begin{proof}[Proof of Claim]
It is clear that any short marking at $w$ is also a short marking for $f_\alpha(w)$; whence the first bullet. The second bullet is immediate for the curve $\alpha = \gamma$. If $\gamma\ne \alpha$ is disjoint from $\alpha$, then $\gamma$ is essential in $\T(V\setminus \alpha)$ and so the lengths $\ell_w(\gamma)$ and $\ell_{f_\alpha(w)}(\gamma)$ coarsely agree.
Finally suppose $\gamma \cut \alpha$. If $\ell_w(\alpha) \ge \thin'$ then $f_\alpha(w) = w$ and there is nothing to prove. Otherwise $\ell_w(\alpha) < \thin$ and so $\ell_{f_\alpha(w)}(\alpha) = \thin'$ by construction. Thus necessarily $\ell_{f_\alpha(w)}(\gamma) > \thin'$ since $\thin'$ is smaller than the Margulis constant.
\end{proof}

Let us list the curves in $\Gamma$ as $\Gamma = \{\alpha_1,\dots,\alpha_m\}$ and write $f_t = f_{\alpha_t}$. For each $0 \le j \le k$, set $\hat{e}_j^0 = \hat{e}_j$  and then recursively set $\hat{e}_j^t = f_t(\hat{e}_j^{t-1})$ for $1 \le t \le m$. Since the points $\hat{e}_0$ and $\hat{e}_k$ are thick by construction, each map $f_t$ fixes these two points and we have $\hat{e}_0^m = \hat{e}_0 = \hat{x}_{i-1}$ and $\hat{e}_k^m = \hat{e}_k = \hat{x}_i$. Hence to prove the proposition it suffices to bound $d_{\T(V)}(\hat{e}_0^m,\hat{e}_k^m)$.

Applying Claim~\ref{claim:length_lenthening_is_nice} successively for the maps $f_1,\dots,f_m$ gives $d_Z(\hat{e}_j, \hat{e}_j^m) \ladd_\wfal 0$ for each $Z\esub V$ (recall that $m\ladd_\wfal 0$). Therefore (\ref{eqn:bounded_projections_for_noncontributing}) can now be restated as
\begin{equation*}
\label{eqn:projection_bound_for_image_points}
d_Z(\hat{e}_{j-1}^m, \hat{e}_j^m) \ladd_\wfal 0\quad\text{ for each $j\in Q$ and every domain $Z\esub V$}.
\end{equation*}
Furthermore, when $j\in Q$, since $\Gamma$ contains every short curve at $\hat{e}_{j-1}$ or $\hat{e}_j$ and each of these curves gets lengthened by one of the maps $f_{t}$, repeated applications of Claim~\ref{claim:length_lenthening_is_nice} shows that the points $\hat{e}^m_{j-1}$ and $\hat{e}^m_j$ are uniformly thick. Therefore the quantity $\hat{R}_j^m$ from Lemma~\ref{lem:thin_with_bdd_projections} associated to this pair is uniformly bounded above (in terms of $\wfal$), and we may promote the above bound to yield the following:
\begin{equation}
\label{eqn:dist_bound_non-contributing_image_points}
d_{\T(V)}(\hat{e}_{j-1}^m, \hat{e}_j^m) \ladd_\wfal 0\quad\text{ for each $j\in Q$}.
\end{equation}

The only remaining step is to promote the bound in (\ref{eqn:sum_paralleling_intervals}) to the new points $\hat{e}^m_{j-1},\hat{e}^m_j$ for $j\in P$. The key point here is that the diameter bound from Claim~\ref{claim:short_annuli_are_noncontributing} implies our transformations $f_t$  are coarsely $1$--Lipschitz for the points in question:

\begin{claim}
\label{claim:length-lengthing_is_Lipshitz}
For any $1 \le j \le k$ and $0 \le t \le m$ we have
\[d_{\T(V)}(\hat{e}_{j-1}^t, \hat{e}_j^t)\ladd_\wfal d_{\T(V)}(\hat{e}_{j-1}, \hat{e}_j).\]
\end{claim}
\begin{proof}
We fix $j$ and proceed by induction on $t$, with the claim being immediate for $t=0$. Fix $t\ge 1$ and  suppose the claim holds for $t-1$.
To ease notation, set $p = \hat{e}_{j-1}^{t-1}$ and $q = \hat{e}_j^{t-1}$. Thus by the induction hypothesis it suffices to prove
\begin{equation}
\label{eqn:tranformaion_is_lipschitz}
d_{\T(V)}(f_t(p), f_t(q)) \ladd_\wfal d_{\T(V)}(p,q).
\end{equation}

Let $[a,b] = \interval_{\alpha_{t}}$ be the (possibly empty) active interval for the curve $\alpha_{t}$ along the geodesic segment $[p,q]$.  Since the length of $\alpha_{t}$ is at least $\thin'$ in the complement of $\interval_{\alpha_t}$, the map $f_t$ is the identity on this complement. 
Thus it suffices to suppose $[a,b]$ is nonempty, for otherwise $f_t$ fixes both points $p,q$ and (\ref{eqn:tranformaion_is_lipschitz}) is immediate.
As $f_t$ is the identity on  $[p,q]\setminus [a,b]$, we have $d_{\T(V)}(f_t(p),f_t(a)) = d_{\T(V)}(p,a)$ and similarly $d_{\T(V)}(f_t(b),f_t(q)) = d_{\T(V)}(b,q)$. Since
\[d_{\T(V)}(p,q) = d_{\T(V)}(p,a) + d_{\T(V)}(a,b) + d_{\T(V)}(b,q),\]
by the triangle inequality it therefore suffices
 to prove that
\begin{equation}
\label{eqn:lipschitz_on_active_interval}
d_{\T(V)}(f_t(a),f_t(b))\ladd_\wfal d_{\T(V)}(a,b).
\end{equation}

Now let $A\esub V$ be the annulus with $\partial A = \alpha_t$. Combining Claim~\ref{claim:short_annuli_are_noncontributing} with $t-1$ applications of Claim~\ref{claim:length_lenthening_is_nice} implies that $d_A(p,q)\ladd_\wfal 0$. By Theorem~\ref{thm:back} it follows that $d_A(a,b) \ladd_\wfal$. Let $a\vert_{\alpha_t},a\vert_{\alpha_t}^\neg$ and $b\vert_{\alpha_t},b\vert_{\alpha_t}^\neg$ respectively denote the $T(\alpha_t)$-- and $T(V\setminus\alpha_t)$--components of the images $\productmap{\alpha_t}(a),\productmap{\alpha_t}(b)\in \productregion{V}{\alpha_t}$. The previous sentence implies that the horizontal coordinates of $a\vert_{\alpha_t}$ and $b\vert_{\alpha_t}$ (viewed in $T(\alpha_t) = \H^2$) differ by an amount bounded in terms of $\wfal$. On the other hand, since $a,b\in \interval_{\alpha_t}$ we have $\ell_a(\alpha_t),\ell_b(\alpha_t) < \thin$. It follows that the vertical coordinates of $h_{\alpha_t}(a\vert_{\alpha_t})$ and $h_{\alpha_t}(b\vert_{\alpha_t})$ both lie between $\frac{1}{\thin}$ and $\frac{1}{\thin'}$. We conclude that $h_{\alpha_t}(a\vert_{\alpha_t})$ and $h_{\alpha_t}(b\vert_{\alpha_t})$ have uniformly bounded (in terms of $\wfal$) distance in $\T(\alpha_t)$. Since the metric on $\productregion{V}{\alpha_t}$ is a supremum, it follows that
\begin{align*}
d_{\productregion{V}{\alpha_t}}&\Big(\mathrm{id}\times h_{\alpha_t}(\productmap{\alpha_t}(a)),\; \mathrm{id}\times h_{\alpha_t}(\productmap{\alpha_t}(b))\Big)\\
&=\sup\Big\{ d_{\T(V\setminus \alpha_t)}\big(a\vert_{\alpha_t}^\neg, b\vert_{\alpha_t}^\neg\big) ,\; d_{\T(\alpha_t)}\big(h_{\alpha_t}(a\vert_{\alpha_t}), h_{\alpha_t}(b\vert_{\alpha_t})\big)\Big\}\\
&\ladd_\wfal d_{\T(V\setminus \alpha_t)}\big(a\vert_{\alpha_t}^\neg, b\vert_{\alpha_t}^\neg\big)
\le d_{\productregion{V}{\alpha_t}}\big(\productmap{\alpha_t}(a), \productmap{\alpha_t}(b)\big).
\end{align*}
Finally, the points $a,b,f_{t}(a)f_t(b)$ lie in the thin region $\mathcal{H}_{\thin,\alpha_t}(V)$ where Minsky's Theorem~\ref{thm:product_regions} ensures the maps $\productmap{\alpha_t}^{\pm 1}$ change distances by at most $\minsky$. The last quantity above thus lies within $\minsky$ of $d_{\T(V)}(a,b)$ and, recalling that $f_t = \productmap{\alpha_t}\inv \circ (\mathrm{id}\times h_{\alpha_t})\circ \productmap{\alpha_t}$, the first quantity lies within $\minsky$ of $d_{\T(V)}(f_t(a),f_t(b))$. Therefore the above estimate establishes (\ref{eqn:lipschitz_on_active_interval}) and completes the proof of the claim.
\end{proof}

The proposition now follows easily by invoking the triangle inequality and successively applying 
equation (\ref{eqn:dist_bound_non-contributing_image_points}), 
Claim~\ref{claim:length-lengthing_is_Lipshitz}, 
and equation (\ref{eqn:sum_paralleling_intervals}):
\begin{align*}
d_{\T(V)}(\hat{x}_{i-1},\hat{x}_i) 
&= d_{\T(V)}(\hat{e}^m_0, \hat{e}^m_k)
\le \sum_{j\in P} d_{\T(V)}(\hat{e}^m_{j-1}, \hat{e}^m_j) + \sum_{j\in Q}d_{\T(V)}(\hat{e}^m_{j-1}, \hat{e}^m_j)\\
&\ladd_\wfal \sum_{j\in P}^k d_{\T(V)}(\hat{e}^m_{j-1}, \hat{e}^m_j)
\ladd_\wfal \sum_{j\in P}^k d_{\T(V)}(\hat{e}_{j-1}, \hat{e}_j) 
\ladd_\wfal \int_y^z \charfunc_{\contr_V^\Omega}.\qedhere
\end{align*}
\end{proof}

\section{Dealing with badness}
\label{sec:bad_sets}
Continue to let $\Omega = (\Omega_1,\dots,\Omega_n)$ be a WISC witness family for a strongly $\wfal$--aligned tuple $(x_0,\dots,x_n)$ in $\T(\Sigma)$. 
We want to use Theorem~\ref{thm:bound_single_resolution_dist} to estimate both the complexity $\cl(\Omega)$ (Definition~\ref{def:complexity_for_tuple}) and the savings $\mathfrak{S}(\Omega)$ of $\Omega$ (Definition~\ref{def:savings_in_complexity}). To this end, we will utilize weighted characteristic functions $\charfunc_{\contr^\Omega_V}$ of contribution sets: Again the ultimate goal is Theorem~\ref{thm:main_complexity_length_bound}.  The obstacle as has been suggested are the existence of nested sets. 
In this section we show how to modify a witness family, if necessary, to deal with this problem.

\begin{definition}[Weight and savings]
\label{def:savings}
For each $V\in \Omega$, use the points $x_0^V,\dots,x_n^V$ to define 
an adjustment function $\adjust{V}\colon [x_0,x_n]\to \R$ whose value is $1$ on those subintervals $[x_{i-1}^V,x_i^V]$ such that $V\in \Omega_i$ is an annulus with $\res{x_{i-1}}{\Omega_i}{V},\res{x_i}{\Omega_i}{V}$ both thick, and whose value is zero elsewhere. Thus $\adjust{V} \equiv 0$ for nonannular $V$.
On $[x_{i-1}^V,x_i^V]$ the values of $\ent{V}-\adjust{V}$ and $\adjust{V}$ thus respectively agree with the coefficients $\entstar{V}$ and $(\ent{V}-\entstar{V})$ of the $d_{\T(V)}(\res{x_{i-1}}{\Omega_i}{V},\res{x_i}{\Omega_i}{V})$ terms appearing in the complexity $\cl(\Omega_i)$ (Definition~\ref{def:complexity}) and savings $\mathfrak{S}(\Omega)$ (Definition~\ref{def:savings_in_complexity}). Accordingly the \define{weight} and \define{savings functions} $[x_0,x_n]\to \R$ of $V\in \Omega$ are defined as the products $\weight{V} =  (\ent{V}-\adjust{V})\charfunc_{\contr^\Omega_V}$  and $\savings{V} = \adjust{V}\charfunc_{\contr^\Omega_V}$ with the characteristic function of the contribution set $\contr^\Omega_V$ (from Definition~\ref{def:contribution_set}). 
Summing now yields the \define{total weight} and \define{total savings} functions $\weight{\Omega},\savings{\Omega}\colon [x_0,x_n]\to \R$ of $\Omega$:
\[\weight{\Omega} = \sum_{V\in \Omega}\omega_V = \sum_{V\in \Omega} (\ent{V}-\adjust{V})\charfunc_{\contr^\Omega_V}
\qquad\text{and}\qquad
\savings{\Omega} = \sum_{V\in \Omega} \savings{V} = \sum_{V\in \Omega} \adjust{V}\charfunc_{\contr^\Omega_v}
\]
\end{definition}

Theorem~\ref{thm:bound_single_resolution_dist} says that the individual terms $(\ent{V}-\entstar{V}) d_{\T(V)}(\res{x_{i-1}}{\Omega}{V},\res{x_i}{\Omega}{V})$ and $\entstar{V} d_{\T(V)}(\res{x_{i-1}}{\Omega}{V},\res{x_i}{\Omega}{V})$ appearing in the savings $\mathfrak{S}(\Omega)$ and complexity $\cl(\Omega)$ are bounded by the respective integrals $\int_{x_{i-1}^V}^{x_i^V}\savings{V}$ and $\int_{x_{i-1}^V}^{x_i^V}\weight{V}$. Since the points $x_0^V,\dots,x_n^V$ appear in order along $[x_0,x_n]$,
summing over all $i$ and $V$ shows that $\cl(\Omega)$ and $\mathfrak{S}(\Omega)$ are bounded by the integrals of the total weight $\weight{\Omega}$ and total savings $\savings{\Omega}$ functions over $[x_0,x_n]$ (up to an additive error depending on $n$, $\irafi$, and the cardinality $\abs{\Omega}$). If we knew $\weight{\Omega}(p) + \savings{\Omega}(p) \le \ent{\Sigma}$ for all $p$, we would thus be able to bound $\cl(\Omega) + \mathfrak{S}(\Omega) \ladd \ent{\Sigma}d_{\T(\Sigma)}(x_0,x_n)$. While this inequality need not hold in general, it can only fail on the following sets:

\begin{definition}[Bad set]
\label{def:bad_set}
We say a point $p\in \contr^\Omega_V$ is \define{bad for $V$} in $\Omega$ if there exists $Z\in\Omega$ such that $Z\esubn V$ and $p\in \contr^\Omega_V\cap \contr^\Omega_Z$.
The \define{bad set for $V$} is then defined to be
\[\bad_V^\Omega = \{p\in \contr^\Omega_V \mid p\text{ is bad for }V\} \subset\contr^\Omega_V.\]
with corresponding \define{badness function} $\badfunc{V} = \ent{V}\charfunc_{\bad^\Omega_V}$. As for the weight and savings, the total badness function is $\badfunc{\Omega} = \sum_{V\in \Omega}\badfunc{V}$
\end{definition}

\begin{remark}
\label{rem:contribution_sets_disjointness}
We note that $\contr^\Omega_V\subset \interval_V$ for every domain $V\in \Omega$. This is immediate from the definition when $V$ annular, and when $V$ is nonannular it follows from Lemmas~\ref{lem:good_subinterval}--\ref{lem:contribution_in_reference_interval}. Therefore any pair of domains $V,W\in \Omega$ with $\contr^\Omega_V\cap \contr^\Omega_W\ne\emptyset$ must either be disjoint or nested (since $V\cut W$ is precluded by Lemma~\ref{lem:our_active_intervals}(\ref{ourinterval-disjoint})). 
\end{remark}

\begin{definition}[Restricted complexity]
\label{def:restricted_complexity}
In our applications, we shall need to restrict our attention to a subcollection $\Omega' = (\Omega'_1,\dots,\Omega'_n)$ of a witness family $\Omega = (\Omega_1,\dots,\Omega_n)$, meaning that $\Omega'_i\subset \Omega_i$ for each $i$. We stress that $\Omega'$ itself need not be a witness family. Associated to such a subcollection, we introduce \define{restricted} versions of the complexity length and savings, and of the weight, savings and badness functions. These are respectively denoted
\[
\cl(\Omega\vert{\Omega'}),\quad 
\mathfrak{S}(\Omega\vert{\Omega'}),\quad
\weight{\Omega\vert{\Omega'}},\quad
\savings{\Omega\vert{\Omega'}},\quad
\badfunc{\Omega\vert{\Omega'}}
\]
and are defined by taking the respective formulas for $\Omega$ and indexing the sums over the subcollection $\Omega'$ instead of all of $\Omega$. For example:
\[\cl(\Omega\vert{\Omega'}) = \sum_{i=1}^n \sum_{V\in \Omega'_i} \entstar{V}d_{\T(V)}(\res{x_{i-1}}{\Omega_i}{V}, \res{x_i}{\Omega_i}{V})
\quad\text{and}\quad
\badfunc{\Omega\vert\Omega'} = \sum_{V\in \Omega'} \ent{V} \charfunc_{\bad^\Omega_V}
\]
To emphasize: each summand is exactly the same as in the formula for $\Omega$ (calculated using the data of the witness family $\Omega$); the restriction simply has fewer terms.
\end{definition}

The relationship between badness and weight is made precise by the next lemma.

\begin{lemma}
\label{lem:complexity_function_bound}
For any WISC witness $\Omega$ for a strongly $\wfal$--aligned tuple $(x_0,\dots,x_n)$ in $\T(\Sigma)$, the weight, savings, and badness functions $\weight{\Omega},\savings{\Omega},\badfunc{\Omega}\colon [x_0, x_n]\to \R$ satisfy
\[\weight{\Omega} + \savings{\Omega} - \badfunc{\Omega} \le \ent{\Sigma}.\]
Moreover, if $Y\esub \Sigma$ is a subsurface and $\Omega'\subset \Omega$ is a subcollection such that $V\esub Y$ for all $V\in \Omega'$, then $\weight{\Omega\vert\Omega'} + \savings{\Omega\vert\Omega'} - \badfunc{\Omega\vert\Omega'} \le \ent{Y}$.
\end{lemma}
\begin{proof}
It suffices to prove the moreover statement, which specializes to the first claim when $\Omega' = \Omega$ and $Y = \Sigma$.
Fix $p\in [x_0,x_n]$ and consider the subcollection
\[\mathcal{G}_p = \{V\in \Omega' \mid p\in \A_V^\Omega\text{ and }p\notin \bad_V^\Omega\}.\]
Notice that if $V\in \Omega'\setminus \mathcal{G}_p$, then necessarily $p\in (\contr^V_\Omega\cap \bad^V_\Omega)\cup([x_0,x_n]\setminus\contr^\Omega_V)$ and thus
\[(\weight{V} + \savings{V} - \badfunc{V})(p)
= \ent{V}\charfunc_{\contr^\Omega_V}(p) - \ent{V}\charfunc_{\bad^\Omega_V}(p) = 0.\]
On the other hand, if $V\in \mathcal{G}_p$ then $\badfunc{V}(p) = 0$ so that
\[(\weight{V} + \savings{V} - \badfunc{V})(p) = (\weight{V} + \savings{V})(p) = \ent{V}\charfunc_{\contr^\Omega_V}(p) \le \ent{V}.\]
Summing over all $V\in \Omega'$, we therefore have $(\weight{\Omega}+\savings{\Omega}-\badfunc{\Omega})(p) \le \sum_{V\in \mathcal{G}_p} \ent{V}$. That this latter quantity is at most $\ent{Y}$ follows from the observation that the domains in $\mathcal{G}_p$ are disjoint subdomains of $Y$. Indeed, all $V,Z\in \mathcal{G}_p$ have $p\in \contr^\Omega_V\cap \contr^\Omega_Z$; hence cutting $V\cut Z$ is impossible by Remark~\ref{rem:contribution_sets_disjointness}, and nesting $Z\esubn V$ is impossible by virtue of $V\in \mathcal{G}_p$ and the definition of $\bad^\Omega_V$.
\end{proof}

We remark that the lemma indicates that in trying to bound  $\weight{\Omega} + \savings{\Omega}$ in terms of   $\ent{\Sigma}$ we need to bound $\badfunc{\Omega}$.  The goal of this section is to construct witness families where that term is small.

We will also need the following feature of bad sets.
\begin{lemma}
\label{lem:how_to_be_bad}
If $p\in \bad_V^\Omega$, then there exists an index $1\le i \le n$ and domains $Z, Y\esubn V$ such that $Z,V\in \Omega_i$. $Y\in \Upsilon(x_{i-1},x_i)$ with $\colsup{Y}{\Omega_i} = V$, and $p\in \contr_Z^\Omega \cap \interval_Y$. Furthermore, $i$ is the unique index for which $p$ lies in the interior of $[x_{i-1}^V, x_i^V]$.
\end{lemma}
\begin{proof}
By definition, there is some $Z\in \Omega$ such $Z\esubn V$ and $p\in \contr_Z^\Omega$. Since $Z,V\in \Omega$ with $Z\esubn V$, it follows from the definition that $\interval_Z\subset M(V)$. Since the contribution set is defined as $\contr_V^\Omega = (\interval_V\setminus M(V))\cup C(V)$, it must be the case that $p\in C(V)$. In particular, $p\in C_i(V)$ for some $1\le i \le n$, which means that $p\in \interval_Y$ for some $Y\esubn V$ satisfying $Y\in \Upsilon(x_{i-1},x_i)$ and $\colsup{Y}{\Omega_i} = V$. From Lemmas~\ref{lem:good_subinterval}--\ref{lem:contribution_in_reference_interval}, we see that $p\in \interval_Y\subset J\subset [x_{i-1}^V,x_{i}^V]$. Since $\interval_Z$ evidently intersects $J$, by Corollary~\ref{cor:other_intervals_miss_J} we must have $Z\notin \Upsilon(x_{j-1},x_j)$ for all $j\ne i$. Hence the fact $Z\in \Omega$ implies $Z\in \Omega_i$.
\end{proof}

\subsection{Fixing badness}
\label{sec:fixing_badness}
If $\Omega = (\Omega_1,\dots,\Omega_n)$ and $\Omega' = (\Omega'_1,\dots,\Omega'_n)$ are both WISC witness families for a tuple $(x_0,\dots,x_n)$, we will write $\Omega \subset \Omega'$ to mean that $\Omega_i\subset \Omega'_i$ for each $i$ and that the subordering on $\Omega'_i$ extends the subordering on $\Omega_i$. Note that in this case, for each $V\in \Omega$ we have $M^{\Omega}(V) \subset M^{\Omega'}(V)$ and $C^{\Omega}(V) \supset C^{\Omega'}(V)$ (since in $\Omega'$  there are more domains subordered below $V$ and thus less domains contributing to $V$). Therefore we observe
\[V\in \Omega \subset \Omega' \implies \contr^{\Omega'}_V \subset \contr^\Omega_V.\]

Recall from \S\ref{sec:complexity_of_tuples} that we have defined $\enc_\Omega(V) = \max_i \enc_{\Omega_i}(V)$.
If $V\in \Omega$ satisfies 
$\enc_\Omega(V) \le \tfrac{1}{3}\indrafi{V} - 9\wfal$, then for each $1\le i \le n$ define
\[\Omega_i^+(V) \colonequals \begin{cases}\Omega_i\cup \maxlset{\lenc_{\Omega_i}(V)}{V} \cup \maxrset{\renc_{\Omega_i}(V)}{V}, & \text{ if }V\in \Omega_i\\ \Omega_i, & \text{ if }V\notin \Omega_i\end{cases}.\]
That is, $\Omega_i^+(V)$ is obtained by taking the left and then right augmentations of $\Omega_i$ along $V$ with parameters $t = \lenc_{\Omega_i}(V)$ and $t = \renc_{\Omega_i}(V)$, respectively. Observe that $\Omega_i^+(V)$ is a witness family by Lemma~\ref{lem:add_chunk} and that it inherits a natural subordering by Lemma~\ref{lem:pad_suborder}. Then let $\Omega^+(V) = (\Omega^+_1(V),\dots, \Omega^+_n(V))$ be the associated witness family, and define
\[\widehat{\Omega}(V) = \overline{\Omega^+(V)} = \left(\overline{\Omega^+_1(V)},\dots, \overline{\Omega^+_n(V)}\right) = \left(\widehat{\Omega}_1(V),\dots,\widehat{\Omega}_n(V)\right)\]
to be the insular completion of $\Omega^+(V)$, equipped with its natural subordering.
Notice that, since $\widehat{\Omega}(V)$ is obtained from $\Omega$ by first performing left- and right-augmentations with parameters $\renc_{\Omega_i}(V)$ and $\lenc_{\Omega_i}(V)$, and then a finite sequence of refinements and augmentations with parameter $0$, Lemmas~\ref{lem:encroach_refine} and \ref{lem:encroach_augment} imply that $\widehat{\Omega}(V)$ is again WISC (since $\Omega$ was wide and $\enc_\Omega(V) \le \tfrac{1}{3}\indrafi{V} - 9\wfal$). 
Observe that $\Omega\subset \widehat{\Omega}(V)$  and hence that $\contr^{\widehat{\Omega}(V)}_V \subset \contr^\Omega_V$. 
The point of this procedure is that it moves the bad set for $V$ entirely off of itself:

\begin{lemma}
\label{lem:disjoint_baddness}
Suppose that $\Omega^\dagger$, $\Omega$, and $\Omega^\ddagger$ are insular, complete, subordered witness families for the strongly $\wfal$--aligned tuple $(x_0,\dots,x_n)$ in $\T(\Sigma)$, that $V\in \Omega^\dagger\subset \Omega$ satisfies $\enc_\Omega(V)\le \tfrac{1}{3}\indrafi{V}-9\wfal$, and that $\Omega^\dagger \subset \Omega \subset \widehat{\Omega}(V) \subset \Omega^\ddagger$. Then
\[\bad_V^{\Omega^\dagger} \cap \bad_V^{\Omega^\ddagger} = \emptyset.\]
\end{lemma}
\begin{proof}
Suppose on the contrary that there is some $p\in \bad_V^{\Omega^\dagger}\cap \bad_V^{\Omega^\ddagger}$. Let $1 \le i \le n$ be the unique index (cf Lemma~\ref{lem:how_to_be_bad}) such that $p$ is in the interior of $[x_{i-1}^V, x_i^V]$. By Lemma~\ref{lem:how_to_be_bad},  for each $\ast\in \{\dagger,\ddagger\}$ we have $V\in \Omega^\ast_i$ and may choose subdomains $Z_\ast,Y_\ast \esubn V$ such that $Z_\ast\in \Omega^\ast_i$, that $Y_\ast\in \Upsilon_i$ contributes to $V$ in $\Omega^\ast_i$, and that $p\in \interval_{Z_\ast}\cap \interval_{Y_\ast}$. In particular, the domains $Z_\dagger,Y_\dagger,Z_\ddagger,Y_\ddagger$ must be pairwise disjoint or nested since we have
\[p\in \interval_{Z_\dagger}\cap \interval_{Y_\dagger}\cap \interval_{Z_\ddagger}\cap \interval_{Y_\dagger}.\]
Since $Z_\dagger,V\in \Omega^\dagger_i$ it must be that $Z_\dagger$ is subordered in $\Omega^\dagger_i$ with respect to $V$. Let us suppose $V\indse{i} Z_\dagger $ in $\Omega^\dagger_i$ (the reverse possibility $Z_\dagger\indsw{i}V$ being symmetric). Since $\partial Y_\ddagger$ and $\partial Z_\dagger$ are disjoint and $\Omega^\dagger_i\subset \Omega_i$, we see that
\[d_V(\partial Y_\ddagger, x_i) \le d_V(\partial Z_\dagger, x_i) + 1 \le \renc_{\Omega^\dagger_i}(V) + 1  \le \renc_{\Omega_i}(V) +1.\]

Since $\Omega^\ddagger_i\supset \Omega_i$ and $Y_\ddagger\in \Upsilon_i$ contributes to $V$ in $\Omega_i^\ddagger$, it must be that $Y_\ddagger$ also contributes to $V$ in $\Omega_i$. By definition of encroachment we may choose a domain $U\in \Omega_i$ with $V\indse{i}U$ and $d_V(\cc(V\vert_U), x_i) = \renc_{\Omega_i}(V)$. We claim that $d_V(\partial Y_\ddagger, x_i) \ge \renc_{\Omega_i}(V) - \rafi$. Indeed, if this were not the case then evidently $d_V(\partial Y_\ddagger, \partial U) \ge d_V(\partial Y_\ddagger, \cc(V\vert_U)) - 1 \ge \rafi -1$ which implies $\partial Y_\ddagger\cut \partial U$. Since $Y_\ddagger$ contributes to $V$ in $\Omega_i$, \ref{subord-contribute} implies we must have the time ordering $Y_\ddagger \tol U$ along $[x_{i-1},x_i]$. On the other hand, the facts that $Y_\ddagger$ and $U$ both have active intervals along $[x_{i-1},x_i]$ and that  $d_V(\partial U, x_i) +1 > d_V(\partial Y_\ddagger, x_i) +\rafi$ imply that we must have the time ordering $U \tol Y_\ddagger$; a contradiction.

The above two paragraphs show that 
\[\renc_{\Omega_i}(V) - \rafi \le  d_V(\partial Y_\ddagger, x_i) \le \renc_{\Omega_i}(V) + 1.\]
Therefore $Y_\ddagger$ is necessarily contained in the set $\lset{\renc{\Omega_i}}{V}$ for the segment $[x_{i-1},x_i]$. Hence we must have $Y_\ddagger \esub W_\ddagger\esubn V$ for some domain
\[W_\ddagger\in \maxlset{\renc{\Omega_i}}{V} \subset \Omega^+_i(V) \subset \widehat{\Omega}_i(V)\subset \Omega^\ddagger_i.\]
But this contradicts the fact that $Y_\ddagger$ contributes to $V$ in $\Omega^\ddagger_i$.
\end{proof}

\subsection{Limited Admissibility}
\newcommand{\inum}{\Delta}
\newcommand{\indnum}[1]{\inum_{#1}}
As conveyed in the discussion before Lemma~\ref{lem:complexity_function_bound}, bounding $\cl(\Omega)$ by Teichm\"uller distance requires controlling the badness of the witness family tuple $\Omega$. While it may not be possible to eliminate badness entirely, we will be content to minimize it by repeatedly applying the operation $\Omega  \rightsquigarrow  \widehat{\Omega}(V)$ along with Lemma~\ref{lem:disjoint_baddness} to move the badness somewhere else. Throughout this process we must carefully control the cardinality of the witness families so that sum of the additive errors from Theorem~\ref{thm:bound_single_resolution_dist} does not blow up.

Recall that we have introduced (at the start of \S\ref{sec:witness-families}) an as-yet unspecified sequence of thresholds $\indrafi{\plex{S}+1},\dotsc,\indrafi{-1}$.  We will shortly explain how these are chosen recursively,
together with accompanying bounds and fractions,
\[ \indnum{j} \ge 1 \quad\text{and}\quad  0 < \indfrac{j}\colonequals \frac{1}{4(j+2)^3\wfal\indnum{j}} < 1\quad\text{ for }\plex{S} \ge j \ge -1,\]
in a manner that only depends on the parameter $\wfal$ and the global complexity $\plex{S}$. We continue to use the notation 
$\indfrac{V} = \indfrac{\plex{V}}$ and 
$\indnum{V} = \indnum{\plex{V}}$ for a domain $V\esub S$.
Before specifying these constants, let us mention the role they will play.

\begin{definition}[Admissible and Limited]
A witness family $\Omega = (\Omega_1,\dots,\Omega_n)$ for a strongly $\wfal$--aligned tuple $(x_0,\dots,x_n)$ in $\T(\Sigma)$ is called:
\begin{itemize}
\item \define{admissible} if $\abs{\bad_V^\Omega}\le \indfrac{V} d_{\T(\Sigma)}(x_0, x_n)$ for all $V\in \Omega$,
\item \define{limited} if $\abs{\Omega}_j \le \indnum{j}$ for every index $\plex{S} \ge j \ge -1$, where here
\[\abs{\Omega}_j \colonequals \max_{1\le i \le n} \abs{\Omega_i}_j = \max_{1\le i \le n} \#\{V\in \Omega_i \mid \plex{V} = j\}\]
\end{itemize}
Adding these conditions to our previous ones, we now say a witness family is \emph{WISCAL} if it is wide, insulated, subordered, complete, admissible, and limited, or \emph{WISCL} when we drop the admissibility condition.
\end{definition}

The significance of such witness families is readily apparent:

\begin{theorem}
\label{thm:complexity_bound_for_WISCAL}
Let $\Omega = (\Omega_1,\dotsc,\Omega_n)$ be a witness family for a strongly $\wfal$--aligned tuple $(x_0,\dots,x_n)$ in $\T(\Sigma)$. If $\Omega$ is WISCAL, then its complexity and savings satisfy
\[\cl(\Omega) + \mathfrak{S}(\Omega)  \ladd_{\wfal,n} 
\left(\ent{\Sigma} + \frac{n}{\wfal}\right) d_{\T(\Sigma)}(x_0,x_n).\]
Moreover, if $Y\esub \Sigma$ is a subsurface and $\Omega'\subset \Omega$ is any subcollection such that $V\esub Y$ for all $V\in \Omega'$, then $\cl(\Omega\vert \Omega') + \mathfrak{S}(\Omega\vert\Omega')  \ladd_{\wfal,n} 
\left(\ent{Y} + \frac{n}{\wfal}\right) d_{\T(\Sigma)}(x_0,x_n)$.
\end{theorem}
\begin{proof}
It suffices to prove the moreover case, as it specializes to the main case by taking $Y = \Sigma$ and $\Omega' = \Omega$. Write $\Omega' = (\Omega_1',\dots,\Omega_n')$ where $\Omega'_i\subset \Omega_i$. 
Set $\inum = \sum_j \indnum{j}$. Since $\Omega$ is limited, each family $\Omega_i$ contains at most $\indnum{j}$ domains of complexity $j$. Thence the full union $\cup_i \Omega_i'$ contains at most $n \inum\ladd_{\wfal, n} 0$ domains. For each domain $V\in \Omega'$, since the points $x_0^V,\dots,x_n^V$ appear in order along $[x_0,x_n]$, Theorem~\ref{thm:bound_single_resolution_dist} implies that
\begin{align*}
\sum_{1\le i \le n \mid V\in \Omega'_i} & \entstar{V}d_{\T(V)}(\res{x_{i-1}}{\Omega_i}{V},\res{x_i}{\Omega_i}{V})
 + \sum_{1\le i \le n \mid V\in \Omega'_i}(\ent{V}- \entstar{V})d_{\T(V)}(\res{x_{i-1}}{\Omega_i}{V},\res{x_i}{\Omega_i}{V})\\
&\ladd_{\wfal,n} \sum_{i=1}^n  \int_{x_{i-1}^V}^{x_i^V} \big((\ent{V} - \adjust{V}) + \adjust{V}\big)\charfunc_{\contr^\Omega_V}
= \int_{x_0}^{x_n} \big(\weight{V} + \savings{V}\big).
\end{align*}
Summing these $\abs{\Omega'}\le n\inum$ inequalities over all $V\in \Omega'$, combining their additive errors, and applying Lemma~\ref{lem:complexity_function_bound} now yields
\begin{align*}
\cl(\Omega\vert\Omega') + \mathfrak{S}(\Omega\vert\Omega')
&\ladd_{\wfal, n} \int_{x_0}^{x_n} \big(\weight{\Omega\vert\Omega'}+\savings{\Omega\vert\Omega'}\big)\\
&\le \int_{x_0}^{x_n} (\ent{Y}+ \badfunc{\Omega\vert\Omega'})
= \ent{Y}d_{\T(\Sigma)}(x_0,x_n) + \sum_{V\in \Omega'} \ent{V}\abs{\bad^{\Omega}_V}
.
\end{align*}
Using the 
definition $\indfrac{j} = (4(j+2)^3\wfal\indnum{j})^{-1}$, 
the fact that $\Omega$ is limited and admissible, and that $\ent{V}\le 2(\plex{V}+2)$ for all $V\esub \Sigma$,  we thus conclude
\begin{align*}
\cl(\Omega\vert\Omega')+ \mathfrak{S}(\Omega\vert\Omega') &\ladd_{\wfal,n} \ent{Y} d_{\T(\Sigma)}(x_0, x_n) + \sum_{V\in \Omega'} \ent{V}\indfrac{V}d_{\T(\Sigma)}(x_0,x_n)\\
&\le d_{\T(\Sigma)}(x_0,x_n) \left(\ent{Y}+ n\sum_{j=-1}^{\plex{\Sigma}}\left(\frac{2(j+2)}{4(j+2)^3 \wfal \indnum{j}}\right)\indnum{j}\right)\\
&= d_{\T(\Sigma)}(x_0,x_n) \left(\ent{Y} + \frac{n}{2\wfal}\sum_{j=-1}^{\plex{\Sigma}}\frac{1}{(j+2)^2}\right)\\
&= d_{\T(\Sigma)}(x_0,x_n) \left(\ent{Y} + \frac{n \pi^2}{12\wfal}\right)
\le d_{\T(\Sigma)}(x,y)\left(\ent{Y} + \frac{n}{\wfal}\right).\qedhere
\end{align*}
\end{proof}

\subsection{Saturation}
It remains to prove that WISCAL witness families exist and, in the process, to specify all of the constants $\indrafi{j}$, $\indfrac{j}$, $\indnum{j}$. Our witness families will be constructed in the following iterative manner. 

To begin with, suppose merely that our constants $\indrafi{j}$ and $\indfrac{j}$ have been specified arbitrarily subject to the conditions
\begin{equation}
\label{eqn:basic_threshold_conditions}
\plex{S} + 30\wfal \frac{\thin}{\thin'} = \indrafi{\plex{S}+1} \le  \dotsb \le \indrafi{-1} = \irafi\quad\text{and}\quad0 < \indfrac{j} < 1\quad\text{for all $j$}.
\end{equation}
We continue to use the notation $\indrafi{V} = \indrafi{\plex{V}}$ and $\indfrac{V} = \indfrac{\plex{V}}$ for any domain $V\esubset S$.

Let $\Sigma \esubset S$ be any domain and let $(x_0,\dots,x_n)$ be a strongly $\wfal$--aligned tuple in $\T(\Sigma)$. For each $1\le i \le n$, let $\Omega^0_i$ be the set of topologically maximal domains in the collection
\[\Upsilon(x_{i-1}, x_i) = \Upsilon^c(x_{i-1},x_i)\cup \Upsilon^\ell(x_{i-1},x_i),\]
where we recall that $\Upsilon, \Upsilon^c, \Upsilon^\ell$ are the sets from Definition~\ref{def:upsilons}. Then $\Omega_i^0$ is a witness family for $[x_{i-1},x_i]$ by definition. Since $\Upsilon^\ell(x_{i-1},x_i)$ consists of at most $2\plex{\Sigma}$ annuli, we see that $\Omega_i^0$ consists of the topologically maximal domains in $\Upsilon^c(x_{i-1},x_i)$ together with a subset of $\Upsilon^\ell(x_{i-1},x_i)$. Thus the number of domains in $\Omega_i^0$ of each complexity is uniformly bounded as described by Lemma~\ref{lem:threshholds}. Since there are no nested domains in $\Omega^0_i$, it is trivially subordered. Now let $\Omega^0 = (\Omega_1^0,\dots,\Omega_n^0)$ be the associated subordered witness family for the tuple $(x_0,\dots,x_n)$, and let $\Omega^1 = \overline{\Omega^0} = (\overline{\Omega_1^0},\dots,\overline{\Omega_n^0})$ be its insular completion (Definitions~\ref{def:completion} and \ref{def:tuple-proj-fam}). Note that by Lemma~\ref{lem:completion_properties} we have $\enc_{\Omega^1}(V) \le 9\wfal$ for all $V\in \Omega^1$.

For $k\in \mathbb{N}$, suppose that we have constructed an increasing chain $\Omega^1\subset\dotsb\subset \Omega^k$ of WISC witness families for $(x_0,\dots,x_n)$. Among all domains $V\in \Omega^k$ satisfying
\[\enc_{\Omega^k}(V)\le \tfrac{1}{3}\indrafi{V}-9\wfal\quad\text{and}\quad\abs{\bad_{V}^{\Omega^k}} > \indfrac{V}d_{\T(\Sigma)}(x_0,x_n)\]
(that is, the Lebesgue measure of $\bad_V^{\Omega^k}\subset [x_0,x_n]$ is more than $\indfrac{V}$--percent of the total measure of $[x_0,x_n]$), choose one of maximal complexity and call it $V_k$. Using the operation $\Omega \rightsquigarrow  \widehat{\Omega}(V) = \overline{\Omega^+(V)}$ from \S\ref{sec:fixing_badness}, we then  define
\[\Omega^{k+1} = \widehat{\Omega^k}(V_k)  = \left(\widehat{\Omega_1^k}(V_k),\dots,\widehat{\Omega_n^k}(V_k)\right).\]

In this way, we obtain a list $V_1,V_2,\dotsc$ of domains and a chain $\Omega^1\subset \Omega^2\subset\dotsb$ of WISC witness families. In fact, this process must terminate in finitely many steps yielding a WISC witness family $\Omega\colonequals\cup_k \Omega^k$ with the property that every domain $V\in \Omega$ satisfies $\enc_{\Omega}(V) > \tfrac{1}{3}\indrafi{V} - 9\wfal$ or $\abs{\bad_V^{\Omega}} < \indfrac{V}d_{\T(\Sigma)}(x_0,x_n)$. 
Indeed, since there are only finitely many domains in each collection $\Upsilon(x_{i-1},x_i)$, 
there are only finitely many possible witness families for $(x_0,\dots,x_n)$. 
Since the bad sets $\bad_{V_k}^{\Omega^k}$ and $\bad_{V_k}^{\Omega^{k+1}}$ are disjoint by Lemma~\ref{lem:disjoint_baddness} (and $\bad_{V_k}^{\Omega^k}$ is nonempty by choice of $V_k$), we see that $\Omega^k \subsetneq \Omega^{k+1}$ for each $k$. Therefore the families $\Omega^k$ are all distinct, showing that the process terminates in finitely many steps. 

\begin{definition}
We refer to any family $\Omega$ obtained in this way as a \define{saturated} witness family for $(x_0,\dots,x_n)$.
\end{definition}

By choosing the constants $\indrafi{j}$ and $\indfrac{j}$ carefully, we will be able to bound the number of domains in $\Omega$ of each complexity and to moreover arrange that every domain $V\in \Omega$ satisfies $\enc_\Omega(V) \le \tfrac{1}{3}\indrafi{V}$ and $\abs{\bad_V^\Omega}\le \indfrac{V} d_{\T(\Sigma)}(x_0,x_n)$. 

To this end, we first observe that a particular domain $Z$ can appear in the list $V_1,V_2,\dotsc$ at most $1/\indfrac{Z}$ times. This is because if $k_1,\dotsc,k_\ell$ are distinct indices with $Z = V_{k_1}=\dotsb= V_{k_\ell}$, then Lemma~\ref{lem:disjoint_baddness} implies the bad sets $\bad_{Z}^{\Omega^{k_1}}, \dotsc, \bad_{Z}^{\Omega^{k_\ell}}$ are all disjoint. Whence
\[d_{\T(\Sigma)}(x_0,x_n) \ge \abs{\bad_{Z}^{\Omega^{k_1}}} +\dotsb +  \abs{\bad_Z^{\Omega^{k_\ell}}} > \ell \indfrac{Z}d_{\T(\Sigma)}(x_0,x_n)\]
and we have $\ell\indfrac{Z}<1$ as claimed. This has the following consequence:

\begin{lemma}
\label{lem:saturated_encroachment_bound}
If $\Omega$ is a saturated witness family, then each domain $Z\in \Omega$  satisfies $\enc_\Omega(Z) \le 9\wfal\left(1 + \frac{1}{\indfrac{V}}\right)$.
\end{lemma}
\begin{proof}
Suppose the saturated family is constructed as $\Omega = \cup_k \Omega^k$ for the increasing chain $\Omega^0\subset \Omega^1\dotsb$ where each $\Omega^0_i$ is the set of topologically maximal domains in $\Upsilon(x_{i-1},x_i)$, where $\Omega^1 = \overline{\Omega^0}$, and $\Omega^{k+1} = \widehat{\Omega}^k(V_k)$ for some domain $V_k\in \Omega_k$. Since there are no nested domains in any collection $\Omega^0_i$ for $1\le i \le n$, we trivially have $\enc_{\Omega^0}(W) = 0$ for every domain $W$. By Lemma~\ref{lem:completion_properties}, its insular completion $\Omega^1$ satisfies $\enc_{\Omega_1}(W)\le 9\wfal$ for all $W\esubset \Sigma$. 

For each augmentation $\Omega^{k}\rightsquigarrow \Omega^{k+1} = \overline{\Omega^{k+}(V_k)}$, Lemmas~\ref{lem:encroach_augment} and \ref{lem:completion_properties} together imply that the encroachment for any $Z\esubset \Sigma$ satisfies
\[\enc_{\Omega^{k+1}}(Z) \le \begin{cases} \enc_{\Omega^k}(Z) + 9\wfal, & Z = V_k\\
\max\{\enc_{\Omega^k}(Z), 9\wfal\},& \text{else}.\end{cases}\]
Therefore, we conclude that $\enc_{\Omega}(Z) \le (m+1)9\wfal$, where $m = \abs{\{k\in \mathbb{N} \mid Z = V_k\}}$ is the number of indices $k$ for which $Z = V_k$. But we have observed that $m$ is at most $1/\indfrac{V}$. This proves the lemma.
\end{proof}

Next observe that if $\Omega$ is complete and insulated, then every domain added during an operation $\Omega \rightsquigarrow \widehat{\Omega}(V)$ is a proper subdomain of $V$. Indeed, fix some $1\le i \le n$ and suppose $Z\in \widehat{\Omega}_i(V) \setminus \Omega_i$. If $Z\in \Omega_i^+(V)\setminus \Omega_i$, then by definition $Z\in \maxlset{\lenc_{\Omega_i}(V)}{V}\cup\maxrset{\renc_{\Omega_i}(V)}{V}$  showing that $Z$ is a proper subsurface of $V$ by definition. Since $\Omega_i$ is already a complete and insular by assumption, it is clear from the construction (Definition~\ref{def:completion}) of the insular completion $\widehat{\Omega}_i(V) = \overline{\Omega_i^+(V)}$ that every $Z\in \overline{\Omega_i^+(V)}\setminus \Omega_i^+(V)$ satisfies $Z\esubn W$ for some $W\in \Omega_i^+(V)\setminus \Omega_i$. Therefore every $Z\in \Omega_{i+1}\setminus \Omega_i$ is a proper subsurface of $V$, as claimed.

Let us next analyze how the cardinalities of a family change under an operation $\Omega \rightsquigarrow \widehat{\Omega}(V)$. The previous paragraph shows that for each $\plex{\Sigma} \ge j \ge \plex{V}$ and $1\le i \le n$, the number $\abs{\Omega_i}_j$ of domains of complexity $j$ stays constant. Hence:
\[\big\vert{\widehat{\Omega}_i(V)\big\vert}_j = \abs{\Omega_i}_j\quad\text{for $\plex{\Sigma}\ge j \ge \plex{V}$},\quad 1\le i \le n.\]
However, surfaces of lower complexity may be added during the augmentation step $\Omega \rightsquigarrow \Omega^+(V)$, and then during the completion step $\Omega^+(V) \rightsquigarrow \widehat{\Omega}(V)$: For each $-1 \le j < \plex{V}$, Lemma~\ref{lem:bounded_pads} shows that 
\[\abs{\Omega_i^+(V)}_j \le \abs{\Omega_i}_j + 2(2\indrafi{j+1})^{\plex{\Sigma}+3}\qquad\text{for $-1 \le j < \plex{V}$},\quad 1\le i \le n.\]
Lemma~\ref{lem:completion_bound} therefore implies that 
\[\big\vert{\widehat{\Omega}_i(V)\big\vert}_j \le \abs{\Omega_i}_j + 2(2\indrafi{j+1})^{\plex{\Sigma}+3} + G_j(\abs{\Omega_i}_{\plex{\Sigma}},\dotsc,\abs{\Omega_i}_{j+1}),\]
where $G_j$ is a function depending only on the thresholds $\indrafi{\plex{\Sigma}},\dotsc,\indrafi{j+1}$. To summarize, since $\abs{\Omega}_j$ is defined as the maximum $\max_i \abs{\Omega_i}_j$, for any operation $\Omega\rightsquigarrow \widehat{\Omega}(V)$ and complexity $j$, we have
\begin{equation}
\label{eqn:num_new_domains_domains}
\big\vert{\widehat{\Omega}(V)\big\vert}_j - \abs{\Omega}_j \le
\begin{cases}
C'_j + G_j(\abs{\Omega}_{\plex{\Sigma}},\dotsc,\abs{\Omega}_{j+1}),& -1 \le j < \plex{V_i}\\
0,& \plex{V_i} \le j \le \plex{\Sigma}
\end{cases},
\end{equation}
where the number $C'_j$ and function $G_j$ depend only on the thresholds $\indrafi{\plex{\Sigma}},\dotsc, \indrafi{j+1}$.

With this, we can now specify our constants $\indrafi{j}$, $\indfrac{j}$, and $\indnum{j}$ recursively:

\begin{proposition}[Choosing the constants]
\label{prop:saturated_constants}
For any $\wfal\ge \rafi$ and $n\ge 1$, there are constants $\indrafi{j}$, $\indnum{j}$ and $\indfrac{j}$ for $\plex{S}\ge j \ge -1$ satisfying (\ref{eqn:basic_threshold_conditions}), such that the following holds: For any domain $\Sigma\esub S$ and strongly $\wfal$--aligned tuple $(x_0,\dots,x_n)$ in $\T(\Sigma)$,  every saturated witness family $\Omega$ for $(x_0,\dotsc,x_n)$ is WISCAL.
\end{proposition} 
\begin{proof}
By construction, every saturated family is  WISC, but we must take care to ensure $\Omega$ is admissible and limited. Fix some complexity $j \le \plex{S}$. Let us say a choice of constants $\indnum{\plex{S}},\dotsc,\indnum{j}$, fractions $\indfrac{\plex{S}},\dotsc,\indfrac{j}$, and thresholds $\indrafi{\plex{S}},\dotsc,\indrafi{j}$ is robust if, \textbf{irrespective of how the remaining fractions $\indfrac{j-1},\dotsc,\indfrac{-1}$ and thresholds $\indrafi{j-1},\dotsc,\indrafi{-1}$ are specified}, subject to equation \eqref{eqn:basic_threshold_conditions}, every saturated resolution family $\Omega$ as in the proposition statement satisfies
\[\abs{\Omega}_{m} \le \indnum{m}\quad\text{and}\quad\enc_\Omega(V)\le \frac{1}{3}\indrafi{V}-9\wfal\quad\text{ for all $\plex{V},m \ge j$};\]
that is, if limited admissibility holds for all complexities at least $j$.

To begin the recursion, let us set
\[\indnum{\plex{S}} = 1,\quad \indfrac{\plex{S}} = \frac{1}{4(\plex{S}+2)^3\wfal \indnum{\plex{S}}},\quad\text{and}\quad \indrafi{\plex{S}} = 30\wfal\frac{\thin}{\thin'}\left(2 + \frac{1}{\indfrac{\plex{S}}}\right).\]
Note this ensures $\indrafi{\plex{S}}\ge \plex{S} +30\wfal\tfrac{\thin}{\thin'}$. We claim this choice is robust. Indeed, let $\Omega$ be any saturated witness family. Since there is only one domain of complexity $\plex{S}$, namely $S$ itself, we trivially have $\abs{\Omega}_{\plex{S}}\le 1 = \indnum{\plex{S}}$. Further, Lemma~\ref{lem:saturated_encroachment_bound} and our choice of $\indrafi{\plex{S}}$ ensure that $\enc_{\Omega}(S)\le 9\wfal(1 + \frac{1}{\indfrac{S}})< \frac{1}{3}\indrafi{S}-9\wfal$.

By induction, fix some complexity $j < \plex{S}$ and suppose that we have already designated robust constants $\indnum{\plex{S}},\dotsc,\indnum{j+1}$, fractions $\indfrac{\plex{S}},\dotsc,\indfrac{j+1}$, and thresholds $\indrafi{\plex{S}},\dotsc,\indrafi{j+1}$. Consider any saturated  resolution family $\Omega = \cup_k\Omega^k$ for a tuple $(x_0,\dots,x_n)$ in $\T(\Sigma)$, where $\Omega^0 = (\Omega_1^0,\dots,\Omega^0_n)$ with each $\Omega^0_i$ equal to the set of maximal domains in $\Upsilon(x_{i-1},x_i)$, where $\Omega^1 = \overline{\Omega^0}$, and $\Omega^{k+1} = \widehat{\Omega}^k(V_k)$ for each $k\ge 1$. Let us consider how many domains of complexity $j$ can arise in $\Omega$: By Lemma~\ref{lem:threshholds}, there are constants $C_{\plex{\Sigma}},\dotsc,C_j$ depending only on $\indrafi{\plex{\Sigma}},\dotsc,\indrafi{j+1}$ such that the original family $\Omega^0$ satisfies $\abs{\Omega^0}_{m}\le C_m$ for all $j \le m \le \plex{\Sigma}$. Therefore Lemma~\ref{lem:completion_bound} implies that
\[\abs{\Omega^1}_{j} = \big\vert{\overline{\Omega^0}\big\vert}_j \le P_j,\]
where $P_j$ is a constant depending only on $C_{\plex{\Sigma}},\dotsc,C_{j},\indrafi{\plex{\Sigma}},\dotsc,\indrafi{j+1}$, and thus ultimately depending only on $\indrafi{\plex{S}},\dotsc,\indrafi{j+1}$. Now, we have seen above that domains of complexity $j$ are only added when we augment along a domain of complexity strictly larger than $j$. For each $ m > j$ and $k$, our induction hypothesis ensures that $\abs{\Omega^k}_m \le \abs{\Omega}_m \le \indnum{m}$. Thus there are at most $n\indnum{m}$ domains of complexity $m$ that are candidates for augmentation, and each such domain can occur on the list $V_1,V_2,\dotsc$ at most $1/\indfrac{m}$ times. 

Therefore, there are at most $\sum_{m=j+1}^{\plex{S}}n\indnum{m}/\indfrac{m}$ indices $k$ such that
\[\abs{\Omega^{k+1}}_{j} > \abs{\Omega^{k}}_{j}.\]
For each such index $k$, our hypothesis $\abs{\Omega^k}_m \le \abs{\Omega}_m \le \indnum{m}$ and equation \eqref{eqn:num_new_domains_domains} together imply the difference $\abs{\Omega^{k+1}}_{j} - \abs{\Omega^{k}}_{j}$is bounded by a number $Q_j$ depending only the thresholds $\indrafi{\plex{S}},\dotsc,\indrafi{j+1}$ and constants $\indnum{\plex{S}},\dotsc, \indnum{j+1}$. This proves $\abs{\Omega}_j \le \indnum{j}$, where
\[\indnum{j}\colonequals
P_j + Q_j\left(n\frac{\indnum{\plex{S}}}{\indfrac{\plex{S}}} +\dotsb+ n \frac{\indnum{j+1}}{\indfrac{j+1}}\right)\]
is a constant depending only on $n$, our previously determined constants $\indnum{m}$, $\indfrac{m}$, and $\indrafi{m}$ for $j < m \le \plex{S}$. Now that we know $\abs{\Omega}_{j} \le \indnum{j}$, we set
\begin{equation}
\label{eqn:choice_of_threholds}
\indfrac{j} \colonequals \frac{1}{4(j+2)^3\wfal\indnum{j}}\qquad\text{and}\qquad\indrafi{j}\colonequals \max\left\{30\wfal\left(2+ \frac{1}{\indfrac{j}}\right),\indrafi{j+1}\right\}.
\end{equation}
By Lemma~\ref{lem:saturated_encroachment_bound}, this d choice of $\indrafi{j}$ ensures that every domain $V\in \Omega$ with $\plex{V}=j$ satisfies $\enc_\Omega(V) \le 9\wfal(1 + \frac{1}{\indfrac{j}}) < \frac{1}{3}\indrafi{V}-9\wfal$. Thus the constants $\indnum{\plex{S}},\dots,\indnum{j}$, fractions $\indfrac{\plex{S}},\dots, \indfrac{j}$ and thresholds $\indrafi{\plex{S}},\dots,\indrafi{j}$ form a robust choice.

Proceeding recursively in this manner, we obtain a complete list of robust constants $\indnum{m}$, $\indfrac{m}$, and $\indrafi{m}$ for $\plex{S}\ge m \ge -1$. Since these constants are robust, any saturated family $\Omega$ built using these constants is necessarily limited. Furthermore, since each $V\in \Omega$ satisfies $\enc_\Omega(V) < \frac{1}{3}\indrafi{V} - 9\wfal$, the fact that $\Omega$ is saturated automatically implies that $\abs{\bad_V^\Omega} < \indfrac{V}d_{\T(\Sigma)}(x_0,x_n)$. Hence $\Omega$ is also admissible.
\end{proof}

\section{Complexity Length}
\label{sec:complexity_length}
We are finally ready to define the key quantity coming from this lengthy construction, namely the complexity length of a tuple.

\begin{definition}
\label{def:complexity-savings_for_tuple_of_points}
Let $\Sigma\esub S$ be a domain. The \define{complexity length} and \define{savings} of a strongly $\wfal$--aligned tuple $(x_0,\dots,x_n)$ in $\T(\Sigma)$ are defined to be 
\[\cl(x_0,\dots,x_n) = \inf_\Omega \cl(\Omega) \qquad\text{and}\qquad \mathfrak{S}(x_0,\dots,x_n) = \inf_{\Omega} \mathfrak{S}(\Omega),\]
where the infima are taken over all WISCL witness families $\Omega$ for the tuple, and where $\cl(\Omega)$ and $\mathfrak{S}(\Omega)$ are as given in Definitions~\ref{def:complexity_for_tuple}--\ref{def:savings_in_complexity}.
Note that the infima are achieved, since the set of such  $\Omega$ is nonempty (e.g.\ by Proposition~\ref{prop:saturated_constants}) and finite by virtue of the sets $\Upsilon(x_{i-1},x_i)$ being finite.
\end{definition}

We now have the following consequences of the construction:

\begin{theorem}
\label{thm:main_complexity_length_bound}
Let $\Sigma\esub S$ be a domain and $(x_0,\dots,x_n)$ be a strongly $\wfal$--aligned tuple in $\T(\Sigma)$. Then for any indices $0 = k_0 < k_1 < \dots < k_m = n$ we have both
\begin{align*}
\sum_{j=1}^m \cl(x_{k_{j-1}},\dots,x_{k_j})  \;&\le\; \cl(x_0,\dots,x_n),\quad\text{and}\\
\cl(x_0,\dots,x_n) + \mathfrak{S}(x_0,\dots,x_n) \;&\ladd_{C,n}\; \left(\ent{\Sigma} + \frac{n}{\wfal}\right ) d_{\T(\Sigma)}(x_0,x_n).
\end{align*}
\end{theorem}
As a special case of the theorem, if $(y,x,z)$ is a strongly $\wfal$--aligned tuple in $\T(S)$,
then $\cl(y,x) + \cl(x,z) \ladd_\wfal (\ent{S} +\frac{2}{\wfal}) d_{\T(S)}(y,z)$.
\begin{proof}
Let $\Omega'$ be a saturated witness family for $(x_0,\dots,x_n)$. 
By Proposition~\ref{prop:saturated_constants}, $\Omega'$ is WISCAL and hence satisfies  $\cl(\Omega') +\mathfrak{S}(\Omega')\ladd_{\wfal,n} (\ent{\Sigma} + \frac{n}{\wfal})\d_{\T(\Sigma)}(x_0,x_n)$ by Theorem~\ref{thm:complexity_bound_for_WISCAL}.
Next let $\Omega$ and $\Omega''$ realize the infima from Definition~\ref{def:complexity-savings_for_tuple_of_points}, so that $\cl(x_0,\dots,x_n) = \cl(\Omega)$ and $\mathfrak{S}(x_0,\dots,x_n) = \mathfrak{S}(\Omega'')$.
Since $\Omega'$ is a candidate for these infima, we trivially have
$\cl(\Omega) + \mathfrak{S}(\Omega'') \le \cl(\Omega') + \mathfrak{S}(\Omega')$.
Combined with the previous observations, this proves the second assertion of theorem.

Next, note that by definition each subfamily $\Omega^{j} = (\Omega_{1+k_{j-1}}, \dots, \Omega_{k_j})$ is a WISCL (but not necessarily admissible) witness family for the strongly $\wfal$--aligned subtuple $(x_{k_{j-1}},\dots, x_{k_j})$. Since complexity length is an infimum, it follows that $\cl(x_{k_{j-1}},\dots, x_{k_j}) \le \cl(\Omega^j)$. On the other hand, the complexity of a tuple witness family is exactly defined so that
\[\sum_{j=1}^m \cl(x_{k_{j-1}},\dots,x_{k_j})  \le\sum_{j=1}^m \cl(\Omega^j) = \sum_{j=1}^m \left(\sum_{i = 1 + k_{j-1}}^{k_j} \cl(\Omega_i)\right) = \sum_{i=1}^n \cl(\Omega_i) = \cl(\Omega).\qedhere\]
\end{proof}

We note that the $\delta$ in the main theorem will come from this theorem. Namely given $\delta$ we will pick $\wfal$ large enough so that $\frac{n}{\wfal}<\delta$. This observation will be repeated in the last section. 

\section{Counting with complexity length}
\label{Sec:countingcomplexity}

\newcommand{\preset}{\mathcal{X}}
\newcommand{\vertset}{\mathcal{V}}
\newcommand{\verset}{\vertset}

In \S\S\ref{sec:witness-families}--\ref{sec:complexity_length} we have gone to great lengths to define a quantity $\cl(x,y)$ that is essentially bounded above by $\ent{S}d_{\T(S)}(x,y)$. We now count points with given complexity length. Recall from Definition~\ref{def:fixed_nets} that for each domain $\Sigma$ of $S$ we have specified a $(\netsep,2\netsep)$ net $\tnet(\Sigma)$ in the Teichm\"uller space $\T(\Sigma)$.
Our goal in this section is:
\begin{theorem}
\label{thm:countcomplexity}
For any parameter $\wfal \ge \rafi$ there exists an integer $k\ge 1$ such that for any domain $\Sigma\esub S$, point $x\in \T(\Sigma)$, and distance $r > 0$ we have
\[ \#\Big\{y\in \tnet(\Sigma) \mid \cl(x,y) \le r\Big\} \le k r^k e^{r}.\]
That is, there are at most $k r^k e^{r}$ net points within complexity length $r$ of $x$.
\end{theorem}

This should be compared with Lemma~\ref{lem:net-point-count} (itself a consequence of Theorem~\ref{thm:ABEM-volumes} by \cite{ABEM}), but with the Teichm\"uller distance replaced by complexity length.

\begin{corollary}
\label{cor:count-complexity-tuples}
There exists an integer $k\ge 1$ depending only on $\wfal,n$ such that for any domain $\Sigma\esub S$, point $x\in \T(\Sigma)$, and distance $r > 0$, there are at most $k r^k e^r$ tuples $(x_1,\dots, x_n)$ of net points such that  $\cl(x,x_1,\dots,x_n)\le r$.
\end{corollary}
\begin{proof}
Since $\sum_{i=1}^n \cl(x_{i-1},x_i) \le \cl(x,x_1,\dots,x_n)\le r$ by Theorem~\ref{thm:main_complexity_length_bound}, for each integer partition $r \ge r_1 + \dots +r_n$ we count the number of strongly $\wfal$--aligned tuples $(x_0,\dots,x_n)$ with $x_0 =x$ and $\cl(x_{i-1},x_i)\le r_i$. By Theorem~\ref{thm:countcomplexity}, once $x_{i-1}$ is determined there are at most $k(r_i)^k e^{r_i}$ options for the next net point $x_i$. Thus in total there are at most $k^n(r_1\dots r_n)^k e^{r_1+\dots+r_n} \le k^n r^{kn} e^{r}$ options for each of the at most $r^n$ such partitions of $r$.
\end{proof}

\subsection{Directed graphs}
The proof will require a bit of setup. Given a WISCL witness family $\Omega$ for a pair $(x,y)$, we define a labeled directed graph $\mathcal{G} = \mathcal{G}(\Omega)$ as follows: The vertex set $\vertset = \vertset(\mathcal{G})$ is the set of domains in $\Omega$ with each vertex $Z\in \Omega$ labeled by the ordered pair $(\entstar{Z},\floor{d_{\T(Z)}(\res{x}{\Omega}{Z},\res{y}{\Omega}{Z})}$), where the first entry $\entstar{Z}$ is the weight used in calculating the complexity $\cl(\Omega)$ (Definition~\ref{def:complexity}) and the second entry is the integer part of the Teichm\"uller distance between the resolution points $\res{x}{\Omega}{Z},\res{y}{\Omega}{Z}$.
Vertices $Y,Z\in \vertset$ are joined by directed labeled edge from $Y$ to $Z$ as  follows:
\begin{itemize}
\item if $Y\cut Z$ with $Y \tol Z$ along $[x,y]$, we have an edge $Y\stackrel{P}{\to} Z$ labeled  ``$P$;''
\item if $Y\esub Z$ with $Y\swarrow Z$, there is an edge $Y\stackrel{SW}{\to} Z$ labeled ``$SW$;''
\item if $Y\esup Z$ with $Y\searrow Z$, there is an edge $Y\stackrel{SE}{\to} Z$ labeled ``$SE$;''
\item if $Y$ and $Z$ are disjoint (that is, $Y\disjoint Z$), there is no edge joining $Y$ and $Z$.
\end{itemize}

\begin{lemma}
\label{lem:directed_graph_partial_order}
These directed edges give a partial ordering on the vertices of $\mathcal{G}$.
\end{lemma}
\begin{proof}
We need to prove transitivity. First consider a concatenation $W\stackrel{P}{\to} Y \stackrel{P}{\to} Z$. Then we have $W\cut Z$ by Corollary~\ref{cor:time-order_and_severing} with $W \tol Z$ along $[x,y]$ by transitivity of time-order; hence $W\stackrel{P}{\to} Z$ as required. 
If the second edge is labeled $Y\stackrel{SE}{\to} Z$, then we must have $W \cut Z$ by \ref{subord-supercutting} and moreover $W\tol Z$ by Corollary~\ref{cor:time-order_subsurf}.
Finally, if the second edge is labeled $Y\stackrel{SW}{\to} Z$, then $Y\esub Z$ and we cannot have $W\disjoint Z$ or $W\esup Z$. If $W\cut Z$, then necessarily  $W\tol Z$ by Corollary~\ref{cor:time-order_subsurf} so that $W\stackrel{P}{\to}Z$ as needed. If instead $W\esub Z$, then we must have $W\swarrow Z$ by \ref{subord-subcutting}, so that $W\stackrel{SW}{\to} Z$.

The cases $W\stackrel{SW}\to Y \stackrel P \to Z$ and $W\stackrel{SE}{\to}Y\stackrel{P}{\to} Z$ are handled by symmetric arguments as above. For the case $W\stackrel{SW}{\to}Y\stackrel{SE}{\to}Z$, axiom \ref{subord-subcutting} exactly gives $W\stackrel{P}{\to}{Z}$. Similarly $W\stackrel{SW}{\to}Y\stackrel{SW}{\to}Z$, and the alternative with two SE edges, follows from \ref{subord-triplenest}. 

For the last remaining case $W\stackrel{SE}{\to}Y\stackrel{SW}{\to}Z$, the domains $W$ and $Z$ cannot be disjoint because they both contain $Y$. If $W\cut Z$ then we necessarily have $W\tol Z$ and thus $W\stackrel{P}{\to} Z$ by \ref{subord-supercutting}. Similarly if $W\esub Z$, then \ref{subord-triplenest} ensures $W\stackrel{SW}{\to} Z$ and, symmetrically, $W\esup Z$ leads to $W \stackrel{SE}{\to} Z$.
\end{proof}

The labeled directed graph $\mathcal{G}$ is a combinatorial object that neither remembers the points $x,y$ nor the family $\Omega$ giving rise to it. To emphasize this combinatorial structure, we will use lowercase letters $v\in\verset$ to denote vertices of $\mathcal{G}$ and write $(h^\ast_v,d_v)\in \N^2$ for the label of the vertex.
Our goal is, essentially, to count the number of witness families that give rise to a given labeled graph $\mathcal{G}$.

\subsection{Realizing initial subsets}
To this end, we say a subset $\preset\subset \vertset$ \define{respects the partial order} if there is no directed edge from the complement $\vertset\setminus\preset$ to $\preset$, that is, if $Y\in \mathcal{X}$ implies $W\in \mathcal{X}$ for any directed edge $W\to Y$. For example, the subsets $\emptyset$ and $\vertset$ both respect the partial order.

Given our combinatorial graph $\mathcal{G}$, a subset $\preset$ respecting the partial order, and an initial point $x\in \T(\Sigma)$, we say a pair of families $\Omega_1,\Omega_2$ are \define{equivalent over $\preset$}
if 
\begin{itemize}
\item each $\Omega_i$ is a WISCL witness family for a segment $[x,y_i]$ starting at $x$, 
\item the graph $\mathcal{G}(\Omega_i)$ associated to each family $\Omega_i$ is isomorphic to $\mathcal{G}$ via an isomorphism $f_i \colon \mathcal{G}\to \mathcal{G}(\Omega_i)$ of labeled directed graphs such that
\item for each vertex $v\in \preset$, the corresponding domains $f_i(v)\in \mathcal{V}(\mathcal{G}(\Omega_i))= \Omega_i$ are equal, call it $f_1(v) = Z = f_2(v)$, and have the same resolution points, meaning that $\res{x}{\Omega_1}{Z} = \res{x}{\Omega_2}{Z}$  and $\res{y_1}{\Omega_1}{Z} = \res{y_2}{\Omega_2}{Z}$ in $\T(Z)$.
\end{itemize}

\begin{definition}
\label{def:realization_of_set_respecting_partial_order}
A \define{realization} of $\preset$ relative to an initial point $x\in \T(\Sigma)$ is an equivalence class $\mathcal{R}$ of families over $\preset$. 
We additionally say a segment $[x,y]$ \define{realizes} $\mathcal{R}$ if the equivalence class contains a witness family $\Omega$ for $[x,y]$.
A realization of $\emptyset$ thus consists of no data, whereas a realization of $\verset$ roughly consists of a witness family $\Omega$ giving rise to $\mathcal{G}$.
\end{definition}

In general, a realization $\mathcal{R}$ of $\preset$ determines for each vertex $v\in \preset$ a domain $Z_v\esub \Sigma$ and a pair of points $\hat{x}_v, \hat{y}_v\in \T(Z_v)$.
These domains moreover satisfy the combinatorial conditions that if $v\stackrel{P}{\to}w$ then $Z_v\cut Z_w$, if $v\stackrel{SW}{\to}w$ then $Z_v\esub Z_w$, if $v\stackrel{SE}{\to}w$ then $Z_v\esup Z_w$, and that $Z_v,Z_w$ are disjoint if there is no edge joining $v$ and $w$. Let us write $\Omega(\mathcal{R}) = \{Z_v \mid v\in \preset\}$ for this set of domains. 

Mimicking the notation from \S\ref{sec:supremums_witness}, let us say that $v\in \preset$ \emph{minimally contains} a domain $U$, denoted $U\clld{\preset} v$, if $Z_v$ is a topologically minimal element of the set $\{Z_w \mid w\in \preset\text{ and }  U\esub Z_w\}$. 

\begin{lemma}
\label{lem:unique_minimal_contains_in_Upsilon}
If distinct vertices $v,w\in \preset$ minimally contain $U\esub \Sigma$, then for every segment $[x,y]$ realizing $\mathcal{R}$ we have $U\notin \Upsilon(x,y)$ and, in particular $d_U(x,y)\le\irafi$.
\end{lemma}
\begin{proof}
Let $\Omega$ be any witness family in $\mathcal{R}$ for the segment $[x,y]$. 
Then $Z_v,Z_w \in \Omega(\mathcal{R})$ both minimally contain $U$. 
By contradiction, let us suppose $U\in \Upsilon(x,y)$. The completeness of $\Omega$ then provides an $\Omega$--supremum $U' = \colsup{U}{\Omega}$ which, by Lemma~\ref{lem:characterize_supremum},  satisfies $U'\esub Z_v$ and $U'\esub Z_w$. 
Observe that the domains $Z_v,Z_w$ cannot be nested, as they both minimally contain $U$. Thus $Z_v\cut Z_w$ and we may suppose, without loss of generality, that they are time-ordered $Z_v \tol Z_w$ along $[x,y]$. The supremum $U'$ is an element of $\Omega$ nested inside $Z_v$ and $Z_w$, hence it must be subordered with respect to them. In fact the subordering must be $Z_v \searrow U' \swarrow Z_w$, since the alternatives $U' \swarrow Z_v \tol Z_w$ and $Z_v\tol Z_w \searrow U'$ are precluded by \ref{subord-subcutting}. This means that the directed graph $\mathcal{G}(\Omega)$ has edges $Z_v \stackrel{SE}{\to} U' \stackrel{SW}{\to} Z_w$. Since $Z_w$ lies in $\Omega(\mathcal{R})$ and $\preset$ respects the partial order, it must be the case that $U'\in \Omega(\mathcal{R})$ as well. But this contradicts the assumption that $v$ and $w$ both minimally containing $U$ in $\preset$.
\end{proof}

\begin{corollary}
\label{lem:minimal_containers_bounded_diam_projs}
For any domain $U\esub \Sigma$, the set $\{\pi_U(\hat{y}_v) \mid U \clld{\preset} v\}$ has uniformly bounded diameter in $\cc(U)$.
\end{corollary}
\begin{proof}
It suffices to bound the diameter $d_U(\hat{y}_v,\hat{y}_w)$ for any distinct pair of vertices $v,w\in \preset$ satisfying both $U\clld{\preset} v$ and $U\clld{\preset} w$.
Let $\Omega$ be any witness family in the equivalence class $\mathcal{R}$, say for a segment $[x,y]$. 
Then by definition $Z_v,Z_w \in \Omega$ with $\hat{y}_v = \res{y}{\Omega}{Z_v}$ and $\hat{y}_w = \res{y}{\Omega}{Z_w}$. Since $U\esub Z_v$, the construction of resolution points (Definitions~\ref{def:projection_tuple} \& \ref{def:resolution})  implies that $\pi_U(\hat{y}_v) = \pi_U(\res{y}{\Omega}{Z_v})$ lies within a uniformly bounded distance of the set $\{\pi_U(x),\pi_U(y)\}$. The same holds for $\pi_U(\hat{y}_w)$. 
Since Lemma~\ref{lem:unique_minimal_contains_in_Upsilon} ensures that $d_U(x,y)\le \irafi$, the bound on $d_U(\hat{y}_v,\hat{y}_w)$ is therefore immediate.
\end{proof}

\subsection{Realization tuples}
We will show that the number of realizations of $\preset$ is controlled by the labels on the vertices of $\preset$. The first step is to show  that a realization relative to $x\in \T(\Sigma)$ determines a companion point $p= p_{\mathcal{R}}\in \T(\Sigma)$. This point will be built via consistency from a tuple $(\tilde{p}_U) \in \Pi_{U\esub \Sigma}\cc(U)$ defined using only the data of the domains $Z_v$ and the points $x\in \T(\Sigma)$ and $\hat{y}_v\in \T(Z_v)$ for $v\in \preset$:

\begin{definition}[Tuples from realizations]
\label{def:tuples_from_realizations}
Let $\mathcal{R}$ be a realization of $\preset$ relative to $x\in \T(\Sigma)$, and 
define a tuple $(\tilde{p}_U)\in \Pi_{U\esub \Sigma} \cc(U)$ as follows:
Given  $U\esub \Sigma$, if no vertices $v\in \preset$ minimally contain $U$, then we set $\tilde{p}_U = \pi_U(x)$, and otherwise we choose some $v\in \preset$ satisfying $U\clld{\preset} v$ and set $\tilde{p}_U = \pi_U(\hat{y}_v)$. Corollary~\ref{lem:minimal_containers_bounded_diam_projs} ensures this is coarsely well defined, independent of the chosen vertex $v$. 
\end{definition}
We observe the following:

\begin{lemma}
\label{lem:describing_consistency_realization_tuple}
If a segment $[x,y]$ realizes $\mathcal{R}$, then for each domain $U\esub \Sigma$ we have:
\begin{itemize}
\item If $U\notin \Upsilon(x,y)$ then $\diam_{\cc(U)}(\tilde{p}_U \cup \pi_U(x) \cup \pi_U(y))\ladd_\wfal 0$.
\item If $U\in \Upsilon(x,y)$, then $d_U(\tilde{p}_U,y)\ladd_\wfal 0$ provided its $\Omega$--supremum $U' = \colsup{U}{\Omega}$ satisfies $U'\in \Omega(\mathcal{R})$, and $d_U(\tilde{p}_U,x)\ladd_\wfal 0$ provided $U'\in \Omega(\mathcal{R})$. 
\end{itemize}
\end{lemma}
\begin{proof}
Let $\Omega$ be any witness family in $\mathcal{R}$ for $[x,y]$.
Notice, as above, that for any  $v\in \preset$, the projection of $\hat{y}_v= \res{y}{\Omega}{Z_v}$ to the curve complex of any subsurface $Y\esub Z_v$ is by construction coarsely either $\pi_{Y}(x)$ or $\pi_{Y}(y)$. In particular, by considering any $v\in \preset$ minimally containing $U$, we see that $\tilde{p}_U = \pi_U(\hat{y}_v)$ is coarsely either $\pi_U(x)$ or $\pi_U(y)$. Therefore, if $U\notin \Upsilon(x,y)$ then $d_U(x,y)\le \indrafi{U}$ and the first bullet follows.

Now suppose $U\in \Upsilon(x,y)$ and let $U'=\colsup{U}{\Omega}$. 
If $U'\in \Omega(\mathcal{R})$, then $U\clld{\Omega(\mathcal{R})} U'$ so that by definition $\tilde{p}_U$ is coarsely given by $\pi_U(\res{y}{\Omega}{U'})$. However, since $U' = \colsup{U}{\Omega}$, the construction in Definitions~\ref{def:projection_tuple}--\ref{def:resolution} implies the projection of $\res{y}{\Omega}{U'}$ to $\cc(U)$ coarsely agrees with $\pi_U(y)$. Hence $\tilde{p}_U= \pi_U(\res{y}{\Omega}{U'})$ is coarsely $\pi_U(y)$ as claimed.

Suppose, on the other hand, that $U'\notin \Omega(\mathcal{R})$. If there is no vertex of $\preset$ that minimally contains $U$, then $\tilde{p}_U = \pi_U(x)$ by definition. However, if $U\clld{\preset}v$ for some $v\in \preset$, then necessarily $U'\esub Z_v$ by Lemma~\ref{lem:characterize_supremum} and in fact we must have $U'\esubn Z_v$ since by assumption $Z_v\in \Omega(\mathcal{R})$ but $U'\notin \Omega(\mathcal{R})$. 
The domains $U',Z_v$ are necessarily subordered in $\Omega$, and the option $U'\swarrow Z_v$ is ruled out by the fact that $\preset\cong\Omega(\mathcal{R})$ respects the partial order. Hence we have $Z_v \searrow U'$ so that by Definitions~\ref{def:projection_tuple}--\ref{def:resolution} the projection of the resolution point $\res{y}{\Omega}{Z_v}$ to $U$ is coarsely $\pi_U(x)$. But since $U\clld{\preset} v$, this projection $\pi_U(\res{y}{\Omega}{Z_v})$ is by construction the $U$--coordinate $\tilde{p}_U$ of our tuple. Thus $\tilde{p}_U = \pi_U(\res{y}{\Omega}{Z_v})$ is coarsely $\pi_U(x)$, as claimed.
\end{proof}

\begin{lemma}
The tuple $(\tilde{p}_U) \in \Pi_{U\esub \Sigma}\cc(U)$ is $k$--consistent, for some $k \ladd_\wfal 0$.
\end{lemma}
\begin{proof}
Let $\Omega$ be a witness family in the equivalence class $\mathcal{R}$, say for a segment $[x,y]$, so that we identify $\preset$ with the subset $\Omega(\mathcal{R}) \subset \Omega$. Fix two domains $U,W\esub \Sigma$ that either cut $U\cut W$ or are nested $U\esubn W$ or $W\esubn U$.
If either $U\notin \Upsilon(x,y)$ or $W\notin \Upsilon(x,y)$, then Lemma~\ref{lem:describing_consistency_realization_tuple} implies  $(\tilde{p}_U,\tilde{p}_W)$ is within bounded distance from either $(\pi_U(x),\pi_W(x))$ or $(\pi_U(y),\pi_W(y))$ and is hence consistent by Theorem~\ref{Thm:Consistency}.

We may therefore suppose $U,W\in \Upsilon(x,y)$. Let  $U' = \colsup{U}{\Omega}\in \Omega$ and $W' = \colsup{W}{\Omega}\in \Omega$ be the $\Omega$--suprema guaranteed by completeness. If neither $U'$ nor $W'$ is in $\Omega(\mathcal{R})$, Lemma~\ref{lem:describing_consistency_realization_tuple} implies we coarsely have $\tilde{p}_U = \pi_U(x)$ and $\tilde{p}_W = \pi_W(x)$. If, on the other hand, $U',W'$ both lie in $\Omega(\mathcal{R})$, then we coarsely have $\tilde{p}_U = \pi_U(y)$ and $\tilde{p}_W = \pi_W(y)$. In either case, the pair $(\tilde{p}_U,\tilde{p}_W)$ satisfies the consistency condition by Theorem~\ref{Thm:Consistency}. It therefore suffices to suppose exactly one of $U'$ or $W'$ lies in $\Omega(\mathcal{R})$ so that, without loss of generality, we suppose $U'\in \Omega(\mathcal{R})$ and $W'\notin\Omega(\mathcal{R})$. 
In particular,  $U'\ne W'$ and, by Lemma~\ref{lem:describing_consistency_realization_tuple}, $d_U(\tilde{p}_U, y)\ladd_\wfal 0$ and $d_W(\tilde{p}_W,x)\ladd_\wfal 0$.

First suppose $U'\esubn W'$. The domains $U',W'$ are then subordered in $\Omega$, and the fact that $\preset$ respects the partial order implies the subordering must be $U'\swarrow W'$.
Note that in this case we must have $U\esubn W$ or $U\cut W$, since the containment $W\esub U$ would imply $W'\esub U'$ by Lemma~\ref{lem:characterize_supremum}.
Also, the domains $U'$ and $W$ cannot be disjoint, as $\partial U$ projects to both of them, and nor can the be nested $W\esub U'$, as that would again imply $W'\esub U'$ by Lemma~\ref{lem:characterize_supremum}, contrary to our assumption. Thus either $U'\cut W$ or $U'\esubn W$.
  If $U'$ cuts $W$, then \ref{subord-contribute} with $U'\swarrow W' = \colsup{W}{\Omega}$ implies they must be time-ordered $U'\tol W$ along $[x,y]$. Since $d_W(\partial U, \partial U') \le 2$, we thus obtain the desired bound
\[d_{W}(\tilde{p}_W,\partial U) \ladd_\wfal d_W(x,\partial U') \le\rafi\]
by Lemma~\ref{lem:characterize_timeorder}. It instead $U'\esubn W$, there are two cases: Firstly, if $W = W'$, then the fact that $\Omega$ is wide with $U'\swarrow W'$ gives
\[d_W(\tilde{p}_W,\partial U) \ladd_\irafi d_W(x, \partial U') = d_{W'}(x,\partial U') \le \tfrac{1}{3}\indrafi{W} \le \irafi,\]
which is the desired condition for consistency. Secondly, if $W\ne W'$, then evidently $W\in \Upsilon(x,y)$ but $W\notin \Omega$. Since we have $U'\in \Omega$ with $U'\esub W$, \ref{projfam-maximal} provides some $Z\in \Omega$ with $U'\esub Z$ and $Z\cut W$. We claim that $Z$ and $W$ must be time-ordered $Z\tol W$ along $[x,y]$. Indeed, if $Z\esubn W'$ then \ref{subord-triplenest} implies $Z\swarrow W'$ so that \ref{subord-contribute} forces $Z\tol W$; similarly if $Z\cut W'$ then we must have $Z\tol W'$ since the alternative would give $U'\swarrow W'\tol Z$ and contradict \ref{subord-supercutting} (since $U'\cut_{W'}Z$ evidently fails). Therefore we conclude that $d_W(\tilde{p}_W,\partial U) \ladd_\irafi d_W(x, \partial Z) \le \rafi$, as desired.

A completely symmetric argument shows that the assumption $U'\esupn W'$ leads to the subordering $U'\searrow W'$. One finds that either $U\esupn W$ or $U\cut W$, and $U\esupn W'$ or $U\cut W'$, and that in any case $d_U(\tilde{p}_U,\partial W) \ladd_\wfal d_U(y,\partial W')$ is uniformly bounded.

It remains to suppose $U'\cut W'$. In this case the fact that $\preset$ respects the partial order implies $U'\tol W'$ along $[x,y]$. We must also have $U\cut W'$, since $U\esup W'$ would give $U'\esup U\esup W'$ and $U\esub W'$ would give $U'\esub W'$ by Lemma~\ref{lem:characterize_supremum}. Similarly we necessarily have $U\cut W$, since either alternative $U\esub W$ or $U\esup W$ would yield nesting $U'\esub W'$ or $U'\esup W'$ again by Lemma~\ref{lem:characterize_supremum}. Corollary~\ref{cor:time-order_subsurf} therefore implies that $U\tol W'$ and $U\tol W$ along $[x,y]$ which gives the desired bound $d_U(\tilde{p}_U, \partial W) \ladd_\wfal d_U(y,\partial W) \le\rafi$.
\end{proof}

We also need to account for the potentiality of short curves:

\begin{definition}[Short multicurve of a realization]
\label{def:short_multicurve_of_realization}
Let $\mathcal{R}$ be a realization of a subset $\preset\subset \mathcal{G}$ that respects the partial order. 
The \define{short multicurve associated to $\mathcal{R}$} is the multicurve $\alpha_{\mathcal{R}}$ consisting of all curves $\gamma$ such that there is some $v\in \preset$ so that $Z_v$ is an annulus whose core $\partial Z_v$ equals $\gamma$ and is short at $\hat{y}_v\in \T(Z_v)$.
\end{definition}

\begin{lemma}
\label{lem:cores_of_short_annuli_can_be_short}
For any realization $\mathcal{R}$, 
we have that $\alpha_{\mathcal{R}}$ is indeed a multicurve and that $d_U(\tilde{p}_U, \alpha_{\mathcal{R}})\ladd_\wfal 0$ for every domain $U\esub \Sigma$.
\end{lemma}
\begin{proof}
Let $\Omega$ be an element of the equivalence class $\mathcal{R}$, say for a segment $[x,y]$.
Given any component $\gamma$ of $\alpha_{\mathcal{R}}$, we may choose some $v\in \preset$ so that $A = Z_v$ is an annulus with $\gamma = \partial Z_v = \partial A$ short at the point $\hat{y}_v\in \T(A)$. 
By construction of this point $\hat{y}_v = \res{y}{\Omega}{A}$ in Definition~\ref{def:resolution}, we necessarily have $\ell_{\hat{y}_v}(\partial A) = \ell_y(\partial A)$; therefore $\gamma= \partial A$ is short at $y$. In particular, since all components of $\alpha_{\mathcal{R}}$ are short at the single point $y\in \T(\Sigma)$, $\alpha_{\mathcal{R}}$ is indeed a multicurve, as claimed.

For the second claim, it suffices to bound $d_U(\tilde{p}_U, \partial A)$ for all $U\esub \Sigma$.
If $\partial A$ is disjoint from $U$, then $d_U(\tilde{p}_U,\partial A)= \diam_{\cc(U)}(\tilde{p}_U)$ is bounded. So we may suppose $A\cut U$ or $A\esubn U$. The fact that $\partial A$ is in the Bers marking at $y$ implies that $d_Y(\partial A,y)\le \lipconst$ for all domains $Y\esub \Sigma$. In particular, since $d_U(y,\partial A)\le \lipconst$, it suffices to bound $d_U(\tilde{p}_U, y)$.
If $d_U(x,y)\le \indrafi{U}$, then $d_U(\tilde{p}_U,y)\ladd_\wfal 0$ by Lemma~\ref{lem:describing_consistency_realization_tuple}, as needed. Hence we may suppose $d_U(x,y)\ge\indrafi{U}$, so that $U\in \Upsilon(x,y)$, and set $U' = \colsup{U}{\Omega}$. 
We claim that necessarily $U'\in \Omega(\mathcal{R})$ so that $d_U(\tilde{p}_U,y) \ladd_\wfal 0$ by Lemma~\ref{lem:describing_consistency_realization_tuple}. Otherwise, the fact that $\preset$ respects the partial order implies that the domains $A,U'\in \Omega$ must be related by $A\swarrow U'$ or $A\tol U'$.
However, the latter option $A\tol U'$ would give $d_{U'}(x,\partial A) \le \rafi$ and, since $\Omega$ is wide, the former option would give $d_{U'}(x,\partial A)  \le \indrafi{U'}/3$. Since $d_{U'}(\partial A, y)\le \lipconst$, either case implies $d_{U'}(x,y) < \indrafi{U'}$ and contradicts our assumption. 
\end{proof}

We can now use our realization data to reconstruct a point in Teichm\"uller space:

\begin{definition}[Realization point]
\label{def:realization_point}
Let $\preset\subset \mathcal{G}$ be a subset respecting the partial order. To each realization $\mathcal{R}$ of $\preset$ we associate a net point $p_{\mathcal{R}}\in \tnet(\Sigma)$ as follows: Let $\alpha_{\mathcal{R}}$ be the associated short multicurve and $(\tilde{p}_U)\in \Pi_{U\esub \Sigma}\cc(U)$  the associated tuple from Definition~\ref{def:tuples_from_realizations}. By Theorem~\ref{Thm:Consistency} and Lemma~\ref{lem:build_marking}, we can build a marking $\mu$ on $\Sigma$ so that $\alpha_{\mathcal{R}}\subset \base(\mu)$ and so that $d_U(\tilde{p}_U,\mu)\ladd_\wfal 0$ for all $U\esub \Sigma$. Now let $p_{\mathcal{R}}\in \T(\Sigma)$ be a net point that has $\mu$ as a Bers marking and so that for each component $\gamma$ of $\alpha_{\mathcal{R}}$, the length $\ell_{p_{\mathcal{R}}}(\gamma)$ coarsely agrees with $\ell_{\hat{y}_v}(\gamma)$, where $v\in \preset$ is the vertex so that $\gamma = \partial Z_v$, and so that all other components of $\base(\mu)$ coarsely have length $\thin$.
\end{definition}

\begin{lemma}
\label{lem:realization_point_near_original_point}
Let $\mathcal{R}$ be any realization of the partially ordered set $\preset = \vertset$  consisting of all vertices of the directed graph $\mathcal{G}$. Then
for any segment $[x,y]$ realizing $\mathcal{R}$ we have $d_{\T(\Sigma)}(p_{\mathcal{R}},y)\ladd_\wfal 0$.
\end{lemma}
\begin{proof}
Consider any domain $U\esub \Sigma$. By Definition~\ref{def:realization_point} we have $d_U(\tilde{p}_U,p_{\mathcal{R}})\ladd_\wfal 0$. Thus if $U\notin \Upsilon(x,y)$, then Lemma~\ref{lem:describing_consistency_realization_tuple} implies  $d_U(p_{\mathcal{R}},y)\ladd_\wfal 0$. If instead $U\in \Upsilon(x,y)$, then $U$ has an $\Omega$--supremum $U' = \colsup{U}{\Omega}\in \Omega$ by completeness. Clearly $U'\in \Omega(\mathcal{R})$, since our subset $\vertset$ consists of all vertices of $\mathcal{G}= \mathcal{G}(\Omega)$, and thus again we find $d_U(p_{\mathcal{R}},y)\ladd_\wfal 0$ by Lemma~\ref{lem:describing_consistency_realization_tuple}.

By the Distance formula, or rather Lemma~\ref{lem:thin_with_bdd_projections}, it remains to show that $p_{\mathcal{R}}$ and $y$ have the same short curves with coarsely the same lengths. By construction of $p_{\mathcal{R}}$ in Definition~\ref{def:realization_point} and $\hat{y}_v = \res{y}{\Omega}{Z_v}$ in Definition~\ref{def:resolution}, each short curve $\gamma$ on $p_{\mathcal{R}}$ is the core of an annulus $Z_v\in \Omega(\mathcal{R})$ on which $\ell_{\hat{y}_v}(\gamma)$ coarsely agrees with $\ell_\gamma(y)$. Conversely, for each short curve $\beta$ on $y$, the annulus $A$ with $\partial A = \beta$ satisfies $A\in \Upsilon^\ell(x,y)$. 
The fact that $\Omega$ is insulated implies (Lemma~\ref{lem:insulated_contains_short_curves}) that $\Upsilon^\ell(x,y)\subset \Omega$. Thus $A \in \Omega = \Omega(\mathcal{R})$ and, again by Definitions~\ref{def:realization_point} and \ref{def:resolution}, the length of $\beta$ at $y$ coarsely agrees with the length of $\beta$ on $\res{y}{\Omega}{A}$ and thus with the length of $\beta$ at $p_{\mathcal{R}}$.
\end{proof}

\subsection{Extending realizations}
Finally, we count how many ways there are to extend a realization to an enlarged subset that respects the partial order:

\begin{proposition}
\label{prop:extend_realization_count}
There is a constant $\wfal'$ depending only on $\wfal$ such that the following holds: Let $\preset\subset \vertset$ be a subset that respects the partial order, and let $v\in \vertset\setminus\preset$ be a vertex so that $\preset' = \preset\cup\{v\}$ also respects the partial order. Then for each realization $\mathcal{R}$ of $\preset$ relative to a point $x\in \T(\Sigma)$, there are at most $\wfal' e^{h^\ast_v d_v}$ realizations $\mathcal{R}'$ of $\preset'$ that extend $\mathcal{R}$. 
\end{proposition}
\begin{proof}
The equivalence class $\mathcal{R}$ is naturally partitioned into subsets $\mathcal{R}'$ that are each realizations of the larger set $\preset'$. Picking such a $\mathcal{R}'$ that extends $\mathcal{R}$ amounts to specifying a domain $Z_v \esub \Sigma$ along with a pair of net points $\hat{x}_v,\hat{y}_v\in \tnet(Z_v)$. We will show that there are at most boundedly many options for $Z_v$ and $\hat{x}_v$ and that, once these are specified, at most $e^{h_v^\ast d_v}$ options for $\hat{y}_v$.

Let $(\tilde{p}_U)_{U\esub \Sigma}$ be the tuple associated to $\mathcal{R}$ and $p_{\mathcal{R}}$ the realization point. Also let $\Omega$ be a witness family in $\mathcal{R}$, say for a segment $[x,y]$. Thus we have an identification of $\mathcal{G}(\Omega)$ with $\mathcal{G}$ under which $v\in \vertset$ corresponds to a domain $W\in \Omega$ and the pair $\hat{x}_v,\hat{y}_v$ are given by the resolution points $\res{x}{\Omega}{W},\res{y}{\Omega}{W}$. These three pieces of data, $W$, $\res{x}{\Omega}{W}$, and $\res{y}{\Omega}{W}$, thus specify the realization $\mathcal{R}'\subset \mathcal{R}$ containing $\Omega$.

\begin{claim}
The domain $W$ satisfies $d_U(p_{\mathcal{R}},\partial W)\ladd_\wfal 0$ for every domain $U\esub \Sigma$ and, consequently, the number of such domains $W$ is bounded by Corollary~\ref{cor:bounded_domains_with_bounded_combinatorics}.
\end{claim}
\begin{proof}[Proof of Claim]
Notice that $W\in \Upsilon(x,y)$ so that $W$ has an active interval along $[x,y]$. 
It suffices to suppose $\partial W$ projects to $U$, that is either $W\esubn U$ or $W\cut U$, and to consider $d_U(\tilde{p}_U,\partial W)$. If $U\notin \Upsilon(x,y)$, then we have
\[d_U(\tilde{p}_U,\partial W) \ladd_\wfal d_U(x,\partial W) \le d_U(x,\partial W) + d_U(\partial W,y) \ladd d_U(x,y) \le \indrafi{U}.\]
by Lemma~\ref{lem:describing_consistency_realization_tuple} and Corollary~\ref{cor:bgit_for_teich}. If instead $U\in \Upsilon(x,y)$, we let $U' = \colsup{U}{\Omega}$. Note that the assumptions $W\esubn U$ or $W\cut U$ imply that either $W\esubn U'$ or $W\cut U'$. Therefore $W$ and $U'$ are related in the directed graph $\mathcal{G}(\Omega)$.

First suppose $U'\in \Omega(\mathcal{R})$, so that $d_U(\tilde{p}_U,y)\ladd_\wfal 0$ by Lemma~\ref{lem:describing_consistency_realization_tuple}. Since $\preset$ respects the partial order and $v\notin \preset$ by assumption, 
we either have $U'\tol W$ along $[x,y]$ (if $U'\cut W$) or $U'\searrow W$ in $\Omega$ (if $W\esubn U'$).
In the former case $U'\cut W$ we necessarily have $U\cut W$ as well (since we cannot have $W\esubn U \esub U'\cut W$). Therefore $U'\tol W$ implies $U\tol W$ by Corollary~\ref{cor:time-order_subsurf} so that $d_U(y,\partial W)\le\rafi$ as desired.
In the latter case $U'\searrow W$, we get $d_{U'}(y,\partial W)\le \irafi$ by wideness of $\Omega$. If $U' = U$ this is the desired bound. 
If instead $U\esubn U'$ then evidently $U\notin \Omega$. We claim there is some $W'\esup W$ so that $W'\in \Omega$ and $W'\cut U$. Indeed: if $W\cut U$ we just take $W' = W$, and if $W\esubn U$ such a $W'$ is provided by \ref{projfam-maximal}. Now observe that if $W'\cut U'$, then we necessarily have $U'\tol W'$ along $[x,y]$ (since $W'\tol U'\searrow W$ is ruled out by \ref{subord-subcutting}) and thus $U\tol W'$ by Corollary~\ref{cor:time-order_subsurf}. Alternatively, if $W'\esubn U'$ then necessarily $U'\searrow W'$ by \ref{subord-triplenest} so that again we get $U\tol W'$ by \ref{subord-contribute}. In any case $U\tol W'$ along $[x,y]$ and we obtain the desired bound:
\[d_U(\tilde{p}_U,\partial W) \ladd_\wfal d_U(y,\partial W') \le \rafi.\]

Next suppose $U'\notin \Omega(\mathcal{R})$ so that $d_U(\tilde{p}_U,x)\ladd_\wfal 0$ by Lemma~\ref{lem:describing_consistency_realization_tuple}. Since $W \ne U'$, we see that $U'$ cannot correspond to one of the vertices of $\preset' = \preset\cup\{v\}$ in $\mathcal{G}$. Since $\preset'$ respects the partial order, this means that either $W\tol U'$ along $[x,y]$ (if $W\cut U'$) or $W \swarrow U'$ in $\Omega$ (if $W\esubn U'$). A symmetric argument to the case $U'\in \Omega(\mathcal{R})$ above now shows that either $U' = U$ and $d_U(\partial W,x)\le \irafi$ by wideness, or $W\esub W'\tol U$ for some $W'\in \Omega$ so that $d_U(\tilde{p}_U,\partial W)\ladd_\wfal d_U(x,\partial W') \le \rafi$.
\end{proof}

\begin{claim}
There are only boundedly many options for the net point $\hat{x}_v\in \tnet(W)$.
\end{claim}
\begin{proof}[Proof of Claim]
We show that the net point $\hat{x}_v = \res{x}{\Omega}{W}$ is coarsely determined by the data captured by the original realization $\mathcal{R}$. 
Specifically, we will show that for any domain $U\esub W$, if there is a unique vertex $u\in \preset$ satisfying $U\clld{\preset} u$,  then $d_U(\res{x}{\Omega}{W},\hat{y}_u)\ladd_\wfal 0$, and otherwise  $d_U(\res{x}{\Omega}{W},x)\ladd_\wfal 0$.

Let us first suppose there does not exist a unique vertex $u\in \preset$ that minimally contains $U$. If $U\notin \Upsilon(x,y)$, then by construction in Definitions~\ref{def:projection_tuple}--\ref{def:resolution} we have $d_U(\res{x}{\Omega}{W},x)\ladd_\wfal 0$, which is the desired bound. If instead $U\in \Upsilon(x,y)$, then we consider $U' = \colsup{U}{\Omega}$ and note that $U'\esub W$ by Lemma~\ref{lem:characterize_supremum}. By completeness, $U'$ is the unique domain of $\Omega$ that minimally contains $U$. Hence the vertex $u\in \vertset = \mathcal{G}(\Omega)$ corresponding to $U'$ cannot lie in $\preset$, as that would contradict our assumption that $\preset$ does not have a unique vertex minimally containing $U$. If $u = v$, that means $U' = W$ and hence that $U$ contributes to $W$ in $\Omega$ so that $d_U(\res{x}{\Omega}{W},x)\ladd_\wfal 0$ by construction. The remaining possibility is $u\notin \preset' = \preset\cup\{v\}$. In this case, the fact that $\preset'$ respects the partial order means  $U'\esubn W$ must be subordered $W\searrow U'$ in $\Omega$. Therefore, by construction, we again have $d_U(\res{x}{\Omega}{W},x)\ladd_\wfal 0$ as desired.

Now suppose there does exist a unique vertex $u\in \preset$ satisfying $U\clld{\preset} u$. We must show $d_U(\res{x}{\Omega}{W},\hat{y}_u)\ladd_\wfal 0$. Let $Z = Z_u$ be the domain in $\Omega$ corresponding to $u$, so that $\hat{y}_u = \res{y}{\Omega}{Z}\in \T(Z)$. If $U\notin \Upsilon(x,y)$, then by definition $d_U(\res{y}{\Omega}{Z},y)\ladd_\wfal 0$ and $d_U(\res{x}{\Omega}{W},x)\ladd_\wfal 0$. Hence by the triangle inequality $d_U(\res{x}{\Omega}{W},\hat{y}_u) \ladd_\wfal d_U(x,y)\le \indrafi{U}\ladd_\wfal 0$, as needed. If instead $U\in \Upsilon(x,y)$, we again consider its supremum $U' =\colsup{U}{\Omega}$ and note that $Z\esup U' \esub W$ by Lemma~\ref{lem:characterize_supremum}. If $U' = Z$, that means $U$ contributes to $Z$ in $\Omega$ so that $d_U(\res{y}{\Omega}{Z},y)\ladd_\wfal 0$ by construction. Since $Z\in \Omega(\mathcal{R})$ and $W\notin \Omega(\mathcal{R})$, the fact that $\preset$ respects the partial order implies the nested domains $Z = U'\esubn W$ are subordered $Z \swarrow W$. Therefore  $d_U(\res{x}{\Omega}{W},y)\ladd_\wfal 0$ by construction and thus $d_U(\res{x}{\Omega}{W},\hat{y}_u)\ladd_\wfal 0$ by the triangle inequality. If $U'\esubn Z$, then the fact that $Z$ minimally contains $U$ in $\Omega(\mathcal{R})$ means we must have $U'\notin \Omega(\mathcal{R})$. Since $\preset$ respects the partial order, the nested domains $U'\esubn Z$ must therefore be subordered $Z\searrow U'$ so that by construction $d_U(\res{y}{\Omega}{Z},x)\ladd_\wfal 0$. Now we either have $U' = W$, or else $U'\notin \Omega(\mathcal{R}')$ and therefore $W\searrow U'$ by the fact that $\preset'$ respects the partial order. In either case we have $d_U(\res{x}{\Omega}{W},x) \ladd_\wfal 0$ by construction. Therefore we again obtain the desired bound $d_U(\res{x}{\Omega}{W},\hat{y}_u)\ladd_\wfal 0$ by the triangle inequality.

The above shows that the curve complex projections $\pi_U(\hat{x}_v)$ of $\hat{x}_v = \res{x}{\Omega}{W}$ to domains $U\esub W$ are coarsely determined by the point $x$ and the data $\{(Z_u, \hat{y}_u) \mid u\in \preset\}$ captured by the realization $\mathcal{R}$ of $\preset$. When $W$ is not an annulus, the point $\res{x}{\Omega}{W}$ is thick by construction and so coarsely determined by its curve complex projections. When $W$ is an annulus, then by construction in Definition~\ref{def:resolution} the core $\partial W$ has coarsely the same length at $x$ and $\res{x}{\Omega}{W}$ so that again $\hat{x}_v$ is coarsely determined by the data of $x$ and $\mathcal{R}$. In either case, we conclude there are only boundedly many options for the net point $\hat{x}_v$.
\end{proof}
To conclude the proof of the proposition, the claims show that there are uniformly boundedly many possibilities for the next domain $W= Z_v$ and initial point $\hat{x}_v\in \tnet(W)$. To finish specifying a realization $\mathcal{R}'$ of $\preset'$, it remains to choose a net point $\hat{y}_v\in \tnet(W)$. But according to the label $(h^\ast_v,d_v)$ of the vertex $v\in V$, this net point must satisfy $d_{\T(W)}(\hat{x}_v,\hat{y}_v) \le d_v$. 
Notice that by definition we have $h_v^\ast = \ent{W}$, unless $W$ is an annulus with both $\hat{x}_v = \res{x}{\Omega}{W}$ and $\hat{y}_v = \res{y}{\Omega}{W}$ $\thin$--thick (Definition~\ref{def:complexity}). In any case, Lemma~\ref{lem:net-point-count}, ensures that once $\hat{x}_v$ is specified there are at most $\netcount e^{h_v^\ast d_v}$ such net points $\hat{y}_v$. 
\end{proof}

\subsection{Finishing the count}

With these tools in hand, it is now a simple matter to complete the proof of Theorem~\ref{thm:countcomplexity}

\begin{proof}[Proof of Theorem~\ref{thm:countcomplexity}]
We are given a point $x\in \T(\Sigma)$ and distance $r >0$ and need to count the number of net points $y\in \tnet(\sigma)$ so that $\cl(x,y)\le r$. By definition of complexity length, for each such point $y$ there is a WISCL witness family $\Omega$ for the segment $[x,y]$ with $\cl(\Omega) = \cl(x,y) \le r$. The corresponding directed graph $\mathcal{G}(\Omega)$ has exactly  $\abs{\Omega}$ vertices. Since $\Omega$ is limited, this number is bounded $\abs{\Omega}\le \indnum{-1}+\dots+\indnum{\plex{\Sigma}} = \inum$ in terms of $\wfal$. Thus there are uniformly boundedly many options for the directed graph $\mathcal{G}(\Omega)$. Each edge has only three possible labels, and for each vertex $v$ there are boundedly many options for the label $h^\ast_v$. The remaining vertex labels $d_v$ satisfy $\sum_{v\in \vertset} h^\ast_v d_v = \cl(\Omega)\le r$. Since there are at most $r^{\abs{\Omega}}$ ways to partition the integer $\floor{r}$ as a sum of $\abs{\Omega}$ nonnegative integers, we conclude there is a constant $\wfal''$ depending only on $\wfal$ such that there are at most $\wfal'' r^{\inum}$ possibilities for the labeled directed graph $\mathcal{G}(\Omega)$.

Let us now fix such a labeled directed graph $\mathcal{G}$ and count the number of points $y$ producing a witness family $\Omega$ with $\mathcal{G}(\Omega) = \mathcal{G}$. Using the partial order (Lemma~\ref{lem:directed_graph_partial_order}), we can enumerate the finite vertex set $\vertset = \vertset(\mathcal{G}) = \{v_1,\dots,v_{\abs{\Omega}}\}$ so that each initial list $\preset_i = \{v_1,\dots,v_i\}$ respects the partial order. Let us count the number of possible realizations $\mathcal{R}_i$ of each of these sets. For the emptyset $\preset_0 = \emptyset$, there is exactly one realization $\mathcal{R}_0$, and for each $1\le i \le \abs{\Omega}$, Proposition~\ref{prop:extend_realization_count} implies there are at most $\wfal' e^{h^\ast_{v_i}d_{v_i}}$ realizations $\mathcal{R}_{i}$ of $\preset_i$ extending each realization $\mathcal{R}_{i-1}$ of $\preset_{i-1}$. Thus by induction there are at most $(\wfal')^{i} \prod_{j=1}^i e^{h^\ast_{v_i}d_{v_i}}$
realizations $\mathcal{R}_i$ of $\preset_i$. In particular, we conclude that there are at most
\[(\wfal')^{\abs{\Omega}} \exp\left( h^\ast_{v_1}d_{v_1} + \dots +  h^\ast_{v_{\abs{\Omega}}}d_{v_{\abs{\Omega}}}\right) \le (\wfal')^{\inum}\exp(r)\]
realizations of the full vertex set $\vertset = \preset_{\abs{\Omega}}$. Furthermore, by Lemma~\ref{lem:realization_point_near_original_point}, each such realization $\mathcal{R}$ determines a point $p_{\mathcal{R}}$ that lies within bounded distance of the original point $y$; hence there are uniformly boundedly many net points $y$ that admit a witness family $\Omega$ in the equivalence class $\mathcal{R}$. All together, there are at most $k r^\inum e^r$ potential net points $y$ for which $\cl(x,y)\le r$, where $k,\inum$ depend only on $\wfal$.
\end{proof}

\newcommand{\order}{{m_0}}
\section{Finding good points and branch points}
\label{sec:finding-good-points}
With our construction of the complexity length machinery finally complete, we now turn our attention to the proof of the upper bound in Theorem~\ref{thm:main}.  
Fix once and for all a  thick point $x_0\in \T_\thin(S)$. For each nontrivial element $\phi\in \Mod(S)$, we will construct a collection of points $x_\phi,y_\phi,a_\phi,b_\phi$ that are related to the geodesic $[x_0,\phi(x_0)]$ and will aid us in counting.

We use the Nielsen--Thurston decomposition (Definition~\ref{def:nielsen-thurston-decomp}) to write $S$ as $S_p \sqcup S_f \sqcup S_a$, where $S_a$ is a small neighborhood of the canonical reducing system $\sigma(\phi)$, and where $S_p$ and $S_f$ consist, respectively, of the components of $S\setminus \sigma(\phi)$ on which $\phi$ is pseudo-Anosov or finite-order. In particular $\phi$ preserves each piece of this decomposition and restricts to a pseudo-Anosov homeomorphism of $S_p$ and a finite-order homeomorphism of $S_f$. Let us denote these restrictions as $\phi_p$ and $\phi_f$. 

We will construct our points by working in the product space $\T(S) \cong \productregion{S}{\sigma(\phi)}$, which we view as a product $\T(S_p)\times \T(S_f)\times T(S_a)$ (c.f.~\S\ref{sec:product_regions}). Note that each factor here may itself be a product when the associated subsurface is disconnected. Since the action of $\phi$ on $\T(S)$ preserves this factorization, by abuse of notation we also use $\phi$ to denote the restricted action to each factor. For example, for $u\in \T(S_f)$ we write $\phi(u)$ to denote the point $\phi_f(u)\in \T(S_f)$.

\subsection{Finite-order part}
\label{sec:good-finite-order-point}
We first consider the finite-order part and construct the $\T(S_f)$--coordinates of our desired points $x_\phi,y_\phi$. To begin, let $x'\in \T_\thin(S_f)$ be a thick point whose subsurface projections coarsely agree with those of $x_0$. This may be constructed by considering the consistent tuple $(\pi_V(x_0))_V \in \Pi_{V\esub S_f}\cc(V)$ for subdomains of $S_f$ and using Theorem~\ref{Thm:Consistency} to realize this as a point $x'\in \T_\thin(S_f)$ with $d_V(x_0,x')\ladd 0$ for all $V\esub S_f$.

Let $m_\phi$ be the order of the element $\phi_f\in\Mod(S_f)$ and note that, since $\Mod(S)$ has only finitely many finite-order conjugacy classes \cite[Theorem 7.13]{FM-Primer}, $m_\phi$ is uniformly bounded in terms of $\plex{S_f}\le\plex{S}$. That is, $m_\phi\le \order$ for some constant $\order$ depending only on $S$.
Let $u'$ be any fixed point of the finite-order element $\phi_f$ in $\T(S_f)$; such a point exists by  Nielsen \cite{Nielsen-finite-order-bound}.  Now for any domain $V \esub S_f$, apply Lemmas~\ref{lem:branch_point_for_many_geodesics}--\ref{lem:branch_tuple_is_consistent} to the points $u', \phi(x'),\dots,\phi^{m_\phi}(x')$ to get a $7\rafi$--consistent branch tuple $(\zeta_V)$ and a corresponding thick point $u\in \T_\thin(S_f)$ provided by Theorem~\ref{Thm:Consistency}.
Since $\phi$ fixes $u'$ and the set $\{\phi(x'),\dots, \phi^{m_\phi}(x')\}$, it follows that $\phi(u)$ and $u$ \emph{both} coarsely satisfy the branch condition of Lemma~\ref{lem:branch_point_for_many_geodesics} for the list $u', \phi(x'),\dots,\phi^{m_\phi}(x')$. In particular, for any domain $V\esub S_f$, if indices $j,k$ are chosen to achieve $\min_{j,k}(\phi^j(x')\vert \phi^k(x'))_{u'}^V$, then for both $v = u$ and $v = \phi(u)$ the triples
\[(u', v, \phi^j(x')), \quad (\phi^j(x'), v, \phi^k(x')), \quad\text{and}\quad (\phi^k(x'), v, u')\]
are each $(2\csistfun(7\rafi)+\rafi)$--aligned in $V$. By Lemma~\ref{lem:aligned_proj_to_quasigeod} it follows that if $\triangle$ is a $\cc(V)$ geodesic triangle with vertices in $\pi_V(u')$, $\pi_V(\phi^j(x'))$ and $\pi_V(\phi^k(x'))$, then $\pi_V(u)$ and $\pi_V(\phi(u))$ both lie within $\csistfun(7\rafi)+\rafi$ of each side of $\triangle$. The set of such points has uniformly bounded diameter, hence we conclude $d_V(u,\phi(u))\ladd 0$. Since $u,\phi(u)$ are thick, the distance formula \eqref{eqn:bdd_projections_implies_bdd_Teich} now implies that $d_{\T(S_f)}(u,\phi(u)) \ladd 0$.

\begin{definition}[Finite-order good fixed point]
\label{def:good_fixed_point}
Apply Durham's result \cite[Theorem 1.3]{Durham-ellptic_actions} to the point $u$ to obtain a fixed point $x^f_\phi$ for $\phi_f$ with $d_{\T(S_f)}(u,x^f_\phi)\ladd 0$. Since $u$ is uniformly thick, we may again adjust $\thin$ if necessary so that $x^f_\phi\in \T_\thin(S_f)$.
To achieve consistency of notation, we set $y^f_\phi = x^f_\phi$ as well.
\end{definition}

\subsection{Pseudo-Anosov part}
\label{sec:good-point-pseudo-Anosov}

We next consider the pseudo-Anosov part and build the $\T(S_p)$--coordinates for the points $x_\phi,y_\phi$. To this end, let $\mathcal{F}_\pm$ be the attracting and repelling measured foliations of the pseudo-Anosov $\phi_p$ of the surface $S_p$. Note that by our convention (Remark~\ref{rem:surfaces_are_open}) $S_p$ is an open surface, which is to say we view the boundary curves $\sigma(\phi)$ as punctures on $S_p$.

For each domain $V\esub S_p$, we will use these foliations to define a coarse geodesic $\beta_V$ in the curve complex $\cc(V)$. Firstly, for each component $\Sigma$ of $S_p$, the foliations $\mathcal{F}_\pm$ determine a pair of points in the Gromov boundary of $\cc(\Sigma)$, and we define $\beta_\Sigma$ to be the union of all bi-infinite geodesics $\gamma$ joining $\mathcal{F}_-$ and $\mathcal{F}_+$.

For each proper domain $V\esubn \Sigma$ of a component $\Sigma$ of $S_p$, the foliations instead determine a pair of uniformly bounded diameter projections $\pi_V(\mathcal{F}_\pm)\subset \cc(V)$, which can be constructed as follows: Take any bi-infinite $\cc(\Sigma)$--geodesic $\gamma$ joining the boundary points $\mathcal{F}_\pm$ and restrict to a subray $\gamma_+$ converging to $\mathcal{F}_+$ that lies outside the $2$--neighborhood of the multicurve $\partial V\subset \cc(\Sigma)$. By Theorem~\ref{thm:bounded_geodesic_image}, the projection $\pi_V(\gamma_+)$ has diameter at most $\bgit$. We then define $\pi_V(\mathcal{F}_+)$ to be the union of the sets $\pi_V(\gamma_+)$ for all such subrays of all such geodesics $\gamma$. 
It is not hard to see that $\pi_V(\mathcal{F}_+)$ has diameter at most $3\bgit$.
Define $\pi_V(\mathcal{F}_-)$ similarly. Note that by construction $\phi(\pi_V(\mathcal{F}_{\pm})) = \pi_{\phi(V)}(\mathcal{F}_\pm)$.
Now define $\beta_V\subset \cc(V)$ be the union of all geodesics joining $\pi_V(\mathcal{F}_-)$ to $\pi_V(\mathcal{F}_+)$. By hyperbolicity and the diameter bound on $\pi_V(\mathcal{F}_\pm)$, any pair of geodesics from $\pi_V(\mathcal{F}_-)$ to $\pi_V(\mathcal{F}_+)$ have uniformly bounded Hausdorff distance, and hence $\beta_V$ has uniformly bounded Hausdorff distance from any of these geodesics. Again observe that by construction $\phi(\beta_V) = \beta_{\phi(V)}$.

The set $\beta_V$ may alternatively be viewed as follows: The foliations $\mathcal{F}_\pm$ of the component $\Sigma$ of $S_p$ define a bi-infinite geodesic in the Teichm\"uller space $\T(\Sigma)$; we denote this geodesic by $\mathcal{A}_\Sigma$ and view it as the axis of the pseudo-Anosov $\phi_p$ in $\Sigma$. When $S_p$ is connected, and so equal to $\Sigma$, then $\mathcal{A}_\Sigma$ is the unique geodesic invariant by $\phi_p$, but when $S_p$ is disconnected we get geodesics in the distinct components that are permuted by rule $\phi(\mathcal{A}_\Sigma) = \mathcal{A}_{\phi(\Sigma)}$. The projection $\pi_\Sigma(\mathcal{A}_\Sigma)$ is a uniform quasi-geodesic in $\cc(\Sigma)$ joining the points $\mathcal{F}_\pm$ in the Gromov boundary and hence fellow travels any bi-infinite geodesic $\gamma$ as above. The further projection $\pi_V(\pi_\Sigma(\mathcal{A}_\Sigma))$ coarsely agrees with $\pi_V(\mathcal{A}_\Sigma)$ and gives a uniform quasi-geodesic coarsely joining the sets $\pi_V(\mathcal{F}_-)$ and $\pi_V(\mathcal{F}_+)$. That is: the set $\beta_V$ has uniformly bounded Hausdorff distance from the projected Teichm\"uller geodesic $\pi_V(\mathcal{A}_\Sigma)$.

Now, for each domain $V\esub S_p$, we define
\[\hat{x}_V\in \beta_V\subset \cc(V)\qquad\text{and}\qquad \hat{y}_V\in \beta_V\subset \cc(V)\]
to respectively be the closest-point projections of $\pi_V(x_0)$ and $\pi_V(\phi(x_0))$ to $\beta_V$. Notice that $\phi(\hat{x}_V)$ is the closest point projection of $\phi(\pi_V(x_0)) = \pi_{\phi(V)}(\phi(x_0))$ to $\phi(\beta_V) = \beta_{\phi(V)}$; that is, we have the equivariance
\begin{equation}
\label{eqn:equivariance_pA_tuples}
\phi(\hat{x}_V) = \hat{y}_{\phi(V)}.
\end{equation}
Since $\beta_V$ coarsely agrees with $\pi_V(\mathcal{A}_\Sigma)$, we may choose points $x'_V,y'_V\in \mathcal{A}_\Sigma$ so that $d_V(\hat{x}_V, \pi_V(x'_V))\ladd 0$ and $d_V(\hat{y}_V,\pi_V(y'_V))\ladd 0$.

\newcommand{\contracting}{{\sf D}}
\begin{lemma}
\label{lem:fellowtravel}
There is a constant $\contracting\ge 0$, depending only on $S$, satisfying the following: For any domain $V\esub S_p$ and points $r,s\in \cc(V)$, let $\hat{r},\hat{s}$ be their respective closest points in $\beta_V$. If $d_V(\hat{r},\hat{s})\ge \contracting$, then any $\cc(V)$--geodesic from $r$ to $s$ contains points $r'$ and $s'$, appearing in order, so that $d_V(\hat{r},r'),d_V(\hat{s},s')\le \contracting$. 
\end{lemma}
\begin{proof}
This is just a formulation of the well-known strong contraction property for closest-point projection to quasi-convex sets in hyperbolic spaces. Since the curve complexes $\cc(V)$ are uniformly hyperbolic (Theorem~\ref{thm:hyp_curve_cplx}) and the sets $\beta_V$ are uniformly quasi-convex by construction, we obtain a uniform contraction constant $\contracting$ independent of the domain $V$ or quasi-convex sets $\beta_V$.
\end{proof}

\begin{lemma}
\label{lem:coarse}
The tuples $(\hat{x}_V)$ and $(\hat{y}_V)$ in $\Pi_{V\esub S_p}\cc(V)$ are uniformly consistent.
\end{lemma}

\begin{proof}
We prove it  for $(\hat{x}_V)_V$;  the proof for  $(\hat{y}_V)_V$ is identical. 
Fix two domains $V,W\esub S_p$ that satisfy $V\cut W$ or $V\esub W$. For this to be possible, it must be that $V,W$ both lie in the same component $\Sigma$ of $S_p$, and we consider the associated Teichm\"uller axis $\mathcal{A}_\Sigma$ and points $x'_V,x'_W\in \mathcal{A}_\Sigma$.

First assume that $\mathrm{diam}(\pi_V(\mathcal{A}_\Sigma)) \le \rafi$ (note that this requires $V\ne \Sigma)$. Then $d_V(x'_V, x'_W)\le \rafi$. If $d_W(\partial V, x'_W)\le \consist$ we have satisfied consistency. Otherwise, Theorem~\ref{Thm:Consistency} implies $d_V(\partial W, x'_W)\le \consist$. Hence $d_V(\partial W, x'_V)\le \consist +\rafi$ and we have again satisfied consistency.  We reach the same conclusion of $\mathrm{diam}(\pi_W(\mathcal{A}_\Sigma))\le \rafi$.

It now suffices to assume the projections of $\mathcal{A}_\Sigma$ to $\cc(V)$ and $\cc(W)$ both have diameter more than $\rafi$. Therefore, when $V\esubn \Sigma$, Lemma~\ref{lem:our_active_intervals} implies  $\mathcal{A}_\Sigma$ contains an active interval $\interval_V$. Since points in $\interval_V$ contain $\partial V$ in their Bers markings, we see conclude that $\pi_W(\partial V) \subset \pi_W(\mathcal{A}_\Sigma)$. In particular, $\pi_W(\partial V)$ (which may be empty) lies bounded distance from $\beta_W$. Similarly, $\pi_V(\partial W)$ lies bounded distance from $\beta_W$ when $W\esubn \Sigma$.

Now suppose $V\pitchfork W$. By consistency for the projections of $x_0\in \T(S)$ (Theorem~\ref{Thm:Consistency}) we may assume that $d_V(x_0, \partial W) \le \consist$ (the alternate possibility being handled symmetrically). In the hyperbolic space $\cc(V)$, closest-point projection to quasi-convex subsets is coarsely Lipschitz. Therefore the closest-point projections of $\pi_V(x_0)$ of and $\pi_V(\partial W)$ to $\beta_V$ lie bounded distance from each other. But the former is by definition $\hat{x}_V$ and, by the previous paragraph, the latter is bounded distance from $\pi_V(\partial W)$ itself. Hence $d_V(\hat{x}_V,\partial W)$ is uniformly bounded, as desired.

Finally consider the case $V\esub W$. 
Let us assume $d_W(\partial V, \hat{x}_W)$ is large. In the hyperbolic space $\cc(W$), since $\partial V$ lies within bounded distance of $\beta_W$, the geodesic from $\pi_W(x_0)$ to $\partial V$ passes within some uniform distance of the closest-point projection $\hat{x}_W$ of $\pi_W(x_0)$ to $\beta_W$. Since $d_W(\partial V, \hat{x}_W)$ is large, it follows that the $\cc(W)$--geodesic from $\pi_W(x_0)$ to $\hat{x}_W$ stays outside the $2$--neighborhood of $\partial V$, and consequently $d_W(\pi_V(\pi_W(x_0)), \pi_V(\hat{x}_W))\le \bgit$ by Theorem~\ref{thm:bounded_geodesic_image}. Similarly, since $\hat{x}_W$ and $\pi_W(x'_W)$ coarsely agree and are far from $\partial V$, $d_V(\pi_V(\hat{x}_W), \pi_V(\pi_W(x'_W)) \le \bgit$ as well. Using Lemma~\ref{lem:composing_projections}, this implies  $d_V(x_0, \pi_V(\hat{x}_W))$ and $d_V(\pi_V(\hat{x}_W), x'_W)$ are both uniformly bounded. But the fact $x'_W\in \mathcal{A}_\Sigma$ implies $\pi_V(x'_W)$ lies within bounded distance of $\beta_V$. Therefore, we see that $\pi_V(x_0)$ and $\pi_V(\hat{x}_W)$ are both within bounded distance of $\beta_V$. That is, the distance $d_V(x_0, \hat{x}_V)$ from $\pi_V(x_0)$ to its closest point projection is bounded. By the triangle inequality,  $d_V(\pi_V(\hat{x}_W), \hat{x}_V)$ is bounded as well, as desired.
\end{proof}

\begin{definition}
[Pseudo-Anosov good points]
Let  $x_\phi^p\in \T_\thin(S_p)$ to be the thick point realizing the consistent tuple $(\hat{x}_V)_{V\esub S_p}$. Also define $y_\phi^p = \phi_p(x_\phi^p)\in \T_\thin(S_p)$; since $\phi(\hat{x}_V) = \hat{y}_{\phi(V)}$, we note that $y_\phi^p$ realizes the consistent tuple $(\hat{y}_V)_{V\esub S_p}$.
\end{definition}

\subsection{Annular part}
\label{sec:good-point-annulus}
It remains to define the coordinates of our desired points $x_\phi,y_\phi$ in the Teichm\"uller space $\T(S_a)$ of the annuli about the canonical reducing system $\sigma(\phi)$. But this is easy: $\T(S_a)$ has one $\T(\alpha) \cong \H^2$ factor for each component $\alpha$ of $\sigma(\phi)$, where the horizontal coordinate is the twist parameter and the vertical coordinate is the reciprocal of the length of $\alpha$. We simply choose the lengths to be $\thin$ and the twists to agree with that of $x_0$ and $\phi(x_0)$. That is, we define  $x_\phi^a\in \T(S_a)$ such that for each component $\alpha$ of $\sigma(\phi)$ we have $\ell_{x_\phi^a}(\alpha) = \thin$ and $\pi_\alpha(x_\phi^a) = \pi_\alpha(x_0)$. Then define $y_\phi^a = \phi(x_\phi^a)\in \T(S_a)$ and note that by construction it satisfies $\ell_{y_\phi^a}(\alpha) = \thin$ and $\pi_\alpha(y_\phi^a) = \pi_\alpha(\phi(x_0))$ for each component $\alpha$ of $\sigma(\phi)$.

\begin{definition}[Good points]
\label{def:the_good poitns}
Putting all this together, we now define
\[x_\phi,y_\phi\in \T(S)\cong \T(S_f)\times \T(S_p)\times \T(S_a)\]
to be the points with coordinates $x_\phi = (x_\phi^f, x_\phi^p, x_\phi^a)$ and $y_\phi = (y_\phi^f, y_\phi^p, y_\phi^a)$. Note that $y_\phi = \phi(x_\phi)$ by construction and that these  are both thick points on which the components of the reducing system $\sigma(\phi)$ each have length $\thin$. 
\end{definition}

\subsection{Controlled backtracking}

Having defined our points $x_\phi,y_\phi\in \T(S)$, we are next interested in those domains $V\esub S$ for which the tuple $(x_0, x_\phi, y_\phi, \phi(x_0))$ fails to be aligned. We will measure this in terms of Gromov products in subsurfaces: By definition (\ref{eqn:gromov_product}), for each domain $V\esub S$ we have
\begin{align*}
d_V(x_0, x_\phi) + d_V(x_\phi, \phi(x_0)) &= d_V(x_0, \phi(x_0)) + 2(x_0 \vert \phi(x_0))_{x_\phi}^V, \quad\text{and}\\
d_V(x_0, y_\phi) + d_V(y_\phi, \phi(x_0)) &= d_V(x_0, \phi(x_0)) + 2(x_0 \vert \phi(x_0))_{y_\phi}^V.
\end{align*}
We thus view $V$ as ``backtracking'' for $\phi$ if either $(x_0 \vert \phi(x_0))^V_{x_\phi}$ or $(x_0 \vert \phi(x_0))^V_{y_\phi}$ is large, since in this case $(x_0, x_\phi,y_\phi,\phi(x_0))$ is poorly aligned in $V$, and the $\cc(V)$--geodesics from $x_0$ to $x_\phi$ to $y_\phi$ and then to $\phi(x_0)$ fellow travel for a large distance.

We first observe (Lemma~\ref{lem:backtracking_unambiguous} below) that these two Gromov products always coarsely agree,  which essentially means the amount of backtracking is the same whether viewed from $x_\phi$ or $y_\phi$. For this we need:

\begin{lemma}
\label{lem:strongly_contracting}
There is a constant $\Lambda_0\ge \branchal\ge 2\rafi$ such that $(x_0,x_\phi,y_\phi)$ and $(x_\phi,y_\phi,\phi(x_0))$ are both $\Lambda_0$--aligned. Further, $d_V(x_\phi,y_\phi)\ge \Lambda_0$ for a domain $V\esub S$, then $(x_0\vert \phi(x_0))_w^V \le \Lambda_0$ for both $w = x_\phi$ and $w = y_\phi$.
\end{lemma}
\begin{proof}
If $V\esub S_f$ then by definition $\pi_V(x_\phi) = \pi_V(y_\phi)$ and so $d_V(x_\phi,y_\phi)\le 2\lipconst$. 
Similarly, if $V\cut \sigma(\phi)$, then $d_V(x_\phi,y_\phi)\le 2\lipconst$ since by construction $\sigma(\phi)$ is part of the Bers markings at both points.
Hence, provided  $\Lambda_0\ge 4\lipconst$, the first statement follows from the triangle inequality and the second statement holds vacuously.

If $V\esub S_a$, then $V$ is an annulus about a component of $\sigma(\phi)$, and $d_V(x_0, x_\phi)$ and $d_V(\phi(x_0),y_\phi)$ are both bounded by construction. Using the triangle inequality, we again obtain alignment in $V$ and, by (\ref{eqn:gromov_product}),  uniform bounds on $(x_0\vert \phi(x_0))_{x_\phi}^V$ and $(x_0 \vert \phi(x_0))_{y_\phi}^V$ regardless of the size of $d_V(x_\phi, y_\phi)$.

The only remaining case is that of a domain $V\esub S_p$. By construction, there is a constant $B\ladd 0$ so that $d_V(x_\phi, \hat{x}_V),d_V(y_\phi, \hat{y}_V)\le B$. Now if $d_V(x_\phi,y_\phi)\le \contracting+B$, then we immediately obtain $2\contracting+2B$ alignment of both triples $(x_0,x_\phi,y_\phi)$ and $(x_\phi,y_\phi,\phi(x_0))$ as above. If instead if $d_V(x_\phi,y_\phi)\ge \contracting+2B$, then $d_V(\hat{x}_V,\hat{y}_V)\ge \contracting$ and Lemma~\ref{lem:fellowtravel} implies any $\cc(V)$--geodesic from $\pi_V(x_0)$ to $\pi_V(\phi(x_0))$ contains points $x',y'$ appearing in order so that $d_V(\hat{x}_V, x')$ and $d_V(\hat{y}_V,y')$ are at most $\contracting$. Since then $d_V(x_\phi, x')$ and $d_V(y_\phi, y')$ are at most $\contracting+B$, it now follows from the triangle inequality that both triples are $(4\contracting+4B+\lipconst)$--aligned and, from (\ref{eqn:gromov_product}), that  both Gromov products are at most $\contracting+B+\lipconst$.
\end{proof}

\begin{lemma}
\label{lem:backtracking_unambiguous}
We have $\abs{(x_0 \vert \phi(x_0))_{x_\phi}^V - (x_0\vert \phi(x_0))_{y_\phi}^V} \le \Lambda_0$ for every $V\esub S$.
\end{lemma}
\begin{proof}
If $d_V(x_\phi,y_\phi)\ge \Lambda_0$, then by Lemma~\ref{lem:strongly_contracting} both Gromov products lie between 0 and $\Lambda_0$ and hence differ by  at most $\Lambda_0$. If instead $d_V(x_\phi,y_\phi)\le \Lambda_0$, then it follows immediately from the definition \eqref{lem:gromov_product_fellow_travel} and the triangle inequality that $(x_0 \vert \phi(x_0))_{x_\phi}^V$ and $(x_0\vert \phi(x_0))_{y_\phi}^V$ differ by at most $\Lambda_0$.
\end{proof}

While in general it may be impossible to eliminate backtracking entirely, as in the example described in \S\ref{Sec:example}, our points $x_\phi,y_\phi$ are designed to minimize backtracking in the sense that there cannot be backtracking in the full orbit $\{\phi^i(V) \mid i\in \Z\}$ of any domain $V$. In fact, we have the following control:

\newcommand{\btbound}{{\sf P}} 
\begin{lemma}
\label{lem:no_total_orbit_backtracking}
There exist constants $\Lambda\ge \Lambda_0$ and $\btbound\ge 1$, depending only on $S$, such that every domain $V\esub S$ satisfies $(x_0\vert \phi(x_0))_{x_\phi}^{\phi^i(V)}\le \Lambda$ and $(x_0\vert \phi(x_0))_{y_\phi}^{\phi^i(V)}\le \Lambda$ for some $-\btbound < i\le 0$. Moreover:
\begin{enumerate}
\item\label{lem-item:finite-orbit-no-total-backtracking} If the $\phi$--orbit $\{\phi^i(V)\}_{i\in \Z}$ is finite, then the orbit size divides $\btbound$.
\item\label{lem-item:infinte-orbit-no-total-backtracking} If $\{\phi^i(V)\}_{i\in \Z}$ is infinite, then $\max\left\{(x_0 \vert \phi(x_0))_{x_\phi}^{\phi^{k\btbound}(V)},(x_0\vert \phi(x_0))_{y_\phi}^{\phi^{k\btbound}(V)}\right\}> \Lambda$ holds for at most one $k\in \Z$.
 \end{enumerate}
\end{lemma}
\begin{proof}
By Lemma~\ref{lem:backtracking_unambiguous} it suffices to bound $(x_0\vert \phi(x_0))_{x_\phi}^{\phi^i(V)}$, since then $(x_0\vert \phi(x_0))_{y_\phi}^{\phi^i(V)}$ is bounded by the slightly larger constant $\Lambda+\Lambda_0$.

First consider the case that $V$ has finite $\phi$--orbit, which occurs if and only if $V\esub S_f$ or if $V$ is an annulus about a component of $\sigma(f)$. Since $\sigma(f)$ has at most $\plex{S}$ components, and since the order $m_\phi$ of $\phi_f$ is uniformly bounded by $\order$, we see that in either case there are only boundedly many possibilities for the size of the orbit $\{\phi^i(V)\}_{i\in \Z}$. Taking $\btbound$ to be a multiple of these orbit sizes thus guarantees conclusion (\ref{lem-item:finite-orbit-no-total-backtracking}) and implies that it suffices to bound $(x_0\vert \phi(x_0))_{x_\phi}^{\phi^i(V)}$ for some $i\in \Z$.

Now, if $V\esub S_a$, so that $V$ is an annulus whose core is a component of $\sigma(f)$, then $d_V(x_0, x_\phi)$ (and $d_V(\phi(x_0),y_\phi)$) are uniformly bounded by construction. Since $(y\vert z)_x^V$ is bounded by $d_V(x,y)$ (\ref{eqn:gromov_product}),
we may choose $\Lambda$ so that $(x_0\vert \phi(x_0))_{x_\phi}^{\phi^i(V)}\le \Lambda$ holds for \emph{all} $i\in \Z$ in this case.

Suppose instead that $V\esub S_f$.
Recall the points $x',u',u\in \T(S_f)$ used in the construction of $x_\phi$ in \S\ref{sec:good-finite-order-point}. 
Let $G = \min_{j,k}(\phi^j(x')\vert \phi^k(x'))_{u'}^V$ and choose indices $1 \le j,k \le m_\phi$ realizing this minimum. Also let $G' = \inf_{i\in \Z} (\phi^i(x')\vert \phi^{i+1}(x'))_{u'}^V$. Then $\abs{k-j}\le \order$ applications of (\ref{eqn:gromov_product_triangle_ineq}) implies $G'-\order(5\delta+3\lipconst) \le (\phi^j(x')\vert \phi^k(x'))_{u'}^V =G$; hence there exists $i$ so that $(\phi^i(x')\vert \phi^{i+1}(x'))_{u'}^V \le G + \order\rafi/2$. By its construction in Lemma~\ref{lem:branch_point_for_many_geodesics}, the branch point $\zeta_V$ thus lies within $24(\delta+\lipconst) + \order \rafi/2 \le \order \rafi$ of any geodesic from $\pi_V(\phi^i(x'))$ to $\pi_V(\phi^{i+1}(x'))$. Since $d_V(u,\zeta_V)\le \csistfun(7\rafi)$, it follows that $(\phi^i(x'),u,\phi^{i+1}(x'))$ is $2(\csistfun(7\rafi)+\order\rafi)$--aligned in $V$, which is equivalent to saying $(\phi^i(x')\vert \phi^{i+1}(x'))_u^V \le \csistfun(7\rafi)+m_0\rafi$. Since the $\T(S_f)$--coordinate of $x_\phi$ is fixed, the uniform bounds $d_V(u,x_\phi),d_V(x_0,x')\ladd 0$ now imply
\[(x_0\vert \phi(x_0))_{x_\phi}^{\phi^{-i}(V)} \ladd (x'\vert \phi(x'))_{x_\phi}^{\phi^{-i}(V)} =  (\phi^i(x')\vert \phi^{i+1}(x'))_{x_\phi}^V \ladd (\phi^i(x')\vert \phi^{i+1}(x'))_{u}^V\ladd 0.\] 

It remains to consider the cases $V\esub S_p$ or $V\cut \sigma(\phi)$ in which the $\phi$--orbit of $V$ is infinite. Here is suffices to prove conclusion (\ref{lem-item:infinte-orbit-no-total-backtracking}), as this implies the main claim of the lemma.  Choose $N\ge 1$ so that $\phi^N$ is in Thurston normal form and fixes each component of $\sigma(\phi)$ and $S\setminus \sigma(\phi)$ (see \S\ref{sec:mapping_class_group}), and recall that $N$ is uniformly bounded in terms of $\plex{S}$. 
If possible, let $\Sigma$ be a component of $S_p$ such that either $V\esub \Sigma$ or $V\cut \Sigma$. If no such component exists (as is the case when $S_p = \emptyset$) then evidently the components of $S\setminus \sigma(\phi)$ intersecting $V$ are all part of $S_f$. In that case, there must be a component $\alpha$ of $\sigma(\phi)$ such that $\alpha \cut V$ and so that $\phi^N$ restricts to the identity on the (one or two) components of $S\setminus \sigma(\phi)$ on either side of $\alpha$. Since $\alpha\in \sigma(\phi)$, the Thurston normal form $\phi^N$ necessarily includes a nontrivial Dehn twist $D_\alpha^j$ for some $j\ne 0$ (Lemma~\ref{lem:twist}). We then let $\Sigma$ denote the annulus about $\alpha$.

We work in the curve complex $\cc(\Sigma)$. Recall from the construction in \S\ref{sec:good-point-pseudo-Anosov} that when $\Sigma$ is a component of $S_p$ there is a constant $B\ladd 0$ so that $d_\Sigma(x_\phi,\hat{x}_\Sigma)\le B$, where $\hat{x}_\Sigma$ is the closest-point projection of $\pi_\Sigma(x_0)$ to $\beta_\Sigma$. If instead $\Sigma$ is an annulus about $\alpha$, then $\pi_\Sigma(x_0) = \pi_\Sigma(x_\phi)$ by construction (\S\ref{sec:good-point-annulus}) and, for notational consistency, we write $\hat{x}_\Sigma = \pi_\Sigma(x_\phi)$ and view $\beta_\Sigma$ as the entire space $\cc(\Sigma)$.

It was shown in \cite[Proposition 4.6]{MasurMinsky-hyp} that there is a uniform lower bound on the curve-complex translation length of any pseudo-Anosov. Hence, there exists $c\in (0,1)$ depending only on $S$ so that, provided $\phi^N\vert_\Sigma$ is pseudo-Anosov, we have $d_\Sigma(\gamma,\phi^{kN}(\gamma))\ge \abs{k}c$ for every $k\in \Z$ and every simple closed curve $\gamma$ on $\Sigma$. If instead $\Sigma$ is an annulus and $\phi^N\vert_\Sigma = D_\alpha^j$ is a Dehn twist, we similarly have 
$d_\Sigma(\gamma, D_\alpha^{kj}(\gamma)) = 2 + \abs{kj}$ 
for any curve $\gamma$ intersecting $\alpha$ by \cite[Equation 2.6]{MasurMinsky-heirarchy}. 
Therefore, by choosing $\btbound = JN$ to be a large multiple of $N$, we may arrange that, regardless of whether $\phi^N\vert_\Sigma$ is a pseudo-Anosov or a Dehn twist, we have 
\[d_\Sigma(\phi^{k_1\btbound}(\hat{x}_\Sigma), \phi^{k_2 \btbound}(\hat{x}_\Sigma)) \ge \abs{k_1J-k_2J} c \ge J c > 2\contracting+4\rafi\]
for all $k_1\ne k_2\in \Z$.
Since $\phi^{k\btbound}(\hat{x}_{\phi^{-k\btbound}(\Sigma)}) = \phi^{k\btbound}(\hat{x}_\Sigma)$ (as $\phi^{k\btbound} = \phi^{kJN}$ fixes $\Sigma$) is the closest-point projection of $\phi^{k\btbound}(x_0)$ to $\beta_\Sigma$, Lemma~\ref{lem:fellowtravel} now implies any $\cc(\Sigma)$--geodesic between the projections of $\phi^{k_1\btbound}(x_0)$  and $\phi^{k_2\btbound}(x_0)$ passes, in order, within $\contracting$ of $\phi^{k_1\btbound}(\hat{x}_\Sigma)$ and $\phi^{k_2\btbound}(\hat{x}_\Sigma)$. By Lemma~\ref{lem:qi-geod-projection}, there must be points $q_1,q_2\in [\phi^{k_1\btbound}(x_0), \phi^{k_2\btbound}(x_0)]$ appearing in order so that $d_\Sigma(q_i, \phi^{k_i\btbound}(\hat{x}_\Sigma))\le \contracting+\rafi/2$ for $i=1,2$.
By the triangle inequality, we also note the bounds
\begin{align}
\label{eqn:lower_bound_for_controlling_backtracking}
d_\Sigma(q_i,\phi^{k_i\btbound}(x_\phi))&\le D+B+\rafi/2\quad\text{for $i$=1,2}\quad\text{and}\quad \notag\\ 
d_\Sigma(q_1,q_2)&\ge d_\Sigma(\phi^{k_1\btbound}(\hat{x}_\Sigma),\phi^{k_2\btbound)}(\hat{x}_\Sigma)) - 2\contracting-M \ge 2\rafi.
\end{align}

By means of contradiction, suppose now that $(x_0\vert\phi(x_0))_{x_\phi}^{\phi^{-k_i\btbound}(V)} \ge \contracting+B+2\rafi$ holds both $i=1,2$, where $k_1\ne k_2$. It then follows from (\ref{eqn:gromov_product}) that for $i=1,2$
\[d_V(\phi^{k_i\btbound}(x_0),\phi^{k_i\btbound}(x_\phi)) = d_{\phi^{-k_i\btbound}(V)}(x_0,x_\phi) \ge \contracting +B + 2\rafi\]
and hence,  by the triangle inequality, that $d_V(q_i, \phi^{k_i\btbound}(x_0))\ge \rafi$. Lemma~\ref{lem:our_active_intervals}(\ref{ourinterval-smallproj}) now ensures the Teichm\"uller geodesic $[\phi^{k_1\btbound}(x_0),\phi^{k_2\btbound}(x_0)]$ has an active interval $\interval_V$ that intersects both subintervals $[\phi^{k_1\btbound}(x_0),q_1]$ and $[q_2,\phi^{k_2\btbound}(x_0)]$. In particular, $q_1,q_2\in \interval_V$ and so both points contain $\partial V$ in their Bers markings. This implies 
\[d_\Sigma(q_1,q_2) \le d_\Sigma(q_1,\partial V) + d_\Sigma(\partial V,q_2) \le 2\lipconst\]
and contradicts (\ref{eqn:lower_bound_for_controlling_backtracking}) above. Therefore $(x_0\vert\phi(x_0)_{x_\phi}^{\phi^{k\btbound}(V)} \ge \contracting+B+2\rafi$ cannot hold for two distinct values of $k\in \Z$, which proves (\ref{lem-item:infinte-orbit-no-total-backtracking}).
\end{proof}

\subsection{Good branch points}
\label{sec:good_branch_pointsS}
In order to avoid over counting, we will need to work with tuples that are aligned. Thus we next use $x_\phi$ and $y_\phi$ to introduce two branch points $a_\phi$ and $b_\phi$ for which the tuple $(x_0, a_\phi, b_\phi, \phi(x_0))$ is aligned. The points are constructed to satisfy several technical properties as well.

To begin, apply Lemma~\ref{cor:adjusted barycenter} with $\theta = \Lambda$ to the ordered triple $x_0, x_\phi, \phi(x_0)$ to obtain the $\branchal$--barycenter $x_\phi'$ and the left/right annular-split barycenters $x_\phi^l,x_\phi^r$ so that $d_V(x_0, x_\phi^l)\le \branchal$ and $d_V(x_\phi^r,\phi(x_0))\le \branchal$ when $V$ is an annulus with $(x_0 \vert \phi(x_0))_{x_\phi}^V\ge \Lambda$. Do the same for the ordered triple $x_0, y_\phi, \phi(x_0)$ to obtain the barycenter $y_\phi'$ and left/right annular-split barycenters $y_\phi^l,y_\phi^r$.

\begin{lemma}
\label{lem:tuple_is_aligned}
The tuple $(x_0, x_\phi^l, y_\phi^r, \phi(x_0))$ is $16\Lambda_0$--aligned.
\end{lemma}
\begin{proof}
Recall that $\Lambda_0\ge \branchal$ (Lemma~\ref{lem:strongly_contracting}).
By definition (\S\ref{sec:alignment}), the quadruple is aligned if each ordered subtriple is. The triples $(x_0, x_\phi^l,\phi(x_0))$ and $(x_0, y_\phi^r, \phi(x_0))$ are both $\branchal$--aligned by construction in Lemma~\ref{cor:adjusted barycenter}. We verify that $(x_0, x_\phi^l,y_\phi^r)$ is aligned; the proof for the remaining subtriple $(x_\phi^l, y_\phi^r,\phi(x_0))$ is symmetric.

Fix any domain $V\esub S$. If $d_V(x_0, x_\phi^l)\le \branchal$,  then $(x_0, x_\phi^l,y_\phi^r)$ is automatically $2\branchal$--aligned in $V$ by the triangle inequality. Similarly, if $d_V(y_\phi^r,\phi(x_0))\le\branchal$ then $3\branchal$--alignment in $V$ of $(x_0, x_\phi^l,y_\phi^r)$ can be deduced from $\branchal$--alignment of $(x_0,x_\phi^l,\phi(x_0))$ via triangle inequality.
Hence, by construction in Lemma~\ref{cor:adjusted barycenter}, it suffices to assume $d_V(x_\phi', x_\phi^l), d_V(y_\phi', y_\phi^r)\le \branchal$; accordingly we focus our attention on the actual barycenters $x_\phi'$ and $y_\phi'$.

If $d_V(x_\phi,y_\phi)\le \Lambda_0$, then Lemma~\ref{lem:barycenter_vs_gromov_product} ensures $d_V(x_\phi',y_\phi')\le 4\branchal+\rafi+\Lambda_0$. Therefore $d_V(x_\phi^l,y_\phi^r)\le 8\Lambda_0$ (recall $\rafi\le\branchal$) and so $(x_0, x_\phi^l, y_\phi^r)$ is $16\Lambda_0$--aligned.
Otherwise we have $d_V(x_\phi,y_\phi)\ge \Lambda_0$ and Lemma~\ref{lem:strongly_contracting} bounds $(x_0\vert \phi(x_0))_{x_\phi}^V$ and $(x_0\vert \phi(x_0))_{y_\phi}^V$ by $\Lambda_0$. Lemma~\ref{lem:barycenter_vs_gromov_product} therefore implies that $d_V(x_\phi, x_\phi')$ and $d_V(y_\phi, y_\phi')$ are at most  $\Lambda_0+\branchal\le 2\Lambda_0$. Since $(x_0, x_\phi,y_\phi)$ is $\Lambda_0$--aligned by Lemma~\ref{lem:strongly_contracting}, it now follows from the triangle inequality that $(x_0, x_\phi^l,y_\phi^r)$ is $16\Lambda_0$--aligned.
\end{proof}

\begin{lemma}
\label{lem:separated_annuli_near_ends}
If $V$ is an annulus so that $V\not\esub S_p$, and $d_V(x_\phi^l,y_\phi^r)>8\branchal$, then both points $w\in \{x_\phi^l,y_\phi^r\}$ satisfy $\min\{d_V(x_0,w),d_V(w,\phi(x_0))\} \le 6\branchal\le 6\Lambda_0$.
\end{lemma}
\begin{proof}
We give the proof for $w = x_\phi^l$, the proof for $w = y_\phi^r$ is symmetric.
First suppose $V\esub S_a$ is an annulus about a component of the reducing system $\sigma(\phi)$. In this case $d_V(x_0, x_\phi)$ is at most $\lipconst$ by construction. Hence $(x_0\vert \phi(x_0))_{x_\phi}^V$ is at most $\lipconst < \Lambda$, and so the construction in Lemma~\ref{cor:adjusted barycenter} implies $d_V(x_\phi^l,x_\phi')\le \branchal$. Lemma~\ref{lem:barycenter_vs_gromov_product} also bounds $d_V(x_\phi, x_\phi')$ by $\branchal+\lipconst$.  Hence, by the triangle inequality,  $d_V(x_0, x_\phi^l)$ is at most $2\branchal+2\lipconst \le 3\branchal$, as required. 

Next suppose $V\not \esub S_a$, so that necessarily $V\esub S_f$ or $V\cut \sigma(\phi)$. 
 In either case $d_V(x_\phi,y_\phi)\le 2\lipconst$ and consequently $d_V(x_\phi', y_\phi')\le 6\branchal$ by Lemma~\ref{lem:barycenter_vs_gromov_product}. 
If $(x_0 \vert \phi(x_0))_{x_\phi}^V$ is more than $\Lambda$, then by construction in Lemma~\ref{cor:adjusted barycenter} we have $d_V(x_0, x_\phi^l)\le \branchal$ and are done. So let us assume $(x_0 \vert \phi(x_0))_{x_\phi}^V \le \Lambda$, which gives $d_V(x_\phi^l,x_\phi')\le \branchal$. If $(x_0 \vert \phi(x_0))_{y_\phi}^V \le \Lambda$ as well, then 
then we symmetrically have $d_V(y_\phi^r,y_\phi')\le \branchal$ so that $d_V(x_\phi^l,y_\phi^r)\le 8\branchal$ by the triangle inequality. As this contradicts our hypothesis, it must be the case that $(x_0 \vert \phi(x_0))_{y_\phi}^V > \Lambda$. 

We will now find $z\in \{x_0,\phi(x_0)\}$ so that $d_V(z,x_\phi')\le 5\branchal$; since  we have observed $d_V(x_\phi^l,x_\phi')\le \branchal$, this will give the desired bound.
For this we appeal to the fact that $\cc(V)$ is a uniform quasi-line. Let $T$ denote the set $\{x_\phi,x_0,\phi(x_0)\}$. By choosing representative vertices in the diameter-at-most $\lipconst$ subsets $\pi_V(p)$ for $p\in T$, \cite[Equation (2.5)]{MasurMinsky-heirarchy} implies there are integers $\overline{x_\phi}, \overline{x_0},\overline{\phi(x_0)}\in \Z$ such that for all $p,q\in T$ we have
\begin{equation}
\label{eqn:annulus_curve_cplx_quasi-line}
\abs{d_V(p,q) - d_\Z(\overline{p},\overline{q}) } \le 2+2\lipconst \le 4\lipconst.
\end{equation}
Now observe that for any $p,q,r\in T$, the integer Gromov product
\[(\overline{p} \vert \overline{r})_{\overline{q}} = \frac{1}{2}\left(d_\Z(\overline{p},\overline{q}) + d_\Z(\overline{r},\overline{q}) - d_\Z(\overline{p},\overline{r})\right) = \frac{1}{2}\left( \abs{\overline{p} - \overline{q}} + \abs{\overline{r}-\overline{q}} - \abs{\overline{p}-\overline{r}}\right)\]
equals zero iff $\overline{q}$ lies between $\overline{p}$ and $\overline{r}$. Furthermore, by (\ref{eqn:annulus_curve_cplx_quasi-line}) we note that
\[\abs{(p\vert r)_q - (\overline{p}\vert \overline{r})_{\overline{q}}} \le 12\lipconst.\]
Since $d_V(y_\phi,x_\phi) \le 2\lipconst$ and we are assuming $(x_0\vert\phi(x_0))_{y_\phi}^V>\Lambda$, it follows that
\[ \left(\overline{x_0}\big\vert\overline{\phi(x_0)}\right)_{\overline{x_\phi}} \ge (x_0\vert \phi(x_0))^V_{x_\phi} -12\lipconst \ge (x_0\vert \phi(x_0))^V_{y_\phi} - 14\lipconst \ge \Lambda - 14\lipconst > 0.\]
Therefore the integer $\overline{x_\phi}$ does not between the integers $\overline{x_0}$ and $\overline{\phi(x_0)}$. Thus there exists $z\in \{x_0, \phi(x_0)\}$ so that  $\overline{z}$ lies in the middle of the other two and is thus a $0$--barycenter for the triple $\overline{x_\phi},\overline{x_0},\overline{\phi(x_0)}$ in $\Z$. It now follows from (\ref{eqn:annulus_curve_cplx_quasi-line}) that $z$ is a $12\lipconst$--barycenter in $V$ for the triple $x_\phi,x_0,\phi(x_0)$. Since $12\lipconst\le \branchal$ and $x_\phi'$ is also $\branchal$--barycenter, Lemma~\ref{lem:barycenter_vs_gromov_product} now implies $d_V(x_\phi',z) \le 5\branchal$ and we are done.
\end{proof}

We now use Lemma~\ref{lem:building_strong_alignment} to promote $(x_0,x_\phi^l,y_\phi^r,\phi(x_0))$ to a strongly aligned tuple. 
The key features of this construction are summarized in the proposition below.
To streamline notation, we define  
\begin{equation}
\label{eqn:backtracking_constant}
\bt^\phi_V \colonequals d_V(y_\phi, b_\phi) \qquad\text{for each domain }V\esub S \text{
and element } \phi \in \Mod(S).
\end{equation}

\newcommand{\good}{\Theta}

\begin{proposition}[Good  point properties]
\label{prop:good-point-features}
There are constants $\good\ge 9\rafi$ and $\btbound\ge 1$, depending only on $S$, such that for each $\phi\in \Mod(S)$ there exist points $x_\phi,y_\phi,a_\phi,b_\phi\in \T(S)$ satisfying the following for every domain $V\esub S$:
\begin{enumerate}
\item\label{gp-fixed-thick-net} $x_\phi,y_\phi\in \T_\thin(S)$ satisfy $\phi(x_\phi) = y_\phi$, and $a_\phi,b_\phi\in \tnet(S)$ are net points;
\item\label{gp-strong-align} the tuple $(x_0,a_\phi,b_\phi,\phi(x_0))$ is strongly $\good$--aligned; 
\item\label{gp-non-pA-annuli-balanced}
if $V$ is an annulus with $V\not\esub S_p$, then the ratio of $\min\{\thin,\ell_{a_\phi}(\partial V)\}$ and $\min\{\thin,\ell_{b_\phi}(\partial V)\}$ lies in $[\tfrac{1}{\good},\good]$.
\item\label{gp-split-short-curves}
if $d_V(a_\phi,b_\phi)>\good$, then either $V\esub S_p$, or else $V$ is an annulus and  $\ell_{a_\phi}(\partial V),\ell_{b_\phi}(\partial V) \ge \thin$.
\item\label{gp-aligned} each quadruple $(x_0, a_\phi,x_\phi,y_\phi)$ and $(x_\phi,y_\phi, b_\phi,\phi(x_0))$ is $\good$--aligned;
\item\label{gp-backtrack-implies_close_xy}
if $\max\{\bt^\phi_V =d_V(y_\phi,b_\phi), d_V(x_\phi, a_\phi)\}\ge \good$, then $d_V(x_\phi,y_\phi)\le \frac{1}{\btbound}\good$;
\item\label{gp-nonannular}
if $V$ is nonannular, then $a_\phi$ and $b_\phi$ are $\good$--barycenters in $V$ for the respective triples $x_0, x_\phi, \phi(x_0)$ and $x_0, y_\phi, \phi(x_0)$; furthermore,  $\abs{d_V(x_\phi, a_\phi) - d_V(y_\phi, b_\phi)}$ and $\abs{d_V(x_\phi, y_\phi) - d_V(a_\phi, b_\phi)}$ are both at most $\good$.
\item\label{gp-no-total-backtracking} for each $V$ there exists $-\btbound < j\le 0$ so that $\bt^\phi_{\phi^j(V)}\le \good$, moreover:
\begin{itemize}
\item if the $\phi$--orbit of $V$ is finite, then this orbit size divides $\btbound$, and
\item if the $\phi$--orbit is infinite, then $\bt^\phi_{\phi^{k\btbound}(V)}>\good$ holds for at most one  $k\in \Z$;
\end{itemize}
\item\label{gp-jump-implies-see-preimage} if $\bt^\phi_{\phi(V)} \ge  \bt^\phi_V+7\good$,  then $d_V(x_0,b_\phi) \ge 2\good > \rafi$;
\item \label{gp-annular_backtracking_bound}
if $V$ is an annulus, then $\bt_V^\phi = d_V(b_\phi,y_\phi)\ge \good$ implies $d_V(b_\phi,\phi(x_0)) <\good$, and $d_V(a_\phi,x_\phi)\ge \good$ implies $d_V(x_0,a_\phi)<\good$;
\item
\label{gp-backtracking-bounded-by-deficit} 
if $V$ is nonannular, then  $\bt^\phi_V \ladd_\good \btbound\lipconst d_{\T(S)}(x_0,\phi(x_0))$.
\end{enumerate}
\end{proposition}

\begin{proof}
By Lemma~\ref{lem:tuple_is_aligned}, the tuple $(x_0,x_\phi^l,y_\phi^r,\phi(x_0))$ is $\theta$--aligned for $\theta = 16\Lambda_0$. For each annulus $A$, we partition $\{x_\phi^l, y_\phi^r\}$ as follows: if $d_A(x_\phi^l,y_\phi^r)\le  16\Lambda_0$, then let $\mathcal{Q}_0^A$ consist of the single element $\{x_\phi^l,y_\phi^r\}$; otherwise $\mathcal{Q}_0^A$ is the trivial partition $\{x_\phi^l\},\{y_\phi^r\}$ into singletons. These partitions are tautologically $\theta$--small. Hence we may apply Lemma~\ref{lem:building_strong_alignment} get a constant $\branchal' = \theta'$, depending only on $\theta = 16\Lambda_0$, and net points $a_\phi,b_\phi\in \tnet(S)$ so that $(x_0,a_\phi,b_\phi,\phi(x_0))$ satisfies the conclusions of Lemma~\ref{lem:building_strong_alignment}. We then take $\good = 40 \btbound(\Lambda + \Lambda_0+\branchal+\branchal'+\rafi)$, where $\btbound$ is from Lemma~\ref{lem:no_total_orbit_backtracking}.

Items (\ref{gp-fixed-thick-net})--(\ref{gp-strong-align}) now follow immediately from the construction of $x_\phi,y_\phi$ in Definition~\ref{def:the_good poitns} and the construction of $a_\phi,b_\phi$ in Lemma~\ref{lem:building_strong_alignment}. 

For (\ref{gp-non-pA-annuli-balanced}), consider an annulus $V\not \esub S_p$. If $d_V(x_\phi^l,y_\phi^r)> 16\Lambda_0$, then Lemma~\ref{lem:separated_annuli_near_ends} ensures $\min\{d_V(x_0, x_\phi^l), d_V(x_\phi^l,\phi(x_0))\}\le 6\Lambda_0 < 4(\theta+\rafi)$. Since we also know $x_0,\phi(x_0)\in \T_\thin(S)$, the second bulleted conclusion of Lemma~\ref{lem:building_strong_alignment} therefore gives $\ell_{a_\phi}(\partial V) \ge \thin$. Symmetrically we also have $\ell_{b_\phi}(\partial V)\ge \thin$. Hence in this case the indicated ratio is $1$. Otherwise $d_V(x_\phi^l,y_\phi^r)\le 16\Lambda_0$ and the construction from Lemma~\ref{lem:building_strong_alignment}, using the partitions $\mathcal{Q}_0^A$, ensures the ratio lies in $[\tfrac{1}{\branchal'},\branchal']$. 

For (\ref{gp-split-short-curves}), assume $d_V(a_\phi,b_\phi) > \good$. The statement is automatic for $V\esub S_p$, so we assume $V\not\esub S_p$. If $V$ is not an annulus, then necessarily $V\esub S_f$ or $V\cut \sigma(\phi)$. In either case  $d_V(x_\phi,y_\phi)\le 2\lipconst$. By Lemma~\ref{lem:barycenter_vs_gromov_product}, this gives $d_V(x_\phi',y_\phi')\le 4\branchal+\rafi+2\lipconst\le 6\branchal$. Since $V$ is nonannular, the constructions from Lemmas~\ref{cor:adjusted barycenter} and \ref{lem:building_strong_alignment} ensure $d_V(a_\phi,x_\phi')$ and $d_V(b_\phi, y_\phi')$ are at most $\branchal+\branchal'$. This gives the contradiction $d_V(a_\phi,b_\phi)\le 8\branchal+2\branchal'<\good$. Therefore $V$ is an annulus. The bounds $d_V(a_\phi,x_\phi^l),d_V(b_\phi,y_\phi^r)\le \branchal'$ from Lemma~\ref{lem:building_strong_alignment} and the triangle inequality now imply $d_V(x_\phi^l,y_\phi^r)>\good -2\branchal' > 16\Lambda_0$. Thus, as in item (\ref{gp-non-pA-annuli-balanced}) above, Lemmas~\ref{lem:separated_annuli_near_ends} and \ref{lem:building_strong_alignment} combine to show $\ell_{a_\phi}(\partial V),\ell_{b_\phi}(\partial V)\ge \thin$.

For (\ref{gp-aligned}), we know from Lemma~\ref{cor:adjusted barycenter} that $(x_0, x_\phi^l,x_\phi)$ and $(y_\phi, y_\phi^r,\phi(x_0))$ are both $\branchal$--aligned and from Lemma~\ref{lem:strongly_contracting} that $(x_0, x_\phi, y_\phi)$ and $(x_\phi,y_\phi,\phi(x_0))$ are $\Lambda_0$--aligned. Using this and the fact $d_V(a_\phi, x_\phi^l)$ and $d_V(b_\phi, y_\phi^r)$ are at most $\branchal'$,  it is straightforward to check that each subtriple of either quadruple is $(2\branchal' +\branchal +\Lambda_0)$--aligned.

For (\ref{gp-backtrack-implies_close_xy}), we show the contrapositive and assume $d_V(x_\phi,y_\phi) > \frac{1}{\btbound}\good > \Lambda_0$. Then $(x_0\vert \phi(x_0))_{w}^V \le \Lambda_0$ holds for $w = x_\phi$ and $w = y_\phi$ by Lemma~\ref{lem:strongly_contracting}. Since $\Lambda_0\le \Lambda$ (Lemma~\ref{lem:no_total_orbit_backtracking}), the construction in Lemma~\ref{cor:adjusted barycenter} ensures $x_\phi^l$ is a $\branchal$--barycenter in $V$ for $x_0, x_\phi, \phi(x_0)$ and  $y_\phi^r$ is a $\branchal$--barycenter in $V$ for $x_0,y_\phi,\phi(x_0)$. Hence $\max\{d_V(x_\phi, x_\phi^l),d_V(y_\phi,y_\phi^r)\}\le \branchal+ \Lambda_0$ by Lemma~\ref{lem:barycenter_vs_gromov_product}. Since $d_V(a_\phi, x_\phi^l)$ and $d_V(b_\phi, y_\phi^r)$ are both at most $\branchal'$, the triangle inequality now implies
\[\max\{d_V(x_\phi, a_\phi), d_V(y_\phi, b_\phi)\} \le \branchal+\Lambda_0 +\branchal' < \tfrac{1}{2}\good. \]

For (\ref{gp-nonannular}), since $V$ is not an annulus, Lemma~\ref{cor:adjusted barycenter} ensures $x_\phi^l$ and $y_\phi^r$ are, respectively, $\branchal$--barycenters in $V$ for the triples $x_0, x_\phi, \phi(x_0)$ and $x_0, y_\phi, \phi(x_0)$. Since $d_V(a_\phi, x_\phi^l) \le \branchal'$ and $d_V(b_\phi, y_\phi^r)\le \branchal'$, it follows that $a_\phi$ is a $\branchal+2\branchal'$ barycenter in $V$ for $x_0, x_\phi, \phi(x_0)$ and that $b_\phi$ is a $\branchal+2\branchal'$ barycenter in $V$ for $x_0, y_\phi, \phi(x_0)$. By Lemma~\ref{lem:barycenter_vs_gromov_product} we now have
\[d_V(a_\phi, b_\phi) \le d_V(x_\phi, y_\phi) + 4\branchal+8\branchal'+\rafi < d_V(x_\phi, y_\phi) + \tfrac{1}{4}\good.\]
Hence, if $d_V(x_\phi, y_\phi)\le \Lambda_0<\tfrac{\good}{4}$ then $d_V(a_\phi, b_\phi) \le \tfrac{\good}{2}$. This implies that their difference $\abs{d_V(a_\phi, b_\phi) - d_V(x_\phi, y_\phi)}$ is at most $\good$ and, by the triangle inequality, that $\abs{d_V(x_\phi, a_\phi) - d_V(y_\phi, b_\phi)} \le \good$. Otherwise, when $d_V(x_\phi, y_\phi) > \Lambda_0$, the proof of (\ref{gp-backtrack-implies_close_xy}) ensures $d_V(x_\phi,a_\phi)$ and $d_V(y_\phi, b_\phi)$ are at most $\frac{\good}{2}$. Hence their difference is at most $\good$ and the triangle inequality implies $\abs{d_V(x_\phi, y_\phi) - d_V(a_\phi, b_\phi)}\le \good$.

For (\ref{gp-no-total-backtracking}), Lemma~\ref{lem:no_total_orbit_backtracking} provides some $-\btbound< j \le 0$ so that $(x_0 \vert \phi(x_0))_{y_\phi}^{\phi^j(V)} \le \Lambda$. Exactly as in the proof of (\ref{gp-backtrack-implies_close_xy}) above, this implies $\bt^\phi_{\phi^j(V)}< \good$. Further, the ``moreover'' claims of (\ref{gp-no-total-backtracking}) follow from the corresponding conclusions of Lemma~\ref{lem:no_total_orbit_backtracking}.

For (\ref{gp-jump-implies-see-preimage}), since $(y_\phi,b_\phi,\phi(x_0))$ is $\good$--aligned and $y_\phi= \phi(x_\phi)$, the hypothesis implies
\[d_V(x_0,x_\phi) = d_{\phi(V)}(\phi(x_0),y_\phi) 
\ge d_{\phi(V)}(b_\phi,y_\phi) - \good 
\ge  d_V(y_\phi,b_\phi) + 6\good.\]
Now, if $d_V(x_\phi,y_\phi)\le 4\good$, then the triangle inequality gives 
\begin{align*}
d_V(x_0,b_\phi) &\ge d_V(x_0, x_\phi) - d_V(b_\phi,y_\phi) - d_V(y_\phi,x_\phi)
\ge 6\good - 4\good  \ge 2\good>\rafi.
\end{align*}
If instead $d_V(x_\phi,y_\phi)> 4\good$, then (\ref{gp-backtrack-implies_close_xy}) implies $d_V(y_\phi,b_\phi)\le \good$. We now use the fact that $(x_0, x_\phi, y_\phi)$ is $\good$--aligned (\ref{gp-aligned}) to find
\[d_V(x_0,b_\phi) \ge d_V(x_0, y_\phi) - \good \ge d_V(x_\phi,y_\phi)-2\good > 2\good \ge \rafi.\]

For (\ref{gp-annular_backtracking_bound}), we show the contrapositive and only consider $a_\phi$ and $x_\phi$; the argument for  $b_\phi$ and $y_\phi$ is symmetric. So suppose $d_V(x_0,a_\phi)\ge\good$. Since $d_V(a_\phi,x_\phi^l)\le \branchal'$, it follows that $d_V(x_0, x_\phi^l) \ge \good-\branchal' > \branchal$. Thus, according to the construction in Lemma~\ref{cor:adjusted barycenter}, it must be that $(x_0\vert \phi(x_0))_{x_\phi}^V \le\Lambda$ and that $x_\phi^l$ is a $\branchal$--barycenter in $V$ for $x_0, x_\phi, \phi(x_0)$. Therefore using Lemma~\ref{lem:barycenter_vs_gromov_product} we have
\begin{align*}
d_V(x_\phi,a_\phi) \le d_V(x_\phi,x_\phi^l)+\branchal' \le (x_0\vert \phi(x_0))_{x_\phi}^V + \branchal+\branchal' \le \Lambda+\branchal+\branchal'<\good.
\end{align*}

Finally, for (\ref{gp-backtracking-bounded-by-deficit}), first recall from  (\ref{gp-strong-align})--(\ref{gp-aligned}) that $(x_0,a_\phi,\phi(x_0))$ and $(y_\phi,b_\phi,\phi(x_0))$ are both $\good$--aligned. Also recall that $y_\phi = \phi(x_\phi)$. Thus for any nonannulus $Y$ we may invoke $\abs{d_Y(x_\phi,a_\phi)-d_Y(y_\phi,b_\phi)}\le \good$ from (\ref{gp-nonannular}) to find that
\begin{align*}
\bt^\phi_{\phi(Y)} 
=d_{\phi(Y)}(y_\phi,b_\phi) 
&\le d_{\phi(Y)}(y_\phi,\phi(x_0))+\good
= d_{Y}(x_0,x_\phi) + \good\\
&\le d_{Y}(x_0, a_\phi) + d_{Y}(a_\phi, x_\phi) + \good\\
&\le (d_{Y}(x_0, \phi(x_0)) + \good) +  (d_Y(y_\phi,b_\phi)+\good) + \good\\
&= d_Y(x_0, \phi(x_0)) + \bt^\phi_Y + 3\good.
\end{align*}
Now fix a nonannulus $V$ and, using (\ref{gp-no-total-backtracking}),  let $0 \le j \le \btbound$ be the smallest integer so that $d_{\phi^{-j}(V)}(y_\phi,b_\phi) \le \good$. Applying the above estimate recursively with the domains $Y = \phi\inv(V), \dots,\phi^{-j}(V)$ we find
\begin{align*}
\bt^\phi_V 
&\le \left( \sum_{n=1}^{j} d_{\phi^{-n}(V)}(x_0,\phi(x_0)) + 3\good\right) + \good \ladd_\good   \sum_{n=-\btbound}^{-1} d_{\phi^n(V)}(x_0,\phi(x_0)).
\end{align*}
Since $V$ is nonannular, the Lipschitz bound \eqref{eqn:lipschitz_proj} implies each term $d_{\phi^n(V)}(x_0, \phi(x_0))$  is at most $\lipconst d_{\T(S)}(x_0,\phi(x_0)) + \lipconst$. Thus this estimate gives the desired bound.
\end{proof}
We remark that the failure of the Lipschitz estimate (\ref{eqn:lipschitz_proj}) for annuli in (\ref{gp-backtracking-bounded-by-deficit}) is the entire reason we have utilized the annular-split barycenters $x_\phi^l,y_\phi^r$ from Lemma~\ref{cor:adjusted barycenter} and the alternate conclusion (\ref{gp-annular_backtracking_bound}).

\section{Bounding the multiplicity of branch points}
\label{Sec:Bounding the multiplicity of branch points}
Recall that we have fixed a point $x_0\in \T_\thin(S)$ and for each element $\phi\in \Mod(S)$  produced in Proposition~\ref{prop:good-point-features} a good pair of points $x_\phi,y_\phi\in \T_\thin(S)$ along with a pair of net points $a_\phi,b_\phi\in\tnet(S)$ satisfying various properties. 

In this section we bound the multiplicity of any given pair $(a,b)\in \tnet(S)$:

\begin{theorem}
\label{thm:bound_multiplicity}
There is a polynomial $p$ such that for any ordered pair $(a,b)\in \tnet(S)$ and $r\ge 0$, there are at most $p(r)$ elements $\phi$
for which $d_{\T(S)}(x_0,\phi(x_0))\le r$ and $(a_\phi,b_\phi)=(a,b)$.
\end{theorem}

\subsection{Agreement}
We shall prove this by using curve complex data to effectively build-up a map $\phi$ on larger and larger subsurfaces. More precisely, we will bound the indicated subset of $\Mod(S)$  by partitioning into smaller and smaller subfamilies that agree on larger and larger subsurfaces.

We begin by establishing some general topological statements that will be useful.
\begin{definition}
\label{def:agreement}
We say  $\phi,\psi\in \Mod(S)$ \define{agree} in a subsurface $A\esub S$ if for each component $Y$ of $A$ we have $\phi(Y) = \psi(Y)$ and $\phi\vert_{Y} = \psi\vert_{Y}$ up to isotopy. So in particular $\phi(A)=\psi(A)$ as subsurfaces and, after adjusting say $\psi$ by an isotopy, $\psi\inv\phi$ pointwise fixes $A$ and its boundary $\partial A$.
In the case of an annular component $Y$, we additionally require 
$d_{\phi(Y)}(\phi(\alpha),\psi(\alpha))\le 4$ for every essential curve $\alpha$ on $S$.
\end{definition}

\begin{remark}
When $Y$ is an annulus with $\phi\vert_Y = \psi_Y$ up to isotopy, we may always precompose with Dehn twists $T^n_{\partial Y}$ in $Y$ so that $\psi$ and $\phi\circ T^n_{\partial Y}$ agree in $Y$.
The key point is that given any two curves $\alpha,\beta$ cutting $Y$ we can choose $n\in \N$ so that $d_Y(\alpha,T_Y^n(\beta))\le 1$. Now, fix some curve $\alpha$ with $\alpha\cut Y$ and choose $n$ so that $d_Y(\alpha, T_Y^n\psi\inv\phi(\alpha))\le 1$. We claim that 
$d_Y(\beta,\psi\inv\phi T_Y^n(\beta))\le 3$ for every curve $\beta$, so that $\psi$ and $\phi T_Y^n$ agree on $Y$. 
Indeed, choose $k\in \N$ so that $d_Y(T_A^k(\beta),\alpha_0)\le 1$. Then, since $T_Y$ and $\psi\inv\phi$ commute, we have
    \begin{align*}
    d_Y(\beta, \psi\inv\phi T_A^n(\beta))
    &= d_Y(T_Y^k(\beta), T_Y^k \psi\inv\phi T_Y^n(\beta))
    = d_Y(T_Y^k(\beta), T_Y^n\psi\inv\phi(T_Y^k(\beta)))\\
    &\le d_Y(T_Y^k(\beta), \alpha_0) + d_Y(\alpha_0, T_Y^n\psi\inv\phi(\alpha_0))\\
    &\hspace{1.5em}+ d_Y(T_Y^n\psi\inv\phi(\alpha_0),T_Y^n\psi\inv\phi(T_Y^k(\beta))
    \le 1 + 1 + 1 = 3.
    \end{align*}
  \end{remark}

Annular components in $A$ will be used to ensure mapping classes do not differ by Dehn twists about boundary components. For example, if $Y\esub S$ is a torus with one boundary component, then agreement on $Y$  conveys no information about twisting in $Y$; indeed the maps $\phi \circ T_{\partial Y}^n$ for $n\in \Z$ all agree on $Y$. If we let $A$ be the union of $Y$ with a disjoint annulus parallel to $\partial Y$, then agreement on $A$ additionally concerns Dehn twists about $\partial Y$ so that the maps $\phi\circ T_{\partial Y}^n$, $n\in \Z$, no longer all agree in $A$.

\begin{lemma}
\label{lem:bounding_submaps}
\label{lem:proper:discontinuity}
For each $k>0$ there exists a constant $k'$ with the following property: For a given point $w\in \T(S)$ and pair of subsurfaces $A,B\esub S$, let $\mathcal{F}\subset \Mod(S)$ be a set of mapping classes such that for all $\phi,\psi\in \mathcal{F}$ we have $\phi(A) = B$ and $d_W(\phi(w),\psi(w)) \le k$ for each domain $W\esub B$. 
 Then, up to agreement on $A$, the collection $\{\phi\vert_A \colon A\to B \mid \phi\in \mathcal{F}\}$ has cardinality at most $k'$. That is, $\mathcal{F}$ may be partitioned into at most $k'$ subfamilies in which all maps agree on $A$.
\end{lemma}
\begin{proof}
There is a universal bound depending on $\plex{S}$ on the number of components of $A$ and the number of boundary components of each component. Hence, after partitioning $\mathcal{F}$ into boundedly many subfamilies, we may assume  $\psi\inv\phi$ fixes up to isotopy each component and boundary component of $A$ for all $\phi,\psi\in \mathcal{F}$. To prove the lemma, we must partition $\mathcal{F}$ to achieve agreement in each component $Y$ of $A$.

First suppose $Y$ is non-annular. Fix some $\phi\in \mathcal{F}$. Then for all $\psi\in \mathcal{F}$, the projections $\pi_V(\psi\inv \phi(w))$ and $\pi_V(w)$ coarsely agree in $\cc(V)$ for each domain $V\esub Y$. If $w'\in \T_\thin(Y)$ is a thick point realizing the consistent tuple $(\pi_V(w))_{V\esub Y}$, it follows that $d_V(\psi\inv\phi(w'),w')\ladd_k 0$ for all such $V$ and hence that $d_{\T(Y)}(\psi\inv\phi(w'),w')$ is bounded by the distance formula. By proper discontinuity of the action on $\T(Y)$, there are boundedly many mapping classes $Y\to Y$ that coarsely fix $w'$.

Now suppose $Y$ is an annulus. Fix $\phi\in \mathcal{F}$, set $V = \phi(Y)$, and  let $\beta\in \mu_w$ be a curve in the short marking with $\beta\cut Y$. 
The assumption implies that each $\psi\in \mathcal{F}$ satisfies $d_{V}(\phi(\beta),\psi(\beta))\le k$ and hence that there exists $n\in \Z$ with $\abs{n}\le k$ so that $d_{V}(T_{V}^n \psi(\beta),\phi(\beta))\le 1$. Letting $\mathcal{F}_n\subset \mathcal{F}$ denote those $\psi$ that work with a given power $n$, we thus get a decomposition $\mathcal{F} = \mathcal{F}_{-k}\cup\dots\cup \mathcal{F}_k$  into at most $2k+1$ subsets. By the triangle inequality, all $\psi,\psi'\in \mathcal{F}_n$ satisfy $d_V(\psi(\beta),\psi'(\beta)) = d_V(T_V^n\psi(\beta), T_V^n \psi'(\beta)) \le 2$.
Now, for any other curve $\alpha$, we may pick $j$ so that $d_Y(T_Y^j(\alpha),\beta)\le 1$. Note that  $\phi' T_Y^j = T_V^j \phi'$ for all $\phi'\in \mathcal{F}$. Then all $\psi,\psi'\in \mathcal{F}_n$ have
\begin{align*}
d_V(\psi(\alpha),\psi'(\alpha)) 
&= d_V(T_V^j\psi(\alpha),T_V^j\psi'(\alpha))
= d_V(\psi T_Y^j(\alpha),\psi'T_V^j(\alpha))\\
&= d_V(\psi T_Y^j(\alpha),\psi(\beta)) + d_V(\psi(\beta), \psi'(\beta)) + d_V(\psi'(\beta),\psi'T_V^j(\alpha))\\
&\le 1 + 2 + 1 = 4.
\end{align*}
Hence we have partitioned into boundedly many subsets that each agree on $Y$.
\end{proof}

\subsection{Backtracking domains}
\label{sec:backtracking_domains}
Recall the streamlined notation $\bt^\phi_V = d_V(y_\phi,b_\phi)$ from (\ref{eqn:backtracking_constant}), which we view as a measure of backtracking in $V$ that indicates non-alignment of the quadruple $(x_0,x_\phi,y_\phi,\phi(x_0))$ in $V$.

 Let $f\colon \R_+\to \R_+$ be the function defined by
\begin{equation}
\label{eqn:defn_of_f}
f(t) = 8t + 53\good
\end{equation}
where $\good\ge 9\rafi$ is from Proposition~\ref{prop:good-point-features}.
We say that a sequence  $V, \phi(V), \ldots, \phi^j(V)$  \define{jumps} for $\phi\in \Mod(S)$ if for all $0\leq i\leq j$ 
\[\bt^\phi_{\phi^i(V)} >  f\left( \bt^\phi_{\phi\inv(V)}\right).\]

\begin{remark}
\label{rem:concat_jump_seqs}
If $0 < i < j$ are such that $V,\dots, \phi^i(V)$ and $\phi^{i+1}(V),\dots, \phi^j(V)$ are both jump sequences, then the concatenation $V,\dots, \phi^j(V)$ is also a jump sequence since for all $i+1 \le k \le j$ the monotonicity of $f$ evidently implies
\[\bt^\phi_{\phi^k(V)} \ge f\big(\bt^\phi_{\phi^i(V)}\big) \ge f^2\big(\bt^\phi_{\phi\inv(V)}\big) \ge f\big(\bt^\phi_{\phi\inv(V)}\big).\]
\end{remark}

\begin{definition}
\label{def:backtracking_domains}
The set of \define{backtracking domains} $D(\phi)$ for $\phi\in [\phi_0]$ is the union of all jump sequences for $\phi$, that is, the set of domains $\phi^i(V)$ for which there is a sequence  $\{V, \phi(V),\dotsc, \phi^i(V), \dotsc,\phi^j(V)\}$ that  jumps for $\phi$. Remark~\ref{rem:concat_jump_seqs} and Proposition~\ref{prop:good-point-features}(\ref{gp-no-total-backtracking}) imply that for every $Z\in D(\phi)$ there exists  $0 \le i < \btbound$ such that $\{\phi^{-i}(Z),\dots,Z\}$ is a jump sequence contained in $D(\phi)$ and $\phi^{-i-1}(Z)\notin D(\phi)$. We call this $i$ the backtracking \define{index} of $Z$ in $D(\phi)$.
\end{definition}

While each $Y\in D(\phi)$ satisfies $\bt^\phi_Y  > 53\good$ by definition, the converse need not hold since there may exist domains $Y$ with $\bt^\phi_Y > 53\good$ but whose orbit $\{\phi^i(Y)\}_{i\in \Z}$ does not contain a jump sequence.
Nevertheless, we have:

\begin{lemma}
\label{lem:which_domains_backtrack}
Every domain $Z\in D(\phi)$ satisfies $\bt^\phi_Z = d_Z(y_\phi,b_\phi)> 53\good$. 
Dually, the collection $D(\phi)$ contains every domain $Z\esub S$ with $\bt^\phi_Z > f^{\btbound}(\good)$. 
\end{lemma}
\begin{proof}
The first claim follows immediately from the definition of jumping.
We now suppose $Z\notin D(\phi)$ and show $\bt^\phi_Z \le f^{\btbound}(\good)$. If $\bt^\phi_Z \le \good$ there is nothing to prove, so we assume $\bt^\phi_Z >\good$. Let us write $Z_k = \phi^k(Z)$ for $k\in \Z$. 
By Proposition~\ref{prop:good-point-features}(\ref{gp-no-total-backtracking})
we may choose $-\btbound <k<0$ so that $\bt^\phi_{Z_k} \le \good$ and $\bt^\phi_{Z_j} > \good$ for all $k < j\le 0$.

For any $k\le j < 0$, the sequence $\{Z_{j+1},\dots,Z_0\}$ evidently does not jump for $\phi$; hence there is some $j < j' \le 0$ so that $\bt^\phi_{Z_{j'}} \le f(\bt^\phi_{Z_j})$. Starting with $k_0 = k$ and recursively using this observation to set $k_{j+1} = k'_j$ produces a sequence integers $k = k_0 < \dots < k_n=0$ such that $\bt^\phi_{Z_{k_{j+1}}}\le f(\bt^\phi_{Z_{k_j}})$ for each $j$. Since $\bt^\phi_{Z_{k_0}}\le \good$ and $n\le k+1 \le \btbound$, applying these inequalities inductively implies that
\[\bt^\phi_Z = \bt^\phi_{Z_{k_n}} \le f^{\circ n}(\bt^\phi_{Z_{k_0}})) \le f^{\btbound}(\good).\qedhere\]
\end{proof}

The next lemma says that the points $x_\phi$ and $y_\phi$ coarsely behave like fixed points from the perspective of backtracking domains. 

\begin{lemma}
\label{lem:good_points_fixed_in_backtracking}
If $\{V, \dots, \phi^j(V)\}\subset D(\phi)$, then 
\[ d_{V}(\phi^{-j}(x_\phi), y_\phi) \le \tfrac{j+1}{\btbound} \good \le  \good.\]
\end{lemma}
\begin{proof}
Note that $j < \btbound$ since $j$ is at most index $\phi^j(V)$. 
For $j = 0$ we indeed have $d_V(x_\phi,y_\phi) \le \tfrac{1}{\btbound}\good$ by Proposition~\ref{prop:good-point-features}(\ref{gp-backtrack-implies_close_xy}), since $\bt_V^\phi \ge \good$ by Lemma~\ref{lem:which_domains_backtrack}.  For $j\ge 1$ we may inductively assume $d_{\phi(V)}(\phi^{1-j}(x_\phi), y_\phi)\le \frac{j}{\btbound}\good$, since evidently $\{\phi(V),\dots, \phi^j(V)\}\subset D(\phi)$. This and the $j=0$ case imply the result since
\[d_V(\phi^{-j}(x_\phi),  y_\phi) 
\le d_V(\phi^{-j}(x_\phi), x_\phi) + \frac{1}{\btbound}\good
= d_{\phi(V)}(\phi^{1-j}(x_\phi),y_\phi) + \frac{1}{\btbound}\good.\qedhere
\]
\end{proof}

For each $Y\in D(\phi)$, let $0\leq i < \btbound$ be the index of $Y$ and define a constant $C_Y$ by
\begin{equation}
\label{eqn:defn_of_C_Y}
C_Y \colonequals 6 \bt^\phi_{\phi^{-i-1}(Y)} + 40\good.
\end{equation}
Also set
$C_0 = f^{\btbound+1}(\good)$. Since $\phi^{-i-1}(Y)\notin D(\phi)$ by definition of the index, we note that $9\rafi \le \good \le C_Y \le C_0 -\good \le C_0 - 2\rafi$ by Lemma~\ref{lem:which_domains_backtrack} and the definition of $f$.
In particular, $C_Y$ is uniformly bounded.

\begin{definition}[Orders on $D(\phi)$]
\label{def:backtracking_orders}
For any $\phi$ we define four asymmetric relations $\lhd_i$ on the set $D(\phi)$ of backtracking domains as follows. For $V,Y\in D(\phi)$,
\begin{itemize}
\item $V \lhd_0 Y$ if there exists $j\ge 1$ so that $\{V,\phi(V),\dotsc, \phi^j(V)=Y\}\subset D(\phi)$.
\item $V \lhd_1 Y$ if $V\cut Y$ and $V \tol  Y$ along $[b_\phi, y_\phi]$.
\item $V \lhd_2 Y$ if $V\esubn Y$ with $d_Y(b_\phi, \partial V) < C_Y$.
\item $V \lhd_{\tilde{2}} Y$ if $V\esubn Y$ with $d_Y(b_\phi, \partial V) < C_Y+\rafi$.
\item $V \lhd_3 Y$ if $V \esupn Y$ with $d_V(b_\phi, \partial Y) \ge C_0 + 2\rafi$.
\end{itemize}
Thus $\lhd_{\tilde{2}}$ is a weaker version of $\lhd_2$; it will serve a minor technical role. Notice that $\lhd_0$ and $\lhd_1$ are non-reflexive partial orders; in particular they are transitive.

In general, for any subcollection $\mathcal{W}\subset D(\phi)$ and $i\in \{0,1,2,\tilde{2},3\}$, we write $\mathcal{W}^{i}$ for the set of domains in $\mathcal{W}$ that are minimal with respect to the order $\lhd_i$, that is:
\[\mathcal{W}^i = \left\{Z\in \mathcal{W} \mid \not\exists Y\in \mathcal{W}\text{ with }Y\lhd_i Z\right\}.\]
We also write $\mathcal{W}^\dag = \mathcal{W}^1\cap \mathcal{W}^2\cap \mathcal{W}^3$ and $\mathcal{W}^\star = \mathcal{W}^0\cap \mathcal{W}^\dag$.
\end{definition}

\begin{lemma}
\label{lem:descent_of_minimal_time_order}
Let $\mathcal{W}$ be any subcollection of $D(\phi)$. If $U\in \mathcal{W}^{1}$ is such that the subcollection $\mathcal{U} = \{Z\in \mathcal{W} \mid Z\esubn U\}$ is nonempty, then $\mathcal{W}^{1} \supset \mathcal{U}^{1} \ne \emptyset$.
\end{lemma}
\begin{proof}
Since $\mathcal{U}$ is nonempty and finite and $\lhd_1$ is transitive, $\mathcal{U}^{1}$ is nonempty. Now consider some $Y\in \mathcal{U}^{1}$ and take any $Z\in \mathcal{W}$ with $Z\cut Y$. Then either $Z\esubn U$ and we have $Y \tol  Z$ by virtue of $Y\in \mathcal{U}^{1}$, or else $Z\cut U$ and hence $U \tol  Z$, since $U\in \mathcal{W}^{1}$, and consequently $Y\tol Z$ by Corollary~\ref{cor:time-order_subsurf}. Thus $Y\in \mathcal{W}^{1}$ as claimed.
\end{proof}

\begin{lemma}
\label{lem:supertriminimal}
Let $\mathcal{W}$ be a subcollection of $D(\phi)$. If $V\in \mathcal{W}^{1}\setminus \mathcal{W}^{3}$, then there exists some $Y\in \mathcal{W}^\dag$ with $Y\esupsetneq V$.
\end{lemma}
\begin{proof}
The assumption $V\notin \mathcal{W}^{3}$ implies that 
\[\Omega = \big\{Z\in \mathcal{W} \mid Z\esupn V \text{ and }d_Z(b_\phi,\partial V) \ge C_0+\rafi\big\}\]
is nonempty.  Hence we may choose a domain $Y\in \Omega$ maximizing the quantity $\plex{Y}$. 
We note that $d_Y(b_\phi,\partial V)\ge C_0 +\rafi$.
We claim $Y\in \mathcal{W}^\dag$. To see this, we consider any $Z\in \mathcal{W}$ and show that $Z\lhd_i Y$ fails for each of $i=1,2,3$.

First consider the cases $Z\cut Y$ and $Z\esubn Y$. Then the multicurve $\partial Z$ projects to $\cc(Y)$. If $\partial Z$ is disjoint from $\partial V$, then we have $d_Y(\partial Z, \partial V) \le 2$ and hence $d_Y(b_\phi,\partial Z) \ge d_Y(b_\phi,\partial V) - 2 > C_0 +\rafi-2> C_Y$. 
When $Z\cut Y$, this ensures $Y \tol  Z$, and when $Z\esubn Y$ it precludes $Z\lhd_2 Y$. 
If, instead, $\partial Z$ and $\partial V$ are not disjoint, then $Z\cut V$ and hence $V \tol  Z$ by the assumption that $V\in \mathcal{W}^{1}$. If $Z\cut Y$ this implies $Y \tol  Z$ by Corollary~\ref{cor:time-order_subsurf}, and if $Z\esubn Y$ it implies via Corollary~\ref{lem:time-order-boundary_distance} that
\[d_Y(b_\phi,\partial Z) \ge d_Y(b_\phi, \partial V) -\rafi/3 \ge C_0 + 2\rafi/3 > C_Y.\]
In either case, we may conclude that $Y\in \mathcal{W}^{1}\cap \mathcal{W}^{2}$.

It remains to suppose $Z\esupsetneq Y$. Since by construction $Y$ is the largest complexity surface in $\Omega$, we must have $Z\notin \Omega$. Hence $d_Z(b_\phi,\partial V) < C_0+\rafi$. On the other hand, the containment $V\esubn Y$ gives $d_Z(b_\phi, \partial Y) \le d_Z(b_\phi, \partial V) + 1 < C_0+2\rafi$ which precludes $Z\lhd_3 Y$ and consequently shows $Y\in \mathcal{W}^{3}$.
\end{proof}

\begin{lemma}
\label{lem:triminimal}
Let $\mathcal{W}\subset D(\phi)$. If $V\in \mathcal{W}^{1}$, then there exists $Y\in \mathcal{W}^\dag$ such that either $Y\esup V$ or else $Y\esubn V$ with $Y\lhd_{\tilde{2}} V$.
\end{lemma}
\begin{proof}
If $V\notin\mathcal{W}^{3}$, then Lemma~\ref{lem:supertriminimal} provides a domain satisfying the conclusion. We may henceforth assume $V\in \mathcal{W}^{3}$. 
If $V\in \mathcal{W}^{2}$, then $V\in \mathcal{W}^\dag$ and the claim holds. Otherwise $V\notin\mathcal{W}^{2}$ and we have $V'\lhd_2 V$ for some $V'\in \mathcal{W}$. In particular, 
\[\mathcal{V} = \{Z\in \mathcal{W} \mid Z\esubsetneq V\}\]
is nonempty. The subcollection $\mathcal{V}^{1}$ is then nonempty, and we may choose a domain $Y'\in\mathcal{V}^{1}$ minimizing the quantity $\plex{Y'}$.
Firstly observe that $Y'\in \mathcal{W}^{1}$ by Lemma~\ref{lem:descent_of_minimal_time_order}. Secondly, observe that $Y'\in \mathcal{W}^{2}$. Indeed, otherwise
\[\mathcal{Y} = \{Z\in \mathcal{W} \mid Z\esubn Y'\} = \{Z\in \mathcal{V} \mid Z\esubn Y'\}\]
is nonempty and hence contains some element $Z\in \mathcal{Y}^{1}$. But then $Z\in \mathcal{Y}^{1}\subset \mathcal{V}^{1}$ by Lemma~\ref{lem:descent_of_minimal_time_order} with $\plex{Z} < \plex{Y'}$, contradicting the choice of $Y'$. 

Next observe that some $Y\esup Y'$ satisfies $Y\in \mathcal{W}^\dag$. Indeed, if $Y'\in \mathcal{W}^3$ we take $Y = Y'$ and if not then Lemma~\ref{lem:supertriminimal} provides such a $Y$.
The domains $Y$ and $V$ cannot be disjoint, since they both contain $Y'$, and nor can we have $Y\cut V$, 
as then they could not both be minimal with respect to time order.
If $Y$ satisfies $Y\esup V$, then the lemma is verified. The only remaining possibility is $Y\esubn V$, in which case we must show $Y\lhd_{\tilde{2}} V$. To see this, recall that we have $V'\lhd_2 V$. If the multicurves $\partial Y$ and $\partial V'$ are disjoint, it follows that $d_V(b_\phi, \partial Y) < d_V(b_\phi, \partial V') + 1$. Otherwise $Y\cut V'$ and we must have $Y \tol  V'$ along $[b_\phi, y_\phi]$, by virtue of $Y$ lying in $\mathcal{W}^1$; thus $d_V(b_\phi, \partial Y) \le d_V(b_\phi, \partial V') + \rafi/3$ by Corollary~\ref{lem:time-order-boundary_distance}. In either case, we conclude
\[d_V(b_\phi, \partial Y) \le d_V(b_\phi, \partial V') + \rafi/3 < C_V + \rafi,\]
which shows that $Y\lhd_{\tilde{2}}V$,  as desired.
\end{proof}

We also have the following observation relating $\lhd_1$ and $\lhd_{\tilde{2}}$ to $\lhd_0$:

\begin{lemma}
\label{lem:chains_and_zero_order}
Suppose $Y,Z\in D(\phi)$ and that $Y \lhd_1 Z$ or $Y\lhd_{\tilde{2}}  Z$. If a chain $\phi^{-j}(Y),\ldots,Y$ is contained in $D(\phi)$, then the corresponding chain $\phi^{-j}(Z),\ldots,Z$  is also contained in $D(\phi)$. That is, $\phi^{-j}(Y)\lhd_0 Y$ implies $\phi^{-j}(Z)\lhd_0 Z$.
\end{lemma}
\begin{proof}
Suppose not. Let $i\ge 0$ be the index of $Z$, so that $\{\phi^{-i}(Z),\dots,Z\}\subset D(\phi)$ is a jump sequence but $\phi^{-i-1}(Z)\notin D(\phi)$.  Set $Y' = \phi^{-i-1}(Y)$ and  $Z' = \phi^{-i-1}(Z)$. The assumption implies $i < j$, so we have $Y'\in D(\phi)$ but $Z'\notin D(\phi)$. Since $j$ is at most the index of $Y$, we also have $j < \btbound$.

We know that $\partial Y$ projects to $\cc(Z)$ and hence that
\[d_Z(b_\phi,y_\phi) \le d_Z(b_\phi,\partial Y) + d_Z(\partial Y, y_\phi).\]
Now if $Y\lhd_1 Z$, then $Y \tol  Z$ along $[b_\phi,y_\phi]$ and hence $d_Z(b_\phi, \partial Y)< \rafi/3$ by Lemma~\ref{lem:characterize_timeorder}. If instead $Y\lhd_{\tilde{2}} Z$  then $d_Z(b_\phi,\partial Y) < C_Z+\rafi$ by assumption. In either case we have 
\begin{equation}
\label{eqn:bound-d_Z-in-terms-of-C_Z}
\bt^\phi_Z = d_Z(b_\phi,y_\phi) < d_Z(\partial Y,y_\phi) + C_Z + \rafi.
\end{equation}
Since $\{\phi^{-i}(Z),\dots, Z\}\subset D(\phi)$, Lemma~\ref{lem:good_points_fixed_in_backtracking} implies
$d_Z(x_\phi,y_\phi)\le \tfrac{1}{\btbound}\good$ and also 
\[d_{Z}(x_\phi, \phi^{i+1}(x_\phi)) = d_{Z}(x_\phi, \phi^i(y_\phi)) = d_{\phi^{-i}(Z)}(\phi^{-i}(x_\phi),y_\phi) \le \tfrac{i+1}{\btbound}\good.\]
Therefore, the triangle inequality gives
\begin{equation}
\label{eqn:go_to_preimages}
d_{Z}(\partial Y, y_\phi) \le d_Z(\partial Y, \phi^{i+1}(x_\phi)) + \tfrac{i+2}{\btbound}\good 
\le d_{Z'}(\partial Y', x_\phi) + \good,
\end{equation}
where here we used $i < j < \btbound$ to conclude $i+2\le \btbound$.

We also know that $\partial Y' = \phi^{-i-1}(\partial Y)$ projects to $\cc(Z')$.  Since $Y'$ and $\phi(Y')$ are in $D(\phi)$ by assumption, Lemma~\ref{lem:which_domains_backtrack} ensures   $d_{Y'}(y_\phi,b_\phi)$ and $d_{\phi(Y')}(y_\phi,b_\phi)$ are both at least $2\good$. As $(y_\phi,b_\phi,\phi(x_0))$ is $\good$--aligned (Proposition~\ref{prop:good-point-features}(\ref{gp-aligned})) we see that
\[d_{Y'}(x_0,x_\phi) = d_{\phi(Y')}(\phi(x_0),y_\phi) \ge d_{\phi(Y')}(\phi(x_0), b_\phi) + d_{\phi(Y')}(b_\phi,y_\phi) - \good \ge \good.\]
Therefore $Y'$ has active intervals along both $[x_0,x_\phi]$ and $[y_\phi,b_\phi]$ and we may choose points $s\in [x_0,x_\phi]$ and $t\in [y_\phi,b_\phi]$ containing $\partial Y'$ in their Bers markings. Since $(x_0, s, x_\phi)$ is $\back$--aligned (Theorem~\ref{thm:back}) and $(x_0, x_\phi,y_\phi)$ is $\good$--aligned (Proposition~\ref{prop:good-point-features}(\ref{gp-aligned})) we see that
\begin{align*}
d_{Z'}(x_0, s) +& d_{Z'}(s, x_\phi) + d_{Z'}(x_\phi,y_\phi) +
\le d_{Z'}(x_0, x_\phi) + d_{Z'}(x_\phi,y_\phi) + \back\\
&\le d_{Z'}(x_0, y_\phi) + \back+\good 
\le d_{Z'}(x_0, s) + d_{Z'}(s,y_\phi)+\back+\good.
\end{align*}
Therefore, subtracting $d_{Z'}(x_0,s)$ from the above inequality, we have
\begin{equation}
\label{eqn:bound-y_phi-and-partial-Y}
d_{Z'}(\partial Y', x_\phi) \le d_{Z'}(s,x_\phi)
\le d_{Z'}(s,y_\phi)+\back+\good
\le d_{Z'}(\partial Y', y_\phi)+\back+\lipconst+\good.
\end{equation}
Combining equations (\ref{eqn:bound-d_Z-in-terms-of-C_Z}), (\ref{eqn:go_to_preimages}) and (\ref{eqn:bound-y_phi-and-partial-Y}), we see that
\[\bt_Z^\phi \le d_{Z'}(\partial Y', y_\phi) + C_Z + \rafi + \back +\lipconst+2\good.\]
But since $t\in [y_\phi, b_\phi]$ contains $\partial Y'$ in its Bers markings, we have
\[d_{Z'}(\partial Y', y_\phi ) \le d_{Z'}(t, y_\phi) \le d_{Z'}(b_\phi, y_\phi) + \back = \bt_{Z'}^\phi+\back.\]
Combining these and using the definition $C_Z = 6 \bt_{Z'}^\phi+40\good$, we now conclude
\begin{align*}
\bt_Z^\phi 
&\le \bt_{Z'}^\phi + C_Z + 3\good
= \bt_{Z'}^\phi + (6 \bt_{Z'}^\phi + 40\good) + 3\good\\
&= 7\bt_{Z'}^\phi + 43\good < 8 \bt_{Z'}^\phi + 53\good = f\left(\bt_{\phi^{-i-1}(Z)}^\phi\right)
\end{align*}
But this exactly means  $\{\phi^{-i}(Z),\dots,Z\}$ does {\sc not} jump for $\phi$, a contradiction.
\end{proof}

\subsection{Initial domains for compatible subsurfaces}

In the spirit of ``building up'' our maps on larger and larger subsurfaces, for $\phi\Mod(S)$  and a possibly empty subsurface $A\esub S$, we write
\[D_{A}(\phi) = \{Y\in D(\phi) \mid Y\not\esub \phi(A)\}\]
for the backtracking domains whose preimages do not land in $A$. We emphasize that if $Y\in D(\phi)$ is an annulus, then  $Y\esub \phi(A)$ if and only if 
either  $Y$ is isotopic to an annular component of $\phi(A)$ or else $Y\esubn V$ for some component $V$ of $\phi(A)$.
Notice that for the empty subsurface we have $D_\emptyset(\phi) = D(\phi)$ and for the whole surface $A = S$ we have $D_S(\phi) = \emptyset$. 
We view $D_A(\phi)$ as the backtracking domains that we still need to account for once we ``know'' $\phi$ on $A$.

We use the notation $D_A^i(\phi) = (D_A(\phi))^i$ for minimal elements as in Definition~\ref{def:backtracking_orders}. 

\begin{definition}[Initial domains]
\label{def:initial_domain}
Given $\phi$, we say a domain $V\esub S$ is \define{$\phi$--initial} for a subsurface $A\esub S$ if 
$V\in D_A^\star(\phi) = \bigcap_{i=0}^3 D_A^i(\phi)$. Note that we do not require minimality with respect to $\lhd_{\tilde{2}}$.
\end{definition}

We will consider subsurfaces $A$ that are constructed by successively adding initial domains, as follows:

\begin{definition}[Compatible]
\label{def:compatible}
A subsurface $A\esub S$ is  \define{compatible} with $\phi$ if  either $A = S$ or else 
$A = \phi\inv(Z_0\ecup \dotsb \ecup Z_m)$ for some sequence $\emptyset = Z_0,Z_1,\dots,Z_m$ in which each  $Z_{i+1}$, with $0\le i < m$, is a $\phi$--initial domain for $A_{i} = \phi\inv(Z_0\ecup \dotsb \ecup Z_i)$ (Recall from Lemma~\ref{lem:BKMM-subsurface-operations} that $Z_0\ecup\dotsb\ecup Z_i$ is the subsurface filled by $Z_0,\dots,Z_i$).
\end{definition}

The next two lemmas explain how subsets $D_A(\phi)$ and relations $\lhd_i$ interact:
\begin{lemma}
\label{lem:compatibility_gives_control_on_orders}
If $A\esub S$ is compatible with $\phi$ and domains $Y,Z\in D(\phi)$ satisfy $Y\lhd_1 Z$ and $Y\in D_A(\phi)$, then $Z\in D_A(\phi)$ as well.
\end{lemma}
\begin{proof}
If $A = S$ then $D_A(\phi)$ is empty and the statement is vacuous. So assume $A\esubn S$ and 
let $\emptyset= Z_0,\dots, Z_m$ be the sequence of domains witnessing compatibility of $A$, so that $\phi(A) = Z_0\ecup\dotsb\ecup Z_m$, and set $B = \phi(A)$.
We must prove $Z\in D_A(\phi)$, which is equivalent to saying $Z\not\esub B$. 
By means of contradiction, let us suppose $Z\esub B$.
Since $Y\cut Z$ but $Y\not\esub B$, we must have  $Y\cut B$. Since the $Z_j$ fill $B$, it must be that $Y\cut Z_j$ for some $j$ (otherwise the subsurface $B$ would be contained in $S\setminus \partial Y$). 

By definition of compatibility, $Z_j$ is $\phi$--initial for $A_{j-1}= \phi\inv(Z_0\ecup\dots\ecup Z_{j-1})$. Since $Y\in D_A(\phi)\subset D_{A_{j-1}}(\phi)$, the Definition~\ref{def:initial_domain} of initial ensures $Z_j \tol  Y$ along $[b_\phi,y_\phi]$. Since $Y \tol  Z$ along $[b_\phi,y_\phi]$ by assumption, it follows that $d_Y(\partial Z_j, \partial Z) \ge d_Y(b_\phi,y_\phi) - 2\rafi/3$. 
However, as $\partial Z$ and $\partial Z_j$ are both disjoint from $\partial B$, we also have  $d_Y(\partial Z, \partial Z_j) \le 4$. But this contradicts the estimate $d_Y(b_\phi,y_\phi)\ge \good$ from Lemma~\ref{lem:which_domains_backtrack}.
\end{proof}

\begin{lemma}
\label{lem:0-minimality_preserved}
Let $A\esub S$ be compatible for $\phi \in \Mod(S)$  and let $Y,Z\in D_A(\phi)$.  If $Z\in D^0_A(\phi)$ and $Y\lhd_i Z$ for $i=1$ or $i = \tilde{2}$, then $Y\in D^0_A(\phi)$ as well.
\end{lemma}
\begin{proof}
By means of contradiction, suppose $Y\notin D_A^0(\phi)$, meaning that there exists some $Y'\in D_A(\phi)$ with  $Y'\lhd_0 Y$. Hence by definition $Y'= \phi^{-j}(Y)$ for some $j\ge 1$ with $\{\phi^{-j}(Y),\dotsc, Y\}\subset D(\phi)$. 
By Lemma~\ref{lem:chains_and_zero_order} it follows that $\{\phi^{-j}(Z),\dots, Z\}\subset D(\phi)$ as well and therefore that $Z' = \phi^{-j}(Z) \lhd_0 Z$.

To obtain a contradiction, it suffices to show  $Z'\in D_A(\phi)$. 
First suppose $Y \lhd_1 Z$. Since $Y\in D(\phi)$, Lemma~\ref{lem:which_domains_backtrack} and Proposition~\ref{prop:good-point-features}(\ref{gp-backtrack-implies_close_xy}) imply $d_Y(x_\phi,y_\phi)\le \good$. Therefore $d_Y(\partial Z,x_\phi)\le d_Y(\partial Z, y_\phi)+\good \le 2\good$ by time ordering. 
Thus using Lemma~\ref{lem:good_points_fixed_in_backtracking} (since  $\{\phi^{-j}(Y),\dots Y\}\subset D(\phi)$), we find that
\[d_{Y'}(\partial Z', y_\phi) \le d_{Y'}(\partial Z', \phi^{-j}(x_\phi))+\good = d_Y(\partial Z, x_\phi) + \good \le 3\good.\]
Since $Y',Z'\in D(\phi)$ with $Y'\cut Z'$, the domains are time-ordered along $[b_\phi,y_\phi]$. The above inequality and fact $d_{Y'}(b_\phi,y_\phi) \ge 4\good$ (Lemma~\ref{lem:which_domains_backtrack}) together preclude $d_{Y'}(b_\phi, \partial Z') < \rafi/3$. This forces $Y' \tol Z'$ along $[b_\phi, y_\phi]$ by Lemma~\ref{lem:characterize_timeorder}. That is, $Y'\lhd_1 Z'$ in $D(\phi)$. Since $Y'\in D_A(\phi)$, Lemma~\ref{lem:compatibility_gives_control_on_orders} therefore implies $Z'\in D_A(\phi)$ as well. If instead $Y\lhd_{\tilde{2}} Z$, then $Y'\esubn Z'$. Since $Y'\in D_A(\phi)$, we have $Y'\not\esub \phi(A)$ and consequently $Z'\not\esub \phi(A)$ as well. Thus $Z'\in D_A(\phi)$ as required.
\end{proof}

The next lemma ensures $\phi$--initial domains exists whenever $D_A(\phi)$ is nonempty.

\begin{lemma}[Initial domains exist]
\label{lem:initial_existence}
Let $A\esub S$ be a compatible subsurface for $\phi$. If $D_A(\phi)$ is nonempty, then $D_A^\star(\phi)$ is nonempty as well. 
\end{lemma}
\begin{proof}
Choose a domain $Z'\in D_A(\phi)$ maximizing the quantity $\plex{Z'}$. Since $\lhd_0$ restricts to a partial order on the finite set $D_A(\phi)$, there exists  $Z\in D_A^0(\phi)$ with $Z = Z'$ or $Z \lhd_0 Z'$. By definition of $\lhd_0$, we have $Z = \phi^{-j}(Z')$ for some $j\ge 0$.

\textbf{Case 1: $Z\in D_A^1(\phi)$}:
Let $Y\in D_A^\dag(\phi)$ be the domain provided by Lemma~\ref{lem:triminimal}. If $Y\esup Z$ then we must have $Y = Z\in D_A^0(\phi)$ by the maximality of $\plex{Z}$. Hence $Y\in D_A^\star(\phi)$ and we are done. Otherwise $Y\esubn Z$ with $Y\lhd_{\tilde{2}} Z$. Since $Z\in D_A^0(\phi)$, Lemma~\ref{lem:0-minimality_preserved} now implies $Y\in D_A^0(\phi)$ and we again conclude $Y\in D_A^\star(\phi)$.

\textbf{Case 2: $Z \notin D_A^1(\phi)$}:
Since $\lhd_1$ is a partial order on $D_A(\phi)$, there necessarily exists some $V\in D_A^1(\phi)$ with $V \lhd_1 Z$. Notice that $V\in D_A^0(\phi)$ by Lemma~\ref{lem:0-minimality_preserved}. Since $V\in D_A^1(\phi)$, we may invoke Lemma~\ref{lem:triminimal} to obtain a domain $Y\in D_A^\dag(\phi)$.
If $Y\esubn V$ with $Y\lhd_{\tilde{2}} V$, then the fact $V\in D_A^0(\phi)$ with Lemma~\ref{lem:0-minimality_preserved} implies that $Y\in D_A^0(\phi)$ and hence $Y\in D_A^\star(\phi)$. If instead $Y\esup V$, then the fact $V\cut Z$ ensures we cannot have $Y\disj Z$ or $Y\esub Z$. But $Z\esub Y$ is also ruled out by the maximality of $\plex{Z}$. The only remaining possibility is $Y\cut Z$. Since $V \tol  Z$, Corollary~\ref{cor:time-order_subsurf} implies that $Y \tol  Z$, which is to say $Y\lhd_1 Z$. Therefore $Y\in D_A^0(\phi)$ by Lemma~\ref{lem:0-minimality_preserved} and we have found the desired domain in $D_A^\star(\phi)$.
\end{proof}

The next lemma says, in light of Corollary~\ref{cor:bounded_domains_with_bounded_combinatorics}, that there are uniformly boundedly many options for the image $\phi(A)$ of a compatible subsurface.

\begin{lemma}[Bounded compatibility]
\label{lem:adjusted_branch_near_initials}
If  $A\esub S$ is a compatible subsurface for $\phi$, then  $B = \phi(A)$  satisfies $d_V(b_\phi, \partial B) \ladd_\good 0$ for every domain $V\esub S$.
\end{lemma}
\begin{proof}
If $A = S$ then $\partial S$ is empty and there is nothing to prove. So suppose $A\esubn S$ and let $\emptyset=Z_0,\dots,Z_m$ be the domains witnessing the compatibility of $A$.
If $V\esub B$ or $V\disj B$ there is nothing to prove, since then $d_V(b_\phi,\partial B)$ is just $\diam_{\cc(V)}\pi_V(b_\phi)\le \lipconst$. Hence we assume $\partial B$ projects to $V$. 
Since $B$ is filled by the $Z_i$, we may choose $1\le j \le m$ such that $\partial Z_j$ projects to $V$. 
As the multicurves $\partial Z_j$ and $\partial B$ are evidently disjoint, it thus suffices to bound $d_V(b_\phi, \partial Z_j)$. 
Setting $B_j = Z_0 \ecup \dotsb \ecup Z_{j-1}$, by definition of compatibility we then know that $Z_j\in D(\phi)$ is  initial for $A_j = \phi\inv(B_j)$.

Let us first suppose $V\in D(\phi)$.
Then it must be that $V\in D_A(\phi)$ since $V\esub \phi(A)$ was excluded above. 
Since $A_{j}\esub A$, we also have  $V\in D_A(\phi)\subset D_{A_{j}}(\phi)$. 
As $Z_j$ is initial for $A_j$, 
 if $V\cut Z_j$, then the failure of $V\lhd_1 Z_j$ implies $Z_j \tol  V$ along $[b_\phi,y_\phi]$ and hence $d_V(b_\phi,\partial Z_j) \le \rafi$. If instead $Z_j \esubn V$, the failure of $V \lhd_3 Z_j$ implies $d_V(b_\phi, \partial Z_j) \le C_0 + 2\rafi$.
We are thus done in this case, as $V\disj Z_j$ and  $V\esub Z_j$ are precluded by $\partial Z_j$ projecting to $V$.

Next suppose $V\notin D(\phi)$ so that $d_V(y_\phi,b_\phi) = \bt^\phi_V\le f^\btbound(\good)$ by Lemma~\ref{lem:which_domains_backtrack}. Since $Z_j\in D(\phi)$ ensures $Z_j$ has an active interval along $[y_\phi,b_\phi]$, there is a point $t\in [b_\phi,y_\phi]$  with $\partial Z_j\subset \base(\mu_t)$ and consequently (by Theorem~\ref{thm:back})
\[d_V(b_\phi,\partial Z_j) \le d_V(b_\phi, t) \le d_V(b_\phi, t) + d_V(t,y_\phi) \ladd d_V(b_\phi, y_\phi)\ladd_\good 0. \qedhere\]
\end{proof}

\subsection{Branch families}
\label{sec:brancH_families}
A family $\mathcal{F}\subset \Mod(S)$ of mapping classes will be called an \define{$(a,b)$--branch family}, or simply a \define{branch family} when the points $a,b\in \T(S)$ are understood, if $a_\phi = a$ and $b_\phi = b$ for all $\phi\in \mathcal{F}$. 
Let us call the \define{displacement} of the family the quantity $r(\mathcal{F}) = \max \{d_{\T(S)}(x_0,\phi(x_0)) \mid \phi \in \mathcal{F} \}$.

With this language, the goal in Theorem~\ref{thm:bound_multiplicity} is to bound the size of an arbitrary branch family $\mathcal{F}$ by a polynomial in its displacement $r(\mathcal{F})$. This will be accomplished by repeatedly imposing additional conditions that have the effect of partitioning $\mathcal{F}$ into smaller subfamilies. Each step will produce a controlled number of subfamilies, and there will be a bounded number of steps.

\subsection{Annular profile}
The first condition we impose concerns annuli:

\begin{definition}
Let us call an annulus $V\esub S$ \define{relevant for $\phi$} if $d_V(a_\phi,b_\phi) > 3\good$. Each relevant annulus is either \define{big for $\phi$} if $d_V(x_\phi,y_\phi) > \good$ or else \define{small for $\phi$} if $d_V(x_\phi,y_\phi) \le \good$. The \define{annular profile of $\phi$} is its set of big and small annuli.
\end{definition}

\begin{remark}
Note that if an annulus satisfies $d_V(x_\phi,y_\phi)> 5\good$, then it is automatically relevant and hence big, since Proposition~\ref{prop:good-point-features}(\ref{gp-backtrack-implies_close_xy}) implies  $d_V(y_\phi,b_\phi)$ and $d_V(x_\phi,a_\phi)$ are less than $\good$ so that $d_V(a_\phi,b_\phi) > 3\good$ by the triangle inequality.
\end{remark}

Note that since relevance is defined in terms of the branch points $a_\phi,b_\phi$, all elements $\phi$ in a branch family have the exact same set of relevant annuli.

\begin{definition}
An $(a,b)$--branch family $\mathcal{F}$ is \define{annular-compliant} if all elements have the same annular profile. That is, if each $\phi\in \mathcal{F}$ induces the same partition of the set $\{V\esub S\text{ annular} \mid d_V(a,b)> 3\good\}$ into big and small annuli.
\end{definition}

\begin{lemma}
\label{lem:bounded:distance:coherent}
If $\mathcal{F}$ is an annular-compliant branch family, then for all $\phi,\psi\in\mathcal{F}$ and every domain $V\esub S$ we have 
$|d_V(x_\phi,x_\psi)-d_V(y_\phi,y_\psi)| \le 10\good$.
\end{lemma} 
\begin{proof}
Write $a = a_\theta$ and $b = b_\theta$ for the branch points common to all $\theta\in \mathcal{F}$. First suppose $V$ is nonannular. If $d_V(a,b) > 2\good$ then Proposition~\ref{prop:good-point-features}(\ref{gp-nonannular}) implies  $d_V(x_\theta, y_\theta)> \good$ for all $\theta\in \mathcal{F}$. Hence Proposition~\ref{prop:good-point-features}(\ref{gp-backtrack-implies_close_xy}) implies $d_V(x_\theta, a)\le \good$ and $d_V(y_\theta, b)\le \good$ for each $\theta$. Thus, by the triangle inequality, $d_V(x_\phi, x_\psi)$ and $d_V(y_\phi, y_\psi)$ are at most $2\good$, and their difference is at most $2\good$. Otherwise $d_V(a,b) \le 2\good$ and Proposition~\ref{prop:good-point-features}(\ref{gp-nonannular}) gives  $d_V(x_\theta, y_\theta)\le 3\good$ for each $\theta$. Hence the triangle inequality implies $\abs{d_V(x_\phi, x_\psi) - d_V(y_\phi,y_\psi)}\le 6\good$.

Next suppose $V$ is an annulus. If $d_V(a,b)\le 3\good$, then for each $\theta\in \mathcal{F}$ either $d_V(x_\theta,y_\theta)\le \good$ or else Proposition~\ref{prop:good-point-features}(\ref{gp-backtrack-implies_close_xy}) implies $\max\{d_V(x_\theta,a),d_V(y_\theta,b)\}\le \good$ so that $d_V(x_\theta,y_\theta)\le 5\good$ by the triangle inequality. Since this holds for all elements of $\mathcal{F}$, as above we obtain $\abs{d_V(x_\phi,x_\psi) - d_V(y_\phi,y_\psi)} \le 10\good$. If instead $d_V(a,b)>3\good$, then $V$ is relevant for each $\theta\in \mathcal{F}$. Since $\mathcal{F}$ is annular-compliant, $V$ is either big or small for all $\theta\in \mathcal{F}$. First suppose it is big, then for each $\theta\in \mathcal{F}$ has $d_V(x_\theta,y_\theta)> \good$ and therefore $\max\{d_V(x_\theta,a),d_V(y_\theta,b)\}\le \good$ by Proposition~\ref{prop:good-point-features}(\ref{gp-backtrack-implies_close_xy}). Hence, as above, $d_V(x_\phi,x_\psi)$ and $d_V(y_\phi,y_\psi)$ are at most $2\good$ and their difference is at most $2\good$. Finally suppose $V$ is small, then $d_V(x_\theta,y_\theta)\le \good$ for all $\theta\in \mathcal{F}$ and we again conclude $\abs{d_V(x_\phi,x_\psi) - d_V(y_\phi,y_\psi)}\le 2\good$ by the triangle inequality. 
\end{proof}

\begin{proposition}
\label{prop:annular-compliance}
There is a degree--$2\plex{S}$ polynomial $q_0$  such that any branch family $\mathcal{F}$ may be partitioned into at most $q_0(r(\mathcal{F}))$ subfamilies $\mathcal{F}'$ that are each annular-compliant. In fact, one may explicitly take $q_0(r) = r^{2\plex{S}}$.
\end{proposition}
\begin{proof}
Let $\mathcal{R}$ denote the common set $\{V\esub S\text{ annular}\mid d_V(a,b)>3\good\}$ of relevant annuli. This set is partially ordered by time-order along $[a,b]$.

For each $\phi\in \mathcal{F}$, the tuple $(x_0,a,b,\phi(x_0))$ is $\good$--aligned by Proposition~\ref{prop:good-point-features}(\ref{gp-strong-align}). Hence each $V\in \mathcal{R}$ has $d_V(x_0,\phi(x_0)) > \good$. We may assume that the uniform constant $\good$ from Proposition~\ref{prop:good-point-features} was chosen so that $\log(\good)$ is larger than some fixed threshold $T$ in the distance formula (Theorem~\ref{thm:distance_formula}) and moreover contributes at least $1$ to the lower bound on $d_{\T(S)}(x_0,\phi(x_0))$ (i.e., that $\log(\good)\ge 2K$). 
It thus follows that $\abs{\mathcal{R}} \le d_{\T(S)}(x_0,\phi(x_0)) \le r(\mathcal{F})$.

Given $\phi\in \mathcal{F}$, let $\mathcal{R}_\phi \subset \mathcal{R}$ be the set of big annuli for $\phi$. Let $\mathcal{R}_\phi^-$ and $\mathcal{R}_\phi^+$ respectively denote the sets of minimal and maximal elements of $\mathcal{R}_\phi$ with respect to time-order on $\mathcal{R}$. Note that all elements of $\mathcal{R}_\phi^+$ and of $\mathcal{R}_\phi^-$ are mutually disjoint, since if two annuli $V,W\in \mathcal{R}_\phi$ satisfy $V\cut W$ they cannot be both be minimal nor both maximal. Thus $\abs{\mathcal{R}_\phi^\pm} \le \plex{S}$. Finally let $\overline{\mathcal{R}}_\phi = \mathcal{R}_\phi^-\cup \mathcal{R}_\phi^+$ denote the union.

\begin{claim}
The set $\overline{\mathcal{R}}_\phi$ determines the big set for $\phi$ via the rule: $Y\in \mathcal{R}$ satisfies  $Y\in \mathcal{R}_\phi$ if and only if $Y\in \overline{\mathcal{R}}_\phi$ or if there exists, $W,Z\in \overline{\mathcal{R}}_\phi$ so that $W \tol Y \tol Z$.
\end{claim}
\begin{proof}[Proof of Claim]
Clearly if $Y\in \mathcal{R}_\phi$, then either $Y$ is minimal or maximal with respect to time order, or else $W \tol Y \tol Z$ for some $W\in \mathcal{R}_\phi^-$ and some $Z\in \mathcal{R}_\phi^+$. For the converse, $\overline{\mathcal{R}}_\phi\subset \mathcal{R}_\phi$ by definition, so it suffices to suppose $Y\in\mathcal{R}$ satisfies $W \tol Y \tol Z$ along $[a,b]$ for some $W,Z\in \overline{\mathcal{R}}_\phi$. Then $d_W(a,\partial Y) \ge 3\good - \rafi/3$ by Lemma~\ref{lem:characterize_timeorder}. Now, since $d_W(x_\phi,y_\phi)> \good$, Proposition~\ref{prop:good-point-features}(\ref{gp-backtrack-implies_close_xy}) implies $d_W(a,x_\phi) < \good$. This means that $Y$ cannot have an active interval along $[a,x_\phi]$, for otherwise we would have $d_W(a,x_\phi) \ge d_W(a, \partial Y) - \back > 3\good -\back-\rafi/3 > \good$. Symmetrically $Y$ cannot have an active interval along $[b,y_\phi]$. Thus $d_Y(a,x_\phi),d_V(b,y_\phi)\le \rafi$ and the triangle inequality gives $d_Y(x_\phi,y_\phi) \ge 3\good - 2\rafi > \good$ so that $Y\in \mathcal{R}_\phi$.
\end{proof}

The claim implies that if  $\overline{\mathcal{R}}_\phi$ and $\overline{\mathcal{R}}_\psi$ agree then $\phi$ and $\psi$ have the same annular profile. We may therefore partition $\mathcal{F}$ according the subsets $\overline{\mathcal{R}}_\phi$. Since $\abs{\mathcal{R}}\le r(\mathcal{F})$ and each set $\overline{\mathcal{R}}_\phi$ has size at most $2\plex{S}$, there are at most $r(\mathcal{F})^{2\plex{S}}$ possibilities for the subset $\overline{\mathcal{R}}_\phi$. Thus we obtain a partition of $\mathcal{F}$ into at most $r(\mathcal{F})^{2\plex{S}}$ subfamilies that are each annular-compliant.
\end{proof}

\subsection{Coherence} 
Our second condition on branch families involves agreement on a subsurface:

\begin{definition}[Coherence]
\label{def:coherence}
Given a subsurface $A\esub S$, we say a family $\mathcal{F}\subset \Mod(S)$ is \define{$A$--coherent} if
\begin{itemize}
\item it is a branch family 
\item it is annular-compliant 
\item $A$ is compatible with each $\phi\in \mathcal{F}$
\item all elements $\phi,\psi\in\mathcal{F}$ agree on $A$.
\end{itemize}
\end{definition}

\begin{definition}[Pre-initial]
\label{def:initial_redux}
Let $\mathcal{F}$ be  an $A$--coherent family. 
We say a domain $Y\esub S$ is \define{pre-initial} for $\mathcal{F}$ if it is the preimage of some initial domain, that is, if $Y = \phi\inv(Z)$ some element $\phi\in \mathcal{F}$ and domain $Z\in D(\phi)$ that is $\phi$--initial for $A$.
\end{definition}

\begin{lemma}[Boundedly many pre-initial domains]
\label{lem:boundedly many pre-initial}
There is a degree $1$ polynomial $q_1$ such that for any subsurface $A\esub S$ and $A$--coherent family $\mathcal{F}$, the cardinality of the set of pre-initial domains for $\mathcal{F}$ is at most $q_1(r(\mathcal{F}))$.
\end{lemma}
\begin{proof}
Let $\mathcal{P} = \{Y_1,Y_2,\dots\}$ be the set of pre-initial domains, and for each $i$ choose $\phi_i\in \mathcal{F}$ such that $Y_i = \phi_i\inv(Z_i)$ for some $\phi_i$--initial domain $Z_i\in D(\phi_i)$. Let $k_i < \btbound$ be the index of $Z_i$, that is, the minimal $k_i\ge 0$ so that $\phi_i^{-k_i -1}(Z_i) \notin D(\phi_i)$. For each $k$, let $\mathcal{P}_k$ be the subset of pre-initial domains $Y_i$ for which $k_i = k-1$. This gives a partition $\mathcal{P} = \mathcal{P}_1\sqcup\dots\sqcup P_{\btbound}$. Restricting to one subcollection $\mathcal{P}_k$ we henceforth assume $k = k_i+1$ for all $i$. 

Let us write $V_i = \phi_i^{-k}(Z_i)$ so that $V_i\notin D(\phi_i)$. Then for each $i$ we have $\phi_i(V_i),\dots, \phi_i^{k}(V_i)\in D(\phi_i)$, with $Z_i = \phi_i^k(V_i)$ being $\phi_i$--initial.
The $\lhd_0$--minimality of $Z_i$ in $D_A(\phi_i)$ implies that $\phi_i(V_i),\dots,\phi_i^{k-1}(V_i)\notin D_A(\phi_i)$.
Thus for all $0 \le n \le k-2$ we have $\phi_i^{n+1}(V_i)\in D(\phi_i)\setminus D_A(\phi_i)$, meaning $\phi_i^{n+1}(V_i) \esub \phi(A)$ or equivalently $\phi_i^n(V_i)\esub A$. 
By assumption, all the maps $\phi\in \mathcal{F}$ agree on $A$; let us write $\psi\colon A\to B$ for this common restriction to $A$. With this notation we conclude that 
\[Y_i = \phi_i\inv(Z_i) =  \psi^{k-1}(V_i)\]
for all $i$. Since the domains $Y_i$ are all distinct by assumption, it follows that the domains $V_i$ and $V_j$ are distinct whenever $i\ne j$.

By coherence, the points $a_\phi,b_\phi$ for $\phi\in \mathcal{F}$ all agree; let us call these common points $a,b$. 
By Definition~\ref{def:backtracking_domains} of backtracking, having $V_i\notin D(\phi_i)$ and $\phi_i(V_i)\in D(\phi_i)$ means that 
\[\bt_{\phi_i(V_i)}^{\phi_i}\ge f( \bt_{V_i}^{\phi_i}) = 8 \bt_{V_i}^{\phi_i} + 53\good, \]
which by Proposition~\ref{prop:good-point-features}(\ref{gp-jump-implies-see-preimage}) implies $d_{V_i}(x_0, b)\ge 2\good$.
Fixing $i=1$ and considering the element $\phi_1$, we know the triple $(x_0, b, \phi_1(x_0))$ is $\good$--aligned in all domains. Therefore, for all $j\ge 1$ we have
\[d_{V_j}(x_0, \phi_1(x_0)) \ge d_{V_j}(x_0, b) - \good \ge \good.\]
We note that we may choose $\good$ large enough so that  $\log (\good)$ is a valid threshold in the distance formula (Theorem~\ref{thm:distance_formula}).  From this  the number of domains $U\esub S$ with $d_U(x_0,\phi_1(x_0))\ge \good$ is bounded linearly in terms of $d_{\T(S)}(x_0,\phi_1(x_0))\le r(\mathcal{F})$, which finishes the proof of the lemma. 
\end{proof}

\subsection{Supercoherence}
We next consider coherent families with additional data:

\begin{definition}[Supercoherence]
\label{def:supercoherence}
A family $\mathcal{F}$ will be called \define{supercoherent} for a subsurface $A\esub S$ and domain $Y\esub S$ if $\mathcal{F}$ is $A$--coherent and for all $\phi,\psi\in \mathcal{F}$:
\begin{itemize}
\item $Z_\phi = \phi(Y)$ is $\phi$--initial for $A$,
\item $\phi(A') = \psi(A')$ where $A'$ is the subsurface filled by $A$ and $Y$,
\item if $Y$ is nonannular, then $d_{Z_\phi}(b_\phi, y_\phi) = d_{Z_\psi}(b_\psi,y_\psi)$,
\item if $Y$ is annular, then $d_{Z_\phi}(b_\phi,\phi(x_0)) = d_{Z_\psi}(b_\psi,\psi(x_0))$.
\end{itemize}
Note that $Z_\phi$ being initial implies $Z_\phi\not\esub \phi(A)$ and thus $Y\not\esub A$. Also note  that the subsurfaces $B = \phi(A)$ and $B' = \phi(A')$ (filled by $B$ and $Z_\phi$) are independent of $\phi$. We also allow $Y$ to denote the empty domain  and say $\mathcal{F}$ is $(A,\emptyset)$--supercoherent to mean that it is $A$--coherent but that $D_A(\phi)$ is empty for each $\phi\in \mathcal{F}$.
\end{definition}

Lemmas~\ref{lem:adjusted_branch_near_initials} and \ref{lem:boundedly many pre-initial} allow us to easily partition coherent families into boundedly many supercoherent ones:

\begin{lemma}
\label{lem:supercoherent}
There is a degree $2$ polynomial $q_2$ such that for any subsurface $A\esub S$, any $A$--coherent family $\mathcal{F}$  may be partitioned into at most $q_2(r(\mathcal{F}))$ subfamilies $\mathcal{F}'$ that are each $(A,Y')$--supercoherent for some domain $Y'$.
\end{lemma}
\begin{proof}
The elements $\phi\in \mathcal{F}$ for which $D_A(\phi)$ is empty comprise a subset of $\mathcal{F}$ that is $(A,\emptyset)$--supercoherent. Excising these, we henceforth suppose each $\phi\in \mathcal{F}$ has $D_A(\phi)$ nonempty and, in particular, that $A\ne S$. Accordingly, for each $\phi\in \mathcal{F}$ we may use Lemma~\ref{lem:initial_existence} to choose some initial domain $Z_\phi\in D_A^\ast(\phi)$.  We then set $Y_\phi = \phi\inv(Z_\phi)$ and let $B'_\phi$ denote the subsurface filled by $Z_\phi$ and $B = \phi(A)$. As $\phi\inv(B'_\phi)$ is clearly compatible for $\phi$ by construction, Lemma~\ref{lem:adjusted_branch_near_initials} and Corollary~\ref{cor:bounded_domains_with_bounded_combinatorics} provide a uniform bound $k_0$ (depending only on $\good$) on the number of domains $B'_\phi$ produced in this way. Similarly Lemma~\ref{lem:boundedly many pre-initial} says there are at most $q_1(r(\mathcal{F}))$ possibilities for the domain $Y_\phi$. Hence after partitioning into $k_0q_1(r(\mathcal{F}))$ subfamilies we may assume $Y = Y_\phi$ and $B' = B'_\phi$ are independent of $\phi$.

Now, if $Y$ (and thus each $Z_\phi$) is nonannular, then Proposition~\ref{prop:good-point-features}(\ref{gp-backtracking-bounded-by-deficit}) provides a uniform constant $k_1$ so that each integer $d_{Z_\phi}(b_\phi,y_\phi)=\bt^\phi_{Z_\phi}$ is at most $k_1 d_{\T(S)}(x_0,\phi(x_0)) + k_1$. Thus we may further partition into at most $k_1 r(\mathcal{F})+k_1$ subfamilies so that $d_{Z_\phi}(b_\phi, y_\phi)$ is independent of $\phi$. If instead $Y$ is annular, then for all $\phi\in \mathcal{F}$ we have $\bt^\phi_{Z_\phi} \ge\good$ by Lemma~\ref{lem:which_domains_backtrack} since $Z_\phi\in D(\phi)$. 
Therefore Proposition~\ref{prop:good-point-features}(\ref{gp-annular_backtracking_bound}) says $d_{Z_\phi}(b_\phi,\phi(x_0))\le\good$  and we may further partition into at most $\good$ subfamilies so that  $d_{Z_\phi}(b_\phi, \phi(x_0))$ is independent of $\phi$. Each subfamily $\mathcal{F}'$ produced in this way is then  $(A,Y)$--supercoherent, where $Y = Y_\phi$ for any $\phi\in \mathcal{F}'$.
\end{proof}

\subsection{Extending supercoherence}
The previous section shows that $A$--coherent families can be refined into $(A,Y)$--supercoherent ones. The remaining ingredient is to show that each $(A,Y)$--supercoherent family can be further refined into $A'$--coherent families for the enlarged subsurface $A'$ filled by $Y$ and $A$, or $A' = S$ in the case $Y = \emptyset$. This is the heart of our reconstructive argument.

\begin{proposition}
\label{prop:fix_on_domain}
There is a constant $q_3$ such that every $(A,Y)$--supercoherent family $\mathcal{F}\subset \Mod(S)$ can be partitioned into at most $q_3$ subfamilies $\mathcal{F}'$ that are each $A'$--coherent, where $A' = S$ when $Y$ is the empty domain and otherwise $A'$ is the subsurface filled by $A$ and $Y$. In particular, in the latter case  $\plex{A'} > \plex{A}$.
 \end{proposition}
\begin{proof}
Since $\mathcal{F}$ is $A$--coherent, we know $a_\phi = a_\psi$ and $b_\phi = b_\psi$ for all $\phi,\psi\in \mathcal{F}$; let us write $a = a_\phi$ and $b = b_\phi$ for these common points. 
The definition of supercoherence ensures $A'$ is compatible with each $\phi\in \mathcal{F}$; indeed, if $Y = \emptyset$ and $A' = S$ compatibility is automatic, and otherwise it follows from the compatibility of $A$ (Definition~\ref{def:compatible}) and fact that $Z_\phi = \phi(Y)$ is $\phi$--initial for $A$.
Supercoherence also gives $\phi(A') = \psi(A')$ for all $\phi,\psi\in \mathcal{F}$; let us call this common subsurface $B'$. 
Hence proving $A'$--coherence amounts to establishing agreement on $A'$. This is equivalent to showing the subset $\{(\phi\psi\inv)\vert_{B'} \mid \phi,\psi\in \mathcal{F}\}$ of $\Mod(B')$ has uniformly bounded cardinality. By Lemma~\ref{lem:bounding_submaps}, for this it suffices to prove that
\begin{equation}
\label{eq:boundeddistance}
d_W(\phi\psi\inv(b), b)\ladd_\good 0 \qquad\text{for all $\phi,\psi\in \mathcal{F}$ and all domains $W\esub B'$.}
\end{equation}
To set notation, for $W\esub B'$ and $\phi,\psi\in \mathcal{F}$ we will write $V = \phi\inv(W)\esub A'$ and $W' = \psi(V) \esub B'$. Note that then 
\begin{equation}
\label{eqn:options_for_what_to_bound}
d_W(\phi\psi\inv(b),b) = d_V(\phi\inv(b), \psi\inv(b)) = d_{W'}(b,\psi\phi\inv(b));
\end{equation}
hence we are free to bound either of these three quantities. 
We first dispense with the case that $Y$ is empty and, accordingly $B'=S=A'$:

\begin{claim}
\label{claim:supercoherent_when_Y_empty}
If $Y=\emptyset$, then $d_W(\phi\psi\inv(b),b)\ladd_\good 0$ for all $W\esub B'$ and $\phi,\psi\in\mathcal{F}$.
\end{claim}
\begin{proof}
First suppose $W\in D(\phi)$. Since $D_A(\phi) = \emptyset$ by definition of supercoherence, evidently $W\notin D_A(\phi)$ which means $W\esub \phi(A)$ and hence $V\esub A$. Since $\phi$ and $\psi$ agree on $A$, we may apply the equal isometries $\phi=\psi \colon \cc(V)\to \cc(W)$ to conclude
\[d_V(\phi\inv(b),\psi\inv(b)) = d_{\phi(V)}(\phi\phi\inv(b), \psi\psi\inv(b)) = d_W(b,b) \le \lipconst.\]
If $W'\in D(\phi)$ we similarly conclude  $d_V(\phi\inv(b),\psi\inv(b))\le \lipconst$.

It remains to suppose $W\notin D(\phi)$ and $W'\notin D(\psi)$. By Lemma~\ref{lem:which_domains_backtrack}, this gives
\[d_V(\phi\inv(b),x_\phi) = d_W(b,y_\phi)\ladd_\good 0 \quad\text{and}\quad d_V(\psi\inv(b),x_\psi) = d_{W'}(b, y_\psi)\ladd_\good 0.\]
Hence by the triangle inequality it suffices to bound $d_V(x_\phi, x_\psi)$. 

Now, observe that $\phi^{-k\btbound}(W)\in D(\phi)$ can hold for at most one of $k\in \{1,2,3\}$. Indeed, if $W$ has infinite $\phi$--orbit, this follows from Proposition~\ref{prop:good-point-features}(\ref{gp-no-total-backtracking}), and if $W$ has finite $\phi$--orbit, our choice of $\btbound$ as a multiple of the orbit size (Proposition~\ref{prop:good-point-features}(\ref{gp-no-total-backtracking})) implies $\phi^{-k\btbound}(W)\notin D(\phi)$ for all $k\in \Z$. Similarly, we can have $\phi^{-k\btbound}(W')\in D(\psi)$ for at most one $k\in \{1,2,3\}$. It follows that $\phi^{1-k\btbound}(V)= \phi^{-k\btbound}(W) \notin D(\phi)$ and $\psi^{1-k\btbound}(V) =\psi^{-k\btbound}(W')\notin D(\psi)$ simultaneously hold for some $k\in \{1,2,3\}$. Therefore, there exists a smallest integer $n\ge 0$ such that both $\phi^{-n}(V)\notin D(\phi)$ and $\psi^{-n}(V)\notin D(\psi)$, and this integer satisfies  $n\le 3\btbound$.

We claim that $\phi^{-i}(V) = \psi^{-i}(V)$ for each $0 \le i \le n$. Indeed, when $0 \le i < n$ either $\phi^{-i}(V)\in D(\phi)$ and hence $\phi^{-i}(V)\esub \phi(A) = B$, or else $\psi^{-i}(V)\in D(\psi)$ and hence $\psi^{-i}(V) \esub\psi(A) = B$. In either case, inductively assuming $\phi^{-i}(V) = \psi^{-i}(V)$, the fact that the maps $\phi\inv,\psi\inv$ agree on $B$ implies that $\phi^{-i-1}(V) = \psi^{-i-1}(V)$.

Let us write $V_i = \phi^{-i}(V) =\psi^{-i}(V)$ for $0\le i \le n$. The facts that $V_n\notin D(\phi)$ and $V_n\notin D(\psi)$ now, by Lemma~\ref{lem:which_domains_backtrack}, give
\[d_{V_n}(y_\phi, y_\psi) \le d_{V_n}(y_\phi, b) + d_{V_n}(b,y_\psi) = \bt^\phi_{V_n} + \bt^\psi_{V_n} \le 2f^{\btbound}(\good)\ladd_\good 0.\]
As we have seen, the maps $\phi,\psi$ agree on each domain $V_n,\dots,V_{1}, V$. Thus for each $0\le i < n$ we may apply the isometries $\phi = \psi \colon \cc(V_{i+1}) \to \cc(V_i)$ and use Lemma~\ref{lem:bounded:distance:coherent} to find
\[d_{V_i}(x_\phi,x_\psi) \le d_{V_i}(y_\phi,y_\psi) + 10\good = d_{V_{i+1}}(x_\phi,x_\psi)+10\good.\]
Inductively chaining these inequalities together now gives the desired bound
\[d_V(x_\phi,x_\psi) \le d_{V_n}(x_\phi,x_\psi)+n10\good \le d_{V_n}(y_\phi,y_\psi)+(n+1)10\good \ladd_\good 0.\qedhere\]
\end{proof}

We henceforth assume that $Y$ is nonempty and thus, by supercoherence, that $Z_{\theta} = \theta(Y)\esub B'$ is initial in $D_A(\theta)$ for each $\theta$.
After partitioning into at most $\btbound$ subfamilies, we may additionally assume each of these initial domains $Z_{\theta}$ has the same index $k<\btbound$ in $D(\theta)$. 
By definition, this means $\{\theta^{-k}(Z_\theta),\dots Z_\theta\}\subset D(\theta)$ but that $\theta^{-k-1}(Z_\theta)\notin D(\theta)$. Recalling that $\theta\inv(Z_\theta) = Y$, the fact that $Z_\theta$ is initial now forces $\theta^{-k+1}(Y),\dots, Y \notin D_A(\theta)$, which means $\theta^{-n}(Y)\esub A$ for each $1\le n \le k$. Since the elements $\theta\in \mathcal{F}$ all agree on $A$, it follows that for each $1\le n \le k$ the common domain $Y_n = \theta^{-n}(Y)$ and map $\theta\colon Y_n \to Y_{n-1}$ are independent of $\theta\in \mathcal{F}$.

The fact that $Y_k = \theta^{-k}(Y) = \theta^{-k-1}(Z_\theta)\notin D(\theta)$ implies by Lemma~\ref{lem:which_domains_backtrack} that
\begin{equation}
\label{eqn:dist_bound_at_index}
d_{Y_k}(y_\theta, b) \le f^{\btbound}(\good) \quad\text{for all }\theta\in \mathcal{F}.
\end{equation}
Since the numbers $d_{Y_k}(y_\theta,b)$ are discrete and uniformly bounded, after further partitioning into at most $f^{\btbound}(\good)$ subfamilies we may assume these distances all agree and hence that the number $R^*\colonequals \bt^\theta_{Y_k} = d_{Y_k}(y_\theta,b)$ is independent of $\theta$.
The constants $C_{Z_\theta} = 6 \bt^\theta_{Y_k} + 40\good = 6R^*+40\good$ are thus also independent of $\theta\in \mathcal{F}$.

For any elements $\phi,\psi\in \mathcal{F}$, by (\ref{eqn:dist_bound_at_index}) and the triangle inequality we have
\begin{equation}
\label{eqn:index_dist_between_y_phi_y_psi}
d_{Y_k}(y_\phi, y_\psi) \le d_{Y_k}(y_\phi,b)+ d_{Y_k}(b,y_\psi) = 2R^* \le 2f^{\btbound}(\good).
\end{equation}
We will instead need bounds on $d_{Y_k}(x_\phi,x_\psi)$ and $d_Y(x_\phi,x_\psi)$:
\begin{claim}
\label{eqn:fixed_pts_close_in_index}
\label{claim:fixed_pts_close_in_index}
$d_Y(x_\phi,x_\psi) -2\good \le d_{Y_k}(x_\phi,x_\psi) \le d_{Y_k}(y_\phi,y_\psi) + 10\good \le 2 R^*+10\good$.
\end{claim}
\begin{proof}
Since $\{Y_{k-1},\dots, Y= \phi^{k-1}(Y_{k-1})\} \subset D(\phi)$, Lemma~\ref{lem:good_points_fixed_in_backtracking} implies that
\[d_{Y_k}(\phi^{-k}(x_\phi), x_\phi) = d_{Y_{k-1}}(\phi^{1-k}(x_\phi),y_\phi)\le \good.\]
Similarly $d_{Y_k}(\psi^{-k}(x_\psi),x_\psi)\le \good$. Thus, the triangle inequality and $k$ successive  applications of the equal isometries $\phi\vert_A=\psi\vert_A \colon \cc(Y_k)\to  \dots \to \cc(Y)$ gives
\[d_{Y_k}(x_\phi,x_\psi) + 2\good  \ge d_{Y_k}(\phi^{-k}(x_\phi),\psi^{-k}(x_\psi)) = d_Y(x_\phi,x_\psi)\]
which is the first claim. The second claim $d_{Y_k}(x_\phi,x_\psi)\le d_{Y_k}(y_\phi,y_\psi)+10\good$ is immediate from Lemma~\ref{lem:bounded:distance:coherent}.
\end{proof}

Using this, we next establish (\ref{eq:boundeddistance}) for the domain $W = Z_\phi$:

\begin{claim}
\label{claim:bound_for_case_of_initial_domain}
$d_{Z_\phi}(\phi\psi\inv(b), b) \le 3 d_Y(x_\phi,x_\psi) + 3\good  \le 6R^*+39\good$ for all $\phi,\psi\in \mathcal{F}$.
\end{claim}
\begin{proof}
The second inequality follows from the bound $d_Y(x_\phi,y_\phi)\le 2R^*+12\good$ in Claim~\ref{claim:fixed_pts_close_in_index}.
Thus it suffices to prove the first inequality.
We know from Proposition~\ref{prop:good-point-features}(\ref{gp-aligned}) that $(y_\psi,b,\psi(x_0))$ is $\good$--aligned; hence:
\[d_{Z_\psi}(y_\psi, b) + d_{Z_\psi}(b,\psi(x_0)) \le d_{Z_\psi}(y_\psi,\psi(x_0))+\good\]
Applying the isometry $\psi\inv\colon\cc(Z_\psi)\to \cc(Y)$ thus gives
\[d_Y(x_\psi, \psi\inv(b)) + d_Y(\psi\inv(b), x_0) \le d_Y(x_\psi,x_0) + \good.\]
 We may swap $x_\psi$ for $x_\phi$ at the cost of $d_Y(x_\psi,x_\phi)$  and then apply $\phi\colon \cc(Y)\to \cc(Z_\phi)$ to conclude $(y_\phi, \phi\psi\inv(b), \phi(x_0))$ is $(2d_Y(x_\phi,x_\psi)+\good)$--aligned in $Z_\phi$:
\[d_{Z_\phi}(y_\phi, \phi\psi\inv(b)) + d_{Z_\phi}(\phi\psi\inv(b),\phi(x_0)) \le d_{Z_\phi}(y_\phi,\phi(x_0)) + 2d_Y(x_\psi,x_\phi) + \good.\]
Since $(y_\phi,b,\phi(x_0))$ is $\good$--aligned (Proposition~\ref{prop:good-point-features}(\ref{gp-aligned})), Lemma~\ref{lem:aligned_proj_to_quasigeod} now implies that $\pi_{Z_\phi}(b)$ and $\pi_{Z_\phi}(\phi\psi\inv(b))$ respectively lie within $\frac{\good}{2} + 4\delta+\lipconst$ and $d_Y(x_\phi,x_\psi) + \frac{\good}{2} + 4\delta+\lipconst$ of any geodesic from $\pi_{Z_\phi}(\phi(x_0))$ to $\pi_{Z_\phi}(x_\phi)$.

If $Y$ is nonannular, then supercoherence implies the distances
\begin{align*}
d_{Z_\phi}(\phi\psi\inv(b),y_\phi) &= d_Y(\psi\inv(b),x_\phi),\quad\text{and}\\
d_{Z_\phi}(b,y_\phi) &= d_{Z_\psi}(b, y_\psi) = d_Y(\psi\inv(b),x_\psi)
\end{align*}
differ by at most $d_Y(x_\phi,x_\psi)$. Otherwise $Y$ is annular and the distances
\[d_{Z_\phi}(b,\phi(x_0))\quad\text{and}\quad d_{Z_{\phi}}(\phi\psi\inv(b),\phi(x_0)) = d_Y(\psi\inv(b),x_0) = d_{Z_\psi}(b,\psi(x_0))\]
agree by supercoherence. In either case, these estimates and the fact that $\pi_{Z\phi}(b)$ and $\pi_{Z_\phi}(\phi\psi\inv(b))$ both lie within controlled distance of a $\cc(Z_\phi)$ geodesic from $x_\phi$ to $\phi(x_0)$ now imply the bound
\begin{equation}
\label{eqn:Z-estimate}
d_{Z_\phi}(\phi\psi\inv(b), b) \le 3d_Y(x_\psi,x_\phi) + 2\good + 16\delta + 5\lipconst < 3 d_Y(x_\phi,x_\psi) + 3\good.\qedhere
\end{equation}
\end{proof}

We continue with the proof of (\ref{eq:boundeddistance}).
Now let $A_0 = A' / A$ and $B_0 = B'/B$. Since all maps in $\mathcal{F}$ send $A$ to $B$ and $A'$ to $B'$, it follows that $\theta(A_0) = B_0$ for all $\theta\in \mathcal{F}$. After partitioning  into boundedly many subfamilies, we additionally assume, for all $\theta,\theta'\in \mathcal{F}$, that $\theta'\theta\inv$ preserves each component and boundary component of $B$, $B'$, and $B_0$. Note that if  $A\disj Y$ then $A'$ is just the union of $A$ and $Y$ so that $A_0= Y$ and $B_0 = \theta(Y)= Z_\theta$ for each $\theta\in \mathcal{F}$.

Since we already know all $\phi,\psi$ agree on $A$, to prove agreement on $A'$ it suffices to establish agreement on $A_0$ and on the boundary components of $A$ that are essential in $A'$. The next claim essentially provides agreement on these boundary components

\begin{claim}
\label{claim:initials_agree_in_boundary_components}
$d_\beta(\partial Z_\phi, \partial Z_\psi) \ladd_\good 0$ for each boundary component $\beta$ of $B$.
\end{claim}
\begin{proof}
Let $U$ be the annulus with core $\partial U =\beta$.  By assumption $\phi\psi\inv(\beta) =\beta$ and $\phi\psi\inv(\partial Z_\psi) = \partial Z_\phi$.  
If $\partial Z_\phi$ is disjoint from $\beta$ there is nothing to prove, so we assume $\partial Z_\phi \cut \beta$, in which case $\partial Z_\psi \cut \beta$ as well. Thus $Z_\phi \cut U$ and $Z_\psi \cut U$.

If $U \esub B$ (that is, if $B$ has an annular component isotopic to $U$) then agreement on $A$ (coming from $A$--coherence) immediately implies the claim. So we suppose $U$ is not a component of $B$. We claim that $d_U(b, \partial Z_\phi)\ladd_\good 0$ and, symmetrically, $d_U(b, \partial Z_\psi)\ladd_\good 0$. The claim will then follow from the triangle inequality.

By means of contradiction, suppose $d_U(b, \partial Z_\phi) > f^{\btbound}(\good) + \rafi$. 
Since $Z_\phi \in D(\phi)$ satisfies $d_{Z_\phi}(b,y_\phi) \ge \rafi$, Corollary~\ref{cor:bgit_for_teich} implies that $d_U(b,\partial Z_\phi) \le d_U(b,y_\phi) + \rafi/3$. Therefore $d_U(b,y_\phi) > f^{\btbound}(\good)$ and consequently $U\in D(\phi)$ by Lemma~\ref{lem:which_domains_backtrack}. Since $U\not\esub B = \phi(A)$, it follows that $U\in D_A(\phi)$. Now the fact that $Z_\phi$ is initial in $D_A(\phi)$ forces $Z_\phi \lhd_1 U$ in $D(\phi)$, meaning that $Z_\phi$ is time ordered before $U$ along $[b, y_\phi]$. But this implies $d_U(b, \partial Z_\phi) \le \rafi/3$ contradicting our above assumption.
\end{proof}

Now consider an arbitrary domain $W\esub B_0$. Since $B'$ is filled by $B$ and $Z_\phi$, and $W$ is disjoint from $B$, it cannot be that $Z_\phi$ and $W$ are disjoint. Further, $Z_\phi \esub W$ occurs only in the case $W = Z_\phi$, which has been dealt with in Claim~\ref{claim:bound_for_case_of_initial_domain} above. Thus we may assume $W\esubn Z_\phi$ or $W\cut Z_\phi$. Let us deal with these two possibilities separately.

\begin{claim}
\label{lem:claim_subset_of_B_0_nested_in_Z}
If $W\esubn Z_\phi$ (and hence $W'\esubn Z_\psi$), then $d_W(\phi\psi\inv(b),b)\ladd_\good 0$.
\end{claim}
\begin{proof}
First suppose $W\in D(\phi)$. Since $W\esub B_0$, we have $W\not\esub B = \phi(A)$ and consequently $W\in D_A(\phi)$ as well. Since $Z_\phi$ is initial in $D_A(\phi)$, it cannot be that $W\lhd_2 Z_\phi$. Thus by definition of $\lhd_2$ it must be that
\[d_{Z_\phi}(b, \partial W)\ge C_{Z_\phi} =  6 d_{Y_k}(b,y_\phi) + 40\good = 6 R^* + 40\good.\]
Claim~\ref{claim:bound_for_case_of_initial_domain} also gives
$d_{Z_\phi}(\phi\psi\inv(b),b) \le 6R^* + 39\good$.
Combining these yields $d_{Z_\phi}(b, \partial W) > d_{Z_\phi}(\phi\psi\inv(b),b) + \rafi$.
The BGIT (Corollary~\ref{cor:bgit_for_teich}) therefore implies the bound  $d_W(b,\phi\psi\inv(b))< \rafi \ladd 0$. If $W'\in D(\psi)$ the same reasoning bounds $d_{W'}(\psi\phi\inv(b),b)$.

It remains to suppose $W\notin D(\phi)$ and $W'\notin D(\psi)$ which  by Lemma~\ref{lem:which_domains_backtrack}, implies
\[d_V(\phi\inv(b),x_\phi) = d_{W}(b,y_\phi)\le f^{\btbound}(\good),\quad d_V(\psi\inv(b),x_\psi) = d_{W'}(b,y_\psi)\le f^{\btbound}(\good).\] 
If $d_V(x_\phi,x_\phi) < \rafi$, then the triangle inequality gives the desired conclusion
\[d_W(\phi\psi\inv(b),b) = d_V(\psi\inv(b),\phi\inv(b)) < 2 f^\btbound(\good)+\rafi \ladd_\good 0.\]
So we may suppose $d_V(x_\phi,x_\psi)\ge \rafi$. Then Corollary~\ref{cor:bgit_for_teich} with Claim \ref{claim:fixed_pts_close_in_index} gives
\[d_Y(x_\phi,\partial V) + d_Y(\partial V, x_\psi) \le d_Y(x_\phi,x_\psi) + \tfrac{\rafi}{3} \le 2R^* + 13\good.\]
By the fact $Z_\phi\in D(\phi)$ and the definition of jumping, we also have that
\[d_Y(\phi\inv(b), x_\phi) = d_{Z_\phi}(b, y_\phi) \ge f(\bt^\phi_{Y_k}) = f(R^*) =  8R^* + 53\good.\]
Hence the triangle inequality gives
\[d_Y(\phi\inv(b),\partial V) \ge d_Y(\phi\inv(b), x_\phi) - d_Y(x_\phi, \partial V) \ge 6R^* + 40\good.\]
On the other hand (\ref{eqn:options_for_what_to_bound}) and Claim~\ref{claim:bound_for_case_of_initial_domain} imply
$d_Y(\psi\inv(b),\phi\inv(b)) \le 6R^*+39\good$.
Putting these together, we see that
\[d_Y(\phi\inv(b),\partial V)  > d_Y(\phi\inv(b),\psi\inv(b)) + \rafi\]
which, by Corollary~\ref{cor:bgit_for_teich} forces $d_V(\phi\inv(b),\psi\inv(b)) < \rafi$, as needed.
\end{proof}

\begin{claim}
\label{claim:case_subset_of_B_0_cutting_Z}
If $W\cut Z_\phi$ (and hence $W'\cut Z_\psi$) then $d_W(\phi\psi\inv(b),b)\ladd_\good 0$.
\end{claim}
\begin{proof}
Since $d_{Z_\phi}(b,y_\phi)\ge \rafi$, Corollary~\ref{cor:bgit_for_teich} implies $d_W(y_\phi, \partial Z_\phi) \le d_W(y_\phi,b)+\tfrac{\rafi}{3}$. Hence if the quantity $d_{W}(y_\phi, \partial Z_\phi) = d_{V}(x_\phi, \partial Y)$ is larger than $\frac{3\rafi}{2}$, we find that $d_W(y_\phi, b) \ge \tfrac{3\rafi}{2}-\tfrac{\rafi}{3}>\rafi$. This ensures $W$ has an active interval and is time-ordered before $Z_\phi$ along $[y_\phi,b]$. In this case $d_{Z_\phi}(y_\phi,\partial W) \le \tfrac{\rafi}{3}$ and so
\[d_{Z_\phi}(b, \partial W) \ge d_{Z_\phi}(b,y_\phi) - d_{Z_\phi}(\partial W,y_\phi) \ge f(\bt^\phi_{Y_k}) -\tfrac{\rafi}{3}\]
by the fact that $Z_\phi$ satisfies the jumping criterion to be in $D(\phi)$. 
Recalling the notation $R^* =  \bt^\phi_{Y_k}$, Claim~\ref{claim:bound_for_case_of_initial_domain} and the definition $f(t) = 8t+53\good$ now give
\begin{align*}
d_{Z_\phi}(\phi\psi\inv(b),b) &\le 6R^* + 39\good 
< f(\bt^\phi_{Y_k}) - 2\rafi.
\end{align*}
Applying $\phi\inv$ to the left side of the previous two displayed inequalities then gives
\[d_Y(\phi\inv(b),\partial V) \ge d_Y(\psi\inv(b),\phi\inv(b)) + \rafi.\]
By Corollary~\ref{cor:bgit_for_teich}, it follows that we must have $d_V(\phi\inv(b),\psi\inv(b))<\rafi\ladd_\good 0$.
This proves the claim in the case that $d_W(y_\phi, \partial Z_\phi) = d_V(x_\phi, \partial Y) \ge \frac{3\rafi}{2}$. The same reasoning applies when $d_{W'}(y_\psi, \partial Z_\psi) = d_V(x_\psi, \partial Y) \ge \frac{3\rafi}{2}$.

Now suppose $W\in D(\phi)$. Since $W$ is disjoint from $B = \phi(A)$, we necessarily have $W\in D_A(\phi)$. The fact that $Z_\phi$ is initial in $D_A(\phi)$ thus implies $Z_\phi \lhd_1 W$ in $D(\phi)$. That is, $W$ is time-ordered before $Z_\phi$ along $[y_\phi,b]$ and so $d_W(y_\phi,\partial Z) \ge d_W(y_\phi,b) - d_W(\partial Z, b)  \ge 2\good - \rafi/3 >\tfrac{3\rafi}{2}$ and we are done by the previous paragraph. The same conclusion holds if $W'\in D(\psi)$.

It remains to suppose $d_V(x_\phi,\partial Y), d_V(x_\psi,\partial Y) < \frac{3\rafi}{2}$ and  $W\notin D(\phi)$, $W'\notin D(\psi)$. In particular $d_V(x_\phi,x_\psi)\leq d_V(x_\phi,\partial Y)+d_V(x_\psi,\partial Y) <3\rafi$. Furthermore,
\[d_V(x_\phi,\phi\inv(b)) = d_W(y_\phi, b)\quad\text{and}\quad d_V(x_\psi, \psi\inv(b)) = d_{W'}(y_\psi, b)\]
are both bounded by $f^{\btbound}(\good)$. Hence by the triangle inequality, $d_V(\phi\inv(b),\psi\inv(b))$ is bounded by $2f^{\btbound}(\good)$ plus $d_V(x_\phi,x_\psi)$ and therefore by $2f^{\btbound}(\good)+3\rafi\ladd_\good 0$.
\end{proof}

With these claims in hand, we may now complete the proof of the proposition. When $Y$ is empty and $B' = S$, Claim~\ref{claim:supercoherent_when_Y_empty} shows $d_W(\phi\psi\inv(b),b))\ladd_\good 0$ for all $\phi,\psi\in \mathcal{F}$. Thus by Lemma~\ref{lem:bounding_submaps} we may partition into boundedly many subcollections to achieve agreement on $A' = S$. When $Y$ is nonempty, we as above set $A_0 = A'/A$ and $B_0 = B'\setminus B$  and partition so that all maps send $A$ to $B$ and $A_0$ to $B_0$ inducing the same bijection of boundary components. By assumption we know all maps agree on $A$. Claims~\ref{claim:bound_for_case_of_initial_domain}, \ref{lem:claim_subset_of_B_0_nested_in_Z} and \ref{claim:case_subset_of_B_0_cutting_Z} together with Lemma~\ref{lem:bounding_submaps} imply that up to partitioning into boundedly many subfamilies we may assume all maps $A_0\to B_0$ agree as well. If $Y$ is disjoint from $A$, then $A'$ is the disjoint union of $A$ and $A_0$ and agreement on $A'$ definitionally follows from agreement on $A$ and $A_0$. Otherwise $Y\cut A$ and $B'$ is obtained by gluing $B$ and $B_0$ along certain boundary components $\beta$ of $\partial B$ that are essential in $B'$. All of our maps $\phi\psi\inv\colon B'\to B'$ preserve these curves and are the identity on the complement $B' \setminus \partial B = B\sqcup B_0$. These compositions $\phi\psi\inv$ thus lie in the kernel of  $\Mod(B')\to \Mod(B'\setminus \partial B) = \Mod(B)\times \Mod(B_0)$ and therefore consist of Dehn twists about these curves $\beta$. Finally, Claim~\ref{claim:initials_agree_in_boundary_components} $d_\beta(\phi\psi\inv(\partial Z_\psi), \partial Z_\psi) = d_\beta(\partial Z_\phi, \partial Z_\psi) \ladd_\good 0$ shows that only boundedly many Dehn twists about $\beta$ arise among these compositions. Therefore pairs $\phi,\psi\in \mathcal{F}$ produce only boundedly many maps $\phi\psi\inv\colon B'\to B'$ and we may again partition so that all pairs agree as maps $A'\to B'$.
\end{proof}

We can now prove the main theorem of this section
\begin{proof}[Proof of Theorem~\ref{thm:bound_multiplicity}] 
 Given a pair points $a,b$, we consider the family $\mathcal{F}$ consisting of those  $\phi\in \Mod(S)$ for which $a_\phi=a$ and $b_\phi=b$ with $d_{\T(S)}(x_0,\phi(x_0))\le r$. 
This is an $(a,b)$--branch family with displacement $r(\mathcal{F}_0)\le r$. Hence we may apply Proposition~\ref{prop:annular-compliance} to partition $\mathcal{F}$ into $q_0(r)$ subfamilies $\mathcal{F}_0$ that are each annular-compliant.
Let $A_0=\emptyset$. Then each of these families $\mathcal{F}_0$ is trivially $A_0$--coherent.

By induction, suppose  we are given a subsurface $A_k$ and subcollection $\mathcal{F}_k\subset \mathcal{F}$ that is $A_k$--coherent. If $A_k = S$, then all elements of $\mathcal{F}_k$ agree on $S$. Otherwise we apply Lemma~\ref{lem:supercoherent} and subdivide $\mathcal{F}_k$ into at most $q_2(r(\mathcal{F}_k)) \le q_2(r)$ subfamilies $\mathcal{F}'_k\subset \mathcal{F}_k$ that are each $(A_k,Y_k')$--supercoherent for some domain $Y'_k$. 
Now apply Proposition~\ref{prop:fix_on_domain} to further partition each $\mathcal{F}_k'$ into $q_3$ subfamilies $\mathcal{F}_{k+1}$ that are each $A_{k+1}$--coherent, where $A_{k+1}$ is $S$ when $Y_k$ is empty and is otherwise the subsurface filled by $A_k$ and $Y_k$.

Since each iteration $\mathcal{F}_k\rightsquigarrow \mathcal{F}_{k+1}$ yields coherence on a strictly larger subsurface $A_{k+1} \esupn A_k$, any chain $\mathcal{F}_0 \supset \mathcal{F}_1 \supset \dots$ produced in this way must terminate at  $A_k = S$ within $k\le \plex{S}$ steps. Since there are at most $q_0(r)$ possibilities for $\mathcal{F}_0$ and at each step each $\mathcal{F}_k$ is partitioned into at most $q_3q_2(r)$ subfamilies, this iterative procedure ultimately produces at most $q_0(r)(q_3 q_2(r))^{\plex{S}}$ subfamilies $\mathcal{F}'$ that are each $S$--coherent. By definition of coherence, all elements $\phi,\psi\in \mathcal{F}'$ agree on $S$ and thus in fact are equal. Hence in total there are at most $q_0(r)(q_3 q_2(r))^{\plex{S}}$ elements $\phi$ in the original collection $\mathcal{F}$
\end{proof}

\section{The upper bounds}
\label{sec:finish}
\label{sec:upper_bound}

With all the setup in place, it is now fairly straightforward to prove the upper bounds in our main theorem.
Fix any $\delta > 0$ and choose the parameter $\wfal$ sufficiently large so that $\wfal\delta > 6$.
Fix a basepoint $x_0\in \T(S)$. For each element  $\phi \in \Mod(S)$, we let $a_\phi,b_\phi$ be the branch points from Proposition~\ref{prop:good-point-features}, so that $(x_0,a_\phi,b_\phi,\phi(x_0))$ is strongly $\good$--aligned.
Recall that $S_p$ denotes the pseudo-Anosov part of the Nielsen--Thurston decomposition of $S$ associated to $\phi$ (Definition~\ref{def:nielsen-thurston-decomp}) We need:

\begin{lemma}
\label{lem:technical_savings_bound}
$\cl(a_\phi,b_\phi) \ladd_\wfal \mathfrak{S}(x_0,a_\phi,b_\phi,\phi(x_0)) + (\ent{S_p}+\frac{3}{\wfal})d_{\T(S)}(x_0, \phi(x_0))$.
\end{lemma}
\begin{proof}
The savings $\mathfrak{S}(x_0,a_\phi,b_\phi,\phi(x_0))$ is defined as an infimum over WISCL witness families for the tuple, say realized by $\Omega = (\Omega_1,\Omega_2,\Omega_3)$. Since $\Omega_2$ is a witness family for $[a_\phi,b_\phi]$, necessarily $\cl(a_\phi,b_\phi)\le \cl(\Omega_2)$. 

Let us analyze the term $\cl(\Omega_2)$. 
By taking $\wfal > \good$, it follows (from \eqref{eqn:indexed_thresholds} and Definition~\ref{def:upsilons}) that $\Omega_2$ consists solely of domains $V$ for which $d_V(a_\phi,b_\phi)\ge \good$ and of annuli $A$ for which $\ell_{a_\phi}(\partial A)$ or $\ell_{b_\phi}(\partial A)$ is less than $\thin'$. However, Proposition~\ref{prop:good-point-features}(\ref{gp-split-short-curves}) says $d_V(a_\phi,b_\phi)> \good$ can only occur if $V\esub S_p$ or if $V$ is an annulus with $\ell_{a_\phi}(\partial V),\ell_{b_\phi}(\partial V) \ge \thin$. Thus we have a partition $\Omega_2 = \Omega_2^p \sqcup \Omega_2^a \sqcup \Omega_2^\ell$, where:
\begin{itemize}
\item  $\Omega_2^p\subset \Omega_2$ is the subcollection consisting of domains nested in $S_p$,
\item $\Omega_2^a\subset \Omega_2\cap \Upsilon^c(a_\phi,b_\phi)$ is the subcollection of annuli $A\not\esub S_p$ which satisfy $d_A(a_\phi,b_\phi)\ge\irafi_A\ge \good$ and consequently have $\ell_{a_\phi}(\partial A), \ell_{b_\phi}(\partial A) \ge \thin$.
\item $\Omega_2^\ell\subset \Omega_2\cap \Upsilon^\ell(a_\phi,b_\phi)$ is the subcollection of annuli $A\not\esub S_p$ which satisfy $d_A(a_\phi,b_\phi)< \irafi_A$ but for which $\min\{\ell_{a_\phi}(\partial A),\ell_{b_\phi}(\partial A)\}<\thin'$
\end{itemize}

Using the notion of restricted complexity (Definition~\ref{def:restricted_complexity}), we may now write
\[\cl(a_\phi,b_\phi) \le \cl(\Omega_2) = \cl(\Omega_2\vert \Omega_2^p) + \cl(\Omega_2\vert \Omega_2^a) + \cl(\Omega_2\vert \Omega_2^\ell).\]
Let us bound these terms separately. Firstly, $\Omega_2^\ell$ has cardinality at most $2\plex{S}$, since there are at most $\plex{S}$ short curves at $a_\phi,b_\phi$. Furthermore, for every $A\in \Omega_2^\ell$, Proposition~\ref{prop:good-point-features}(\ref{gp-non-pA-annuli-balanced}) implies that the ratio of $\min\{\thin,\ell_{a_\phi}(\partial A)\}$ and $\min\{\thin,\ell_{b_\phi}(\partial A)\}$ is contained in $[\tfrac{1}{\good},\good]$. Since we also have $d_A(a_\phi,b_\phi) \le \irafi_A = \irafi$, it follows by construction (Definition~\ref{def:resolution}) that the resolution points in $\T(A) = \H^2$ have $x$--coordinates that differ by a bounded amount and $y$--coordinates that differ by a bounded ratio; hence $d_{\T(A)}(\res{a_\phi}{\Omega_2}{A},\res{b_\phi}{\Omega_2}{A})$ is uniformly bounded for each $A\in \Omega_2^\ell$. Consequently, the final term above satisfies
\[\cl(\Omega_2\vert \Omega_2^\ell) = \sum_{A\in \Omega_2^\ell} \entstar{A} d_{\T(A)}(\res{a_\phi}{\Omega_2}{A},\res{b_\phi}{\Omega_2}{A}) \ladd_C 0.\]

For the $\Omega_2^a$ term, since each $A\in \Omega_2^a$ satisfies $\ell_{a_\phi}(\partial A), \ell_{b_\phi}(\partial A) \ge \thin$, the construction of resolution points (Definition~\ref{def:resolution}) implies $\res{a_\phi}{\Omega_2}{A}$ and $\res{b_\phi}{\Omega_2}{A}$ are also $\thin$--thick. Consequently, for each $A\in \Omega^a_2$ we use $\entstar{A} = 1$ when computing complexity length $\cl(\Omega_2)$ (Definition~\ref{def:complexity}) and correspondingly $1 = \ent{A}-\entstar{A}$ when calculating the savings $\mathfrak{S}(\Omega_2)$ (Definition~\ref{def:savings_in_complexity}).
Since $\entstar{A} = 1 = 2-1 = \ent{A}-\entstar{A}$, we therefore see that
\begin{align*}
\cl(\Omega_2\vert\Omega_2^a) 
&= \sum_{A\in \Omega_2^a} \entstar{A} d_{\T(A)}(\res{a_\phi}{\Omega_2}{A}, \res{b_\phi}{\Omega_2}{A})
= \sum_{A\in \Omega_2^a} (\ent{A} - \entstar{A}) d_{\T(A)}(\res{a_\phi}{\Omega_2}{A}, \res{b_\phi}{\Omega_2}{A})\\
&\le \sum_{i=1}^3 \sum_{V\in \Omega_i} (\ent{V} - \entstar{V}) d_{\T(V)}(\res{a_\phi}{\Omega_i}{V}, \res{b_\phi}{\Omega_i}{V})
= \mathfrak{S}(\Omega) = \mathfrak{S}(x_0,a_\phi, b_\phi, \phi(x_0)).
\end{align*}

Finally, since all the domains in $\Omega_2^p$ are nested in $S_p$ by definition, setting $\Omega' = (\emptyset,\Omega_2^p,\emptyset)$ in the ``Moreover\dots'' conclusion of Theorem~\ref{thm:complexity_bound_for_WISCAL} gives
\[\cl(\Omega_2 \vert\Omega_2^p) \le \cl(\Omega\vert\Omega') 
\le \cl(\Omega\vert\Omega') + \mathfrak{S}(\Omega\vert\Omega')
\ladd_\wfal \left(\ent{S_p} + \tfrac{3}{\wfal}\right) d_{\T(S)}(x_0, \phi(x_0)).\]
Combining these inequalities gives the claimed bound.
\end{proof}
\begin{corollary} There exists a constant $\wfal' > 0$ depending only on $\wfal$ so that
\label{cor:for_the_upper_bound}
\[\cl(x_0,a_\phi)+2\cl(a_\phi,b_\phi)+\cl(b_\phi,\phi(x_0))
\le
\left(\ent{S} + \ent{S_p} + \tfrac{6}{\wfal}\right) d_{\T(S)}(x_0, \phi(x_0)) + \wfal'.\]
\end{corollary}
\begin{proof}
By Theorem~\ref{thm:main_complexity_length_bound}, the left hand side is bounded by $\cl(x_0,a_\phi,b_\phi,\phi(x_0)) + \cl(a_\phi,b_\phi)$. Using Lemma~\ref{lem:technical_savings_bound} to bound the $\cl(a_\phi,b_\phi)$ term, we see that there is a constant $\wfal''$ depending only on $\wfal$ so that the left hand side is bounded by
\[\cl(x_0,a_\phi,b_\phi,\phi(x_0)) + \mathfrak{S}(x_0,a_\phi,b_\phi,\phi(x_0)) + \left(\ent{S_p} + \tfrac{3}{\wfal}\right) d_{\T(S)}(x_0, \phi(x_0))+\wfal''.\]
Applying Theorem~\ref{thm:main_complexity_length_bound} to these first two summands now gives the result.
\end{proof}

Using Corollary~\ref{cor:for_the_upper_bound}, we now complete the proof of the general upper bound:

\begin{theorem}
\label{thm:general_upper_bound}
Fix $0\le m \le \ent{S}$ and let $\Gamma\subset \Mod(S)$ be any subset of elements $\phi$ for which $\ent{\phi}\le m$. Then for any $\delta > 0$ and $x,y\in \T(S)$ the following holds for all large $R$:
\[\abs{\big\{\phi\in \Gamma \mid d_{\T(S)}(\phi(x),y)\le R\big\}} 
\le \exp\left(\left(\tfrac{\ent{S}+m}{2} + \delta\right)R\right)\]
\end{theorem}
\begin{proof}
By Observation~\ref{obs:can_use_any_points} we may assume $x=y=x_0$. For each $\phi\in \Gamma$ we may, by symmetry,
suppose that $\cl(x_0,a_\phi) \le \cl(b_\phi,\phi(x_0))$. Indeed, otherwise, we may replace $\phi$ with $\phi\inv$ and observe that the good points for $\phi\inv$ (c.f.~Proposition~\ref{prop:good-point-features}) may be chosen so that $x_{\phi\inv} = \phi\inv(y_\phi)$, $y_{\phi\inv} = \phi\inv(x_\phi)$, $a_{\phi\inv} = \phi\inv(b_\phi)$, and $b_{\phi\inv} = \phi\inv(a_\phi)$. Thus $(x_0, a_{\phi\inv}, b_{\phi\inv}, \phi\inv(x_0))$ is $\good$--strongly aligned with $\cl(x_0,a_{\phi\inv}) = \cl(\phi(x_0),b_\phi) < \cl(x_0,a_\phi) = \cl(b_{\phi\inv},\phi\inv(x_0))$, and so we proceed in the same way counting $\phi\inv$. (That is, we divide $\Gamma$ according to whether the assumption holds for $\phi$ or $\phi\inv$ and count these subsets separately).
With this assumption, it follows that $\cl(x_0,a_\phi) + \cl(a_\phi,b_\phi)$ is at most half the quantity from Corollary~\ref{cor:for_the_upper_bound}. Therefore, since $6 < \delta\wfal$ and $\ent{\phi}=\ent{S_p}\le m$, we have
\[\cl(x_0,a_\phi) + \cl(a_\phi,b_\phi) \le \frac{\left(\ent{S} +m+\delta\right)}{2}d_{\T(S)}(x_0,\phi(x_0)) +\wfal'.\]

Applying Corollary~\ref{cor:count-complexity-tuples}, we obtain a constant $k$ such that the number of such pairs $(a_\phi,b_\phi)$ is at most 
\[k \left(\tfrac{(\ent{S}+m+\delta)}{2}R + \wfal'\right)^k \exp\left({\tfrac{(\ent{S}+m+\delta)}{2}R} + \wfal'\right) \le k' R^k \exp\left({\tfrac{\left(\ent{S}+m+\delta\right)}{2}R}\right),\]
for some larger constant $k' > k$ depending only on $\wfal$, $\ent{S}$, and $\delta$.
Lastly, Theorem~\ref{thm:bound_multiplicity} provides a polynomial $p$ such that each such pair $(a,b)$ arises as $(a_\phi,b_\phi)$ for at most $p(R)$ elements with $d_{\T(S)}(x_0,\phi(x_0))\le R$. Hence the total number of such $\phi$ is at most $p(R)$ times the above, and we finally conclude that for all sufficiently large $R$ our set of elements has cardinality at most
\[
p(R) k' R^k \exp\left({\tfrac{\left(\ent{S}+m+\delta\right)}{2}R}\right) \le \exp\left({\left(\tfrac{\ent{S}+m}{2} + \delta\right)R}\right).\qedhere\]

\end{proof}
\bibliographystyle{alphanum}
\bibliography{counting}

\end{document}